\documentclass[reqno, oneside]{amsbook}
\usepackage{amssymb, color}
\usepackage{amsmath}
\usepackage{upgreek}
\usepackage{url}

\usepackage{etoolbox}

\usepackage[colorlinks=true, linkcolor=blue, citecolor=blue, urlcolor=blue]{hyperref} 

\AtBeginEnvironment{thebibliography}{\interlinepenalty=10000}

\renewcommand{\thesection}{\thechapter.\arabic{section}}

\makeatletter
\@addtoreset{equation}{section}

\makeatother

\makeatletter
\newcommand{\startchaptertight}[1]{
  \begingroup
  \let\clearpage\relax
  \let\cleardoublepage\relax
  \chapter{#1}
  \endgroup
}

\makeatother

\renewcommand\thefigure{\thesection.\@arabic\c@figure}
\renewcommand\thetable{\thesection.\@arabic\c@table}

\newtheorem{theorem}{Theorem}[section]

\newtheorem{lemma}[theorem]{Lemma}
\newtheorem{proposition}[theorem]{Proposition}
\newtheorem{corollary}[theorem]{Corollary}

\newtheorem{definition}[theorem]{Definition}
\newtheorem{remark}[theorem]{Remark}
\newtheorem{example}[theorem]{Example}
\newtheorem{exercise}[theorem]{Exercise}

\newcommand{\mc}[1]{{\mathcal #1}}
\newcommand{\mf}[1]{{\mathfrak #1}}

\newcommand{\bb}[1]{{\mathbb #1}}

\def\Z{\mathbb Z}
\def\R{\mathbb R}
\def\N{\mathbb N}

\def\L{\mathcal L}
\def\I{\mathcal I}
\def\B{\mathcal B}
\def\F{\mathcal F}

\def\M{\mathcal M}
\def\T{\bb T}

\def\E{\mathbb E}
\def\P{\mathbb P}
\def\S{{\Z^d}}
\def\g{{\mf g}}

\def\H{\mathcal{H}}
\def\f{\hat{f}}

\DeclareRobustCommand{\SkipTocEntry}[5]{}

\begin{document}

\pagenumbering{gobble}          
\frontmatter

\begin{titlepage}

\begin{center}
{\Huge \textbf{Notes on Hydrodynamic Limits and Related Topics}}\\[.9 cm]
\end{center}
\vskip .2cm

{\huge {Sunder Sethuraman}}
\vskip .1cm

{\Large University of Arizona}
\vskip 1.5cm

\addtocontents{toc}{\SkipTocEntry}
\vskip .7cm

\section*{}

In these lecture notes, we discuss various `hydrodynamic LLN' and `CLT' scaling limits, among others, in types of stochastic interacting particle systems, connecting `microscopic' behaviors to continuum laws. 
Via `short stories', 
the aim is to present some of the `basics' for students and those entering the field, as a complement to books such as \cite{KL}, \cite{Ko_La_Ol}, \cite{Liggett1}, \cite{Liggett2}.  
 To be concrete, attention is restricted to a few `mass conservative' systems on discrete spaces, namely exclusion and zero-range processes, that have proved robust in the study of different phenomena.

After preliminaries, we discuss the `entropy' and `relative entropy' methods to prove hydrodynamic limits of the bulk mass in finite volume, as well as other items such as construction of systems in infinite volume and the structure of their invariant measures, and scaling limits of local functionals, such as occupation times of sites and the motion of a tagged particle.  In the last part, we also discuss equilibrium fluctuations of the bulk mass when the process starts from an invariant measure.  

While such a list forms a brief introduction, there are of course several results not covered.
For instance, those with respect to large deviations, mixing times, metastability and condensation phenomena, non-equilibrium fluctuations, the `KPZ' equation and `KPZ' fixed point, and integrable probability, though referenced, are not much discussed.

\vskip 1.5cm

\noindent {\it Acknowledgements.}  These notes were begun for a topics course given in Spring 2012 at the University of Arizona, with input from research papers and texts, and revised during a 2026 sabbatical, supported in part by the Simons Foundation. They have also served as background for recent minicourses in Crete 2023 (Stochastic Methods in Finance and Physics IV), Rio 2024 (IMPA), and Hermosillo 2025 (XV Symposium on Probability and Stochastic Processes).
Many thanks to the  
participants
in these venues for their comments and interest.

\end{titlepage}
\frontmatter

\pagenumbering{gobble}          
\tableofcontents
\clearpage
\pagenumbering{arabic} 

\mainmatter

\chapter[Section $1$]{Preliminaries and hydrodynamics of independent random walks}
\label{lec1}

Before discussing a basic example, illustrating possibilities in the study of `hydrodynamics of stochastic particle systems', we recall some basic notions in Markov chains to set the stage.

\section{Markov chains}

\subsection{Construction}
A family of random variables $\{X_n: n\geq 0\}$ taking values on a countable state space $\Omega$ is called a `discrete time Markov chain' if the `stationary Markov property' is satisfied:
\begin{eqnarray*}
P(X_n = x_n| X_0 = x_0, \ldots, X_m=x_m) & = & P(X_n = x_n| X_m=x_m) \\
& = & P(X_{n-m}=x_n|X_0=x_m)\end{eqnarray*}
for all $x_0,\ldots, x_m\in \Omega$ and $n>m\geq 0$.  When $n=1$ and $m=0$, the last quantity on the right-side represents a `transition probability' of the Markov chain, denoted $p(x,y) = P(X_1=y|X_0=x)$.  The $n$-step probability $p^{(n)}(x,y) = P(X_n=y|X_0=x)$ satisfies a recurrence, $p^{(n+1)}(x,y)= \sum_{z\in \Omega} p^{(n)}(x,z)p(z,y)$.

We now construct a `continuous time Markov chain' on $\Omega$ with `skeleton' $\{X_n: n\geq 0\}$ and a transition probability vanishing on the diagonal, that is $p(x,x)=0$ for all $x\in \Omega$.  Let $\{\lambda_x: x\in \Omega\}$ be a collection of positive numbers, and let $\{W_n: n\geq 0\}$ be a collection of independent identically distributed exponential random variables with rate $1$, independent of the skeleton discrete time chain.  

Define the process $\{Z_t: t\geq 0\}$ as follows:  Initially, $Z_0 = X_0\in \Omega$.  After time $\lambda_{X_0}^{-1}W_0$, the process jumps to value $X_1$, and after a subsequent time $\lambda_{X_1}^{-1} W_1$, the process jumps to value $X_2$, and so on.  Let $T_k = \sum_{i=0}^k \lambda_{X_i}^{-1}W_i$ for $k\geq 0$.  Then,
$$
Z_t  \ = \ \left\{\begin{array}{rl}
x& \ {\rm for \ } 0\leq t< T_0\\
X_1& \ {\rm for  \ } T_0\leq t< T_1\\
\vdots& \ \ \vdots\\
X_n& \ {\rm for \ } T_{n-1}\leq t< T_n\\
\vdots& \ \ \vdots
\end{array}\right.
$$
where $0\leq t< T_\infty = \lim_{n\uparrow\infty}T_n$.
Sufficient conditions for $T_\infty = \infty$ include the cases that the state space $\Omega$ is finite, or that $\sup_{x\in \Omega}\lambda_x<\infty$.  We will assume from now on that the process is `regular', that is $T_\infty = \infty$, so that the process $Z_t$ is defined for all time $t\geq 0$.  

One can show that the `continuous time' chain $\{Z_t: t\geq 0\}$ satisfies the stationary Markov property, which in this context is equivalent to
\begin{eqnarray}
\label{ct_MP}
P(Z_t = y| Z_{t_0}=x_0, \ldots, Z_{t_m} = x_m, Z_s = x) & = & P(Z_t=y| Z_s=x) \\
& = & P(Z_{t-s}=y|Z_0 = x)\nonumber
\end{eqnarray}
for all $x,y,x_0,\ldots, x_m\in \Omega$, $0\leq t_0< \cdots t_m < s<t$ and $m\geq 0$.

Conversely, given a process $\{Z_t: t\geq 0\}$ satisfying \eqref{ct_MP} and also the `jump property' that there exists a sequence of strictly increasing stopping times $\{T_n: n\geq 0\}$ such that $T_0>0 = T_{-1}$ and $Z_t$ is constant on intervals $[T_n,T_{n+1})$ and $Z_{T_n^-} \neq Z_{T_n}$ for $n\geq 0$, one can determine a unique skeleton discrete time Markov chain $\{X_n: n\geq 0\}$ and positive jump parameters $\{\lambda_x: x\in \Omega\}$ such that $X_n = Z_{T_n}$, $p(x,y) = P(Z_{T_{n+1}}=y|Z_{T_n}=x)$, and $T_{n}-T_{n-1}\big |\{X_n\}_{n\geq 0}$ are independent exponentials with rates $\lambda_{X_n}$ for $n\geq 0$.  Chains with the same skeleton and jump parameters have the same joint distributions.  The condition $Z_{T_n^-}\neq Z_{T_n}$ ensures $p$ vanishes on the diagonal.

\subsection{Generators and Kolmogorov equations}
In the discrete time setting, we can define the transition operator $P = \big(p(x,y): x,y\in \Omega\big)$, which is a matrix when $\Omega$ is finite.  Then,
the $n$th powers give the $n$th step probabilities, $P^n(x,y) = P(X_n=y|X_0=x)$.
  
Computing the $t$-time probabilities for the continuous time Markov chain $Z_t$ is more complicated.  Define the transition probability
$$P_t(x,y) \ = \ P(Z_t = y| Z_0 = x).$$
Then, by the Markov property we have
$$P_{t+s}(x,y) \ = \ \sum_{z\in \Omega} P_t(x,z)P_s(z,y)$$
or in terms of operators $P_{t+s} = P_tP_s$, the semigroup property.

Given regularity of the process, the transition functions are differentiable in time, and satisfy the `first jump' relation
\begin{align*}
P_t(x,y) &= P(Z_t=y, T_0>t|Z_0=x) + P(Z_t=y, T_0\leq t|Z_0=x)\\
&= 1(x=y)e^{-\lambda_x t} + \int_0^t \lambda_xe^{-\lambda_x s} \sum_{z\neq x}p(x,z)P_{t-s}(z,y),
\end{align*}
from which the backward equation is seen:
\begin{eqnarray*}\frac{d}{dt} P_t(x,y) &=& \sum_{z\in \Omega} \lambda_xp(x,z)[P_t(z,y) - P_t(x,y)]\\
P_0(x,y) &=& 1(x=y).
\end{eqnarray*} 
Similarly, from a decomposition of the jump time just before time $t$, one has the forward equation
$$\frac{d}{dt}P_t(x,y) \ = \ \sum_{z\in \Omega} [P_t(x,z)-P_t(x,y)]\lambda_zp(z,y). $$

\begin{exercise}\rm Review and perform these derivations.  
\end{exercise}

Define the operator
$$L(x,y) \ = \ \left\{\begin{array}{rl}
\lambda_x p(x,y) & \ {\rm for \ }y\neq x\\
-\lambda_x & \ {\rm for \ } y=x.
\end{array}\right.$$
Then,
neatly expressed, the backward and forward equations become
$$\frac{d}{dt}P_t \ = \ LP_t \ \ \ {\rm and \ \ \ } \frac{d}{dt}P_t \ = \ P_tL.$$
Also, $\lim_{t\downarrow 0} t^{-1}[P_t - I] = L$ and $P_t(x,y) = \delta_{x,y} + tL(x,y) + o(t)$ as $t\downarrow 0$.

When the space $\Omega$ is finite, $L$ is a `generator' matrix, that is $L(x,y)\geq 0$ for $x\neq y$ and $L(x,x) = -\sum_{y\neq x}L(x,y)$, and by solving the ODE's, one obtains $P_t = e^{tL}$ which can be computed in some cases.  More generally, the semigroup $P_t$ can be understood in terms of the Hille-Yosida formulation.

Let $f: \Omega\rightarrow \R$ be a bounded function on the state space.  In the discrete time case, define $Pf(x) = \sum_{y\in \Omega}p(x,y)f(y)$, which is the conditional expectation of $f(X_1)$ given $X_0=x$.  Then, $P^nf(x)= \sum_{y\in \Omega} p^{(n)}(x,y)f(y)$ is the conditional expectation of $f(X_n)$ given $X_0=x$, where $p^{(n)}(x,y)$ is the $n$-fold convolution, or $n$th step transition probability.

In the continuous case, define $P_tf(x) = \sum_{y\in \Omega} P_t(x,y)f(y)$ which is the conditional expectation of $f(Z_t)$ given $Z_0=x$.  
In this framework, the generator $L$ is often expressed in terms of its action on compactly supported $f$:
\begin{eqnarray*}
(Lf)(x) & = & \sum_{y\in \Omega} L(x,y) [f(y)-f(x)]\\
&=& \sum_{y\in \Omega}\lambda_xp(x,y)[f(y)-f(x)].
\end{eqnarray*}

\subsection{Invariant measures}
\label{sec:invariantmeasurefinite}
Let $\mu$ be a probability measure on $\Omega$.  In the discrete time situation, define $\mu P(x) = \sum_{y\in \Omega} \mu(y)p(y,x)$.  Hence, we see that $\mu P^n$ is the distribution at time $n$ when the initial state is distributed according to $\mu$.  
  
In the continuous-time  model, define 
$$\mu P_t(x) \ = \ \sum_{y\in \Omega} \mu(y)P_t(y,x), \ \ {\rm and } \ \  \mu L(x) \ = \ \sum_{y\in \Omega} \mu(y)L(y,x).$$

We say that $\mu$ is an `invariant measure' for discrete time chains if $\mu P = \mu$, and for continuous time chains if $\mu P_t = \mu$ for all $t\geq 0$.  We also say that $\mu$ is a `reversible' invariant measure in discrete time chains if $\mu(x)p(x,y) = \mu(y)p(y,x)$ for all $x,y\in \Omega$.  In continuous time chains, $\mu$ is `reversible' when $P_t$ is self-adjoint, that is $\sum_{x\in \Omega}\mu(x)f(x) P_t(x,y)g(y) = \sum_{x\in \Omega}\mu(x)g(x)P_t(x,y)f(y)$, or in terms of the inner product, $\langle f, P_tg\rangle_\mu = \langle P_t f, g\rangle_\mu$ for all compactly supported $f,g$.

One can verify that in the continuous time setting that $\mu$ being invariant is equivalent to $\mu L = 0$ or in terms of expectations $E_\mu\big[Lf\big]=0$ for all compactly supported $f$.  Also, $\mu$ being reversible is equivalent to $\mu(x) L(x,y) = \mu(y) L(y,x)$ for all $x,y\in \Omega$, or in terms of inner products $\langle f, Lg\rangle_\mu = \langle Lf, g\rangle_\mu$ for all compactly supported $f,g$.

In the finite state space case, invariant measures always exist, and if the skeleton chain can reach every state from any state in finite time, that is $p$ is irreducible, the invariant measure is unique.  However, in the countable state case there may be no invariant measures.

Reversibility has the following interesting implication.  Fix a time $t>0$, and consider $R_s = Z_{t-s}$ for $0\leq s\leq t$.  
Suppose that initially $Z_0$ is distributed according to an invariant measure $\mu$.  Then, it can be seen that $R_s$ is a continuous time Markov chain with 
transition probability $Q_s(x,y) = (\mu(y)/\mu(x))P_s(y,x)$.  In particular, when $\mu$ is reversible, $Q_s(x,y) = P_s(x,y)$ and in this case the `forward in time' and `backward in time' chains have the same distribution.
We remark that 
it may be easier to find directly a reversible measure and therefore an invariant measure by checking the reversibility conditions.

\subsection{Examples}

We will content ourselves for the moment with two basic continuous time examples, the two-state Markov chain, and random walk.

\begin{example}
\label{example1}\rm
Let $\Omega= \{0,1\}$ correspond to states `on' and `off', or sometimes `empty' and `occupied'.  Here, state $0$ can transition to state $1$ and vice versa.  Let $\lambda_0$ and $\lambda_1$ be the corresponding jump rates.  The skeleton chain probabilities are $p(0,1)=p(1,0)=1$.  The generator matrix is
$$L \ = \ \left[\begin{array}{cc}-\lambda_0& \lambda_0\\
\lambda_1& - \lambda_1\end{array}\right].$$
and correspondingly, it is an exercise in diagonalization to compute $P_t = e^{tL}$ and to find the unique invariant measure.

\end{example}

\begin{exercise}\rm  Compute the invariant measure and $P_t$.  Is it reversible?
\end{exercise}

\begin{example}\rm We define the Poisson process with parameter $\alpha$.  Let $\Omega=\N_0 := \{0,1,2,\ldots\}$, the whole numbers.  Let $\lambda_x \equiv \alpha$.  The skeleton chain is deterministic where transitions are to the nearest right site:  $p(x,x+1)=1$ for $x\geq 0$.  Then, with $Z_0=0$, the continuous time chain $Z_t$ counts the number steps made up to time $t\geq 0$.  The trajectories are step functions, which are right-continuous, with left limits.

\begin{exercise}  
\rm  Compute the backward/forward equation, and show that the transition probability is in `Poisson' form
$P_t(0,x)=P_t(y,y+x) = e^{-\alpha t} (\alpha t)^x/x!$.
\end{exercise}

\end{example}

\begin{example}\label{example2}
\rm
Let $\Omega = \T^d_N$ the $d$-dimensional torus where $\T_N = \Z/ N\Z$, or integers modulo $N$.  Let also $\lambda_x \equiv 1$, and $p$ be a finite-range, translation-invariant transition probability:  For all $x,y\in \Omega$, $p(x,y) = 0$ if $|x-y|\geq R$ some $R<\infty$, and $p(x,y) = p(0,y-x)=:p(y-x)$.  We will also assume that $p$ is irreducible.  For instance, the nearest-neighbor, symmetric case is one possibility.

Then, $L(x,y) = p(x,y)$ for $x\neq y$ and $L(x,x) = -\sum_{y\neq x}L(x,y)$.  In particular, the uniform distribution $\mu(x) \equiv N^{-1}$ is the unique invariant measure.  Moreover, $\mu$ is reversible exactly when $p(x)=p(-x)$ for all $x\in \Omega$.

Now let $R_t$ be the number of jumps before time $t$.  Since the jump rates are all $1$, we see that $R_t$ is a Poisson process with rate $1$.  In particular, the transition probability satisfies
\[P^N_t(y-x):= P_t(x,y) \ = \ E\big[p^{(R_t)}(x,y)\big] \ = \ \sum_{n\geq 0} \frac{e^{-t}t^n}{n!} p^{(n)}(x,y).\]

We now consider a sequence of chains $Z_t=Z^{(N)}_t$ on a sequence of torii $\T_N^d$ as $N\geq 1$ increases.
Define $m = \sum_{x\in \Omega} xp(x)$ to be the mean displacement of the position.  Then, it is not difficult to extablish the following law of large numbers,
$$\lim_{N\uparrow\infty} \frac{Z^{(N)}_{Nt}}{N} \ = \ mt \ \ \ {\rm in \ probability}.$$
Here, $m\in \T^d$ where $\T$ is the unit circle.

\begin{exercise}\rm
Show this LLN (law of large numbes) by say variance computations.  Noting $R_t$ is Poisson with variance $t$ is useful.  One can do it also by regeneration, and other methods.   \end{exercise}

Also, if $m=0$, let $C$ be the matrix of covariances, $C_{i,j} = \sum_{x}x_ix_jp(x)$ for $0\leq i,j\leq d$.  We have the central limit theorem,
$$  \frac{Z^{(N)}_{N^2 t}}{N} \ \Rightarrow \ {\rm N}(0,C t).$$
\begin{exercise}\rm
There are a few ways to show this CLT (central limit theorem).  One way is to compute the moment generating function or characteristic function, noting independence of $R^{(N)}_t$ and the displacements.  
\end{exercise}

\end{example}

Finally, when $m = \sum_x xp(x) \neq 0$, we say the walk is asymmetric, and when $m=0$ and $p(\cdot)$ is not symmetric, we say the walk is mean-zero asymmetric, and when $p(\cdot)$ is symmetric, we say the walk is symmetric.  

\section{Evolution of the mass of independent random walks}

We would like to understand the space-time evolution of the mass in a system of particles with a conservation law.  Perhaps the simplest model is that of unlabeled non-interacting random walks on a $d$-dimensional torus with $N$ locations.  When $N$ is large, and one looks at the system from afar, after long times, one can discern better the motion of the bulk of the mass rather than individual components.  The goal is to make precise the evolution of the mass in this scale in terms of a continuum equation.

We will be working with a Markov chain on $\Omega= \N_0^{\T^d_N}$, where we recall $\N_0 = \{0,1,2,\ldots\}$, which governs the motion of $K$ independent random walks on $\T^d_N$ specified in Example \ref{example2}.  Since we are interested in the `mass' of particles, we will consider the occupation numbers at each location on the lattice $\T^d_N$.  That is, let $Z^i_t$ be the position of the $i$th particle at time $t$.  Define
$$\eta_t(x) \ = \ \sum_{i=1}^K 1(Z^i_t = x).$$

We now observe that the process $\eta_t = \{\eta_t(x): x\in \T^d_N\}$ is a Markov chain.  Indeed, given independence and the Markov property of the individual particle movement, by splitting over all possibilities, the Markov property of $\eta_t$ can be deduced.

\begin{exercise}
\rm Show this Markov property.
\end{exercise}

\subsection{Invariant measures and distribution at time $t\geq 0$}
What are the invariant measures for the process?  Since the process $\eta_t$, corresponding to $K$ particles, is irreducible, there is a unique invariant measure.  It is not so easy to characterize it immediately.

Indeed, let us relax the assumption that there are $K$ particles in the system.  If we do not specify the initial number of particles, then $\eta_t$ is no longer irreducible, since there is no birth or death possible:  For instance, a system with $10$ particles cannot evolve to one with $20$ random walks.  Nevertheless, we may specify in good form several invariant measures for this `relaxed' reducible system.

Recall the Poisson distribution with parameter $\alpha$: $\kappa_\alpha(k) = e^{-\alpha}\alpha^k/k!$ for $k\geq 0$.  Its moment generating function is given by
$$\sum_{k\geq 0} e^{\lambda k} e^{-\alpha}\frac{\alpha^k}{k!} \ = \ e^{\alpha(e^\lambda -1)}.$$
For a nonnegative function $\rho_0: \T^d \rightarrow \R_+$, define the product measure
$\nu^N_{\rho_0(\cdot)}=\prod_{x\in \T_N^d} \kappa_{\rho_0(x/N)}$ on $\N_0^{\T^d_N}$ so that the means
$$E_{\nu^N_{\rho(\cdot)}}[\eta(x)] \ = \ \sum_{k\geq 0}k\kappa_{\rho_0(x/N)}(k) \ = \ \rho_0(x/N).$$
When $\rho_0(\cdot) \equiv \alpha$ is constant, we denote $\nu^N_\alpha=\nu^N_{\rho_0(\cdot)}$.

The process $\{\eta_t: t\geq 0\}$ belongs to the space $D([0,\infty), \Omega)$ of right-continuous paths with left limits in $\Omega=\N_0^{\T^d_N}$.  
 We will denote by $\P_\mu$ and $\E_\mu$ the probability measure and expectation with respect to the evolution of the process when initially $\eta_0$ is distributed according to $\mu$.  On the other hand, $E_\mu$ refers to the expectation with respect to $\mu$ on $\Omega$.     

Now, starting from $\nu^N_{\rho(\cdot)}$, we observe that we may calculate the distribution at later times $t>0$.  

\begin{proposition}
\label{later time prop}
Under $\P_{\nu^N_{\rho_0(\cdot)}}$, the distribution of $\eta_t$ is the inhomogeneous product of Poisson measures $\prod_{x\in \T_N^d} \kappa_{\psi_{N,t}(x)}$ where
$\psi_{N,t}(x) = E\big[\rho_0(N^{-1}(x-Z_t^N))\big]$ is the expectation with respect to the position of a random walk $Z^N_t$ with rates $p(\cdot)$ starting from the origin at time $t$.
\end{proposition}
That the later time $t$-distribution would still be a product measure is a feature of the independence of the particles--not a generic feature in more general interacting particle systems!

When $\rho_0\equiv \alpha$ is constant, we immediately arrive at the following characterization.

\begin{corollary}
The measures $\nu^N_\alpha$ are invariant for the Markov chain $\eta_t$.
\end{corollary}

\noindent
{\it Proof of Proposition \ref{later time prop}.}  We need only compute the moment generating function of $\eta_t$. 
 Write
$$\eta_t(x) \ = \ \sum_{y\in \T^d_N} \sum_{k=1}^{\eta_0(y)} 1(Z^{y,k}_t = x)$$
and, for $\theta:\T_N^d\rightarrow \R$,
$$\sum_{x\in \T^d_N}\theta(x) \eta_t(x) \ = \ \sum_{x\in \T^d_N}\sum_{y\in \T^d_N} \sum_{k=1}^{\eta_0(y)} \theta(x)1(Z^{y,k}_t = x) \ = \ \sum_{y\in \T^d_N} \sum_{k=1}^{\eta_0(y)} \theta(Z^{y,k}_t)$$
where $Z^{y,k}_t$ denotes the position at time $t$ of the $k$th particle initially at location $y$.
 
Since particles move independently,
and initially there are a Poisson number of particles on each site of the lattice,
\begin{eqnarray*}\E_{\nu^N_{\rho_0(\cdot)}}\Big[\exp \sum_{x\in \T^d_N}\theta(x) \eta_t(x)\Big]
& = & \prod_{y\in \T^d_N} \E_{\nu^N_{\rho_0(\cdot)}} \Big[\exp \sum_{k=1}^{\eta_0(y)} \theta(Z^{y,k}_t)\Big]\\
& = & \prod_{y\in \T^d_N} E_{\nu^N_{\rho_0(\cdot)}} \Big( E\big[\exp \theta(Z^{y,1}_t)\big]\Big)^{\eta_0(y)}\\
 &=&\prod_{y\in \T^d_N} \exp \Big[ \rho_0(y/N) \big(E\big[e^{\theta(y+ Z_t)}\big] - 1\big)\Big]
\end{eqnarray*}
where $Z_t=Z_t^N$ is the position of a random walk on $\T^d_N$ starting at the origin.  

Now,
$$E\big[e^{\theta(y+ Z_t)}\big] \ = \ \sum_{x\in \T^d_N} P^N_t(x-y)e^{\theta(x)}$$
and
$$
E\big[e^{\theta(y+Z_t)}\big] - 1 \ = \ \sum_{x\in \T^d_N} P^N_t(x-y) (e^{\theta(x)} -1).$$
Hence, after a calculation,
\begin{eqnarray*}
\E_{\nu^N_{\rho_0(\cdot)}} \Big[ \exp \sum_{x\in \T^d_N} \theta(x)\eta_t(x)\Big] & = &\exp \sum_{y\in \T^d_N}\rho_0(y/N) \sum_{x\in \T^d_N} P^N_t(x-y)\big(e^{\theta(x)} -1\big).
\end{eqnarray*}
Since
\begin{eqnarray*}
\sum_{y\in \T^d_N} P^N_t(x-y)\rho_0(y/N) & = & \sum_{z\in \T^d_N} P^N_t(z)\rho_0(N^{-1}(x-z))\\
&=& E\big[\rho_0\big(N^{-1}(x - Z^N_t)\big)\big] \ = \ \psi_{N,t}(x),
\end{eqnarray*}
the proof concludes. \qed
\medskip

To come back to our initial question, we remark that the invariant measure $\nu^N_\alpha$ can be decomposed in terms of its restrictions to the sets $\Omega_K=\{\eta \in \Omega: \sum_{x\in \T^d_N}\eta(x)=K\}$ for $K\geq 0$.  These are closed and irreducible for the motion. Then, each $\nu_{N,K} = \nu^N_\alpha(\cdot | \sum_{x\in \T^d_N} \eta(x) = K)$ is the unique invariant measure on the restriction, and does not depend on $\alpha$.  In physics terminology, $\nu^N_\alpha$ is the `grand canonical' measure and $\nu_{\T^d_N,K}$ is the `canonical' one.

Since the mean of $\eta(x)$ under $\nu^N_\alpha$ equals $\alpha$, it makes sense to call $\alpha$ the `mass density' of the process.  In this way, $\{\nu^N_\alpha: \alpha\geq 0\}$ is a family of invariant measures indexed by density $\alpha$.

When $\rho_0$ is not constant, the measures $\nu^N_{\rho_0(\cdot)}$ are examples of `local equilibrium' measures, as near a continuity of point $u\in \T^d$ of $\rho_0(\cdot)$, they distribute particles near $\lfloor uN\rfloor$ near that of the invariant measure $\nu^N_{\rho_0(u)}$.  One may consider other non-product `local equilibrium' and `nonequilibrium' measures, but, as we have seen, these local equilibrium measures allow for some computations.

\subsection{Hydrodynamics in mean-value}
If we start with a local equilibrium measure $\nu^N_{\rho_0(\cdot)}$, then initially, the means $\{\rho_0(x/N): x\in \T_N^d\}$ are a discretization of the density `profile' $\rho_0$ defined on the continuous space $\T^d$.  At a later time $t$, the means have evolved to $\{\psi_{N,t}(x): x\in \T_N^d\}$.  But, how to understand a macroscopic picture?

The answers depend on the particular time and space scales chosen in the problem.  We will think of $\T^d_N$ as embedded in $\T^d$ where grid points on $\T^d_N$ are separated by distance $N^{-1}$.  In this way, a `macroscopic'
 point $u$ on $\T^d$ corresponds to the `microscopic' point $\lfloor uN\rfloor$.  As we will see, time should now be appropriately speeded up to see movement of the system.  How fast this speed up should be will depend on the structure of the underlying jump probability $p(\cdot)$.

When $t$ is fixed, one can see that $N^{-1}Z^N_t$ converges weakly to $0$.
Hence, at a continuity point $u$ for $\rho_0(\cdot)$, we have $\lim_{N\uparrow\infty} \psi_{N,t}(\lfloor uN\rfloor) = \rho_0(u)$.
So, if we do not speed up time at all, 
the system does not move.

\vskip .1cm

\subsubsection{Asymmetric motions}

In the asymmetric setting, let $m = \sum xp(x)\neq 0$.  Since $N^{-1}Z^N_{tN}\rightarrow mt$ in probability, we have, for $\epsilon>0$,
$$\lim_{N\uparrow\infty} \sum_{|z/N - mt|\leq \epsilon} P^N_{tN}(z) \ = \ \lim_{N\uparrow\infty} P\Big[\Big|\frac{Z^N_{Nt}}{N} - mt\Big| \leq \epsilon\Big] \ = \ 1.$$
 In this case, when the initial profile $\rho_0$ is continuous, $$\lim_{N\uparrow\infty}\psi_{N,Nt}(\lfloor uN\rfloor) \ =\  \rho_0(u-mt)\ :=\  \rho(t,u).$$
Therefore, if we speed up time by a factor of $N$, we see that the density profile has translated by $mt$.  In this sense $Nt$ is referred to as the `microscopic' time, and $t$ as the `macroscopic' time.  Moreover, the density $\rho(t,u)$ satisfies
\begin{equation}
\label{hyd_ind1}
\partial_t \rho + m\cdot \nabla \rho \ = \ 0.\end{equation}
This makes sense as individual particles displace an order $N$ microscopic locations at microscopic time $Nt$.

\vskip .1cm
\subsubsection{Mean-zero motions}
However, when $m=0$, particles do not displace as much, but follow the `square root' law, when say the underlying jump rates are finite-range.  In other words, displacements are of order $\sqrt{N}$ at time $Nt$, or alternatively of order $N$ at times $N^2t$.  The latter version fits in nicely with our space scaling of $N^{-1}$.  

By the central limit theorem for random walks in this case, $N^{-1}Z^N_{N^2t}\Rightarrow {\rm N}(0,C t)$ and, when again $\rho_0$ is continuous, we have
\begin{eqnarray*}\lim_{N\uparrow\infty}\psi_{N,N^2t}(\lfloor Nu\rfloor) & = & \lim_{N\uparrow\infty} \sum_{z\in \T^d_N} P^N_{N^2t}(z)\rho_0\big(N^{-1}(\lfloor Nu\rfloor - z)\big) \\
& = & \lim_{N\uparrow\infty} E\Big[\rho_0(u - N^{-1}Z^N_{N^2t})\Big] \ = \ \int_{\R^d} \bar\rho_0(x)G_t(u-x)dx.\end{eqnarray*}
Here, $\bar\rho_0$ is the periodic extension of $\rho_0$ with period $\T^d$, and $G_t$ is the Gaussian density with covariance $Ct$.
It follows that $\rho(t,u) := \int_{\R^d} \bar\rho_0(x)G_t(u-x)dx$, being a convolution with the $\bar\rho_0$, satisfies the heat equation on $\T^d$:
\begin{eqnarray}
\partial_t \rho &=& \triangle_C\rho := \sum_{1\leq i,j\leq d} C_{i,j} \partial^2_{u_i,u_j}\rho\nonumber\\
\rho(0,u)&=& \rho_0(u).
\label{hyd_ind2}
\end{eqnarray}

\subsection{Conclusion}

What we have done so far is to derive a macroscopic `hydrodynamic limit' for the mean values over $x\in \T_N^d$.  In the next Section \ref{lec2}, we will view the hydrodynamic limit as a full fledged law of large numbers of the empirical measure of particles.  We also give some physical motivation for the name `hydrodynamics'.

We call the equations \eqref{hyd_ind1} and \eqref{hyd_ind2} and their solutions $\rho(t,u)$ as `hydrodynamic' equations and `hydrodynamic' solutions for the space-time evolution of the macroscopic density.  For independent particles, to summarize, we have proved the following:

\begin{theorem}
Suppose $\rho_0:\T^d\rightarrow \R_+$ is continuous.  Let $v(N)= N$ if $m\neq 0$ and $v(N) = N^2$ if $m=0$.  Then, starting from the sequence $\nu^N_{\rho_0(\cdot)}$, the density average $\psi_{N, v(N)t}(\lfloor Nu\rfloor)$ at location $\lfloor Nu\rfloor$ and time $v(N)t$ converges to $\rho(t,u)$
which solves
equation \eqref{hyd_ind1} if $m\neq 0$ and equation \eqref{hyd_ind2} if $m=0$.
\end{theorem}
 
\section{Notes}
The material on Markov chains is a standard treatment based on the development in \cite{Karlin-Taylor}, \cite{Resnick} and \cite[Appendix 1]{KL}.  
The discussion on hydrodynamics of independent walks follows that in \cite[Chapter 1]{KL} and \cite{Landim-notes}.   See also \cite{Martin-lof} for a more general treatment.  An early rigorous work on hydrodynamics of independent particle systems is \cite{Dob}.

Other work on invariant measures and hydrodynamics of systems of independent particles includes \cite{JLTeixeira}, \cite{JP}, \cite{Lig_ind}, \cite{lpsx}, \cite{Peterson}.

\newpage
\chapter[Section $2$]{Empirical measures, and a first look at hydrodynamics of exclusion processes}
\label{lec2}

We introduce a notion of `hydrodynamics' via empirical measures, which will be used in the analysis of more general interacting particle systems. Then, after a preliminary discussion of certain martingales in the Markov chain context, we discuss the hydrodynamics of exclusion processes.  At the end, we comment on the physical motivation behind the name `hydrodynamics'.

\section{Empirical measures and view of `hydrodynamics' as a LLN}
We may recast the derivation of `hydrodynamics' in systems of independent random walks in terms of the random space-time scaled empirical distribution
$$\pi^N_{v(N)t} \ = \ \frac{1}{N^d}\sum_{x\in \T^d_N} \eta_{v(N)t}(x) \delta_{x/N}.$$
Here, $\pi^N_{v(N)t}$ is a member of ${\mathcal M}_+(\T^d)$, the nonnegative measures on $\T^d$.  Recall our definition of the independent particle process in Section \ref{lec1}:  We derived the distribution of the particle numbers $\eta_{v(N)t}$ at later times, starting from a local equilibrium measure $\nu^N_{\rho_0(\cdot)}$.
However, it will be easier in what follows to consider the weaker notion of the asymptotic behavior of the empirical distribution.   In this sense, what is meant by `hydrodynamics' is a law of large numbers, characterizing `first-order' behavior.

 Let $G$ be a smooth, bounded function on $\T^d$, and write  
\begin{equation}
\label{Ginner}
\langle G, \pi^N_{v(N)t}\rangle \ := \ \frac{1}{N^d}\sum_{x\in \T^d_N} G(x/N)\eta_{v(N)t}(x).\end{equation}
Here, the pairing $\langle F, \pi\rangle$ is another way to write the expectation $E_\pi[G]$.
Recall, from Section \ref{lec1}, the distribution of $\eta_{v(N)t}$ starting from $\nu^N_{\rho_0(\cdot)}$ is a product of Poisson measures with intensity $\psi_{N, v(N)t}(\cdot)=E\big[\rho_0\big(N^{-1}(x - Z^N_{v(N)t})\big)\big]$, where the expectation is with respect to the random walk $Z^N_\cdot$. Then, in mean-value, starting in $\nu^N_{\rho_0(\cdot)}$, we have
\begin{eqnarray*}
&&\E_{\nu^N_{\rho_0(\cdot)}}\Big[ \frac{1}{N^d}\sum_{x\in \T^d_N} G(x/N)\eta_{v(N)t}(x) \Big] = \frac{1}{N^d}\sum_{x\in \T^d_N} G(x/N) \E_{\nu^N_{\rho_0(\cdot)}}[\eta_{v(N)t}(x)]\\
&&\ \ \ \ \ \ \ \ \ \ \ \ \ = \ \frac{1}{N^d}\sum_{x\in \T^d_N}G(x/N)E\big[\rho_0\big(N^{-1}(x-Z_{v(N)t})\big)\big].\end{eqnarray*}
Depending on whether $m\neq 0$ or $m=0$, in which case $v(N)=N$ or $v(N)=N^2$, the last quantity as before tends respectively to
\begin{equation}\label{hyd_sols_indep} \int \rho_0(u-mt)G(u)du \ \ {\rm or \ \ } \int G(u)\int \bar\rho_0(w)G_t(u-w)dwdu.\end{equation}

However, the variance of \eqref{Ginner}, when starting in $\nu^N_{\rho_0(\cdot)}$, since the occupation numbers 
$\eta_{v(N)t}$ are independent Poisson variables with intensities $\{\psi_{N,v(N)t}(x): x\in \T_N^d\}$, vanishes:
\begin{eqnarray*}
{\rm Var} \Big[ \frac{1}{N^d}\sum_{x\in \T^d_N} G(x/N)\eta_{v(N)t}(x)\Big] 
&=& \frac{1}{N^{2d}}\sum_{x\in \T^d_N} G^2(x/N)\psi_{N,v(N)t}(x)\\
&=& O(N^{-d}) \ \rightarrow \ 0,\end{eqnarray*}
as $G$ and $\rho_0$ are bounded.  Hence, $\langle G, \pi^N_{v(N)t}\rangle$ converges variously to the expressions in \eqref{hyd_sols_indep} in probability, depending on the drift $m$.  

In particular, we have shown, with respect to the initial distribution $\mu^N$ for the process $\eta_0$, the empirical measure $\pi^N_{v(N)t}$, a random element of $\mathcal{M}_+(\T^d)$, converges in probability to the deterministic measure $\rho(t,u)du$ corresponding to the macroscopic space-time mass evolution.  

Here, the topology on $\mathcal{M}_+(\T^d)$ used is as follows:  Consider $C(\T^d)$, the space of real continuous functions on $\T^d$ endowed with the sup-metric.  Let $\{f_k: k\geq 1\}$ be a dense, countable family of continuous functions in $C(\T^d)$.  Then, define the distance $\delta(\mu,\nu)$ on $\mathcal{M}_+(\T^d)$ by 
$$\delta(\mu,\nu)\ = \ \sum_{k=1}^\infty \frac{1}{2^k} \frac{|\langle \mu, f_k\rangle - \langle \nu, f_k\rangle|}{1+ |\langle \mu, f_k\rangle - \langle \nu, f_k\rangle|}.$$
Therefore, capturing the limit behavior of $\langle G,\pi^N_{v(N)t}\rangle$ for each continuous function $G$ is enough to compute the limits of $\pi^N_{v(N)t}$.

\section{Exclusion processes}

The exclusion process is one of the `canonical' interacting particle systems, introduced in \cite{Spitzer}.  Consider particles on $\T^d_N$ with the minimal interaction that no particle can jump onto another.  Accordingly, a configuration of occupation numbers $\eta_t$ belongs to the finite state space $\Omega = \{0,1\}^{\T^d_N}$ where $\eta_t(x) =0$ or $1$ depending on whether $x\in \T^d_N$ is empty or occupied at time $t\geq 0$.  Informally, $\eta_t$ updates in that each particle is a continuous time random walk carrying an exponential $1$ clock.  When a clock rings, the particle tries to move with skeleton jump probability $p$.  However, if the site chosen is already occupied, the jump is suppressed, and all clocks reset.  We will restrict attention to jump probabilities $p$ which are tranlsation-invariant:  $p(x,y)=p(y-x)$.

More formally, infinitesimally $\eta_t$ can change to $\eta_t^{x,x+y}$ when a particle at $x$ displaces by increment $y$ where
$$\eta^{a,b}(z) \ = \ \left\{\begin{array}{rl} 
\eta(b)& \ {\rm when \ } z=a\\
\eta(a)& \ {\rm when \ }z=b\\
\eta(z)& \ {\rm when \ }z\neq x,y
\end{array}\right.
$$
with rate $\eta(x)(1-\eta(x+y))p(y)$.  The `exclusion' factor `$\eta(x)(1-\eta(x+y))$' is $1$ exactly when $x$ is occupied and the destination $x+y$ is unoccupied; otherwise, it is $0$.
The generator of the process $\eta_t$ is given by
\begin{align}
\label{eq:exclu_gen}
(Lf)(\eta) \ = \ \sum_{x\in \T^d_N}\sum_{y\in \T^d_N} \eta(x)(1-\eta(x+y))p(y)\big[f(\eta^{x,x+y}) - f(\eta)\big]
\end{align}
for functions $f:\Omega \rightarrow \R$.

In the following, we will also assume that $p$ is finite-range, that is $p(z)=0$ for $|z|>R$ for some $R$.  Of course, a case is when $p$ is nearest-neighbor, that is when the range $R=1$.

When $p$ is symmetric, the process is called the `symmetric simple exclusion process'.  When $p$ is asymmetric, $\eta_t$ is termed the `asymmetric exclusion process'.  When the range $R=1$, the label `simple' is sometimes added to the names.

There is a simplification of the form of $L$ when $p$ is symmetric.  Namely, since $\eta^{x,x+y} = \eta^{x+y,x}$ and $p(y)=p(-y)$, we have
\begin{eqnarray}
\label{simplification}
(Lf)(\eta) &=& \frac{1}{2}\sum_{x\in \T^d_N}  \sum_{y\in \T^d_N} \big\{\eta(x)(1-\eta(x+y)) + \eta(x+y)(1-\eta(x))\big\}p(y)\nonumber\\
&&\ \ \ \ \ \ \ \ \ \ \ \ \ \ \ \ \ \ \ \  \times \big[f(\eta^{x,x+y})-f(\eta)\big]\\
&=& \frac{1}{2}\sum_{x\in \T^d_N}  \sum_{y\in \T^d_N} p(y)\big[f(\eta^{x,x+y})-f(\eta)\big].\nonumber
\end{eqnarray}
The last line follows as the term in curly braces equals $|\eta(x) - \eta(x+y)|=1$ exactly when the difference in square brackets vanishes.

Let $\nu^N_\alpha = \prod_{x\in \T_N^d}{\rm Bern}(\alpha)$ be the product measure on $\T^d_N$ with Bernoulli marginals with success probability $\alpha$.  Recall from Subsection \ref{sec:invariantmeasurefinite} that $E_\mu$ stands for the expectation under $\mu$, and $\P_\mu$ and $\E_\mu$ for the process measure and expectation when starting in $\mu$.
\begin{proposition}
\label{prop:exclusion_invariant}
The measures $\{\nu^N_\alpha: 0\leq \alpha\leq 1\}$ are invariant measures for $\eta_t$, for both symmetric and asymmetric $p$.  In the symmetric case, $\nu^N_\alpha$ is also reversible.
\end{proposition}

\begin{proof}  First note that $f(\eta^{x, x+y}) - f(\eta)$ vanishes if $\eta(x)=\eta(x+y)$.  Hence, we may write
$$(Lf)(\eta) \ = \ \sum_{x\in \T^d_N}\sum_{y\in \T^d_N} \eta(x)p(y)\big[f(\eta^{x,x+y}) - f(\eta)\big],$$
where we dropped the factor `$1-\eta(x+y)$'.  

Note also, under the change of measure $\zeta = \eta^{x,y}$, which exchanges values $\eta(x)$ and $\eta(y)$, the measure $\nu^N_\alpha$ remains the same.  Hence, for functions $f,h$, the term
\begin{align*}
&E_{\nu^N_\alpha}\big[h(\eta)\eta(x)(1-\eta(x+y))p(y)f(\eta^{x,x+y})\big] \\
&\ \ \ \ \ \ \  = \ E_{\nu^N_\alpha}\big[h(\eta^{x,x+y})\eta(x+y)(1-\eta(x))p(y)f(\eta)].
\end{align*}
Then, by collecting terms, with simple manipulation,
$$E_{\nu^N_\alpha}\big[hLf\big] \ = \ E_{\nu^N_\alpha}\big[(L^*h)f\big],$$
where the $\nu^N_\alpha$-adjoint $L^*$ is seen as the exclusion generator with single particle jump probability $q(z)= p(-z)$.  

Now, since $L^* 1 = 0$, by inspection, we have that $E_{\nu^N_\alpha}[Lf]=0$ for all $f$.  This shows $\nu^N_\alpha$ is invariant. 

Moreover, as $L^*=L$ when $p$ is symmetric, $\nu^N_\alpha$ is reversible.
\end{proof}

From this result, we can deduce 
when there are exactly $K$ particles in the system that $\nu_\alpha^N(\cdot |\sum_{x\in \T^d_N}\eta(x)=K)$ is the unique `canonical' invariant measure for the motion on $\Omega_K = \big\{\eta: \sum_{x\in \T_N^d}\eta(x)=K\big\}$.

\begin{exercise}[Duality]
\label{ex:duality} \rm  Suppose $p$ is symmetric.  Show that the space of linear combinations of occupation variables $\eta(x)$ for $x\in \T^N_d$ remains invariant under action by generator $L$.  In fact, the space of linear combinations of $\prod_{j=1}^n\eta(x_j)$ for $\{x_j\in \R^N_d: 1\leq j\leq n\}$ is closed.  We comment that this property is sometimes referred to as `duality'. 
Hint:  A case of the property is straightforwardly seen by computing the action on the variable $\eta(x)$.
\end{exercise}

\section{Martingales and Markov chains}
\label{sec:2.3}
Recall that a martingale $M_t$ corresponding to sigma-fields $\mathcal{F}_t$ is a random process, adapted to the filtration $\{\mathcal{F}_t\}$, which satisfies
$$E[M_t|\mathcal{F}_s] \ = \ M_s \ \ {\rm and \  \ } E|M_t|<\infty$$
for all $t\geq s\geq 0$.

\begin{exercise}\rm
Let $N(t)$ be a Poisson process with rate $\lambda$.  Then, $M_t = N(t) - \lambda t$ is a martingale with respect to the `natural' sigma-fields $\mathcal{F}_t = \sigma\{N_u: u\leq t\}$.  Also, $Q_t = M^2_t - \lambda t$ is also a martingale with respect to $\{\mathcal{F}_t\}$.
\end{exercise}

Let $X_t$ be a Markov process on a countable state space $\Omega$. Let $F:\R_+\times \Omega\rightarrow \R$ be a twice continuously differentiable function whose first and second partial time derivatives are uniformly bounded.  Define
\begin{eqnarray*}
M^F_t &=& F(t,X_t) - F(0,X_0) - \int_0^t \Big(\partial_s+L\Big )F(s,X_s) ds \\
N^F_t&=& (M^F_t)^2 - \int_0^t \Big[(LF^2)(s,X_s) - 2F(s,X_s)(LF)(s,X_s)\Big]ds.\end{eqnarray*}

These two processes, which we show below are martingales, will be very useful in our stochastic analysis of Markov systems.  The term,
$$\langle M^F\rangle_t \ = \ \int_0^t \Big[(LF^2)(s,X_s) - 2F(s,X_s)(LF)(s,X_s)\Big]ds$$
is the `(predictable) quadratic variation' of the martingale $M^F_t$.

\begin{proposition}\label{martingale}  With respect to natural sigma-fields $\mathcal{F}_t = \sigma\{X_u: u\leq t\}$, both $M^F_t$ and $N^F_t$ are martingales.
\end{proposition}

\begin{proof}  We will show that $M^F_t$ is a martingale when $F$ does not depend on time.  Generalizations and verification of $N^F_t$ as a martingale are left to the reader.  We need only show that
$$E\big[F(X_t)|\mathcal{F}_s\big] -F(X_s) - \int_s^t E\big[(LF)(X_u)|\mathcal{F}_s\big] du \ = \ 0.$$
Now, $E[F(X_t)|\mathcal{F}_s] = P_{t-s}F(X_s)$ and $E[(LF)(X_u)|\mathcal{F}_s]= P_{u-s}(LF)(X_s)$ from the Markov property.  From the forward equation, the derivative of the left-side of the above display in $t$ equals
$$P_{t-s}(LF)(X_s) - P_{t-s}(LF)(X_s) \ = \ 0.$$
At time $t=s$, the left-side also vanishes.  This concludes the proof. \end{proof}

\begin{exercise}\rm
Complete the proof of Proposition \ref{martingale}.  Hint:  With  respect to $M^F_t$, we need to show
$$E\big[F(t,X_t)|\mathcal{F}_s\big] - F(s,X_s) = \int_s^t E\Big[ \big(\partial_s + L\big) F(u, X_u)\Big |\mathcal{F}_s\Big]ds.$$
When $t=s$, the relation holds.  Hence, if one shows the derivatives with respect to $t$ match, that is
$$\partial_t E\big[F(t, X_t)|\mathcal{F}_s\big] = P_{t-s} \big(\partial_t F\big)(t, x)|_{x=X_s} + P_{t-s} (LF)(t,x)|_{x=X_s}.$$ 
\end{exercise}

\section{Sketch of hydrodynamics for exclusion processes}
\label{sec:sketch-exclusionhydro}

 Let $\rho_0:\T^d\rightarrow \R_+$ be a continuous function.  We will denote by $\nu^N_{\rho_0(\cdot)}=\prod_{x\in \T_N^d}{\rm Bern}(\rho_0(x/N))$ the product measure with Bernoulli marginal at site $x\in \T_N^d$ with success probability $\rho_0(x/N)$.

Our goal is to analyze the asymptotic behavior of the empirical measure with respect simple exclusion process $\eta_t$,
$$\pi^N_{v(N)t} \ = \ \frac{1}{N^d}\sum_{x\in \T^d_N} \eta_t(x)\delta_{x/N},$$
in a time scale $v(N)$ to be chosen later.  We will start the process from local equilibrium measures $\mu^N = \nu^N_{\rho_0(\cdot)}$.

In general, one doesn't expect $\{\eta_{v(N)t}(x): x\in \T_N^d\}$ to consist of independent random variables, even if initially at $t=0$ they are independent.
Instead of computing the mean and variance, as with independent particles, we will use the martingale formulation.

To understand main ideas, let $G:\T^d\rightarrow \R$ be a smooth function, not depending on time.  Treating 
$$\langle G, \pi^N_{v(N)t}\rangle \ = \ \frac{1}{N^d}\sum_{x\in \T^d_N}G(x/N)\eta_{v(N)t}(x)$$ as a function $F$ of the Markov process $\eta_{v(N)t}$ on the state space $\Omega$, we obtain a microscopic evolution equation:
$$\langle G, \pi^N_{v(N)t}\rangle \ = \  \langle G, \pi^N_0\rangle + \int_0^{v(N)t} L\langle G, \pi^N_{s}\rangle ds + M^N_{t}(G)$$
where $M^N_t(G)$ is a martingale (cf. Proposition \ref{martingale}) with quadratic variation
$$\langle M^N(G)\rangle_t \ = \ \int_0^{v(N)t} \Big[L\big(\langle G,\pi^N_s\rangle \big)^2- 2\langle G,\pi^N_s\rangle \big(L\langle G,\pi^N_s\rangle\big)\Big]ds.$$

Now, the martingale is negligible in the $N\uparrow\infty$ limit:  We compute that
\begin{align}\label{quad_var_calc}
&\E_{\mu^N}\big[ (M^N_{t}(G))^2 \big] \\
&\ \ \ \ \ \  =  \int_0^{v(N)t} \frac{1}{N^{2(d+1)}}\sum_{x,y\in \T^d_N}\big(\nabla^N_{x,x+y}G\big)^2 p(y)\eta_s(x)(1-\eta_s(x+y))ds\nonumber\\
&\ \ \ \ \ \  \leq  \frac{v(N)t}{N^{2(d+1)}} \sum_{x,y\in \T^d_N}\big(\nabla^N_{x,x+y}G\big)^2p(y) 
 \ = \ O(v(N)N^{-d-2}).\nonumber\end{align}
In the last line, we have used that the occupation variables are bounded by $1$ and that the jump probabilities are finite-range.

A calculation shows that 
\begin{align}\label{summation} 
&L\langle G, \pi^N_{s} \rangle \\
&\ \ \ \ \ \ \ =  \frac{1}{N^{d+1}}\sum_{x,y\in \T^d_N} \eta_s(x)(1-\eta_s(x+y))p(y)\big[\nabla^N_{x+y,x}G\cdot (\eta_{s}(x) - \eta_{s}(x+y))\big]\nonumber\\
&\ \ \ \ \ \ \ = \frac{1}{N^{d+1}}\sum_{x,y\in \T^d_N} \eta_s(x)(1-\eta_s(x+y))p(y)\big[\nabla^N_{x+y,x}G\big]\nonumber
\end{align}
where $\nabla^N_{u,v}G = N[G(u/N)-G(v/N)] \sim (u-v)\cdot \nabla G(v/N)$.

\begin{exercise}\rm
 Verify the derivation of the quadratic variation given in \eqref{quad_var_calc}.\end{exercise}

\subsection{Symmetric case.}
\label{sec:symm-hyd}
  When $p$ is symmetric, recalling the simpler form of the generator \eqref{simplification}, a further summation by parts is possible and we obtain
\begin{eqnarray*}
L\langle G, \pi^N_s\rangle & = & \frac{1}{2N^{d+1}}\sum_{x,y\in \T^d_N} \big[\eta_s(x) - \eta_s(x+y)\big]p(y)\nabla^N_{x+y,x}G\\
&=&  \frac{1}{2N^{d+2}}\sum_{x,y\in\T^d_N}p(y)\eta_s(x) \triangle^N_{x,y}G
\end{eqnarray*}
where $\triangle^N_{x,y}G = N^2[G(x+y/N) - 2G(x/N)+G(x-y/N)]$.

Now, $\sum_y p(y)\triangle^N_{x,y}G = \triangle_CG(x/N) + o(1)$ where we recall
$$\triangle_C \ = \ \sum_{1\leq i,j\leq d} C_{i,j} \partial^2_{x_i,  x_j}, \ \ {\rm and \ covariances \ } C_{i,j} \ = \ \sum_{z\in \T^d_N} z_iz_j p(z).$$ 

Hence, if $v(N)=N^2$, we have that
$$\langle G, \pi^N_{v(N)t}\rangle \ = \  \langle G, \pi^N_0\rangle + \frac{1}{2N^2}\int_0^{v(N)t} \langle \triangle_C G, \pi^N_{s}\rangle ds + M^N_t(G) + o(1).$$

Putting these estimates together, along with \eqref{quad_var_calc}, for symmetric $p$, we have `closed' the equation:
$$\langle G, \pi^N_{v(N)t}\rangle \ = \ \langle G, \pi^N_0\rangle + \frac{1}{2}\int_0^{t} \langle \triangle_C G, \pi^N_{v(N)s}\rangle ds + o(1).$$
This suggests in the $N\uparrow\infty$ limit that the empirical measures $\pi^N_{N^2t}$ converge in the weak sense to a solution of the heat equation 
$$\partial_t \rho = \frac{1}{2}\triangle_C \rho.$$
The goal of Chapter \ref{lec3} is to make precise this statement for symmetric simple exclusion.

The reader will have noticed that the above display is virtually the same equation derived for independent particles.  This is due to the definition of the empirical measure $\pi^N_t$, which uses the occupation variables $\eta(x)$ as the masses at locations $x/N$.  The `duality' property of such functions with respect to symmetric exclusion mentioned earlier (cf. Exercise \ref{ex:duality}) allows to recover the heat equation.  There are differences however in the associated fluctuations, for instance as seen in Sections \ref{lec11}, and \ref{lec12} (specialized to independent particles).

\subsection{Drift case.}
When $p$ is asymmetric, say $m=\sum_z zp(z)\neq 0$, as in the independent particle model, we should choose $v(N)=N$.  But, one cannot `close' the equation.  One has to deal with the term
$$\frac{N}{N^{d+1}}\sum_{x,y\in \T^d_N} \eta_s(x)(1-\eta_s(x+y))p(y)\big[\nabla^N_{x+y,x}G\big]$$
composed of `two-point' functions $\eta(x)\eta(x+y)$.  In the limit, such a term due to `local averaging' should be replaced by a quadratic function of the empirical density.  

Formally, one would obtain a form of Burger's equation
$$\partial_t + m\cdot \nabla \rho(1-\rho) \ = \ 0.$$
There is much work on such hyperbolic mass conservation laws.  In particular, there may be several solutions even starting from smooth initial data.  It will turn out that the solution found by hydrodynamics is the unique `entropy' or `vanishing viscosity' solution.  We will discuss the $m\neq 0$ case in the context of `TASEP', that is in totally asymmetric simple exclusion setting in $d=1$ when $p(1)=1$ and $p(j)=0$f or $j\neq 1$ in Section \ref{lec6}.

\subsection{Asymmetric, mean-zero case.}
\label{sec:asym_meanzero}
In the final case when $p$ is asymmetric, but mean-zero, that is $p(z)\neq p(-z)$ for some $z$, but $\sum_z zp(z) = 0$, there are other complications.  The system is referred to as a `non-gradient' model.
Although, we should speed up time by $v(N)=N^2$, a second summation-by-parts as in the symmetric situation cannot be done.  Indeed, multiplying \eqref{summation} by $N^2$, we have
\begin{align*}
&\frac{N^2}{2N^{d+1}} \sum_{x\in \T^d_N}\nabla G(x/N)\\
&\ \ \ \ \ \ \cdot \sum_{y\in \T_N^d}\big\{\eta_s(x)(1-\eta_s(x+y))yp(y) - \eta_s(x+y)(1-\eta_s(x))yp(-y)\big\}.
\end{align*}
Unless $p$ is symmetric, one cannot rewrite the expression in curly braces as the exact difference of a function $f$ and its translate.

Let $\{e_i\}_{i=1}^d$ be the standard basis in $\Z^d$.  It can be shown that the sum over $y$ dotted with $e_i$ can be approximated by a difference $\sum_{j=1}^d a_{i,j}(\eta(x))-a_{i,j}(\eta(x+e_j))$, a `gradient'.  The function $a_{i,j}$ depends on $p$.  With respect to a certain  'homogenization' $\bar a_{i,j}$ of $a_{i,j}$, the hydrodynamic equation can then be written down as a nonlinear Heat equation,
$$\partial_t \rho = \frac{1}{2}\sum_{i=1}^d \partial_{x_i, x_j} \bar a_{i,j}(\rho(t,\theta)).$$
	We refer to \cite{Sasada}, \cite{KL}, \cite{Varadhan_notes}  for more discussion and details.

 \section{Why is it called `hydrodynamics'?}
 \label{sec:why-hydro}
 Finally, we comment on why the scaling limit is called a `hydrodynamic limit'.  The origins go back to the study of $L=\rho N$ identical mass $1$ particles with certain positions $\{q^\ell=\langle q^\ell_i: i=1,2,3\rangle\}_{\ell=1}^L$ and momenta $\{p^\ell = \langle p^\ell_i: i=1,2,3\rangle\}_{\ell=1}^L$ moving on a torus $N\T^3$, with width $N$, according to Newtonian dynamics: 
\begin{align*}
\frac{dq^\ell_i}{dt} &=  \frac{\partial H}{\partial p^\ell_i} \ = \ p_i^\ell \ \ {\rm and \ \ } \\
\frac{dp^\ell_i}{dt} & =  -\frac{\partial H}{\partial q^\ell_i} \ = \ -\sum_k V_i(q^\ell - q^k),
\end{align*}
where the energy $H$ is given by
\begin{align*}
H  & =  \ \frac{1}{2}\sum_\ell\sum_i |p^\ell_i|^2 + \frac{1}{2}\sum_{\ell\neq k} V\big(q^\ell - q^k\big)
\end{align*}
and $\nabla V = \langle V_1, V_2, V_3\rangle$
in terms of an even, nonnegative, smooth, compactly supported function $V:\R^3 \rightarrow \R$. 

In the system, the total mass, the three components of momentum, and energy are conserved.  The question is how to understand in a scaling limit, as time is speeded up by $N$, and space is rescaled by $N^{-1}$, the evolution of the local density $\rho(t,x)$, momenta $w=\langle w_1(t,x), w_2(t,x), w_3(t,x)\rangle$, and energy $e(t,x)$, that is the `hydrodynamic flow' of the system.

  One can form the scaled empirical measures corresponding to these quantities:
  Writing $\delta(\cdot)$ for the point mass $\delta_\cdot$,
\begin{eqnarray*}
\xi^0_t &=& \frac{1}{N^3}\sum_\ell \delta\Big(\frac{q^\ell(Nt)}{N}\Big)\\
\xi^i_t &=& \frac{1}{N^3}\sum_\ell \delta\Big(\frac{q^\ell(Nt)}{N}\Big)p^\ell_i(Nt) \ \ \ {\rm for \ }i=1,2,3\\
\xi^4_t &=& \frac{1}{N^3}\sum_\ell \delta\Big(\frac{q^\ell(Nt)}{N}\Big)h^\ell(Nt)
\end{eqnarray*}
where energy of the $\ell$th particle is
$$h^\ell(Nt) \ = \ \frac{1}{2}|p^\ell(Nt)|^2 + \frac{1}{2}\sum_k V\big(q^\ell(Nt) - q^k(Nt)\big).$$

On the torus $N\T^3$, one can define a $5$ parameter family of canonical point measures $\mu^N_{\rho, w, \beta}$, corresponding to density $\rho$, velocity $w$, and inverse temperature $\beta$, which are invariant for the dynamics: 
Here, $L = \rho N^3$ points $\{q^\ell\}$ are distributed in $N\T^3$ according to joint density, with respect to Lebesgue measure,
$$\frac{1}{Z} \exp \Big\{ -\frac{\beta}{2} \sum_{\ell\neq k} V(q^\ell - q^k)\Big\}$$
where $Z$ is a normalization.  The momenta $\{p^\ell\}$ are i.i.d. vectors with independent Gaussian components $\big\{N(w_i, \beta^{-1}): i=1,2,3\big\}$, independent of the positions $\{q^\ell\}$.  An infinite volume limit of these measures, as $N\uparrow\infty$, might be taken under some conditions.

The goal is to show in some sense the limits $\xi^i_t\rightarrow y^i(t,x)du$ for $i=0,1, 2,3,4$ where $y=\langle y_i: i=0,1,2,3,4\rangle$ satisfies an `Euler' equation
$$\partial_t y + \nabla_x F(y)  \ = \ 0.$$
Here $F$ is a $5\times 3$ matrix.
The rigorous passage to such a limit in general is open!  In computing $\frac{\partial}{\partial t} \xi^i$, one has to `close' the expressions in terms of functions of the empirical measures.  Formally, one can do this and derive the form of $F$, subject to a certain ansatz. 

For instance, let us derive the equation for the formal limiting `local average' density $y^0 = \rho(t,x)$ in terms of the formal limiting `local average' momentum $\langle y^1, y^2, y^3\rangle = w(t,x)$.  With respect to a test function $G$,
\begin{align*}
\frac{d}{d t} \frac{1}{N^3}\sum_\ell G\Big(\frac{q^\ell(Nt)}{N}\Big) & = \frac{1}{N^3}\sum_\ell \nabla G\Big(\frac{q^\ell(Nt)}{N}\Big) \cdot p^\ell (Nt)\\
&\sim \int_{\T^3} \big[\nabla G(x)\cdot w(t,x)\big]\rho(t,x)dx.
\end{align*}
Then,
$$\partial_t \rho + \nabla \cdot (\rho w) = 0.$$

To derive an equation for $w(t,x)$, however, we will need to understand how to average time-dependent quantities involving nonlinear terms such as 
$$
 p^\ell_i(Nt)p^\ell_j(Nt) \ {\rm \ and  \ \ }
 \big(q^\ell(Nt) - q^k(Nt)\big)V_i(q^\ell(Nt) -q^k(Nt)).$$  
One would expect that these quantities could be replaced by space averages with respect to an infinite volume local equilibrium measure.

Such an ergodic theorem for the purely deterministic coupled ODE dynamics to `close' the equation is however difficult to obtain.  One of the best results is in \cite{OVY} where a small amount of noise is added to the dynamics to make it into a reasonable Markov process so that local averaging can be done, and a rigorous limit can be proved.  See also the recent developments in the solution of Hilbert's sixth problem with respect to Boltzmann's equation \cite{Dengetal}.

One motivation, among others, for our study of stochastic interacting particle systems is that the randomness of the microscopic motions may allow for suitable ergodic theorems.  In a sense, to quote S.R.S. Varadhan, instead of `approximating an exact problem', we will try to solve `exactly an approximate problem'.  Of course, in the modeling of various phenomena, adding noise to the dynamics is often natural.

\begin{exercise}
\label{hyd_det_exercise}
\rm
Consider the ansatz, with respect to the canonical measures $\mu^N(\rho, w, \beta)$, that 
$$p^\ell_i(Nt)p^\ell_j(Nt) \sim w_i(t,x)w_j(t,x) + 1(i=j)\beta^{-1}(t,x)$$
and the `pressure'
\begin{align*}
& \frac{-1}{2}\sum_k\big(q_j^\ell(Nt) - q_j^k(Nt)\big)V_i\big(q^\ell(Nt) -q^k(Nt)\big) \sim
P_{j,i}\big(\rho(t,x), w(t,x), \beta^{-1}(t,x)\big),
\end{align*}
Note that 
\begin{align*}
&\sum_\ell G\Big(\frac{q^\ell(Nt)}{N}\Big)\sum_k V_i\big(q^\ell(Nt) - q^k(Nt)\big)\\
&\quad\quad = \frac{1}{2}\sum_{\ell, k} \Big(G\Big(\frac{q^\ell(Nt)}{N}\Big) - G\Big(\frac{q^k(Nt)}{N}\Big)\Big) V_i\big(q^\ell(Nt) - q^k(Nt)\big)
\end{align*}
as $V_i$ is anti-symmetric.  Derive formally from the ansatz that the equation for $w(t,x)$ is given by
\begin{align*}
&\frac{\partial}{\partial t}\big(w_i(t,x)\rho(t,x)\big) +\frac{\partial}{\partial x_i} \big(\beta^{-1}(t,x)\rho(t,x)\big) \\
&\quad \quad +\sum_{j=1}^3\frac{\partial}{\partial x_j} \big(w_j(t,x)w_i(t,x)\rho(t,x)\big)\\
&\quad\quad + \sum_{j=1}^3 \frac{\partial}{\partial x_j}\big[P_{j,i}\big(\rho(t,x), w(t,x),\beta^{-1}(t,x)\big) \rho(t,x)\big]= 0.
\end{align*}

We note that the last equation for the energy $y^5=e$ should be
\begin{align*}
&\frac{\partial}{\partial t}\big(e(t,x)\rho(t,x)\big) + \sum_{j=1}^3 \frac{\partial}{\partial x_j}\big(e(t,x)w_j(t,x)\rho(t,x)\big) \\
&\quad\quad + \sum_{j=1}^3 \sum_{i=1}^3 \frac{\partial}{\partial x_j} \big[P_{j,i}\big(\rho(t,x),w(t,x), \beta^{-1}(t,x)\big)w_i(t,x)\rho(t,x)\big]=0.
\end{align*}

\end{exercise}

\section{Notes}

The material on martingales can be found in \cite{EK} for instance, and other places.  Similar treatments of the hydrodynamics of exclusion can be found in \cite{KL} and \cite{Varadhan_notes}.  

The exclusion process--see \cite{Griffeath} for a retrospective--has many properties which make it amenable to calculation.  It has proved to be a versatile model, which can be defined on very general graphs. See \cite{dmp}, \cite{Liggett1}, \cite{Liggett2}, \cite{Liggett_notes}, \cite{Liggett3}, \cite{Seppalainen-book} for detailed studies using different properties.

We comment one may formulate other empirical measures of interest, those of `higher order' \cite{ChenJ}, or with respect to `local functions' \cite{KL}[Chapter III, Lemma V.5.5].

As one can see in the sketch of the hydrodynamics for symmetric exclusion processes, requirements may be weakened on the initial measure.  In fact, what is needed is a guarantee of a law of large numbers at time $0$.  The concept of `very weak local equilibrium' given below (cf. Chapter III in \cite{KL}) is sufficient and somewhat general.

Let $\rho_0:\T^d\rightarrow \R_+$ be a function.  We say that a sequence of probability measures $\mu^N$ on $\T^d_N$ is a `very weak local equilibrium' according to profile $\rho_0$ if 
 $$\lim_{N\uparrow\infty} E_{\mu^N}\Big[ \Big| \frac{1}{N^d} \sum_{x\in \T^d_N} G(x/N) \eta(x)  - \int_{\T^d} G(u)\rho_0(u)du\Big|\Big] \ = \ 0.$$
 for all bounded, continuous $G: \T^d \rightarrow \R$.
 
We remark that $\mu^N$ may be degenerate, that is supported on a single configuration, and that the sequence $\{\mu^N\}$ may consist of deterministic configurations which satisfy the law of large numbers in the definition above.

 With respect to Subsection \ref{sec:why-hydro}, for further discussion of hydrodynamics of deterministic evolutions see  \cite{Spohn}[Part 1], \cite{Saint-Raymond}, \cite{Dengetal}.  An early work in this regard is \cite{Morrey}.  We have followed the scheme in \cite{Varadhan_notes} which also elaborates more on the calculations surrounding Exercise \ref{hyd_det_exercise}.

\newpage
\chapter[Section $3$]{Proof of hydrodynamics for symmetric exclusion processes}
\label{lec3}

After stating the main theorem, we provide an outline of the proof and provide details on the associated steps.  Such an outline can be used to prove `hydrodynamic limits' in other models, in particular in zero-range processes in Chapter \ref{lec4}.

\section{Statement of hydrodynamics}

Recall the notation from the last subsection, with respect to $d$-dimensional symmetric exclusion processes with finite-range symmetric jump rate $p$, in particular the definition of the operator
$\triangle_C \ = \ \sum_{1\leq i,j\leq d} C_{i,j}\partial^2_{x_i, x_j}$
with $C_{i,j}  =  \sum_{z\in \T^d_N} z_iz_j p(z)$, and the measure $\mu^N= \nu^N_{\rho_0(\cdot)}=\prod_{x\in \T_N^d}{\rm Bern}(\rho_0(x/N))$.  Recall also that $\P_{\mu^N}$ and $\E_{\mu^N}$ stand for the process measure and expectation starting from $\mu^N$.

\begin{theorem}
Consider the symmetric exclusion process with finite-range translation-invariant jump probabilities.  Then, starting from the local equilibrium measure $\mu^N$, associated with continuous profile $\rho_0(\cdot)$, the empirical measure $\pi^N_{N^2t}$ converges in probability to the measure $\rho(t,u)du$ where $\rho(t,u)$ satisfies the hydrodynamic equation
\begin{equation}
\label{hyd_ssep}\partial_t\rho(t,u) \ = \ \frac{1}{2}\triangle_C\rho(t,u), \ \ \ {\rm and \ \ \ }\rho(0,u) = \rho_0(u).\end{equation}
\end{theorem}

The proof is given through the following rough steps, which are then explained in more detail in later subsections.

\vskip .1cm

Step 1.  Consider the trajectory of empirical measures indexed in an interval of time, $\pi^N=\langle \pi^N_{N^2t}: t\in [0,T]\rangle$.  Here, $T>0$ is time length, fixed throughout.  Let $Q^N$ be the law of the trajectory.  The first step is to show that the laws $\{Q^N: N\geq 1\}$ are tight in the space of measure-valued right-continuous trajectories with left limits, the Skorohod space $\mathcal{D}([0,T]; \mathcal{M}_+(\T^d))$.  We will show this tightness in the stronger uniform topology.

\vskip .1cm

Step 2.  Given tightness, we will show that any subsequential limit $Q$ of $Q^N$ must be supported on trajectories $\langle \pi_t: t\in [0,T]\rangle$ satisfying
\begin{equation}
\label{limpt_sep}\langle G(t,\cdot), \pi_t\rangle - \langle G(0,\cdot), \rho_0\rangle \ = \ \int_0^t \Big\langle \Big(\partial_s + \frac{1}{2}\triangle_C\Big)G(s,\cdot),\pi_s\Big\rangle ds.\end{equation}
This equation is similar to what was derived in the last subsection when $G$ did not depend on time.  However, to input into PDE uniqueness results, we will need to derive the equation with respect to this wider class of functions $G$.  Moreover, after also showing in Step 3 below that trajectories under a limit point $Q$ are absolutely continuous, we will be able to conclude the associated densities satisfy a weak formulation of \eqref{hyd_ssep}.
\vskip .1cm

Step 3.  We will show that $Q$ supports trajectories such that for each $t\in [0,T]$, $\pi_t$ is absolutely continuous with respect to Lebesgue measure on $\T^d$, and hence can be written as $\pi_t = \rho(t,u)du$, which is a priori random.
Therefore, from Step 2, $\rho(t,u)$ is a weak solution of the hydrodynamic equation \eqref{hyd_ssep}.  By uniqueness of weak solutions to the heat equation in the class of bounded solutions, we see that $\rho(t,u)$ is actually deterministic.  In particular, all subsequential limits of $Q^N$ converge to the point mass supported on the trajectory $\langle \rho(t,u)du: t\in [0,T]\rangle$.  Since tightness was proved in the uniform topology, this trajectory is continuous in time. [This can also be inferred from regularity results in PDE.]  Hence, it can be concluded, at a fixed time $t\in [0,T]$, that $\pi^N_{N^2t}$ converges weakly to a constant, the measure $\rho(t,u)du$, and hence in fact converges in probability.

\subsection{Proof of Step 2: Identification}
We recall that we have almost shown Step 2 in the last subsection.
To this end, let $G:\R_+\times \T^d \rightarrow \R$ be a smooth $C^2$ function.  By the development in Subsection \ref{sec:symm-hyd}, we obtain
   \begin{eqnarray*}
   \langle G(t,\cdot), \pi^N_{N^2t}\rangle & = &  \langle G(0,\cdot), \pi^N_0\rangle + \int_0^{t} (\partial_s+N^2L)\langle G(s,\cdot), \pi^N_{s}\rangle ds + M^N_{t}(G)\\
   & = &  \langle G(0,\cdot), \pi^N_0\rangle + \int_0^{t} \Big\langle\Big(\partial_s + \frac{N^2N^{-2}}{2}\triangle_C\Big)G(s,\cdot), \pi^N_{s}\Big\rangle ds + o(1)\\
   &=& \langle G(0,\cdot), \pi^N_0\rangle + \int_0^{t} \Big\langle \Big(\partial_s + \frac{1}{2}\triangle_C \Big)G(s,\cdot), \pi^N_{s}\big\rangle ds + o(1).
   \end{eqnarray*}
   Therefore, if $\langle\pi_t: t\in [0,T]\rangle$ is a limit point of $\langle \pi^N_t: t\in [0,T]\rangle$, we obtain \eqref{limpt_sep}.

   We now discuss more carefully some topological considerations and argue Step 1, the most difficult part. Afterwards, we concentrate on Step 3 in a subsequent subsection.

\section{Topology and compactness}

Before making calculations with respect to Step 1, we first recall some definitions and results on weak convergence, and the spaces $C\big([0,T]; \mathcal{M}_+(\T^d)\big)$ and $\mathcal{D}\big([0,T]; \mathcal{M}_+(\T^d)\big)$ with a view toward characterizing when a sequence of probability measures $Q$ on these spaces is tight.  More details can be found in \cite{Billingsley}, \cite{KL}[Section 4.1]. 

\vskip .1cm
First, we recall that a family of probability measures $Q^N$ on a metric space is relatively compact if any subsequence of the family has a weakly convergent subsequence.  We say that the family is tight if for each $\epsilon>0$ there is a compact set $K_\epsilon$ such that all measures give at least weight $1-\epsilon$ to $K_\epsilon$.  Recall that $\Omega$ is a complete metric space if all Cauchy sequences converge in $\Omega$, and $\Omega$ is a separable space if it contains a countable dense set of points.
\begin{proposition}[Prokhorov's theorem]
Let $\Omega$ be a complete, separable metric space.  Then, a family $Q^N$ of probability measures on $\Omega$ is relatively compact exactly when the family is tight.

Moreover, if $\Omega$ is a complete metric space, then a family $Q^N$ of probability measures on $\Omega$ is relatively compact when the family is tight.
\end{proposition}

We will consider in the following a generic complete, separable space $\Omega$ with metric $\updelta$. 
Often, $\Omega$ will be $\mathcal{M}_+$ the space of finite measures on $\T^d$ equipped with the metric
$$\updelta(\mu, \nu) \ = \ \sum_{k\geq 1} \frac{1}{2^k}\frac{|\langle f_k, \mu\rangle - \langle f_k, \nu\rangle|}{1+|\langle f_k, \mu\rangle - \langle f_k, \nu\rangle|}$$
where $\{f_k\}$ is a countable dense set of functions in $C(\T^d)$, the space of continuous functions on the torus $\T^d$ endowed with the `sup' metric $\|f-g\|_\infty$ between $f$ and $g$.   We may take $f_1\equiv 1$.  Here, $\mathcal{M}_+=\mathcal{M}_+(\T^d)$ is a complete, separable metric space.  Moreover, a set $K\subset \mathcal{M}_+$ is relatively compact exactly when $\sup_{\mu\in K}|\langle 1, \mu\rangle|<\infty$.

At other times, $\Omega$ may be $\R$ with usual Euclidean metric.

 Now, we consider the space $C([0,T]; \Omega)$ with the `uniform' distance,
 $$d(\pi, \chi) \ = \ \sup_{t\in [0,T]}\updelta(\pi_t,\chi_{t}).$$
 It is known that this space is a complete separable space, and therefore amenable to Prokhorov's theorem.

 Since our basic building blocks in our study of hydrodynamics involve jump processes, the space of continuous trajectories is not sufficient for our purposes.  We will often focus on the space $\mathcal{D}([0,T],\Omega)$ of right-continuous trajectories with left limits.  Unfortunately, the uniform distance will not make this space a complete, separable metric space.  Define $\Lambda$ to be the set of strictly increasing continuous functions $\lambda$ of $[0,T]$ into itself, and
 $$\|\lambda\| \ = \ \sup_{s\neq t} \left| \log \frac{\lambda(t)-\lambda(s)}{t-s}\right|.$$
 Then, define the Skorohod distance between elements in $\mathcal{D}([0,T],\mathcal{M}_+)$ as
 $$d(\pi, \chi) \ = \ \inf_{\lambda \in \Lambda}\max \left\{\|\lambda\|, \sup_{t\in [0,T]} \updelta(\pi_t, \chi_{\lambda(t)})\right\}.$$
 In some sense, the Skorohod distance compares two trajectories allowing small variation both in space and in time.  To contrast, the uniform distance only compares variation in space.  With the Skorohod distance, $\mathcal{D}([0,T], \Omega)$ is a complete, separable metric space.

How to characterize compact sets in these spaces?  Consider the following moduli of continuity:
\begin{eqnarray*}
w_\pi(\gamma) & = & \sup_{\stackrel{|t-s|\leq \gamma}{s,t\in [0,T]}}\updelta(\pi_t, \pi_s)\\
w'_\pi(\gamma) &=& \inf_{\{t_i\}_{i=0}^r}\max_{0\leq i\leq r-1}\sup_{t_i\leq s<t< t_{i+1}}\updelta(\pi_s,\pi_t)
\end{eqnarray*}
where $\{t_i\}_{i=0}^r$ refers to a partition $0=t_0< \cdots < t_r = T$ such that $t_{i+1}-t_i> \gamma$ for $0\leq i\leq r-1$.

Then, $\pi$ belongs to $C([0,T], \Omega)$ exactly when $\lim_{\gamma\downarrow 0} w_\pi(\gamma) = 0$, and $\pi$ belongs to $\mathcal{D}([0,T], \Omega)$ exactly when $\lim_{\gamma\downarrow 0} w'_\pi(\gamma) =0$.

\begin{exercise}\label{moduli_ex}\rm Relate the two moduli by showing that
$$w'_\pi(\gamma) \ \leq \ w_\pi(2\gamma).$$
\end{exercise}

\begin{exercise}\rm Show, when $\Omega=\R$ and $\pi = 1\big([0,T/2)\big)$, that $\lim_{\gamma\downarrow 0}w'_\pi(\gamma)=0$ but $\liminf_{\gamma\downarrow 0}w_\pi(\gamma)>0$.
\end{exercise}

Compact sets in these spaces, since they are complete, are characterized as follows.
\begin{proposition}
A set $A$ belonging to $C([0,T], \Omega)$ or $\mathcal{D}([0,T],\Omega)$ is compact exactly when
\begin{itemize}
\item[(a)] $\{\pi_t: \pi\in A\}$ is relatively compact in $\Omega$ for each $t\in[0,T]$.
\item[(b)] $\lim_{\gamma\downarrow 0} \sup_{\pi\in A} \bar{w}_\pi(\gamma)  =  0$
where $\bar{w}_\pi = w_\pi$ on $C([0,T], \Omega)$ and $\bar{w}_\pi = w'_\pi$ on $\mathcal{D}([0,T], \Omega)$.
\end{itemize}
\end{proposition}

We remark, when $A\subset C([0,T],[a,b])$, the above characterization reduces to the Ascoli-Arzela condition with respect to equicontinuous families of trajectories on the finite interval $[a,b]$, and (a) can be replaced by `$\{\pi_0: \pi\in A\}$ is relatively compact in $\Omega$'.

Recall now Exercise \ref{moduli_ex}.  We now have the following claim.
\begin{proposition}
\label{tightness_prop}
A family $Q^N$ of probability measures on $\mathcal{D}([0,T], \Omega)$ is tight exactly when
\begin{itemize}
\item[(1)] For each $t\in [0,T]$, the distributions of $\pi_t$ under $Q^N$ are tight.
\item[(2)] For every $\epsilon>0$, $\lim_{\gamma\downarrow 0}\lim_{N\uparrow\infty} Q^N(\pi: w'_\pi(\gamma)>\epsilon) = 0$.
\end{itemize}
Moreover, a sufficient condition for $(2)$ is
\begin{itemize}
\item[(2')] In condition $(2)$, replace $w'_\pi$ with $w_\pi$.
\end{itemize}
\end{proposition}

\begin{exercise}\rm
Observe that any limit point $Q$ of $\{Q^n\}$ satisfying $(2')$ is supported on continuous paths.  This is a well-known property; see \cite{Billingsley}.
\end{exercise}

It is not so easy to work with $(2')$ directly.  However, when $\Omega=\mathcal{M}_+$, one can understand a family $Q^N$ of probability measures on $\mathcal{D}([0,T],\mathcal{M}_+)$ by their actions on smooth functions on $\T^d$.  
\begin{proposition}\label{effective_tightness}
A family $\{Q^N\}$ of probability measures on $\mathcal{D}([0,T], \mathcal{M}_+)$ is tight if the distributions of $\{\langle G, \pi_t\rangle: t\in [0,T]\}$ under $Q^N$ for $N\geq 1$ are tight in $\mathcal{D}([0,T],\R)$, that is satisfying (1) and (2') in Proposition \ref{tightness_prop}, for each $G\in C^2(\T^d)$.
\end{proposition}

\begin{exercise}
\rm
Give a proof of Proposition \ref{effective_tightness}.  Hint: See Proposition IV.1.7 in \cite{KL}.
\end{exercise}

 \section{Proof of Step 1: Tightness}
 \label{sec:3.3.3}
We are now back to considering the tightness of $\pi^N = \langle \pi^N_{N^2t}: t\in [0,T]\rangle$ for $N\geq 1$ which are elements of $\mathcal{D}([0,T], \mathcal{M}_+)$.  From Proposition \ref{effective_tightness}, we need only show for a smooth function $G:\T^d\rightarrow \R$, tightness of the distributions of $\{\langle G, \pi^N_{N^2t}\rangle:  t\in [0,T]\}$ for $N\geq 1$ which are elements of $\mathcal{D}([0,T], \R)$.

From Proposition \ref{tightness_prop}, we need to show conditions $(1)$ and $(2')$.  Condition $(1)$ is the simplest, and follows straightforwardly as for each $t\in [0,T]$,
$$|\langle G, \pi^N_{N^2t}\rangle| \ \leq \ \|G\|_\infty\cdot \frac{1}{N^d} \sum_{x\in \T^d_N}\eta_{N^2t}(x) \ \leq \ \|G\|_\infty.$$
Then, $\langle G, \pi^N_{N^2t}\rangle$ is a tight sequence in $\R$ as the sequence is uniformly bounded in $N$ with full probability.

Verifying Condition $(2')$, and therefore tightness in the uniform topology, is a little more involved.  Since
$$\langle G, \pi_{N^2t}^N\rangle \ = \ \langle G, \pi^N_0\rangle + \frac{1}{2}\int_0^t \frac{1}{N^d}\sum_{x,y\in \T^d_N}\eta_{N^2s}(x)p(y)\triangle^N_{x,y}G ds  + M^N_t(G),$$
we need only show condition (2') for each term separately.  The initial term $\langle G, \pi^N_0\rangle$ does not contribute in this respect.

We now bound the second term. 
Condition (2') follows as
$$\sup_{|t-s|\leq \gamma} \Big| \int_s^t \frac{1}{N^d}\sum_{x,y\in \T^d_N}\eta_{N^2s}(x)p(y)\triangle^N_{x,y}Gds\Big| \ \leq \ \gamma R\|\triangle G\|_\infty$$
where $R$ is the range of the probability $p$.

For the third term, 
to treat condition (2'), we first recall Doob's inequality (cf. \cite{Durrett}):  For a martingale $(M_t, \mathcal{F}_t)$, $a,b\in[0, T]$ and $\lambda>0$, we have
\begin{eqnarray*}
\P_{\mu^N}\Big(\sup_{t\in [a,b]} |M_t-M_a|>\lambda\Big) 
&\leq& \frac{1}{\lambda^2}\E_{\mu^N}\Big[|M_{b}-M_a|^2\Big].
\end{eqnarray*}

Now, partition the interval $[0,T]$ into $\lfloor T/\gamma\rfloor +1$ divisions $\{l_k\}$ of sublength $\gamma$, assuming $\gamma$ does not divide $T$ and $l_0=0$.  Noting, when $|t-s|\leq \gamma$, that either both $t,s$ lie in the same subinterval or lie in adjacent intervals, we have
\begin{eqnarray}
\label{condition_2'}
&&\P_{\mu^N}\Big(\sup_{|t-s|\leq \gamma} |M^N_t(G)-M^N_s(G)|>\lambda\Big)\nonumber\\
 &&\ \ \ \leq \  \P_{\mu^N}\Big(\sup_{\stackrel{l_k<t\leq l_{k+1}}{0\leq k\leq \lfloor T/\gamma\rfloor}}|M^N_t(G) - M^N_{l_k}(G)|>\lambda/3\Big)\nonumber\\
 && \ \ \ \leq \ \sum_{k=0}^{\lfloor T/\gamma\rfloor} \P_{\mu^N}\Big(\sup_{l_k<t\leq l_{k+1}}|M^N_t(G) - M^N_{l_k}(G)|>\lambda/3\Big)\nonumber\\
 &&\ \ \ \ \leq \ \frac{9}{\lambda^2}\sum_{k=0}^{\lfloor T/\gamma\rfloor}\E_{\mu^N}\Big[|M^N_{l_{k+1}}(G) - M^N_{l_k}(G)|^2\Big].
 \end{eqnarray}
 Note the quadratic variation bound for $M^N_t(G) - M^N_{l_k}(G)$ near \eqref{quad_var_calc},
 \begin{eqnarray*}
\E_{\mu^N}\Big[(M^N_{l_{k+1}}(G)-M^N_{l_k}(G))^2\Big] & \leq & \E_{\mu^N}\Big[\langle M^N_{l_{k+1}}(G) - M^N_{l_k}(G)\rangle_t
\Big]\\
 &=& \E_{\mu^N}\Big[\frac{N^2}{2N^{2d +2}}\int_{l_k}^{l_{k+1}}\sum_{x,y\in \T^d_N}p(y)(\nabla^N_{x,y}G)^2du\Big]\\
&=&O\big(R\|\nabla G\|_\infty N^{-d}|l_{k+1} - l_k|\big), \end{eqnarray*}
to bound \eqref{condition_2'} of order $O(\gamma^{-1}\gamma R\|\nabla G\|_\infty N^{-d})$ which vanishes as $N\uparrow\infty$.

\section{Proof of Step 3: Absolute continuity}
To show $\pi_t$ is absolutely continuous, observe for any continuous function $G:\T^d\rightarrow \R$ that
\begin{eqnarray*}
\sup_{t\in [0,T]}|\langle G, \pi^N_{N^2t}\rangle| & \leq & \frac{1}{N^d}\sum_{x\in \T^d_N} |G(x/N)|\eta_{N^2t}(x)\\
&\leq & \|G\|_{L^1},\end{eqnarray*}
since at most one particle is allowed per site.
Hence, since the function
$$\langle \pi_t: t\in [0,T]\rangle \ \mapsto \ \sup_{t\in [0,T]} |\langle G, \pi_t\rangle |$$
is continuous with respect to the Skorohod topology, any limit point satisfies, with full probability (see Portmanteau theorem \cite{Billingsley}[Chapter 1]),
$$\sup_{t\in [0,T]}|\langle G, \pi_t\rangle| \ \leq \ \|G\|_{L^1}.$$  Hence, any limit point $Q$ of $\{Q^N\}$ is supported on trajectories with the property that $\pi_t$ is absolutely continuous for all $t\in [0,T]$.

\begin{exercise}\rm Detail the use of the Portmanteau theorem in the previous paragraph.
\end{exercise}

At this point, as noted in the initial strategy, for each $t\in [0,T]$, $\pi_t$ can be written as $\pi_t = \rho(t,u)du$ where $\rho$ may be random.  However, by Step 2, $\rho$ is a weak solution to the hydrodynamic equation, which has a unique bounded solution.  Therefore, $\rho$ is deterministic, and $\pi_t = \rho(t,u)du$ where $\rho$ is the hydrodynamic density.  We remark when $\rho_0\in C^2(\T^d)$ then $m(t,u)$ is actually a `classical' solution given in terms of a Gaussian kernel convolution with $\rho_0$.  

What we have shown is that the law of $\langle\pi^N_{N^2t}: t\in [0,T]\rangle$, where initial configurations are distributed according to $\mu^N$, converges to the point mass on trajectory $\langle\rho(t,u)du: t\in [0,T]\rangle$.  To conclude convergence at a fixed time $t\in [0,T]$, we note that in Step 1, tightness of $Q^N$ was obtained in the uniform topology.  Hence, the trajectory in the support of the limit $Q$ is continuous for all $t\in [0,T]$.  For $t\in [0,T]$, let $h_t$ be the projection function,
$$\langle \pi_t: t\in [0,T]\rangle \ \mapsto \ \pi_t.$$
Now, $h_t$ is not continuous on $\mathcal{D}([0,T]; \mathcal{M}_+)$.  However, since the limit $Q$ is supported on a continuous trajectory, $h_t$ is continuous on the support of $Q$. Now it is known that if $Q^N\Rightarrow Q$ and $h=h_t$ is a continuous function almost surely on the support of $Q$, then $Q^N\circ h_t^{-1} \Rightarrow Q\circ h_t^{-1}$ (see \cite{Billingsley}[Chapter 1]).  In other words, the projection $\pi^N_{N^2t}$ converges in law to the point mass at $\rho(t,u)du$, and therefore also in probability.

 \section{Notes}
 
There are other proofs of hydrodynamics for symmetric exclusion processes, for instance using correlation functions and `duality' as in \cite{dmp} or superexponential estimates as in \cite{Varadhan_notes}.  We have followed mostly the treatment with some caveats in \cite{KL}, following the strategy, in the symmetric exclusion context, expounded in \cite{GPV}, a classic in the field.  It is worth noting by this method we have shown existence of weak solutions to the PDE \eqref{hyd_ssep}.

Besides \cite{Billingsley}, other good references for `weak convergence' in general spaces include
 \cite{EK}, \cite{JS}, \cite{Partha}.

We note, a subject of interest has been hydrodynamics of exclusion and other processes where the jump parameters are chosen from random environments, or when there are boundary conditions and reservoirs (cf. among others, \cite{Bernardinetal},
 \cite{FJL}, \cite{Goncalves-tams}, \cite{JLTeixeira}, \cite{JP}, \cite{lpsx}, \cite{S-Xue}, \cite{ZZ}, and references therein).

\newpage
\chapter[Section $4$]{Entropy, Dirichlet form, and averaging principle in zero-range processes}
\label{lec4}

We now consider another well-studied `mass conservative' interacting particle system, the so called `zero-range' model.  Unlike for symmetric simple exclusion, the evolution equation for the mass empirical measure does not close, and one must invoke an `averaging principle' to approximate a nonlinear term in terms of a function of the empirical measure.  Here, we develop the `entropy method' for this rigorous replacement in the associated hydrodynamics.  To this end, notions of entropy and Dirichlet form, as well as the `basic coupling', are introduced before defining the zero-range process.  At the end, we give an outline to derive the associated hydrodynamical equation.   

The following Chapter \ref{lec5} is devoted to proving steps in the outline, including the famous so called `1-block' and `2-block' lemmas introduced in \cite{GPV}.

\section{Relative entropy}
Let $\Omega$ be a countable space.   For probability measures $\mu$ and $\nu$ on $\Omega$, define the relative entropy of $\mu$ with respect to $\nu$ as
\begin{equation}
\label{var_entropy}
H(\mu|\nu) \ = \ \sup_f\left\{E_\mu[f] - \log E_\nu\big[e^f\big]\right\}
\end{equation}
where the supremum is over bounded functions $f:\Omega\rightarrow\R$.   Recall, as before, $E_\kappa$ stands for the expectation under $\kappa$.

Since the expression 
\begin{equation*}
E_\mu[(f+c)] - \log E_\nu\big[e^{(f+c)}\big] \ = \ E_\mu[f] - \log E_\nu\big[e^f\big]
\end{equation*}
is invariant for all constants $c$, the above supremum can be taken over bounded nonnegative functions $f$.

\begin{exercise}\rm Show that $H(\mu|\nu)$ is nonnegative, convex and lower semi-continuous in the argument $\mu$ (where $\mu_n\rightarrow\mu$ if $\mu_n$ converges weakly to $\mu$).  Hint:  For nonnegativity, substitute $f$ equal to a constant.
\end{exercise}

\begin{lemma}[Entropy inequality]
\label{lem:entropyinequality}
For all bounded functions $f:\Omega\rightarrow \R$ and $\alpha>0$, we have
$$E_\mu[f] \ \leq \ \frac{1}{\alpha}H(\mu|\nu) + \frac{1}{\alpha}\log E_\nu\big[e^{\alpha f}\big].$$
Also, when $f= 1(A)$, for $A\subset \Omega$, one has 
$$\mu(A) \ \leq \frac{\log 2 + H(\mu|\nu)}{\log\big(1+ \frac{1}{\nu(A)}\big)}.$$
\end{lemma}

\begin{exercise}\rm
\label{exercise:entropy}
Prove this lemma.  Hint: Multiply and divide by $\alpha$.
\end{exercise}

\begin{lemma}
\label{lem:4.14} We have
$H(\mu|\nu)<\infty$ implies that $\mu \ll \nu$.  Also, if $\mu \ll \nu$, then
$$H(\mu|\nu) \ = \ \sum_{x\in \Omega} \mu(x) \log \frac{\mu(x)}{\nu(x)} = \sum_{x\in \Omega}\nu(x) \frac{\mu(x)}{\nu(x)}\log \frac{\mu(x)}{\nu(x)}.$$
Moreover, when $\Omega$ is finite, we have $\mu \ll \nu \Leftrightarrow H(\mu|\nu)= \sum_{x\in \Omega}\mu(x)\log \frac{\mu(x)}{\nu(x)}<\infty$.
\end{lemma}

\begin{proof} If $\mu\not \ll \nu$, there is an $x_0\in \Omega$ such that $\mu(x_0)>0$ but $\nu(x_0)=0$.  In the variational definition of $H(\mu|\nu)$, we can insert the function $f(x) = 1(x=x_0)$ to see $H(\mu|\nu)=\infty$.  Hence, $H(\mu|\nu)<\infty \Rightarrow \mu\ll\nu$.

We now suppose $\mu\ll \nu$ and prove the equality in the display.  
By simple truncations, approximate $H(\mu|\nu)$ in \eqref{var_entropy} by the supremum over functions which vanish except for a finite number of points in $\Omega$.  
We now find the maximum of 
\begin{equation*}
f\mapsto \sum_{x\in \Omega_k}\mu(x)f(x) - \log \Big[\sum_{x\in \Omega_k}\nu(x)e^{f(x)} + (1-\nu(\Omega_k)\Big],
\end{equation*}
over functions $f$ supported on $\Omega_k$ where $\Omega_k\subset \Omega$ and $|\Omega_k|=k<\infty$.
Indeed, the functional is concave and takes its maximum where its gradient vanishes.  Computing the gradient, we obtain the maximum occurs when
$$\mu(x_0) \ = \ \frac{\nu(x_0)e^{f(x_0)}}{\sum_{x\in\Omega_k}\nu(x)e^{f(x)} + (1-\nu(\Omega_k)}$$
for each $x_0\in \Omega_k$.  We can add a constant $c$ to $f$, noting $1-\nu(\Omega_k) = \sum_{x\not\in \Omega_k}e^{f(x)}$ and $e^{f(x)+c} = e^c$ for $x\not\in \Omega_k$, and not change this equation--the mapping is extended to functions which are constant off $\Omega_k$.
Hence, we may choose $c$ so that
$$\sum_{x\in \Omega_k}\nu(x)e^{f(x) +c} + e^c(1-\nu(\Omega_k)) \ = \ 1.$$
With this choice, we obtain $f(x_0)+c = \log\big(\mu(x_0)/\nu(x_0)\big)$ for $x_0\in \Omega_k$, with convention $0/0=1$ .  As $k\uparrow\infty$ and $\Omega_k\uparrow \Omega$, since $\log\big[\mu(\Omega_k) + e^c(1-\nu(\Omega_k))\big]$ vanishes, we have the upper bound 
\begin{align*}
H(\mu|\nu) &\leq \sum_{x\in \Omega} \mu(x)\log\frac{\mu(x)}{\nu(x)} -\lim_{k\uparrow\infty} \log\Big[\sum_{x\in \Omega_k} \nu(x)\frac{\mu(x)}{\nu(x)} + e^c(1-\nu(\Omega_k))\Big]\\  
&= \sum_{x\in \Omega}\mu(x)\log \frac{\mu(x)}{\nu(x)}.
\end{align*}

We obtain the lower bound $H(\mu|\nu)\geq \sum_{x\in \Omega}\mu(x)\log \frac{\mu(x)}{\nu(x)}$ by inserting $f = 1(x\in \Omega_k)\log\big( \mu(x)/\nu(x)\big)$ into the variational formula for $H(\mu|\nu)$, and then taking $k\uparrow\infty$.  Hence, $\mu\ll \nu$ implies that $H(\mu|\nu) = \sum_{x\in \Omega}\mu(x)\log \frac{\mu(x)}{\nu(x)}$.

The second statement follows finiteness of $\Omega$.
    \end{proof}

\section{Entropy with respect to Markov chains}
\label{sec:entropy-MC}
Let $\eta_t$ be a continuous time  Markov chain on a finite space $\Omega$.   Let $\pi$ be an invariant measure.  Recall $P_t$ and $L$ stand for the semigroup and generator of the process.  We write in the following $\langle f, g\rangle_\mu := E_\mu[fg]$.

\begin{lemma}
\label{lem:finite-entropy}
For probability measures $\mu$ such that $\mu\ll\pi$, we have $\mu P_t\ll \pi$, and 
$H(\mu P_t|\pi) \leq H(\mu|\pi)$ for $t\geq 0$.
\end{lemma}

\begin{proof}
 We show first that $\mu P_t$ is absolutely continuous with respect to $\pi$, and therefore $H(\mu P_t|\pi) = \sum_x \mu P_t(x)\log \frac{\mu P_t(x)}{\pi(x)}$ by Lemma \ref{lem:4.14}.
As $|\Omega|<\infty$, and $\mu\ll\pi$, we have 
$\max_{x\in \Omega}\mu(x)/\pi(x) \ =\ C\ < \ \infty$.
Then, by the invariance $\sum_x \pi(x)P_t(x,y) = \pi(x)$, we conclude
$$\mu P_t(y) \ = \ \sum_{x\in \Omega} \mu(x) P_t(x,y) \ \leq \ C\sum_{x\in \Omega}\pi(x)P_t(x,y) \ = \ C\pi(x),$$ 
showing the absolute continuity.

Now, write $H(\mu P_t|\pi)$ in form $\sum_x \pi(x) \Phi\big(\frac{\mu P_t(x)}{\pi(x)}\big)$, where $\Phi(u) = u\log u$ is convex.  Since 
$$\frac{\mu P_t(x)}{\pi(x)} = \sum_z  \frac{\mu(z)}{\pi(z)} \frac{\pi(z)P_t(z,x)}{\pi(x)},$$
we have to finish,
\[\sum_x \pi(x)\Phi\Big(\frac{\mu P_t(x)}{\pi(x)}\Big) \leq \sum_x \pi(x) \sum_z \frac{\pi(z)P_t(z,x)}{\pi(x)} \Phi\Big( \frac{\mu(z)}{\pi(z)} \Big) = H(\mu|\pi). \qedhere \]
\end{proof}

\begin{exercise}
\label{ex:finite-entropy}
\rm
One can update Lemma \ref{lem:finite-entropy} to countable state space $\Omega$ by showing that $\mu P_t \ll \pi$ if $\mu\ll \pi$. Perform this deduction when $\Omega$ is countable.
\end{exercise}

We now find the rate of decrease in the entropy.

\begin{lemma}
\label{entropy_diff}
Under the assumptions of Lemma \ref{lem:finite-entropy}, we have
$$\frac{d}{dt} H(\mu P_t|\pi) \ = \ \Big\langle \frac{d\mu P_t}{d\pi}, L \log \frac{d\mu P_t}{d\pi}\Big\rangle_\pi.  $$
\end{lemma}

\begin{proof}
Recall from the forward and backward equations that $\partial_t(\mu P_t(x)) = \mu P_t L(x) = \mu L P_t(x)$.  Then, as there is no problem to interchange limits and sums in finite state space,
\begin{eqnarray*}
\frac{d}{dt}\sum_x \mu P_t(x) \log \frac{\mu P_t(x)}{\pi(x)} &=& \sum_x \mu P_t L(x) \log \frac{\mu P_t(x)}{\pi(x)}\\
&& \ + \ \sum_x \mu P_t(x)\frac{\pi(x)}{\mu P_t(x)} \mu LP_t(x).
\end{eqnarray*}
The first term on the right-hand side equals
\begin{eqnarray*}\sum_x \sum_z \mu P_t(z) L(z,x) \log \frac{\mu P_t(x)}{\pi(x)} & = & \sum_z \mu P_t(z) \Big(L \log \frac{\mu P_t}{\pi}\Big)(z)\\
&=& \Big\langle \frac{d\mu P_t}{d\pi}, L\log \frac{d\mu P_t}{d\pi}\Big\rangle_\pi
\end{eqnarray*}
after interchanging the sum on $z$ and $x$.

On the other hand, the second term vanishes:
$$\sum_x \pi(x) \sum_z \mu(z) LP_t(z,x) \ = \ \sum_z \mu(z) E_\pi[Lg_z] \ = \ 0.$$
Here, $g_z = P_t(z, \cdot)$ is a function on $\Omega$ for each $z$, and so $E_\pi[Lg_z]=0$ given $\pi$ is invariant.
\end{proof}

\subsection{Dirichlet forms}
For a continuous time Markov process $\eta_t$ on a finite state space $\Omega$ with invariant measure $\pi$,
define the Dirichlet form on functions $f$ by
$$D(f) \ = \ -\langle f, Lf\rangle_\pi.$$

We now define the adjoint $L^*(x,y)$ by the relation $\pi(x)L(x,y) = \pi(y)L^*(y,x)$ for $x,y\in \Omega$.  Then, a simple computation shows
$$\langle g, Lf\rangle_\pi \ = \ \langle L^*g, f\rangle.$$

The operators $S = (L + L^*)/2$ and $A= (L-L^*)/2$ may also be defined.  If $\pi$ is reversible, $L=L^*=S$.  In particular, $S$ is a reversible operator, $\pi(x) S(x,y) = \pi(y)S(y,x)$, and $A$ is anti-symmetric in that $\pi(x)A(x,y)=-\pi(y)A(y,x)$ for all $x,y\in\Omega$.
Recall, also as we have seen before,
$$Lf(x) \ = \ \sum_{y} L(x,y)\big[f(y)-f(x)\big].$$

Then, by straightforward calculation, as also $D(f) = -\langle L^*f, f\rangle_\pi$,
\begin{eqnarray*}
D(f) &=& -\sum_{x,y} \pi(x) f(x)L(x,y)\big[ f(y)-f(x)\big]\\
&=& -\sum_{x,y} \pi(x) f(x) L^*(x,y)\big[f(y)-f(x)\big]\\
&=&   -\sum_{x,y}\pi(x)f(x)S(x,y)\big[f(y)-f(x)\big].
\end{eqnarray*}
Since $S$ is reversible, by interchanging $x$ and $y$ in the above sum and averaging the two expressions, we obtain the first of the following formulas.  The second follows by anti-symmetry of $A$.
\begin{lemma} We have
\begin{eqnarray}
\label{dirichlet_expression}
D(f) &=& \frac{1}{2}\sum_{x,y}\pi(x) S(x,y)\big[f(y)-f(x)\big]^2\nonumber\\
&=& \frac{1}{2}\sum_{x,y}\pi(x)L(x,y)\big[f(y)-f(x)\big]^2.
\end{eqnarray}
\end{lemma}

We have the following properties of the Dirichlet form.  For a nonnegative function, let $I(f) = D(\sqrt{f})$.
\begin{lemma}
We have $f\mapsto I(f)$ is nonnegative, convex, and lower semi-continuous.  When the chain is irreducible and $I(f)=0$, then $f$ is a constant function.
\end{lemma}

\begin{proof}
Nonnegativity follows from the expression \eqref{dirichlet_expression}.  After squaring out terms, convexity follows by Jensen inequality applied to the crossterm.
Clearly, the Dirichlet form is continuous, and therefore lower semi-continuous, in its argument as the space is finite.

When $I(f)=0$, we have $(\sqrt{f}(y)-\sqrt{f}(x))^2=0$ for all $x,y$ where $S(x,y)>0$.  If the chain is irreducible, all values of $f$ must be the same.
\end{proof}

\begin{remark}
\rm
Naturally, the above notions of Dirichlet form extend to countable state Markov chains, with respect to compactly supported functions $f$, and to those in $L^2(\pi)$ in the domain of the generator $L$.
\end{remark}

\subsection{Connection between entropy and Dirichlet form}
Let $\eta_t$ be a continuous time finite space Markov chain with invariant measure $\pi$.
Recall in Lemma \ref{entropy_diff}, that the derivative of $H(\mu P_t|\pi)$ equals
$$\Big\langle \frac{d\mu P_t}{d\pi}, L \log \frac{d\mu P_t}{d\pi}\Big\rangle_\pi.$$
We now bound this expression in terms of the form $I(\cdot)$.

\begin{lemma}
\label{lem:sec4-relent}
We have
\begin{eqnarray*}
\Big\langle \frac{d\mu P_t}{d\pi}, L \log \frac{d\mu P_t}{d\pi}\Big\rangle_\pi & \leq &
-2D\Big(\sqrt{\frac{d\mu P_t}{d\pi}}\Big).\end{eqnarray*}
\end{lemma}

\begin{proof} The argument follows from an application of the inequality $a(\log b - \log a) \leq 2\sqrt{a}(\sqrt{b}-\sqrt{a})$ for $a,b\geq 0$, 
which we leave to the reader.
\end{proof}

\begin{remark}
\rm
We comment that the conclusion of Lemmas \ref{entropy_diff} and \ref{lem:sec4-relent} may be compared with the more general Lemma \ref{lem:sec61-1}, which bounds the derivative of relative entropy with a not necessarily invariant measure, given in Section \ref{lec6}.  
\end{remark}

\section{Zero-range models}
\label{sec:zero-rangemodels}

We now discuss the `zero-range' system of interacting random walks on the $d$-dimensional torus $\T^d_N$ with $N^d$ locations.  Informally, particles interact infinitesimally only with those on their sites through their jump times, hence the name `zero-range'. 

 More specifically, at a vertex with $k$ particles, a particle displaces by $y$ with rate $[g(k)/k]p(y)$ where again we assume $p(\cdot)$ is a finite-range translation-invariant jump probability on $\T^d_N$ and $g: \N_0\rightarrow \R_+$ is a function such that $g(0)=0$ and $g(k)>0$ for $k\geq 1$.  An alternate description is that each vertex has its own exponential clock which rings at rate $g(k)$ when there are $k$ particles at the vertex.  When rung, one of the $k$ particle is selected at random and then it displaces according to $p(\cdot)$.

Formally, consider $\eta_t = \big\{\eta_t(x):x\in \T^d_N\big\}$ where $\eta_t(x)\in \N_0$ denotes the number of particles at $x$ at time $t\geq 0$.  The process on the countable configuration space $\Omega = \N_0^{T^d_N}$ is a continuous time Markov chain with semigroup $P_t$ and generator
\begin{equation}
\label{eq:Lf}
(Lf)(\eta) \ = \ \sum_{x,y} g(\eta(x))p(y) \big[f(\eta^{x,y})-f(\eta)\big]
\end{equation}
where $\eta^{x,y}$ is the configuration obtained from $\eta$ by moving a particle from $x$ to $y$:
$$\eta^{x,y}(z) \ = \ \left\{\begin{array}{rl} \eta(x)-1 & \ {\rm when \ }z=x\\
\eta(y)+1 & \ {\rm when \ } z=y\\
\eta(z)& {\rm \ otherwise.}\end{array}\right.
$$

\subsection{Invariant measures}
As with the exclusion process, there is an associated family of invariant measures depending on particle density.  Let
$\bar\nu^N_\Psi = \prod_{x\in \T_N^d}\kappa_\Psi$ be the product measure with common marginal on $\N_0$ given by
$$\kappa_\Psi(k) \ = \ \left\{\begin{array}{rl}
\frac{1}{Z(\Psi)} \frac{\Psi^k}{g(1)\cdots g(k)} & \ {\rm when \ }k\geq 1\\
\frac{1}{Z(\Psi)} & \ {\rm when \ }k=0.
\end{array}\right.
$$
Here, $Z(\Psi)$ is the normalization which converges for $0\leq \Psi <\liminf g(k)$.

An interesting point is that when $g(k)\equiv k$, this is the model of `independent' random walks considered in Section \ref{lec1} when there is no interaction.  In this case, of course $\kappa$ are Poisson measures.

We now index the family in terms of `density'.  Let $\rho(\Psi)$ be the density,
$E_{\bar\nu^N_{\Psi}}[\eta(0)] = \rho(\Psi)$.  One can check that the function $\Psi\mapsto \rho(\Psi)$ is a strictly increasing function which may or may not diverge when $\Psi\uparrow \liminf g(k)$, depending on the structure of $g$.
In any case, for $0\leq \Psi <\liminf g(k)$ we can invert $\rho(\Psi)$ and define
$$\nu^N_\rho \ := \ \bar\nu^N_{\Psi(\rho)},$$
the product measure which places a mean $\rho$ number of particles at each location in $\T_N^d$.  The term $\Psi(\rho)$ is sometimes called a `fugacity' or `chemical potential'. One may calculate that is also the mean of the rate function: $E_{\nu^N_\rho}[g(\eta(0))] = \Psi(\rho)$.  In the following, we drop the superscript `$N$' and write $\nu_\rho = \nu^N_\rho$ to simplify notation.

\begin{exercise}\label{ex:zr-particle}
\rm
For bounded, local $f$, that is $f$ supported only on a finite number of variables $\{\eta(x): x\in \T^d_N\}$, 
show
$$E_{\nu_\rho}\big[g(\eta(x))f(\eta^{x,y})\big] = \Psi(\rho) E_{\nu_\rho}\big[f(\eta + \delta_y)\big] = E_{\nu_\rho}\big[g(\eta(y))f(\eta)\big]$$
where $\eta^{x,y} = \eta -\delta_x + \delta_y$ and $\delta_z$ is the configuration with exactly one particle at $z$.  
\end{exercise}

\begin{lemma}
\label{lem:finite_inv}
Let $\rho$ be such that $\Psi(\rho)<\liminf g(k)$.  Then, $\nu_\rho$ is an invariant measure for process on $\T_N^d$.  

Moreover, $\nu_{N,K}:= \nu_\rho\big(\cdot|\sum_{x\in \T^d_N}\eta(x)=K\big)$ does not depend on $\rho$ and is the unique invariant measure on $\Omega_K=\big\{\eta: \sum_{x\in \T^d_N}\eta(x)=K\big\}$ when $p(\cdot)$ is irreducible.
\end{lemma}

\begin{proof}  Recall the expression for $Lf(\eta)$ in \eqref{eq:Lf}.
We need only show that
$E_{\nu_\rho}[Lf(\eta)] = 0$ for all 
bounded local functions $f$.   Since $p(x,y)=p(y-x)$ is translation-invariant, it is doubly stochastic, $\sum_{x\in \T_N^d} p(x,y) = \sum_{y\in \T_N^d} p(x,y) = 1$.

Noting Exercise \ref{ex:zr-particle},
we have
\begin{align*}
\sum_{x,y\in \T_N^d} p(x,y) E_{\nu_\rho}\big[g(\eta(x))f(\eta^{x,y})\big] &=  \sum_{x,y\in \T_N^d} p(x,y)E_{\nu_\rho}\big[g(\eta(y))f(\eta)\big]\\
&= \sum_{y\in \T_N^d} E_{\nu_\rho}\big[g(\eta(y))f(\eta)\big].
\end{align*}
But, also,
$\sum_{x,y\in \T_N^d} p(x,y)E_{\nu_\rho}\big[g(\eta(x))f(\eta)\big] = \sum_{x\in \T_N^d}E_{\nu_\rho}\big[g(\eta(x))f(\eta)\big]$.  Hence, $E_{\nu_\rho}[Lf]=0$.

The second statement for $\nu_{N,K}$,  the restriction of $\nu_\rho$ to $\Omega_K$, as for independent particles, follows by noting $\Omega_K$ is a closed, irreducible set when $p(\cdot)$ is irreducible.
 \end{proof}

\noindent {\it Local equilibrium.} In the following, with respect to a continuous nonnegative function $\rho_0:\T^d \rightarrow \big[0,\Psi^{-1}(\liminf g(k))\big)$, define
$$\nu^N_{\rho_0(\cdot)}= \prod_{x\in \T_N^d} \kappa_{\Psi(\rho_0(x/N))}$$ 
to be the (inhomogeneous) product measure with marginals $\kappa_{\Psi(\rho_0(x/N))}$ over sites $x\in \T^d_N$.  Sometimes $\nu^N_{\rho_0(\cdot)}$ is referred to as a `local equilibrium' measure.  Again, we will drop the superscript `$N$' and write $\nu_{\rho_0(\cdot)} = \nu^N_{\rho_0(\cdot)}$ to simplify notation.

\begin{remark}
\label{ex:ginc}
\rm
We comment that divergence of $\rho(\Psi)$ as $\Psi\uparrow\liminf g(k)$ allows to define $\nu_\rho$ and $\nu_{\rho_0(\cdot)}$ for any $0\leq \rho<\infty$ and $\rho_0:\T^d \rightarrow [0,\infty)$, useful in the formulation of the hydrodynamic limit.

A sufficient condition is that $Z(\Psi)$ diverges as $\Psi \uparrow\liminf g(k)$.  This is the case when $g$ is an increasing function.  See Lemma II.3.3 and Section II.3 in \cite{KL} for an argument and more discussion.
  \end{remark}

\begin{exercise}
\rm
Suppose $g(k) = 1(k\geq 1)$.  Then, $\liminf g(k)=\lim g(k)=1$. 
By explicit calculation, show that $\rho(\Psi)$ diverges as $\Psi\uparrow 1$.  
\end{exercise}

\subsection{Connection with entropy}

The next results will be used to bound a certain `entropy production', discussed in Section \ref{lec5}, when starting from $\nu_{\rho_0(\cdot)}$.  Since here $\Omega = \N_0^{\T_N^d}$ is countable but not finite, we now update the results in Subsection \ref{sec:entropy-MC} to this context.

\begin{lemma}
\label{ex:4.1}
 Consider $\rho>0$ and $\bar \rho = \|\rho_0(\cdot)\|_\infty$ such that $\Psi(\rho), \Psi(\bar \rho)< \liminf g(k)$.  
 Then, $\nu_{\rho_0(\cdot)}\ll \nu_\rho$, and $H(\nu_{\rho_0(\cdot)}P_t | \nu_\rho) \leq H(\nu_{\rho_0(\cdot)}|\nu_\rho) = O(N^d)$ for $t\geq 0$.
\end{lemma}

\begin{proof}
The absolute continuity follows as $\nu_\rho(\eta)>0$ for all $\eta\in \Omega$.  The bound $H(\nu_{\rho_0(\cdot)}P_t|\nu_\rho)\leq H(\nu_{\rho_0(\cdot)}|\nu_\rho)$ follows by the proof of Lemma \ref{lem:finite-entropy} and Exercise \ref{ex:finite-entropy}.  The last part is left as an exercise.
\end{proof}

\begin{exercise}\rm
Prove $H(\nu_{\rho_0(\cdot)}|\nu_\rho)=O(N^d)$ by explicit calculation when $\rho = \bar \rho = \|\rho_0(\cdot)\|_\infty$.  Update to other densities $\rho>0$, by use of the entropy inequality in Lemma \ref{lem:entropyinequality} with $\alpha = \alpha(\rho, \bar \rho)>0$ chosen small.  Hint: See Remark V.1.2 in \cite{KL}.
\end{exercise}

Extend the definition of $I(\cdot)$ with respect to \eqref{dirichlet_expression} in the current setting:
\begin{equation}
\label{eq:I-def}
I(h) = D_K\big(\sqrt{h}\big) = \frac{1}{2}\sum_{x,y}p(x,y)E_{\nu_{N,K}}\Big[g(\eta(x))\big(\sqrt{h(\eta^{x,y})}-\sqrt{h(\eta)}\big)^2\Big].
\end{equation}

\begin{lemma}
\label{lem:4.3.6}
Under the assumptions of Lemma \ref{ex:4.1}, we have, for $0\leq r\leq t$,
\begin{align*} 
H(\nu_{\rho_0(\cdot)}P_t| \nu_\rho) -H(\nu_{\rho_0(\cdot)}P_r|\nu_\rho)
&
= \int_r^t \Big\langle \frac{d\nu_{\rho_0(\cdot)}P_s}{d\nu_\rho}, L \log \frac{d\nu_{\rho_0(\cdot)}P_s}{d\nu_\rho}\Big\rangle_{\nu_\rho} ds\\
& \leq -2\int_r^tI\Big(\sqrt{\frac{d\nu_{\rho_0(\cdot)}P_s}{d\nu_\rho}}\Big)ds. 
\end{align*}
\end{lemma}

\begin{proof}
By Lemma \ref{ex:4.1}, $\nu_{\rho_0(\cdot)}\ll \nu_\rho$ and therefore by Lemma \ref{lem:4.14}, we may write
$$H(\nu_{\rho_0(\cdot)}P_t| \nu_\rho) = \sum_\eta \nu_{\rho_0(\cdot)}(\eta)\log f_t(\eta)<\infty$$
where $f_t(\eta)={\nu_{\rho_0(\cdot)}P_t(\eta)}/{\nu_\rho (\eta)}$ for each $t\geq 0$.  
Decomposing on the number of particles in $\T_N^d$, the entropy $H(\nu_{\rho_0(\cdot)}P_t|\nu_\rho)$ equals
$$\sum_K \nu_\rho(\Omega_K) \sum_{\eta\in \Omega_K}\nu_{N,K}(\eta)f_t(\eta)\log f_t(\eta) = \sum_K\nu_\rho(\Omega_K) E_{\nu_{N,K}}\big[ f_t \log f_t\big].$$
Recall, here $\Omega_K=\{\eta: \sum_x \eta(x) = K\}$ is a finite set, and $\nu_{N,K}(\cdot) = \nu_\rho(\cdot | \Omega_K)$.  Let also $\nu_{\rho_0, K}(\cdot) = \nu_{\rho_0(\cdot)}(\cdot | \Omega_K)$ and $f_{t, K}(\eta) = \nu_{\rho_0, K}P_t(\eta)/\nu_{N, K}(\eta)$.

On $\Omega_K$, note by mass conservation that $\nu_{\rho_0(\cdot)}P_t(\eta) = \nu_{\rho_0(\cdot)}(\Omega_K) \nu_{\rho_0, K}P_t(\eta)$.  Then, on $\Omega_K$, $f_t(\eta) = \frac{\nu_{\rho_0(\cdot)}(\Omega_K)}{\nu_\rho(\Omega_K)} f_{t, K}(\eta)$.  Therefore, inputting these relations, 
\begin{align*}
E_{\nu_{N,K}}\big[ f_t \log f_t\big] &=  \frac{\nu_{\rho_0(\cdot)}(\Omega_K)}{\nu_\rho(\Omega_K)} \log
\frac{\nu_{\rho_0(\cdot)}(\Omega_K)}{\nu_\rho(\Omega_K)}  +
\frac{\nu_{\rho_0(\cdot)}(\Omega_K)}{\nu_\rho(\Omega_K)} E_{\nu_{N, K}}\big[ f_{t, K}\log f_{t, K}\big].
\end{align*}
Since $\sum_K \nu_\rho(\Omega)\big[ \frac{\nu_{\rho_0(\cdot)}(\Omega_K)}{\nu_\rho(\Omega_K)} \log
\frac{\nu_{\rho_0(\cdot)}(\Omega_K)}{\nu_\rho(\Omega_K)} \big]$ is itself a relative entropy of the measures $\mu_1(K)= \nu_{\rho_0(\cdot)}(\Omega_K)$ and $\mu_2(K)=\nu_\rho(\Omega_K)$ on $\N_0$, it is nonnegative.  Also, the relative entropy $E_{\nu_{N,K}}\big[f_{t,K}\log f_{t,K}\big]\geq 0$ for each $K\geq 0$.  Note also that $E_{\nu_{N,K}}\big[f_t\log f_t\big]$, as $\Omega_K$ is a finite space, is differentiable.

 Hence, given finiteness of $H(\nu_{\rho_0(\cdot)}P_t|\nu_\rho)$, we may write it as a sum of two finite, nonnegative terms, one of which, $H(\mu_1|\mu_2)$, does not depend on $t$.
Therefore, we are able to write the difference
\begin{align*}
H(\nu_{\rho_0(\cdot)}P_{t}| \nu_\rho) - H(\nu_{\rho_0(\cdot)}P_r| \nu_\rho)
&= \sum_K \nu_\rho(\Omega_K)\int_r^{t} \partial_s E_{\nu_{N, K}}\big[f_s\log f_s\big] ds.
\end{align*}
The right-hand side, by 
the arguments of Lemmas \ref{entropy_diff} and \ref{lem:sec4-relent}, since $\nu_{N,K}$ is an invariant measure on $\Omega_K$, is evaluated as
\begin{align*}
  \sum_K \nu_\rho(\Omega_K) \int_r^{t} E_{\nu_{N,K}}\big[f_s L \log f_s\big]ds
& \leq -2\sum_K \nu_\rho(\Omega_K)\int_r^{t}  D_K\big(\sqrt{f_s}\big)ds.
\end{align*}

Since the summands, as $D_K(\cdot)\geq 0$, are all negative, we may interchange the sum on $K$ and the integral.  The right-hand side above equals
$\int_r^{t} -2D\big(\sqrt{f_s}\big)ds$.
We obtain the desired statement as
$I(f_s) = D\big(\sqrt{f_s}\big)$.
\end{proof}

\begin{remark}\rm
Since $P_t$ and so $f_t$ is continuous in $t$, by Fatou's lemma, one may bound the upper derivative:
\[\limsup_{\delta\downarrow 0} \frac{1}{\delta}\big[H(\nu_{\rho_0(\cdot)}P_{t+\delta}|\nu_\rho) - H(\nu_{\rho_0(\cdot)}P_t|\nu_\rho)\big] \leq -2 I\big(\sqrt{f_t}\big). \]
See also discussion in Appendix I.1.9 in \cite{KL}.
\end{remark}

\section{Basic coupling}
\label{coupling}
We discuss a useful coupling in the context of zero-range processes when $g$ is an increasing function.  First, we define the notion of `stochastic domination' on the partially ordered set $\Omega$.  We say two probability measures $\mu_1$, $\mu_2$ on a countable set $\Omega$ are ordered, $\mu_1 \ll \mu_2$, if $E_{\mu_1}[f] \leq E_{\mu_2}[f]$ for all coordinatewise increasing functions $f:\Omega \rightarrow \R$.

There is an interesting characterization of ordered measures for whose proof we refer to \cite{Liggett1}[Theorem II.2.4].
\begin{proposition}
\label{Liggett_prop}
We have $\mu_1\ll \mu_2$ exactly when there is a bivariate distribution on $\Omega\times \Omega$ such that marginally variables $\eta$ and $\xi$ are distributed according to $\mu_1$ and $\mu_2$ and
$P(\eta \leq \xi) = 1$.
\end{proposition}

In this next result, $g$ does not have to be increasing.

\begin{lemma} Let $\alpha, \beta< \liminf g(k)$.  The marginals $\kappa_\alpha \ll\kappa_\beta$ exactly when $\alpha\leq \beta$.
\end{lemma}

\begin{proof}
If $\kappa_\alpha \ll\kappa_\beta$, since $f(\eta)=\eta(0)$ is increasing, we have $\rho(\alpha)\leq \rho(\beta)$ from which one deduces $\alpha\leq \beta$.

For the converse, since linear combinations of $1(\eta\in [L,\infty))$ are dense in the set of increasing functions, it is enough to show that
$$P_{\kappa_\beta}(\eta\geq L) \ \geq \ P_{\kappa_\alpha}(\eta\geq L)$$
for all $L\geq 0$.
This is equivalent, after cancellation, to showing
$$\sum_{k\geq L} \frac{\beta^k}{g(1)\cdots g(k)} Z(\alpha) \ \geq \ \sum_{k\geq L} \frac{\alpha^k}{g(1)\cdots g(k)} Z(\beta) $$
which is equivalent to 
$$\sum_{k\geq L}\sum_{\ell\leq L-1} \frac{\beta^k \alpha^\ell}{g(k)!g(\ell)!} \ \geq \ \sum_{k\geq L}\sum_{\ell\leq L-1} \frac{\alpha^k\beta^\ell}{g(k)!g(\ell)!}.$$
The last inequality will be shown if term by term the inequality is true.  But,
$\beta^k \alpha^\ell \geq \alpha^k\beta^\ell$ since $k\geq \ell$.\end{proof}

We now show the existence of the so-called `basic coupling' for the zero-range process when $g$ is increasing.  Another name for such an existence is that the zero-range process is `attractive'.

\begin{proposition}
 When $g$ is increasing, there is a joint process $(\eta_t, \xi_t)$ on $\Omega\times \Omega$, starting from $\mu_1\times \mu_2$ where $\mu_1\ll \mu_2$, such that at all later times $t\geq 0$, $\eta_t$ and $\xi_t$ marginally are the zero-range processes starting from $\mu_1$ and $\mu_2$ respectively and $\eta_t\leq \xi_t$ a.s.
\end{proposition}

\begin{proof}
Let $\bar\mu = \mu_1\times \mu_2$ be the joint probability on $\Omega\times \Omega$ such that $\eta_0\leq \xi_0$ a.s. (Proposition \ref{Liggett_prop}).  Consider the Markov process on $\Omega\times \Omega$ with initial distribution $\bar\mu$ and generator
\begin{eqnarray*}
(\bar L_N) f(\eta,\xi) &=& \sum_{x,y\in \T^d_N}p(y)\min\{g(\eta(x)),g(\xi(x))\}\big[f(\eta^{x,x+y},\xi^{x,x+y})-f\big]\\
&&\ + \ \sum_{x,y\in \T^d_N}p(y)\big(g(\eta(x))-g(\xi(x))\big)_+\big[f(\eta^{x,x+y},\xi)-f\big]\\
&& \ + \ \sum_{x,y\in \T^d_N}p(y)\big(g(\xi(x))-g(\eta(x))\big)_+\big[f(\eta,\xi^{x,x+y})-f\big].
\end{eqnarray*}
One can see, by inserting a function of coordinate $\eta$ only or of coordinate $\xi$ only, that the marginal processes are as desired.

To check that $\eta_t \leq \xi_t$ a.s., recall the joint process is a continuous time Markov chain on a countable state space.  At time $t=0$, we have arranged that $\eta_{0}\leq \xi_{0}$.  At the next jump time $\tau$, by the specification of the rates, we see that still $\eta_\tau \leq \xi_\tau$.  Hence, the joint process remains ordered for all time $t\geq 0$.
\end{proof}

\begin{exercise}\rm There is another way to show that $\eta_t\leq \xi_t$ for all $t\geq 0$ a.s.  (cf. Theorem II.5.2 in \cite{KL}).  Compute that $\bar L_N 1(\eta\leq \xi)\geq 0$.  Then, with $\bar E_{\bar \mu}$ denoting the process expectation starting from $\bar \mu$,
$\bar E_{\bar\mu}[1(\eta_t\leq \xi_t)]$ is increasing in $t$:  $\partial_t\bar E_{\bar\mu}[1(\eta_t\leq \xi_t)] = \bar E_{\bar \mu}[\bar L_N1(\eta_t\leq \xi_t)] \geq 0$.  Hence,
$$1\geq \bar E_{\bar\mu}[1(\eta_t\leq \xi_t)] \ \geq\  \bar E_{\bar\mu}[1(\eta_0\leq \xi_0)] \ = \ 1.$$
\end{exercise}

The next exercise will be useful in the later proof of hydrodynamics.  Recall that $\nu_\rho$ and $\nu_{\rho_0(\cdot)}$ are well-defined when $g$ is increasing by Remark \ref{ex:ginc}. 
\begin{exercise}
\label{ex:464}
\rm
Let $\bar\rho = \|\rho_0(\cdot)\|_\infty$.  Show that $\nu_{\rho_0(\cdot)}\ll \nu_{\bar\rho}$, and consequently $\P_{\nu_{\rho_0(\cdot)}}\ll\P_{\nu_{\bar\rho}}$, when $g$ is increasing.
\end{exercise}

We will also need later an estimate on $\Psi(\rho)= E_{\nu_\rho}[g(\eta(0))]$ if $g$ is a Lipschitz function, $|g(k)-g(l)|\leq C|k-l|$ for all $k,l\geq 0$.
\begin{lemma}
\label{lem:4.4.5}
If $g$ is Lipschitz, then $\Psi$ is also Lipschitz.
\end{lemma}
\begin{proof}
Let $\delta \geq \beta$.  Then, by the basic coupling,
\begin{align*}
\Psi(\delta) - \Psi(\beta) &= E_{\nu_\delta}[g] - E_{\nu_\beta}[g]\\
&=\bar E[g(\eta(0)) - g(\xi(0))]\\
&\leq \bar E\big[|\eta(0) - \xi(0)|\big] \\
&= \bar E[ \eta(0)-\xi(0)]\ = \ \delta - \beta. \qedhere
\end{align*} 
\end{proof}

As we have seen, the zero-range model allows several `closed form' calculations. The following is another one of these. We will not use this exercise in the sequel.

\begin{exercise}\rm
Calculate that $\Psi'(\rho) = \Psi(\rho)/\sigma^2(\rho)$ where $\sigma^2(\rho)$ is the variance of $\eta(0)$ under $\nu_\rho$.
\end{exercise}

\section{What is the hydrodynamical equation?}
\label{sec:4.7}
We will start the process from initial configurations distributed according to $\mu^N=\nu_{\rho_0(\cdot)}$.  
Denote, as before, for fixed $T<\infty$, that $\P_{\nu_{\rho_0(\cdot)}}$ and $\E_{\nu_{\rho_0(\cdot)}}$ are the distribution of $\{\eta_t: t\in [0,T]\}$ and its expectation, when $\eta_0$ is governed by $\mu$.

We now informally compute the generator action to guess the hydrodynamic behavior.  Recall $\pi^N_t$ stands for the empirical measure
$$\pi^N_t \ = \ \frac{1}{N^d}\sum_{x\in \T^d_N} \eta_t(x)\delta_{x/N}.$$
For a smooth $G:\T^d\rightarrow\R$, consider the equation
$$\langle G,\pi^N_{v(N)t}\rangle \ = \ \langle G, \pi^N_0\rangle + v(N)\int_0^t L \langle G, \pi^N_{v(N)s}\rangle ds + M^N_{t}(G)$$
where $M^N_{t}(G)$ is a martingale.
One may compute the quadratic variation,
\begin{align*}
&\langle M^N(G)\rangle_{t}  = 
v(N)\int_0^t \Big[L\big(\langle G, \pi^N_{v(N)s}\rangle\big)^2 - 2 \langle G, \pi^N_{v(N)t}\rangle L\langle G, \pi^N_{v(N)s}\rangle \Big]ds\\ 
&\quad = \frac{v(N)}{N^{2d+2}}\int_0^t\sum_{x,y} \big(\nabla^N_{x,x+y}G\big)^2g(\eta_{v(N)s}(x))p(y)ds  =  O(v(N)tN^{-d-2}),
\end{align*}
which shows, easily if $g$ is bounded, that the martingale is negligible in the limit.

We now compute
\begin{eqnarray*}v(N)L\langle G, \pi^N_{v(N)t}\rangle & = & \frac{v(N)}{N^{d+1}}\sum_{x,y} g(\eta(x))p(y)\big[N(G(x+y/N) - G(x/N))\big].
\end{eqnarray*}
Here, again $v(N)=N^2$ when $p$ is symmetric (or mean-zero), and $v(N)=N$ otherwise.

When $p$ is symmetric, similar to symmetric exclusion, one can change $y$ to $-y$ and average.
The right-side equals
\begin{eqnarray*}&&\frac{v(N)}{2N^{d+2}}\sum_{x,y}g(\eta(x))p(y)\triangle^N_{x,y}G.
\end{eqnarray*}
where we recall $\triangle^N_{x,y}G = N^2[G(x+y/N) - 2G(x/N)+G(x-y/N)]$.

The trouble now is that this weighted average of functions $g(\eta(x))$ does not close in terms of the empirical measure.  The main point of `hydrodynamics' is to approximate this average by a function of the empirical measure.

What should it be?  In the microscopic scale, near point $x\in \T^d_N$ there is lots of local particle movement.  One expects the distribution of particles in an $N\epsilon$ neighborhood of $x$ at time $N^2t$ to be in `local' equilibrium given by the nearby density, such as $\nu_{\eta_{N^2t}^{N\epsilon}(x)}$ where
\begin{equation}
\label{eq:eta-l}
\eta^\ell(x) = \frac{1}{(2\ell +1)^d}\sum_{|y-x|\leq \ell}\eta(x)
\end{equation}
and $\epsilon>0$ is a small parameter.

Then, one might expect that
$$\frac{v(N)}{2N^{d+2}}\sum_{x,y}g(\eta(x))p(y)\triangle^N_{x,x+y}G \ \sim \ \frac{1}{2N^d}\sum_{x,y} \Psi\big(\eta_{N^2t}^{N\epsilon}(x)\big)p(y)\triangle^N_{x,x+y}G$$
where we recall
$\Psi(\rho)= E_{\nu_\rho}[g(\eta(0))]$.

Now, this `local' density, for $\epsilon>0$, can be re-expressed in terms of the empirical measure:
\begin{equation}
\label{eq:4.eta}
\eta^{N\epsilon}_{N^2t}(x) \ = \ \frac{(2\epsilon)^dN^d}{(2N\epsilon+1)^d}\big\langle (2\epsilon)^{-d}1([x/N-\epsilon,x/N+\epsilon]^d), \pi^N_{N^2t}\big\rangle
\end{equation}
which if $\pi^N_{N^2t}(x) \sim \pi(t,u)du$ is further approximated by
$$\frac{1}{(2\epsilon)^d}\int 1_{[-\epsilon,\epsilon]^d}(u-x/N) \pi(t,u)du.$$

Putting this together, one `closes' the equation and obtains, after first $N\uparrow\infty$ and then $\epsilon\downarrow 0$, that
$$\int_{\T^d}G(u)\pi(t,u)du \ = \ \int_{\T^d}G(u)\rho_0(u)du + \frac{1}{2}\int_0^t \triangle_C G(u) \Psi(\pi(t,u))du$$
which is the weak formulation of 
$$\partial_t \rho(t,u) \ = \ \frac{1}{2}\triangle_C \Psi\big(\rho(t,u)\big) \ \ \ {\rm and \ \ \ } \rho(0,u)=\rho_0(u)$$
where $\triangle_C$ is as defined before,
$\triangle_C  =  \sum_{i\leq i,j\leq d} C_{i,j}\frac{\partial^2}{\partial_{x_i}\partial_{x_j}}$
and $C_{i,j} = \sum_zz_iz_jp(z)$.

\section{Assumptions and statement of hydrodynamics}
\label{sec:4.8}

We will make the following assumptions to help simplify the proof and introduce main ideas.  We remark however that hydrodynamics has been shown for much more general zero-range processes; see the later Notes for details on extensions.

We will assume that $g$ is a Lipschitz function, bounded and increasing, with invariant measures at all densities $\rho\geq 0$:
\begin{itemize}
\item 
$|g(k+1)-g(k)| \ \leq \ a_0$
for $k\geq 0$
\item $\|g\|_\infty=\sup_{k\geq 0} g(k) \ \leq \ a_1$,
\item $g(k+1)\geq g(k) \ {\rm for \ }k\geq 0$.
\end{itemize}
Note by Remark \ref{ex:ginc} that $\rho(\Psi)\uparrow\infty$ as $\Psi\uparrow\lim g(k)$, and so $\nu_\rho$ is defined for all densities $0\leq \rho<\infty$. 

A standard example is $g(k) = 1(k\geq 1)$, under which $\kappa_\Psi$ is Geometric, that is $\kappa_\Psi(k) = (1-\Psi)\Psi^k$ for $k\geq 0$ and $\Psi\in [0,1)$.

Also, to reduce notation, we assume $p$ is nearest-neighbor and symmetric:
\begin{itemize}
\item $p(e) \ = \ (2d)^{-1}$ for coordinate unit vectors $e$, so that $\triangle_C  =\frac{1}{d} \triangle$.
\end{itemize}
Recall that $\P_\mu$ stands for the process measure when starting in $\mu$.

\begin{theorem}
\label{thm:481}
We have for all $G\in C^2(\T^d)$ and $\delta>0$ that
$$\lim_{N\uparrow\infty} \P_{\nu_{\rho_0(\cdot)}}\Big[ \Big| \langle G, \pi^N_{N^2t}\rangle - \int_{\T^d} G(u)\rho(t,u)du\Big|>\delta\Big] \ = \ 0$$
where $\rho$ is the unique weak solution of the hydrodynamic equation
$$\partial_t\rho \ = \ \frac{1}{2d}\triangle \Psi(\rho) \ \ \ {\rm and \  \ \ } \rho(0,u) \ = \ \rho_0(u).$$
\end{theorem}

\subsection{Strategy}
We now separate the proof into two rough steps.

\medskip 
Step 1.  We will again consider the measures $Q^N$ which govern the trajectories of empirical distributions $\pi^N_{N^2t}$ on $D([0,T]; M_+(\T^d))$.  The first task is to show that $\{Q^N\}$ is tight, and all limit points are concentrated on trajectories of measures with densities $\pi(t,u)du$.  
\medskip

Step 2.  We show that all limit points are supported on weak solutions to the hydrodynamical equation in $L^2([0,T]\times \T^d)$.  It is a result in PDE that such weak solutions are unique.  Hence, the measures $Q^N$ converge to a single limit point.
One concludes convergence at a fixed time, as for simple exclusion processes.
The development on entropy, Dirichlet forms, and coupling will be useful in proving Step 2 in Section \ref{lec5}.

 \section{Notes}
The zero-range process was introduced in \cite{Spitzer}; see also \cite{Andjel}, \cite{Evans-Hanney}.  It has been a good vehicle to study a variety of phenomena, beyond hydrodynamics, such as condensation and metastability (cf. \cite{AGL}, \cite{Beltran}, \cite{Lan-meta}, \cite{Seo}, \cite{Loulakis-hyd}).

Much of the development of entropy and its use in various applications, including hydrodynamics, stems from the work of Guo, Papanicolaou, and Varadhan (GPV) \cite{GPV}.  
On the other hand the basic coupling, and its use in particle systems analysis was advanced by Liggett \cite{Liggett1}.

In terms of hydrodynamics, one may weaken the assumptions on $g$.  In particular, the increasing or boundedness assumptions are not needed when $p$ is mean-zero; see \cite{KL} and the Notes of Section \ref{lec5}.  See the Notes of Section \ref{lec6} for remarks when $p$ is asymmetric and not mean-zero.

There are at least five different proofs of the hydrodynamic behavior when $p$ say is symmetric.  One way, the `entropy' method \cite{GPV}, to be outlined in Section \ref{lec5}, shows that the limiting hydrodynamic density satisfies a weak form of hydrodynamical PDE.  Another way, the `relative entropy' method \cite{yau}, assumes that a classical, smooth solution $\rho=\rho(t,u)$ exists, and then shows that the relative entropy of $\nu_{\rho_0(\cdot)}$ with respect to $\nu_{\rho(\cdot)}$ vanishes; see Section \ref{lec6}.  One may use a Hopf-Lax formula, and the basic coupling, to prove in $d=1$ the hydrodynamical limit for systems with drift \cite{Seppalainen}, \cite{Andjel-Vares}, \cite{Ravi}.  A method using a logarithmic Sobolev inequality and compensated compactness ideas can be employed in $d=1$ \cite{Fritz}, \cite{Fritz-Nagy}.  Also, the hydrodynamic limit may be seen as a `gradient flow', and found via Gamma convergence methods \cite{Menz}, \cite{Simon}, \cite{Grunewaldetal}.  The first two methods are also discussed in \cite{KL}.  See also \cite{dmp} for a treatment of the GPV `entropy' method, and \cite{Funaki}, \cite{Varadhan_notes} for treatments of the `relative entropy' method.

\newpage

\chapter[Section $5$]{Entropy method for hydrodynamics of zero-range processes}
\label{lec5}

A rigorous proof of hydrodynamics (Theorem \ref{thm:481}) is given for a class of zero-range models through the `entropy' method of \cite{GPV}.  The main idea is that the drift computed from the generator action, an average of a nonlinear function of the occupation variables, may be understood in the scaling limit as a function of the limiting empirical density. 

 This sort of ergodic theorem, which proceeds in two replacement steps, the `1-block' and `2-block' lemmas which introduce additional scales, is facilitated by estimates on the Dirichlet form of a nonstationary Radon-Nikodym density function, and local central limit theorem asymptotics.

\section{Proof of Step 1: Tightness and absolute continuity}

Recall the notation from the Section \ref{lec4}, in particular the development and assumptions from Subsections \ref{sec:4.7} and \ref{sec:4.8}.  As before $E_\mu$ stands for the expectation under $\mu$, and $\P_\mu$ and $\E_\mu$ denote the process measure and expectation when starting in $\mu$.

Analogously as for simple exclusion, the first step can be divided into the tasks:
\begin{itemize}
\item Show that the measures $Q^N$ on $D([0,T]; \M_+(\T^d))$ which govern $\langle\pi_{N^2t}: t\in [0,T]\rangle$ starting from local equilibrium $\nu_{\rho_0(\cdot)}$ are tight.

\item Show that all limit points of $\{Q_N\}$ are supported on trajectories with densities $\pi(t,u)du$.
\end{itemize}

\begin{lemma} $\{Q^N\}$ is tight.
\end{lemma}

\begin{proof}  Considering Proposition \ref{effective_tightness}, we need only show items (1) for each $t\in [0,T]$ that $\langle G, \pi^N_{N^2t}\rangle$ is tight on $\R$, and
$${\rm (2')\ \ } \lim_{\gamma\downarrow0}\lim_{N\uparrow\infty}\P_{\nu_{\rho_0(\cdot)}}\Big(\sup_{|t-s|\leq \gamma} \big| \langle G, \pi^N_{N^2t}\rangle - \langle G, \pi^N_{N^2s}\rangle \big| > \epsilon\Big) \ = \ 0.$$

Now, by the basic coupling and Exercise \ref{ex:464}, we have that $\nu_{\rho_0(\cdot)}\ll \nu_{\bar\rho}$ where $\bar\rho = \|\rho_0\|_\infty$.
Hence, (1) follows as the first moments, $\E_{\nu_{\rho_0(\cdot)}}[\eta_{N^2t}(x)] \leq E_{\nu_{\bar \rho}}[\eta(0)]$, are uniformly bounded:
$$\E_{\nu_{\rho_0(\cdot)}}|\langle G, \pi^N_{N^2t}\rangle | \ \leq \ \frac{1}{N^d}\sum_{x\in \T^d_N} |G(x/N)| E_{\nu_{\bar\rho}}[\eta(0)] \ \leq \ C(G, \bar\rho).$$

To establish (2'), we consider, as we did for the exclusion process, the drift and martingale terms separately.
For the drift term, since $g$ is bounded, we have the uniform bound on $N$,
\begin{align*}
\sup_{|t-s|\leq \gamma} v(N)\int_s^t\big| L\langle G, \pi^N_{N^2u}\rangle \big| du & \leq \gamma \frac{v(N)}{2N^{d+2}}\sum_{x,y}p(y) \big(\triangle^N_{x,x+y}G\big)\|g\|_\infty \\
&\leq \gamma {N} C(R,g, p).
\end{align*}

Noting the form of the quadratic variation $\langle M^N(G)\rangle_t$ with $v(N)=N^2$ in Subsection \ref{sec:4.7}, the proof is similar to that as for exclusion processes.
\end{proof}

We now establish the following characterization of limit points.

\begin{lemma} All limit points $Q$ of $\{Q^N\}$ are supported on trajectories with densities $\pi(t,u)du$ such that, for a constant $C=C(\bar\rho)$,
$$E_Q \Big[\int_0^T \int \pi(t,u)^2dudt\Big] \ < \ C.$$
\end{lemma}

Before beginning the proof, we remark that absolute continuity of the limit trajectories means that particles cannot pile up to form point masses.  Given the process is `attractive', that is the basic coupling holds, one can bound the probability of large piles in terms of $\nu_{\bar\rho}$ estimates where $\bar\rho$ is an absolute bound on the hydrodynamic density.  Without attractiveness, the proof is harder and we refer to \cite{KL} for a general argument.

\begin{proof}
We show first that the trajectories are absolutely continuous under a limit point $Q$.  Let $G:[0,T]\times \T^d\rightarrow \R$ be a smooth function. 
Note, for $A>0$ and for all large $N$, by attractiveness and truncation bounds, that
\begin{eqnarray*}
&&\int_0^T\frac{1}{N^d}\sum_{x\in \T_N^d} G(s,x/N)\eta_{N^2s}(x)ds \\
&&\ \  \leq \ \big(A + E_{\nu_{\bar\rho}}[\eta(0)1(\eta(0)>A)]\big)\|G\|_{L^1([0,T]\times \R)} \\
&&\ \    + \int_0^T\frac{1}{N^d}\sum_{x\in \T_N^d} |G(s,x/N)|\big(\eta_{N^2s}(x)1(\eta_{N^2s}(x)>A) -E_{\nu_{\bar\rho}}[\eta(0)1(\eta(0)>A)]\big) .
\end{eqnarray*}
The last term is an increasing function of $\eta_{N^2s}$, which we will show is small. The role of $A$ is to introduce a truncation.  Let $\phi(\eta_{N^2s}(x)) = \eta_{N^2s}(x)1(\eta_{N^2s}>A) -E_{\nu_{\bar\rho}}[\eta(0)1(\eta(0)>A)$.  Therefore, by the basic coupling, Chebychev's inequality and invariance of $\nu_{\bar\rho}$, for $\epsilon>0$, we have
\begin{eqnarray*}
&&\P_{\nu_{\rho_0(\cdot)}}\Big(\int_0^T\frac{1}{N^d}\sum_x |G(s,x/N)|\phi(\eta_{N^2s}(x)) > \epsilon\Big) \\
&&  \ \ \ \ \ \ \  \leq \ 
\P_{\nu_{\bar\rho}}\Big(\int_0^T\frac{1}{N^d}\sum_x |G(s,x/N)|\phi(\eta_{N^2s}(x)) > \epsilon\Big)\\
&& \ \ \ \ \ \ \ \leq \ \frac{2T\|G\|_{L^2}}{N^d\epsilon^2} {\rm Var}_{\nu_{\bar\rho}}\big(\eta(0)1(\eta(0)>A)\big)\ = \ O(N^{-d}).
\end{eqnarray*}
Hence, by simple estimates,
$$Q^N\Big(\int_0^T \langle G(s, \cdot), \pi_s\rangle ds \leq \big( A + E_{\nu_{\bar\rho}}[\eta(0)1(\eta(0)>A)]\big)\|G\|_{L^1} + \epsilon\Big) \ \geq \ 1-O(N^{-d})$$
and
by weak convergence, since the function $\int_0^T \langle G(s,\cdot), \pi_s\rangle ds$ is continuous in the Skorohod topology and $Q(F)\geq \lim Q^N(F)$ for closed sets,
$$Q\Big( \int_0^T \langle G(s, \cdot), \pi_s\rangle ds \leq \big(A + E_{\nu_{\bar\rho}}[\eta(0)1(\eta(0)>A)]\big)\|G\|_{L^1} + \epsilon\Big)  \ = \ 1.$$
Since $\epsilon>0$ is arbitrary, we have a.s. under $Q$ that all trajectories satisfy
$$\int_0^T \langle G(s,\cdot),\pi_s\rangle ds \ \leq \ C\|G\|_{L^1}.$$
Now, note by the tightness estimate (2') already proven in the previous lemma that all trajectories under $Q$ are continuous in time in the underlying topology (e.g. $t\mapsto \langle H, \pi_t\rangle$ is continuous for $H\in C^2(\T^d)$).
Then, by choosing $G$ to approximate $\delta^{-1}1(t,t+\delta)1(B)$ for $B\subset \T^d$, we obtain $\pi_t(B) \leq C|B|$, that is $\pi_t$ is absolutely continuous for each $t\in [0,T]$.

To show the display in the lemma, recall \eqref{eq:eta-l} and consider the bound
$$\sup_{N\geq 1} \E_{\nu_{\rho_0(\cdot)}}\Big[ \int_0^T du \frac{1}{N^d}\sum_{x\in \T^d_N} \big(\eta_{N^2s}^{N\epsilon}(x)\big)^2\Big] \ \leq \ T E_{\nu_{\bar\rho}}[\eta(0)^2],$$
via the basic coupling and Schwarz inequality 
$$\big(\eta^{N\epsilon}(x)\big)^2 \leq \frac{1}{(2N\epsilon+1)^d}\sum_{|y-x|\leq N\epsilon} \eta(y)^2.$$

Hence, as $Q$ is a limit point, by Fatou's lemma again, we have
$$E_Q \Big[\int_0^T ds \int_{\T^d} dx \Big((2\epsilon)^{-d}\int_{B(u,\epsilon)} \pi(s,v) dv\Big)^2 \Big] \ \leq \ C(T, \bar\rho)$$
where $B(u,\epsilon)$ is the ball of radius $\epsilon$ around $u$.
Taking limit on $\epsilon\downarrow 0$, and another application of Fatou's lemma along with Lebesgue's differentiation theorem (a.e. $u$ is a Lebesgue point), we finish the proof. \end{proof}

\section{Proof of Step 2: Identification given replacement homogenization}
We now supply, modulo a replacement estimate, the proof of Theorem \ref{thm:481}.
Note that $\Psi(\rho) = E_{\nu_\rho}[g(\eta(0))]$ is a bounded function as $g$ is bounded.  Recall also the notation $\eta^\ell(x)$ in \eqref{eq:eta-l}.  Let $\tau_x$ be the shift operator, $(\tau_x\eta)(y)=\eta(x+y)$ and $\tau_x f(\eta) = f(\tau_x \eta)$.
\begin{theorem}[Replacement]
\label{replacement}
For $J\in C([0,T]\times \T^d)$, we have that
$$\limsup_{\epsilon\downarrow 0}\limsup_{N\uparrow\infty} \E_{\nu_{\rho_0(\cdot)}}\Big[\Big|\int_0^T \frac{1}{N^d} \sum_{x\in \T^d_N}J(s, x/N) \tau_x V_{N\epsilon}(\eta_{N^2s})ds \Big|\Big] \ = \ 0$$
where 
$$V_\ell(\eta) \ = \ g(\eta(0)) - \Psi(\eta^\ell(0)).$$
\end{theorem}
\noindent Notice that $V_\ell$ is a bounded function as $g$ is bounded.

Let us now see how this replacement allows to finish the proof of hydrodynamics.

\subsection{Proof of Theorem \ref{thm:481}} 
 First, we establish an equation that the densities $\pi(t,u)$ under a limit point $Q$ must satisfy.
Let $G\in C^2([0,T]\times \T^d)$.  Then, following the derivation when $G$ did not depend on time in Subsection \ref{sec:4.7}, we obtain
\begin{eqnarray*}
&&\langle G(t,\cdot), \pi^N_{N^2t}\rangle \ =\ \langle G(0,\cdot), \pi^N_0\rangle \\
&&\ \ \ \ + \int_0^t \frac{N^2}{N^{d+2}}\sum_{x\in \T^d_N}\Big[\partial_sG(s,x/N)\eta_{N^2s}(x)
+ \frac{1}{2d}\triangle G(s,x/N)g(\eta_{N^2s}(x))\Big]ds \\
&&\ \ \ \ \ \ \ \ \ \ \ \ \ \ \  + M^N_{t}(G) + o(1)
\end{eqnarray*}
where by Doob's inequality,
$$\E_{\nu_{\rho_0(\cdot)}}\big[\sup_{t\in [0,T]} \big(M^N_{t}(G)\big)^2 \big] \ \leq \ \E_{\nu_{\rho_0(\cdot)}} \big[\big(M^N_{T}(G)\big)^2 \big]\ = \ O(N^2 TN^{-d-2})$$ since $g$ is bounded.

On the other hand, by Theorem \ref{replacement}, since $\triangle G$ is continuous, we have
$$\lim_{\epsilon\downarrow 0}\lim_{N\uparrow\infty}\E_{\nu_{\rho_0(\cdot)}}\Big[\Big|\int_0^T \frac{1}{N^d}\sum_{x\in \T^d_N} \triangle G(s,x/N)\Big\{g(\eta_{N^2s}(x)) - \Psi\big(\eta_{N^2s}^{N\epsilon}(x)\big)\Big\}ds\Big|\Big] \ = \ 0.$$

Putting this together, noting \eqref{eq:4.eta}, we have
\begin{eqnarray}
\label{epsilon_function}
&&\lim_{\epsilon\downarrow 0}\lim_{N\uparrow\infty}
\E_{\nu_{\rho_0(\cdot)}}\Big[\Big|\langle G(T, \cdot), \pi^N_{N^2T}\rangle -  \langle G(0,\cdot), \pi^N_0\rangle - \int_0^T \Big\langle \partial_sG(s, \cdot), \pi^N_{N^2s}\Big\rangle ds\nonumber\\
&&\ \ \ \ \ \ \ \ \ \ \ \ \ \ - \frac{1}{2d}\int_0^t \int_{\T^d}\triangle G(s,u)\Psi(\pi^{(N, \epsilon)}_s(u))ds \Big|\Big] \ = \ 0
\end{eqnarray}
where
$\pi^{(N, \epsilon)}_s(u) = \big\langle (2\epsilon)^{-d}1_{[-\epsilon,\epsilon]^d}(\cdot - u), \pi^N_{N^2s}\big\rangle$.

Now, again the quantity in absolute value in \eqref{epsilon_function} is a continuous function of $\pi$ in the Skorohod topology.  Hence, any limit point $Q$, by Fatou's lemma, satisfies
\begin{eqnarray*}&&\lim_{\epsilon\downarrow 0} 
E_Q\Big[ \Big|\langle G, \pi_T\rangle -  \langle G, \pi_0\rangle - \int_0^T \Big\langle \partial_sG, \pi_s\Big\rangle ds\nonumber\\
&&\ \ \ \ \ \ \ \ \ \ \ \ - \frac{1}{2d}\int_0^T \int_{\T^d}\triangle G(s,u)\Psi(\pi^{(\epsilon)}_s(u))duds \Big| \Big] \ = \ 0,
\end{eqnarray*}
where $\pi^{(\epsilon)}_s(u) = \big\langle (2\epsilon)^{-d}1_{[-\epsilon,\epsilon]^d}(\cdot - u), \pi_{s}\big\rangle$.

But, since trajectories have densities under $Q$, and $\Psi$ is bounded and continuous, by Lebesgue's differentiation theorem and dominated convergence,
$$\lim_{\epsilon\downarrow 0}
E_Q\Big[ \int_0^T \int_{\T^d}\big|\Psi(\pi^{(\epsilon)}_s(u)) - \Psi(\pi(u,s))\big|duds\Big] \ = \ 0.$$

Finally, we obtain therefore $Q$ is supported on trajectories satisfying
$$\langle G(T, \cdot),\pi_T\rangle =  \langle G(0, \cdot), \pi_0\rangle + \int_0^T \langle G_s, \pi_s\rangle ds + \frac{1}{2d}\int_0^T\int_{\T^d}\triangle G(s,u)\Psi(\pi(s,u))duds$$
which are weak solutions to
$$\partial_t \rho \ = \ \frac{1}{2d}\triangle \Psi(\rho) \ \ \ {\rm \ \ }\rho(0,u)=\rho_0(u).$$

At this point, it is known in PDE that weak solutions with $L^2$ integrable densities are unique (cf. \cite[Appendix 2.4]{KL}).  Hence, following now the same argument as for simple exclusion, we obtain Theorem \ref{thm:481}.
\qed

\section{Proof of the replacement homogenization}
After preliminaries and some reductions, we prove Theorem \ref{replacement} at the end of this subsection.  
Since we start out of a stationary distribution, we have to understand the nonstationary contribution.  Controlling the entropy and Dirichlet form of the associated Radon-Nikodym probability density is the first step.

For $\bar\rho= \|\rho_0(\cdot)\|_\infty<\infty$, let $\nu_{\bar\rho}$ be a reference measure throughout the rest of Section \ref{lec5}.  Let also $f^N_t = d\nu_{\rho_0(\cdot)}P_{N^2t}/d\nu_{\bar \rho}$ be the Radon-Nikodym density at macroscopic time $t\geq 0$.  Let $H(f^N_t) = H(\nu_{\rho_0(\cdot)}P_{N^2t}|\nu_{\bar\rho})$.  By Lemma \ref{ex:4.1}, we know $H(f^N_0) \leq CN^d$ for some constant $C<\infty$.  
Since the time scaling is $v(N)=N^2$, we have by Lemma \ref{lem:4.3.6}
for each $t\in [0,T]$ that
$$H(f^N_t) + 2N^2\int_0^tI(f^N_s)ds \ \leq \ H(f^N_0) \ \leq \ CN^d.$$

Let $\bar f^N_t = \frac{1}{t}\int_0^t f^N_s ds$.  Then, by convexity of the entropy and Dirichlet form, we have
\begin{equation}
\label{time-average}
H\left(\bar f^N_t\right) \leq C N^d \ \ {\rm and \ \ } I\left(\bar  f^N_t\right) \leq C(2t)^{-1}N^{d-2}.
\end{equation}

For later reference, the following estimate will allow to introduce a truncation corresponding to the number of particles in a region.
\begin{lemma}
\label{truncation} For $\ell\geq 1$, we have
$$\sup_{x,y\in \T^d_N, t\in [0,T]}\E_{\nu_{\rho_0(\cdot)}}\Big[\eta_{N^2t}^\ell(x) 1\big(\eta_{N^2t}^\ell(y) > A\big)\Bigg] \ \leq \ A^{-1}E_{\nu_{\bar\rho}}[\eta^2(0)].$$
\end{lemma}

\begin{proof}
By Markov's inequality, the basic coupling as $\eta^\ell(x)\eta^\ell(y)$ is an increasing function, and Schwarz inequality, we have the left-hand side is bounded by
\[
\sup_{x,y\in \T^d_N, t\in [0,T]}A^{-1}\E_{\nu_{\rho_0(\cdot)}}\Big[\big(\eta_{N^2t}^\ell(x) \big)^2\Big]^{1/2} \E_{\nu_{\rho_0(\cdot)}}\Big[\big(\eta_{N^2t}^\ell(y) \big)^2\Big]^{1/2}\ \leq \ A^{-1}E_{\nu_{\bar\rho}}[\eta^2(0)]. \qedhere\]
\end{proof}

\subsection{Reductions}
\label{sec:reductions}
Consider the following bound of the integrand in Theorem \ref{replacement}, by adding and subtracting terms.
Write
\begin{align}
\label{eq:5.4.1}
&\frac{1}{N^d} \sum_{x} J\Big(\frac{x}{N}\Big)\
\tau_x V^{N\epsilon}(\eta)\nonumber\\
&  = \ 
\frac{1}{N^d} \sum_x J\Big(\frac{x}{N}\Big)\
\Big\{g(\eta(x)) - \frac{1}{(2\ell + 1)^d}\sum_{|z|\leq \ell} g(\eta(z+x))\Big\}
\nonumber\\
& + \ 
\frac{1}{N^d} \sum_x J\Big(\frac{x}{N}\Big)\
\Big\{\frac{1}{(2\ell + 1)^d}\sum_{|z|\leq \ell} g(\eta(z+x)) - \Psi\big(\eta^{(\ell)}(x)\big)\Big\}
\nonumber\\
&  + \  
\frac{1}{N^d} \sum_x J\Big(\frac{x}{N}\Big)\
\Big\{\Psi \big(\eta^{(\ell)}(x)\big)- \Psi\big(\eta^{(N\epsilon)}(x)\big)\Big\}.
\end{align}

The first term on the right-hand side introduces the scale $\ell$ and more averaging: As $g$ is bounded and $J$ is uniformly continuous, it vanishes when $N\uparrow\infty$ and $\ell\uparrow\infty$.  If $J$ were $C^1$, it would be of order $O(\ell/N)$.

The second term, bringing an absolute value inside the sum, is bounded by
$\frac{\|J\|_\infty}{N^d}\sum_x \tau_x |V_\ell(\eta)|$
where we observe
$$|V_\ell(\eta)| = \Big| \frac{1}{(2\ell+1)^d}\sum_{|y|\leq \ell} g(\eta(y)) - \Psi(\eta^\ell(0))\Big| \leq 2\|g\|_\infty.$$
Moreover, we may introduce a truncation by Lemma \ref{truncation}, so that what remains to bound is for $A>0$ the expectation of 
$$\frac{1}{N^d}\sum_x \tau_x \Big(|V_\ell(\eta)| 1(\eta^\ell(0)\leq A)\Big).$$

In the third term, as $\Psi$ is Lipschitz, we bound it by
\begin{align*}
\frac{\|J\|_\infty}{N^d}\sum_x \Big | \Psi(\eta^\ell(x)) - \Psi(\eta^{N\epsilon}(x))\Big|
\leq \frac{\|J\|_\infty}{N^d}\sum_x \big|\eta^\ell(x) - \eta^{N\epsilon}(x)\Big|.
\end{align*}
We may further write the $N\epsilon$-window term $\eta^{N\epsilon}(x)$ in terms of an average of disjoint $\ell$-window terms $\eta^\ell(x+z_i)$ for $1\leq i\leq M:= \lfloor (2N\epsilon+1)^d/(2\ell+1)^d\rfloor$ as
$$\eta^{N\epsilon}(x) = \frac{1}{M}\sum_{i=1}^M \eta^\ell(x+z_i) + \frac{1}{(2N\epsilon+1)^d}\sum_j \eta(x+j),$$
where $\{z_i\}$ are centers of the cubes in the decomposition and the last sum is over at most $(2N\epsilon + 1)^d - (2\ell+1)^d\lfloor (2N\epsilon+1)^d/(2\ell+1)^d\rfloor \leq O\big((N\epsilon)^{d-1} (2\ell+1)^d\big)$ variables in possibly uncompleted $\ell$-windows.  One may also add into this sum the contributions from the $\ell$-windows neighboring $x$, say those with centers within $2\ell$ of $x$. Then, noting the $\ell^d$-overcount of these `edge' items, the third term is further bounded by
\begin{align*}
&\frac{\|J\|_\infty}{N^d}\frac{1}{M}\sum_{i=1}^M 1(2\ell\leq |z_i|\leq N\epsilon) \sum_x \Big| \eta^\ell(x) - \eta^\ell(x+z_i)\Big| \\
&\quad\quad\quad+ \frac{C\|J\|_\infty \ell^d(N\epsilon)^{d-1}}{(N\epsilon)^d}\frac{1}{N^d}\sum_x \eta(x).
\end{align*}

The $\nu^N_{\rho_0(\cdot)}$-process expectation of the last term is of order $O\big(\ell^d/(N\epsilon)\big)$ since by the basic coupling $\E_{\nu^N_{\rho_0(\cdot)}}\big[\eta_{N^2t}(x)\big] \leq E_{\nu_{\bar \rho}}\big[\eta(0)\big] = \bar\rho$.
We may also introduce truncations, for $A>0$ by Lemma \ref{truncation}, so that we need only bound the expectation of
\begin{align}
\label{third_term}
\frac{1}{N^d}\frac{1}{M}\sum_{i=1}^M 1(2\ell\leq |z_i|\leq 2N\epsilon) \sum_x \Big| \eta^\ell(x) - \eta^\ell(x+z_i)\Big| 1(\eta^\ell(x) + \eta^\ell(x+z_i) \leq A).
\end{align}

\subsection{Statements of $1$ and $2$-block lemmas}
Recall the definition of the density $\bar f^N_t$ near \eqref{time-average}.  By the development in the previous subsection, to bound the second term in \eqref{eq:5.4.1}, it is enough to estimate
\begin{align*}
&\E_{\nu_{\rho_0(\cdot)}}\Big[\int_0^T \frac{1}{N^d}\sum_x \tau_x \big\{|V_\ell(\eta_{N^2 s})|1(\eta_{N^2s}^\ell(0)\leq A)\big\} ds\Big]\\
&\quad\quad =T E_{\nu_{\bar\rho}}\Big[ \bar f^N_T(\eta) \frac{1}{N^d}\sum_x \tau_x \big\{|V_\ell(\eta)|1(\eta^\ell(0)\leq A)\big\} \Big].
\end{align*}
Since we know $I(\bar f^N_T)=O\big(N^{d-2}\big)$ from \eqref{time-average}, choosing $A>\bar\rho$, it will be sufficient to show the following.
\begin{proposition}[1-block lemma]
\label{1-block}
We have for all $A>\bar \rho$ that
$$\lim_{\ell\uparrow\infty}\lim_{N\uparrow\infty} \sup_{I(f)\leq C_0N^{d-2}}E_{\nu_{\bar\rho}}\Big[f(\eta)
\frac{1}{N^d}\sum_{x\in \T_N^d} \tau_x |V_\ell(\eta)|1(\eta^\ell(0)\leq A)\Big] \ = \ 0$$
\end{proposition}
where the supremum is over nonnegative densities $f$ with $\|f\|_{L^1(\nu_{\bar \rho})}=1$.
\vskip .1cm

The third term in \eqref{eq:5.4.1} is handled similarly.  Noting the formulation in \eqref{third_term}, it is enough to show the following limit.

\begin{proposition}[2-block lemma]
\label{2-block}
 We have for $A>\bar\rho$ that
\begin{eqnarray*}&&\lim_{\ell\uparrow\infty}\lim_{\epsilon\downarrow 0}\lim_{N\uparrow\infty} \sup_{I(f)\leq C_0N^{d-2}}\sup_{2\ell<|y|\leq 2N\epsilon}\\
&&\ \ \ \ \ \ E_{\nu_{\bar\rho}}\Big[f(\eta)
\frac{1}{N^d}\sum_x \big |\eta^\ell(x) - \eta^\ell(x+y)\big|1(\eta^\ell(x) + \eta^\ell(x+y)\leq A)\Big] \ = \ 0.\end{eqnarray*}
\end{proposition}
\vskip .1cm

\noindent {\bf Proof of Theorem \ref{replacement}.}  With the $1$ and $2$-block replacements Propositions \ref{1-block}, \ref{2-block} in hand, the proof follows from the reductions in Subsection \ref{sec:reductions}. \hfill \qed

\section{Proof of 1-block lemma}
To prove Proposition \ref{1-block}, the main idea is that the Dirichlet form of the density $f$ vanishes as $N\uparrow\infty$.  In some sense, the optimal density $f$ is almost constant.  Things then reduce to a standard ergodic theorem or law of large numbers associated with i.i.d. random variables distributed according to $\kappa_{\bar\rho}$.  

To make this strategy precise, define
$${\rm Av}(f) \ = \ \frac{1}{N^d} \sum_x \tau_x f(\eta).$$
Then, by translation-invariance of $\nu_{\bar\rho}$,
\begin{align}
&E_{\nu_{\bar\rho}}\Big[f(\eta)
\frac{1}{N^d}\sum_x \tau_x |V_\ell(\eta)|1(\eta^\ell(0)\leq A)\Big] \nonumber\\
&\quad\quad \quad = \ E_{\nu_{\bar\rho}}\Big[{\rm Av}(f) |V_\ell(\eta)|1(\eta^\ell(0)\leq A)\Big].
\label{average_density}\end{align}
Observe that ${\rm Av}(f)$ is a translation-invariant density.

Let now 
$$\Lambda_\ell = \{-\ell, \ldots, \ell\}^d \ \ {\rm  and  \ \ }\mathcal F_\ell = \sigma\{\eta(x): x\in \Lambda_\ell\}.$$  Denote $f_\ell = E_{\nu_{\bar\rho}}[{\rm Av}(f)|\mathcal{F}_\ell]$.  Since $|V_\ell(\eta)|1(\eta^\ell(0)\leq A)$ depends only on variables $\eta(x)$ where $x\in \Lambda_\ell$, we have the right-side of \eqref{average_density} equals
$$E_{\nu_{\bar\rho}}[f_\ell(\eta) |V_\ell(\eta)|1(\eta^\ell\leq A)] \ = \ E_{\nu^\ell_{\bar\rho}}[f_\ell(\eta) |V_\ell(\eta)|1(\eta^\ell\leq A)]$$
where $\nu^\ell_{\bar\rho}$ is the product measure over sites in $\Lambda_\ell$ with marginal $\kappa_{\Psi(\bar\rho)}$.

Now, recall that the Dirichlet form has explicit form \eqref{eq:I-def}.  For $x,y$ such that $|x-y|=1$, let
$$I_{x,x+y}(h) = \frac{1}{2(2d)}E_{\nu_{\bar\rho}}\Big[g(\eta(x))\big(\sqrt{h(\eta^{x,x+y})} - \sqrt{h(\eta)}\big)^2\Big].$$
Then, $I(h) = \sum_{x,|y|=1}I_{x,x+y}(h)$; note there are $dN^d$ edges in $\T_N^d$ .  

Define also the Dirichlet form corresponding to the dynamics restricted to $\Lambda_\ell$ with invariant measure $\nu^\ell_{\bar\rho}$.
$$I_\ell(h)\ =\ \sum_{\stackrel{x, |y|=1}{x,x+y\in \Lambda_\ell}}I_{x,x+y}(h);$$
note there are $d(2\ell +1)^{d-1}(2\ell)$ edges in $\Lambda_\ell$.
\begin{lemma}
\label{1-block_dirichlet}
We have that
\begin{eqnarray*}I_\ell(f_\ell) &\leq & I_\ell({\rm Av}(f))\\
& = & (2\ell +1)^{d-1}(2\ell)N^{-d}I({\rm Av}(f))\\
&\leq&   (2\ell +1)^{d-1}(2\ell)N^{-d}I(f) \\
&\leq&  C\ell^dN^{-2}.
\end{eqnarray*}
\end{lemma}

\begin{proof} Note $f_\ell$ and ${\rm Av}(f)$ are averages:  For instance, $f_\ell$ is a conditional expectation.  Then, the first and second inequalities follow as the Dirichlet form is convex.  The equality follows as ${\rm Av}(f)$ and the measure $\nu_{\bar\rho}$ are translation-invariant:  Namely, $I_{x,x+y}({\rm Av}(f)) = I_{z,z+w}({\rm Av}(f)$ for $x,z\in \T_N^d$ and $|y|=|w|=1$.
The last line follows from the bound $I(f)\leq C_0N^{d-2}$. \end{proof}

Now, taking into account \eqref{average_density} and the display after, and Lemma \ref{1-block_dirichlet}, we need to show
\begin{equation}
\label{1-block-ell}
\lim_{\ell\uparrow\infty}\lim_{N\uparrow\infty} \sup_{I_\ell(f)\leq C\ell^dN^{-2}}E_{\nu^\ell_{\bar\rho}}\big[f(\eta)|V_\ell(\eta)|1(\eta^\ell(0)\leq A)\big] \ = \ 0\end{equation}
where the supremum is over densities $f$ with respect to $\nu^\ell_{\bar\rho}$.

In fact, since $\nu^\ell_{\bar\rho}(\eta^\ell\leq A)>0$ is uniformly bounded below for all large $\ell$ because $A>\bar\rho$, in \eqref{1-block-ell}, we may replace the integrating measure $\nu^\ell_{\bar\rho}$ by $\nu^\ell_{\bar \rho, A}=\nu^\ell_{\bar\rho}(\cdot |\eta^\ell(0)\leq A)$ which supports only a finite number of configurations.  Also, as 
\[ E_{\nu^\ell_{\bar \rho, A}}[f] = E_{\nu^\ell_{\bar\rho}}[f1(\eta^\ell(0)\leq A)]/\nu^\ell(\eta^\ell(0)\leq A) \leq 1/\nu^\ell(\eta^\ell(0)\leq A) \leq C,\]
 we may view the supremum as over the collection of nonnegative $f$ with respect to $\nu^\ell_{\bar \rho, A}$ which are uniformly bounded $\|f\|_{L^1(\nu^\ell_{\bar \rho, A})}\leq C$.
The corresponding collection of subprobablity measures $fd\nu^\ell_{\bar\rho, A}/C$, being on a compact space, is tight and has a converging subsequence by a form of Prokhorov's theorem.

Hence, to evaluate \eqref{1-block-ell}, for fixed $\ell$, we consider a sequence in $N$ which approaches the limit supremum.  Densities $f^{N}$ can be found on which the supremum value is well approximated.  By tightness, we can find a subsequence where $f^N$ converges to $f^*$ for which necessarily $I_\ell(f^*)=0$.  Hence, it is enough to show
\begin{equation*}
\lim_{\ell\uparrow\infty} \sup_{I_\ell(f)=0} E_{\nu^\ell_{\bar\rho}}\big[f(\eta)|V_\ell(\eta)|1(\eta^\ell(0)\leq A)\big] \ = \ 0.
\end{equation*}

Now, given $I_\ell(f^*)=0$ and $\eta^\ell(0)\leq A$, we know that $f^*$ is constant on the configurations such that $\eta^\ell(0) = a$ for $a\leq A$.   Therefore, as $f^*\leq C$, decomposing along such `hyperplanes', it is enough to show that
\begin{equation}
\label{1-block-ell-2}
\lim_{\ell\uparrow\infty} \sup_{a\leq A}E_{\nu^\ell_{\bar\rho}}\big[|V_\ell(\eta)|\big |\eta^\ell(0)=a\big] \ = \ 0.
\end{equation}

The left-hand side expectation is the same as
$$E_{\nu_{\bar\rho}}\Big[\Big|\frac{1}{(2\ell+1)^d}\sum_{|x|\leq \ell} g(\eta(x)) - E_{\nu_a}\big[g(\eta(0))\big]\Big| \Big|\eta^\ell(0) = a\Big].$$
Here, we removed the restriction of the measure to $\Lambda_\ell$ since it does not matter.
Notice also that the integrating measure is the canonical measure on $\Lambda_\ell$ with $a(2\ell + 1)^d$ particles.  Hence, by the form $\nu_{\bar \rho}$, the parameter $\bar\rho$ of the underlying grand-canonical measure $\nu_{\bar\rho}$ does not matter, and can be chosen as we like, say $a$.  

We may rewrite the above display as

\begin{align}
\label{1-block-final}
&\frac{1}{\sqrt{(2\ell +1)^d}\nu_a(\eta^\ell(0) = a)}\\
&\quad\times  E_{\nu_a}\Big[\Big|\frac{1}{\sqrt{(2\ell+1)^d}}\sum_{|x|\leq \ell} \big(g(\eta(x)) - E_{\nu_a}\big[g(\eta(0))\big]\big) 1\big(\eta^\ell (0)= a\big)\Big].\nonumber
\end{align}
It is an exercise to see that the denominator is bounded away from $0$.
\begin{exercise}
\label{lclt}
\rm
Observe by a local central limit theorem, namely an expansion in the characteristic function, that
$$\lim_{\ell\uparrow\infty}\sqrt{(2\ell +1)^d}\nu_a(\eta^\ell(0)=a) \ = \ (2\pi)^{-d/2}.$$
\end{exercise}

Now, the variables $g(\eta(x)) - E_{\nu_a}[g(\eta(0))]$ in \eqref{1-block-final} are mean-zero and in $L^2(\nu_a)$.  Hence, by an application of Schwarz inequality and Exercise \ref{lclt} again, we have that \eqref{1-block-final} is at most of order $O(\ell^{-d/4})$ which vanishes.

This concludes the proof of Proposition \ref{1-block}. \hfill \qed

\section{Proof of 2-block lemma}

The argument for the proof of Proposition \ref{2-block} is similar to the proof of the `1-block Lemma' Proposition \ref{1-block}.  Now, we have to control the differences of averages in two separated blocks of width $2\ell+1$.  To compare these averages, we need again to show that the localized Dirichlet form of the density function is small.  Since the jump probabilities are nearest-neighbor, to localize on the two blocks of width $O(\ell)$, we will need to extend the Dirichlet form and dynamics to include a long jump from one block to the other.  In this way, the localized system can mix, and the difference in the averages will wash out.  There is a cost for this extension which however can be overcome.

As before, with the `$1$-block Lemma', the first step is to estimate the display in Proposition \ref{2-block} in terms of the averaged density over shifts ${\rm Av}(f)$.  We need to show that
\begin{eqnarray*}
&&\lim_{\ell\uparrow \infty}\lim_{\epsilon\downarrow 0}\lim_{N\uparrow\infty}  \sup_{I(f) \leq C_0N^{d-2}}\sup_{2\ell <|y|\leq 2N\epsilon}\\
&& \ \ \ E_{\nu_{\bar\rho}}\Big[ f(\eta) |\eta^\ell(0)-\eta^\ell(y)|1(|\eta^\ell(0) + \eta^\ell(y)\leq A)\Big] \ = \ 0
\end{eqnarray*}
Here the supremum is over translation invariant densities $f$.

Next, we localize to the union of the two blocks $\Lambda_\ell$ and $\Lambda^y_\ell = y + \Lambda_\ell$.  Let $\nu^{\ell,y}_{\bar\rho}$ be the product measure over $x\in \Lambda_\ell \cup \Lambda^y_\ell$ with marginal $\kappa_{\bar\rho}$.  Let also 
$f^{\ell,y}=E_{\nu^{\ell,y}_{\bar\rho}}[f|\mathcal{F}_{\ell,y}]$ where $\mathcal{F}_{\ell, y} = \sigma\{\eta(x): x\in \Lambda_\ell \cup \Lambda^y_\ell\}$.  Then, as before, the above display reduces to 
\begin{eqnarray*}
&&\lim_{\ell\uparrow \infty}\lim_{\epsilon\downarrow 0}\lim_{N\uparrow\infty}  \sup_{I(f) \leq C_0N^{d-2}}\sup_{2\ell <|y|\leq 2N\epsilon}\\
&& \ \ \ E_{\nu^{\ell,y}_{\bar\rho}}\Big[ f^{\ell,x}(\eta) |\eta^\ell(0)-\eta^\ell(y)|1(|\eta^\ell(0) + \eta^\ell(y)\leq A)\Big] \ = \ 0,
\end{eqnarray*}
for say $A> 2\bar \rho$.

The question now is how to treat the Dirichlet form $I(f)$.  If we use the argument for the `$1$-block Lemma', then a density in the limit would be constant on each hyperplane $\{\eta_{\Lambda_\ell \cup \Lambda^y_\ell}: \eta^\ell(0)+\eta^\ell(y)=a\}$ for $a\leq A$ where $\eta_B = \langle\eta(x): x\in B\rangle$.  But, the constants could be different on each block!

The trick is to add a (non-local) Dirichlet form bond corresponding to a jump say from $0\in \Lambda_\ell$ to $y\in \Lambda_\ell^y$.  Any such jump from a site in $\Lambda_\ell$ to $\Lambda^y_\ell$ would work.  Then, if the Dirichlet form localized to $\Lambda_\ell\cup \Lambda^y_\ell$, including this extra bond term, vanishes, the density would take the same constant value on both blocks.
 In terms of dynamics, the zero-range process associated to this modified Dirichlet form is irreducible on configurations in $\Omega^{y,\ell}=\{0,1,\ldots\}^{\Lambda_\ell \cup \Lambda^y_\ell}$.

For $h: \Omega^{y,\ell}\rightarrow \R$, define
\begin{align*}
I_b(h) &= \frac{1}{2}E_{\nu_{\bar\rho}}\Big[g(\eta(0))\big(\sqrt{h(\eta^{0,y})} - \sqrt{h(\eta)}\big)^2\Big]\\
&= \frac{\Psi(\bar\rho)}{2}E_{\nu_{\bar\rho}}\Big[\big(\sqrt{h(\eta + \delta_x)} - \sqrt{h(\eta+\delta_0)}\big)^2\Big],
\end{align*}
noting Exercise \ref{ex:zr-particle}.  By itself, $I_b$ is the Dirichlet form for the dynamics which moves a particle from $0$ to $y$ and back according to zero-range dynamics, and as such shares all the convexity and other Dirichlet form properties that we used in the proof of the `1-block Lemma'.

By the estimate
$\big(\sum_{j=1}^k q_j\big)^2 \ \leq \ k \sum_{j=1}^k q_j^2$
and adding and subtracting several terms,
\begin{align*}
I_b(h) & = \frac{\Psi(\bar\rho)}{2} E_{\nu_{\bar\rho}}\Big[\Big(\sum_{k=1}^{m_x} \sqrt{h(\eta+ \delta_{q_k})} - \sqrt{h(\eta+ \delta_{q_{k+1}})}\Big)^2\Big] \\
&\leq  
C|x|\sum_{j=1}^{M_x} I_{q_j,q_{j+1}}(h)
\end{align*}
where $\{q_k\}$ corresponds to a nearest-neighbor path from $0$ to $x$ in $M_x = O(|x|)$ steps.

We will apply the above estimate to the translation invariant density $f$.  Since $I_{e_j,e_{j+1}}(f) \leq CN^{-d} I(f)$, we have that
$I_b(f) \leq C|x|^2N^{-d}I(f)$.  Now, if $I(f) \leq CN^{d-2}$ and $|x|\leq C N^2\epsilon^2$, since the separation between $0$ and $x$ is of order $O(N\epsilon)$, we conclude $I_b(f) \leq C\epsilon^2$.

Let now $I^{y,\ell}(h) = I_\ell(h) + I_\ell^y(h) + I_b(h)$ where $I_\ell^y(h) = \sum_{\stackrel{z,|w|=1}{z,z+w\in \Lambda^y_\ell}}I_{z,z+w}(h)$.  By convexity, we have
\begin{align}
\label{eq:5-bond}
I^{y,\ell}(f^{\ell,x}) \ \leq \ I^{y,\ell}(f) \ \leq \ 2C\ell^dN^{-2} + C\epsilon^2 \ \leq \ C\epsilon^2
\end{align}
for fixed $\ell, \epsilon$ and all large $N$.

Hence, it is enough to show for each constant $C$ that
  \begin{eqnarray*}
&&\lim_{\ell\uparrow \infty}\lim_{\epsilon\downarrow 0}\lim_{N\uparrow\infty}  \sup_{I^{x,\ell}(f) \leq C\epsilon^2}\sup_{2\ell <|y|\leq 2N\epsilon}\\
&& \ \ \ E_{\nu^{\ell,x}_{\bar\rho}}\Big[ f(\eta) |\eta^\ell(0)-\eta^\ell(y)|1(|\eta^\ell(0) + \eta^\ell(y)\leq A)\Big] \ = \ 0.
\end{eqnarray*}
Here, the supremum is over densities with respect to $\nu_{\bar\rho}^{\ell,x}$.

But, at this point, the proof can follow the method as for the `$1$-block Lemma'.  We just point out that in the above display that, as
$\eta^\ell(0)$ and $\eta^\ell(y)$ share no terms, with respect to $\nu_{\bar\rho}^{\ell,x}$, they are independent.   This was the reason to avoid the nearest-neighbor cubes in the original decomposition of $\eta^{N\epsilon}(x)$ in Subsection \ref{sec:reductions}.
This finishes the proof of Proposition \ref{2-block}. \hfill \qed

\section{Notes}

We have followed the scheme of Chapter V in \cite{KL}, although there are differences. 
 The reader will have noted that translation invariance is much used in the above derivations.  We remark that there are ways to handle non-translation invariant dynamics if the inhomogeneity is `slowly varying'; see \cite{CR}, and \cite{GJ-random}, \cite{Faggionato}, \cite{FRS}, \cite{lpsx} when the particles move in a random environment.

Also, we comment that diffusive scaling $v(N)=N^2$ was crucial to prove the `$2$-block' estimate; see \eqref{eq:5-bond}.  In asymmetric models, with hyperbolic scaling $v(N)=N$, although the `$1$-block' result still holds, the `$2$-block' estimate is not generally available. 
Such a `$1$-block Lemma' is useful for the `relative entropy' method discussed in Section \ref{lec6} to deduce hydrodynamics, at least for short times, in asymmetric processes.

Finally, we comment, beyond Zero-range models, the `entropy' method works well in symmetric rate processes ($v(N)=N^2$) with conserved quantities that is of `gradient' type, e.g. allowing a twice sum-by-parts evaluation 
\[L\langle G, \pi^N_{V(N)t}\rangle \sim \frac{v(N)}{2N^{d+2}}\sum_{x,y}h(\tau_x \eta_{v(N)t}) p(y)\triangle^N_{x,y}G\]  
for some local function $h$ (depending only on a finite number of variables $\{\eta(x)\}$) to be homogenized, and invariant measures indexed to the conserved quantities with sufficient decorrelation properties (cf. for instance \cite{FHU}).  As alluded in Subsection \ref{sec:asym_meanzero}, diffusively scaled `nongradient' models can also be analyzed by versions of the `entropy' method (cf. \cite{Sasada}, \cite{KL}, \cite{Varadhan_notes}).

\newpage

\chapter[Section $6$]{Relative entropy method and hydrodynamics of TASEP}
\label{lec6}

We discuss the `relative entropy' method in the context of hydrodynamics of the totally asymmetric simple exclusion process (TASEP) in $d=1$. 
The underlying notion is to measure how far the distribution $\mu^N_t$ of the process at time $t$ is away from a specified product measure $\nu^N_t$ on $\T_N$.  If distance, measured in terms of relative entropy, is not so far away, then hydrodynamics will follow from calculations with respect to independent variables governed by $\nu^N_t$.   Such methods and ideas have proved useful for other processes and in other problems.

\section{Statement of hydrodynamics and sketch of the argument}
Recall the notation in Chaper \ref{lec2}, with respect to asymmetric exclusion processes $\eta_t$ on $\T^d_N$.  As before, $E_\mu$ denotes the expectation under $\mu$, and $\P_\mu$ and $\E_\mu$ the process measure and expectation when starting in $\mu$.

When $d=1$ and $p$ allows movement only to the nearest right site, that is $p(1)=1$ and $p(j)=0$ for $j\in \T_N$ otherwise, the process is known as the totally asymmetric simple exclusion process or TASEP for short.  To reduce notation, we will concentrate in the following on TASEP, although calculations may be extended to finite-range asymmetric processes in $d\geq 1$ with non-zero drift.

As mentioned in Section \ref{lec2}, the nontrivial time scaling is when $\upsilon(N) = N$, the `hyperbolic' or `Euler' scale, and the hydrodynamic density $\rho(t, u)$ satisfies $\partial_t \rho + m\cdot\nabla \big[\rho(t, u)(1-\rho(t,u))\big] = 0$.  In the context of TASEP, as $m = \sum jp(j) =1$, the equation simplifies to
\begin{equation}
\label{eq:asym-hyd}
\partial_t \rho + \partial_u \big(\rho(1-\rho)\big)=0.
\end{equation}

Solutions to this one-dimensional Burgers-type hyperbolic conservation law on $\T$ are not necessarily unique.  However, when starting from smooth $C^1$ initial profile $\rho_0(u) = \rho(0, u)$, bounded above and below on $\T$,
\[0<\rho_-\leq \rho_0(u)\leq \rho_+<1,\]
 it is known that a unique classical $C^1$ smooth solution, also bounded above and below by $\rho_+$ and $\rho_-$,  with continuous bounded derivatives, exists up to a short time $T>0$, which we now fix.  We mention one may specify uniquely a `physical' (or by other names, `entropy' or `viscosity') solution for all times $t\geq 0$ by imposing that the solution satisfies extra conditions; see \cite{Evans} as a general reference and for a comprehensive discussion.

Recall the empirical measure on $\T_N$:
\[\pi^N_t = \frac{1}{N} \sum_{x\in \T_N} \eta_t(x)\delta_{x/N}.\]
Define, for $t\in [0,T]$, with respect to the hydrodynamic density, the product measures on $\Omega = \{0,1\}^{\T_N}$:
\[\nu^N_t = \prod_{x\in \T_N} {\rm Bern}(\rho(t, x/N)).\]
Suppose that initially the process $\eta_0$ is distributed according to $\mu^N=\nu^N_0=\nu^N_{\rho_0(\cdot)}$.   In terms of the process semigroup $P_t$, let $\mu^N_t=\mu^NP_{Nt}$ be the distribution of $\eta_{Nt}$ starting from $\mu^N$.

We will show the following form of the hydrodynamic limit.

\begin{theorem}
\label{thm:sec61}
Let $\rho_0:\T\rightarrow [0,1]$ be a $C^1$ function.  Then, for $t\in [0,T]$, we have the convergence in probability,
\[\lim_{N\rightarrow\infty} \langle \pi^N_{Nt}, J\rangle = \frac{1}{N}\sum_{x\in \T_N} J(x/N)\eta_{Nt}(x) = \int_{\T}J(u)\rho(t,u)du\]
where $\rho(t,u)$ satisfies \eqref{eq:asym-hyd} with $\rho(0,u) = \rho_0(u)$.
\end{theorem}

\begin{proof}
The sketch of the argument is in a few steps.
\medskip

Step 1.  
Since $\nu^N_t(\eta)>0$ for all $\eta\in \Omega$, we have $\mu^N_t  \ll \nu^N_t$.  Noting Lemma \ref{lem:4.14}, we will show that the relative entropy $H(\mu^N_t|\nu^N_t)= E_{\mu^N_t}\big[ \log \frac{\mu^N_t}{\nu^N_t}\big]$ between the distribution $\mu^N_t$ of $\eta_{Nt}$ and $\nu^N_t$ is $o(N)$.  This step is the main part of the proof.  It takes the $1$-block lemma as input, and is discussed in the next subsection.

\medskip
Step 2.  The entropy inequality (cf. Lemma \ref{lem:entropyinequality}, Exercise \ref{exercise:entropy}) yields for sets $A\subset \Omega$ that
\[\mu^N_t(A) \leq \frac{\log(2) + H(\mu^N_t|\nu^N_t)}{\log \big(1 + 1/\nu^N_t(A)\big)}.\]

 Let now $\epsilon>0$ and
\[ A = \Big\{ \eta: \big | \frac{1}{N}\sum_{x\in \T_N} J(x/N)\big[\eta(x) - \rho(t, x/N)\big]\big|>\epsilon \Big\}.\]
An exponential bound $\nu^N_t(A)\leq e^{-CN}$ may be found as the variables $\{\eta_{Nt}(x):x\in \T_N\}$ are independent under $\nu^N_t$ with means $\{\rho(t, x/N): x\in \T_N\}$.  Therefore, given the relative entropy estimate in Step 1, the theorem follows. 
\end{proof}

\begin{exercise}  \rm Show the `Chernoff' exponential bound in Step 2, by computing the moment-generating functions.  The inequality $0<\rho_-\leq \rho(t, x/N)\leq \rho_+<1$ will be useful.
\end{exercise}

\section{Proof of Step 1: $o(N)$ entropy estimate}
Let $\nu_{1/2}=\prod_{x\in \T_N} {\rm Bern}(1/2)$ denote the product invariant measure of TASEP with density $1/2$.  Recall the form of $L$ in \eqref{eq:exclu_gen} (in $d=1$ and $p(1)=1$, $p(j)=0$ otherwise).  One may compute that the $L^2(\nu_{1/2})$-adjoint operator $L^*_N= NL^*$ of $L_N = NL$ is the generator of the asymmetric exclusion process, speeded up by $\upsilon(N) = N$, where particles jumps left, that is when $p^*(-1)=1$ and $p^*(j)=0$ otherwise:
\begin{align*}
L_N^*h(\eta) = \sum_{x\in \T_N} \eta(x)\big(1-\eta(x-1)\big)\big[h(\eta^{x, x-1}) - h(\eta)\big].
\end{align*}

The probabilities $\mu^N_t$ satisfy the forward equation (cf. Exercise \ref{exercise:6.2.1}),
\begin{equation}
\label{eq:sec61-0}\partial_t \mu^N_t = L^*_N \mu^N_t.
\end{equation}
However, $\mu^N_t$ is no longer a product measure, the process having mixed things up from the initial condition $\nu^N_0$. One feels though from the intuition given in Section \ref{lec2} that the process may not be far from $\nu^N_t$.  The relative entropy $H(t):=H(\mu^N_t|\nu^N_t)$ is a convenient, and analytically tractable measure of how far apart they are.  

\begin{exercise}
\label{exercise:6.2.1}\rm
Perform the computation of $L^*_N$.  Show also $\mu^N_t$ satisfies \eqref{eq:sec61-0} by differentiating 
$\E_{\nu_{1/2}}\big[h(\eta_{Nt})\big]=E_{\nu_{1/2}}\big[f^N_t(\eta)h(\eta)\big]= \sum_\eta \nu_{1/2}(\eta)f^N_t(\eta)h(\eta)$
where the density $f^N_t(\eta) = \frac{\mu^N_t(\eta)}{\nu_{1/2}(\eta)}$ and $h$ is a bounded function.  The relation $\nu_{1/2}(\eta) \equiv (1/2)^N$ for all $\eta\in \{0,1\}^{\T_N}$ will be useful.
\end{exercise}
\medskip

{\it Step A.} We first update the conclusion of Lemmas \ref{entropy_diff} and \ref{lem:sec4-relent}, one of the differences being that $\mu^N_t$ is not the invariant measure.

\begin{lemma}
\label{lem:sec61-1}
We have
\begin{align*}
H'(t) &\leq -\sum_\eta \nu^N_t(\eta) \eta(x)\big(1-\eta(x+1)\big) \left(\sqrt{\frac{\mu^N_t}{\nu^N_t}(\eta^{x,x+1})} - \sqrt{\frac{\mu^N_t}{\nu^N_t}(\eta)}\right)^2\\
&\quad \quad + E_{\mu^N_t}\Big[ \frac{1}{\nu^N_t(\eta)}\big(L^*_N - \partial_t\big)\nu^N_t(\eta)\Big].
\end{align*}
\end{lemma}

\begin{proof}
Write
\begin{align*}
\partial_t H(\mu^N_t|\nu^N_t) & = \partial_t \sum_\eta \mu^N_t(\eta)\big[\log \mu^N_t(\eta) - \log \nu^N_t(\eta)\big]\\
&=\sum_\eta L^*_N\mu^N_t(\eta)\big[[\log \mu^N_t(\eta) - \log \nu^N_t(\eta)\big]\\
&\quad\quad + \sum_\eta \mu^N_t(\eta) \big[ \frac{1}{\mu^N_t(\eta)}\partial_t \mu^N_t(\eta) - \partial_t \log \nu^N_t(\eta)\big].
\end{align*}
Note that $\partial_t \sum_\eta \mu^N_t(\eta) =0$ as $\sum_\eta \mu^N_t(\eta)\equiv 1$.  Also, 
\begin{align*}
\sum_\eta L^*_N\mu^N_t(\eta) \log \frac{\mu^N_t(\eta)}{\nu^N_t(\eta)}
&= \sum_\eta \nu_{1/2}(\eta) \Big(L^*_N \frac{\mu^N_t(\eta)}{\nu_{1/2}(\eta)} \Big)\log  \frac{\mu^N_t(\eta)}{\nu^N_t(\eta)}\\
&= \sum_\eta \nu_{1/2}(\eta) \frac{\mu^N_t(\eta)}{\nu_{1/2}(\eta)}\Big(L_N  \log  \frac{\mu^N_t(\eta)}{\nu^N_t(\eta)}\Big)\\
&=\sum_\eta \mu^N_t(\eta)L_N  \log  \frac{\mu^N_t(\eta)}{\nu^N_t(\eta)}.
\end{align*}
Then,
\begin{align}
\label{eq:sec61-lem}
\partial_t H(\mu^N_t|\nu^N_t)
&= \sum_\eta \mu^N_t(\eta) L_N \log \frac{\mu^N_t(\eta)}{\nu^N_t(\eta)} - \sum_\eta \mu^N_t(\eta) \partial_t \log \nu^N_t(\eta).
\end{align}

Now, consider the relation
 \[a(\log b - \log a)\ \leq \ 2\sqrt{a}\big(\sqrt{b}-\sqrt{a}\big)\ = \ -\big(\sqrt{b}-\sqrt{a}\big)^2 + \big(b-a\big),\]
  for $a,b>0$.  Then, as $L g(\eta) = \sum_x \eta(x)(1-\eta(x+1))\big[g(\eta^{x,x+1})-g(\eta)\big]$, with $a=\frac{\mu^N_t}{\nu^N_t}(\eta)$ and $b= \frac{\mu^N_t}{\nu^N_t}(\eta^{x,x+1})$,
we observe
\begin{align*}
& \sum_\eta \mu^N_t(\eta)L_N\log \frac{\mu^N_t(\eta)}{\nu^N_t(\eta)} \\
& \ \leq \sum_\eta \nu^N_t(\eta) \sum_x \eta(x)\big(1-\eta(x+1)\big)\cdot 2 \sqrt{\frac{\mu^N_t}{\nu^N_t}(\eta)} \left(\sqrt{\frac{\mu^N_t}{\nu^N_t}(\eta^{x, x+1})} - \sqrt{\frac{\mu^N_t}{\nu^N_t}(\eta)}\right),
\end{align*}
which further equals
\begin{align*} 
&-\sum_\eta \sum_x \nu^N_t(\eta)\eta(x)\big(1-\eta(x+1)\big)\left( \sqrt{\frac{\mu^N_t}{\nu^N_t}(\eta^{x, x+1})} - \sqrt{\frac{\mu^N_t}{\nu^N_t}(\eta)}\right)^2 \\
&\quad \quad\quad + \sum_\eta \nu^N_t(\eta) L_N 
 \frac{\mu^N_t(\eta)}{\nu^N_t(\eta)}.
 \end{align*}
 We may insert this expression into \eqref{eq:sec61-lem}.  Dividing/multiplying by $(1/2)^N\equiv \nu_{1/2}(\eta)$ again and moving $L_N$ to $L^*_N$ on the other side, we obtain the expression desired.
 \end{proof}

We remark that for the `relative entropy' method in the sequel, we will not use the first term, a `Dirichlet' form expression, bounding it above by $0$.  However, the the full force of Lemma \ref{lem:sec61-1} has been useful in other problems; see the Notes.
\medskip

{\it Step B.} Next, we evaluate the second term on the right-hand side in Lemma \ref{lem:sec61-1}.  Explicitly,
\begin{align*}
\nu^N_t(\eta) &= \prod_x \rho(t, x/N)^{\eta(x)}(1-\rho(t,x/N))^{1-\eta(x)}\\
&=\exp\big\{\psi_N(t,\eta)\big\},
\end{align*}
where
\[\psi_N(t,\eta) = \sum_x \Big\{\eta(x)\log \rho(t,x/N) + (1-\eta(x))\log (1-\rho(t,x/N))\Big\}.\]

Hence,
\begin{align*}
&\frac{1}{\nu^N_t(\eta)}L^*_N\nu^N_t(\eta) \\
&\quad\quad= N\sum_x \eta(x+1)(1-\eta(x))\\
&\quad\quad\quad \quad\quad \cdot \Big[\frac{1}{\exp\big\{\psi_N(t, \eta)\big\}} \Big(\exp\big\{\psi_N(t, \eta^{x, x+1})\big\} - \exp\big\{\psi_N(t, \eta)\big\} \Big)\Big]\\
&\quad\quad= N\sum_x \eta(x+1)(1-\eta(x))\Big[\exp\big\{\psi_N(t, \eta^{x, x+1}) - \psi_N(t, \eta)\big\} -1\Big].
\end{align*}
Observe
\begin{align*}
&\psi_N(t, \eta^{x, x+1}) -\psi_N(t,\eta) \\
&\quad= \Big\{\eta(x+1)\log \rho(t, x/N) + (1-\eta(x+1))\log (1-\rho(t, x/N)) \\
&\quad\quad\quad + \eta(x)\log \rho(t, (x+1)/N) + (1-\eta(x))\log (1-\rho(t, (x+1)/N))\Big\}\\
& \quad- \Big\{\eta(x)\log \rho(t, x/N) + (1-\eta(x))\log (1-\rho(t, x/N))\\
&\quad \quad \quad+ \eta(x+1)\log \rho(t, (x+1)/N) + (1-\eta(x+1))\log \rho(t, (x+1)/N)\Big\}.
\end{align*}
Hence, under the condition $\eta(x+1) = 1$ and $\eta(x)=0$, corresponding to $\eta(x+1)(1-\eta(x))=1$, we have by Taylor approximation that
\begin{align*}
&\exp\big\{\psi_N(t, \eta^{x, x+1}) - \psi_N(t, \eta)\big\} -1\\
&\quad= \exp\big\{\log \rho(t, x/N) - \log \rho(t, (x+1)/N) \\
&\quad\quad \quad+ \log (1-\rho(t, (x+1)/N)) - \log (1-\rho(t, x/N))\big\} -1\\
&\quad= \exp\big\{-\frac{1}{N} \frac{\rho'}{\rho}(t, x/N) - \frac{1}{N}\frac{\rho'}{1-\rho}(t, x/N) + o(1/N)\big\} -1\\
&\quad= \frac{-1}{N} \frac{\rho'(t, x/N)}{\rho(1-\rho)(t, x/N)} + o(1/N).
\end{align*}

Therefore, cancelling the `$N$', we obtain
\begin{align*}
\frac{1}{\nu^N_t(\eta)}L^*_N\nu^N_t(\eta) &= -\sum_x \eta(x+1)(1-\eta(x))\frac{\rho'(t, x/N)}{\rho(1-\rho)(t, x/N)} + o(N).
\end{align*}

On the other hand,
\begin{align*}
-\frac{1}{\nu^N_t(\eta)}\partial_t \nu^N_t(\eta) & = -\frac{1}{\exp\big\{\psi_N(t)\big\}}\partial_t \exp\big\{\psi_N(t)\big\}
\ = \ -\partial_t \psi_N(t)\\
&= -\sum_x \eta(x)\frac{\partial_t \rho}{\rho}(t, x/N) + \sum_x (1-\eta(x))\frac{\partial_t \rho}{\rho}(t, x/N).
\end{align*}

Recall $\partial_t \rho = -\partial_u\big(\rho(1-\rho)\big) = (2\rho -1)\partial_u \rho$.  Then, 
by Lemma \ref{lem:sec61-1}, dropping the negative `Dirichlet' term, and dividing by $N$, we conclude
\begin{align*}
\frac{1}{N}H'(t) &\leq E_{\mu^N_t}\Big[ \frac{1}{N} \sum_x \Big\{\eta(x+1)(1-\eta(x)) \frac{-\partial_u \rho}{\rho(1-\rho)}(t, x/N) \\
&\quad -\eta(x)\frac{(2\rho -1)\partial_u\rho}{\rho}(1, x/N)
+ (1-\eta(x))\frac{(2\rho-1)\partial_u\rho}{1-\rho}(t, x/N)\Big\}\Big] + o(1).
\end{align*}

\medskip
{\it Step C.} Now, we invoke a form of the $1$-block lemma in Section \ref{lec5}.  We state it in $d=1$ in the context of TASEP, although it can be generalized to $d\geq 1$ and other processes.  Recall that $\eta^\ell(x) = \frac{1}{2\ell+1}\sum_{|y-x|\leq \ell}\eta(y)$.

\begin{proposition}
\label{prop:sec61}
Let $G:[0,T]\times \T\rightarrow \R$ be continuous.  Let $h$ be a local function and $\tilde h(\rho) = E_{\nu_{\rho}}[h]$.  Then, for $t\in [0,T]$,
\[\lim_{\ell\rightarrow\infty}\lim_{N\rightarrow\infty} \E_{\mu^N}\Big[\Big|\int_0^t \frac{1}{N}\sum_x G(s, x/N)\tau_x \big[h(\eta_{Ns}) - \tilde h(\eta^\ell_{Ns}(0))\big] ds\Big|\Big] = 0.
\]
\end{proposition}

\begin{proof}[Sketch of proof]  Consider the $1$-block shown in Section \ref{lec5}, namely Theorem \ref{replacement} with time-dependent $G(s,x/N)$ replacing $J(x/N)$,  $V_\ell$ replacing $V_{N\epsilon}$, and limits $\lim_{\ell\uparrow\infty}\lim_{N\uparrow}$ replacing $\lim_{\epsilon\downarrow 0}\lim_{N\uparrow\infty}$.  The same scheme can be followed here:  With respect to the first two terms on the right-hand side of \eqref{eq:5.4.1}, continuity of $G(s,x/N)$ in the space variable handles the first term, and as it is Proposition \ref{1-block} takes care of the second term.

A difference is that the time scaling previously was $v(N)=N^2$.  Although this was important for the later $2$-block estimate Proposition \ref{2-block}, for the $1$-block estimate, the Euler scale $v(N)=N$ suffices.  Indeed, we will now obtain the Dirichlet form of the density is of order $O(N^{d-1})$, instead of $O(N^{d-2})$ when $v(N)=N^2$.  Such an $O(N^{d-1})$ estimate is enough to obtain the $1$-block limit.  These computations are left to the reader.
\end{proof}

By Proposition \ref{prop:sec61}, applied with $h(\eta) = \eta(1)(1-\eta(0)$, $\eta(0)$, and $1-\eta(0)$, and $G$ in terms of $\rho(t, u)$, we may bound
\begin{align}
\label{eq:sec61-2}
& \frac{1}{N}H(t) - \frac{1}{N}H(0) \nonumber \\
&\quad\leq \int_0^t 
E_{\mu^N_s}\Big[ \frac{1}{N} \sum_x \Big\{\eta^\ell(x)(1-\eta^\ell(x)) \frac{-\partial_u \rho}{\rho(1-\rho)}(s, x/N) \nonumber\\
&\quad\quad\quad -\eta^\ell(x)\frac{(2\rho -1)\partial_u\rho}{\rho}(s, x/N)
+ (1-\eta^\ell(x))\frac{(2\rho-1)\partial_u\rho}{1-\rho}(s, x/N)\Big\}\Big]ds + o(1) \nonumber
\\
&\quad= \int_0^t E_{\mu^N_s}\Big[\frac{1}{N}\sum_x \partial_u\rho(s, x/N) F\Big(\eta^\ell(x), \rho(t, x/N)\Big)\Big]ds + o(1)
\end{align}
where, for $m,\rho\in [0,1]$,
\[F(m, \rho) = -\frac{m(1-m)}{\rho(1-\rho)} - \frac{m(2\rho-1)}{\rho} + \frac{(1-m)(2\rho-1)}{1-\rho}.\]

Observe that
\[\partial_m F(m, \rho) = -\frac{1-2m}{\rho(1-\rho)} - \frac{2\rho-1}{\rho} - \frac{2\rho-1}{1-\rho}\]
and that it vanishes when $m=\rho$.
Hence,
\[\big| F(m, \rho) - F(\rho, \rho)\big| \leq C\big|m-\rho\big|^2\]
with respect to a constant $C$.

Subtracting and adding terms $F(\rho(s, x/N), \rho(s, x/N))\equiv -1$ in the summation \eqref{eq:sec61-2} is negligible:  Indeed,
\begin{align*}
&\frac{1}{N}\sum_x \partial_u\rho(s, x/N) F\Big(\rho(s, x/N), \rho(s, x/N)\Big)\\
&\quad\quad= -\frac{1}{N}\sum_x \partial_u\rho(s, x/N)\\
&\quad\quad= -\frac{1}{N}\sum_x \Big[N\big(\rho(s, (x+1)/N) - \rho(t, x/N)\big) + O(1/N)\Big] = O(1/N).
\end{align*} 

Therefore, we have that 
\begin{align*}
& \frac{1}{N}H(t) - \frac{1}{N}H(0) \\
&\quad\quad \leq \int_0^t E_{\mu^N_s}\Big[\frac{1}{N}\sum_x \partial_u\rho(s, x/N) \\
&\quad\quad\quad\quad \times\Big\{F\Big(\eta^\ell(x), \rho(s, x/N)\Big) - F\Big(\rho(s, x/N), \rho(s, x/N)\Big)\Big\}\Big]ds + o(1)\\
&\quad\quad \leq \int_0^t E_{\mu^N_s}\Big[\frac{C\|\partial_u\rho\|_\infty}{N}\sum_x \big|\eta^\ell(x) - \rho(s, x/N)\big|^2\Big]ds + o(1).
\end{align*}

\medskip
{\it Step D.} We now estimate the expectation above,  via use of the entropy inequality (Lemma \ref{lem:entropyinequality}):  Let $C=C\|\partial_u \rho(s, x/N)\|_\infty$.  Then, with $\delta>0$,
\begin{align*}
&E_{\mu^N_s}\Big[\frac{C}{N}\sum_x \big|\eta^\ell(x) - \rho(s, x/N)\big|^2\Big] \\
&\quad\quad \leq \frac{C}{\delta N}H(s) + \frac{C}{\delta N}\log E_{\nu^N_s}\Big[\exp\Big\{\delta\sum_x \big|\eta^\ell(x) - \rho(s, x/N)\big|^2\Big\}\Big].
\end{align*}

The variables $\{\eta^\ell(x): x\in \T_N\}$ are $2\ell+1$-dependent.  To gain some independence, divide the interval $\{0,\ldots, N-1\}$ into $\lceil N/(2\ell+1)\rceil$ blocks of width $2\ell+1$, where the possible last block may be of length less than $2\ell+1$.  We may separate the sum over $x\in \T_N$ as a sum over $z$ such that $|z|\leq \ell$ and $w\in \Pi_{N, z, \ell}$ where $\Pi_{N, z, \ell}$ refers to the set $\{z + (2\ell+1)r: r=0, \ldots, \lfloor N/(2\ell+1)\rfloor -1\}$, and also a sum over the remaining $O(\ell)$ indices in the possible overflow block $\Pi'$. 

Let $q_x = \big|\eta^\ell(x) - \rho(s, x/N)\big|^2$.  The sum corresponding to the overflow block can be bounded $\sum_{x\in \Pi'} q_x \leq C_1\ell$.  
We may also center $q_x \leq 2\big| \eta^\ell(x) - \rho^\ell(s, x/N)\big|^2 + 2\big|\rho^\ell(s, x/N) - \rho(s, x/N)\big|^2$ where $\rho^\ell(s, x/N) = (2\ell + 1)^{-1}\sum_{|y|\leq\ell} \rho(s, (y+x)/N)$.  Note $|\rho(s, (x+y)/N) -\rho(s, x/N)| \leq C\|\partial_u\rho\|_\infty\ell/N$ for $|y|\leq\ell$.

Then, for each $|z|\leq \ell$, since $\{q_w: w\in \Pi_{N, z, \ell}\}$ are independent, by H\"older's inequality, 
\begin{align*}
&\frac{1}{N}\log E_{\nu^N_s}\Big[\exp\Big\{\delta \sum_x q_x\Big\}\Big] \\
&\quad \leq \frac{1}{N \ell}\sum_{|z|\leq \ell}\sum_{w\in \Pi(N, z, \ell)} \log E_{\nu^N_s}\Big[ \exp\Big\{2\delta \ell \big |\eta^\ell(w) - \rho^\ell(s, w)\big|^2\Big\}\Big] \\
&\quad\quad\quad + \frac{C_1\delta \ell}{N} + \frac{C(\|\partial_u\rho\|_\infty)\delta \ell}{N}.
\end{align*}

\medskip
{\it Step E.} We now apply a concentration inequality for subgaussian random variables.  We say $X$ is $\sigma^2$-subgaussian
if 
\[\log E\big[e^{\theta X}\big] \leq \sigma^2\theta^2/2\]
for all $\theta\in \R$.  
\begin{lemma}
\label{lem:concentration}
Let $X$ be $\sigma^2$-subgaussian.  Then, if $\gamma \leq (4\sigma^2)^{-1}$,
\[E\big[e^{\gamma X^2}\big] \leq 3.\]
\end{lemma}

In our context, the variable $X_w:=\sqrt{\ell}\big(\eta^\ell(w) - \rho^\ell(s, w/N)\big)$, with respect to $\nu^N_s$, is subgaussian with parameter
\[\sigma^2(w) = \frac{2}{(2l+1)} \sum_{|y-w|\leq l} \rho(s, y/N)\big(1-\rho(s, y/N)\big).\]

\begin{exercise} \rm Show Lemma \ref{lem:concentration}.  Also, show that $X_w$ is subgaussion with respect to $\sigma^2(w)$.  Hint:  This is Proposition E.7 in \cite{Jara-Menezes}.
\end{exercise}

Since $0<\rho_-\leq \rho(s, u)\leq \rho_+<1$ on $[0,T]\times \T$, we have $0<2\rho_-\big(1-\rho_+)\leq \sigma^2(w)\leq 2\rho_+\big(1-\rho_-\big)<1$.   Then,
\[E_{\nu^N_s}\Big[\exp\Big\{ 2\delta l \big| \eta^\ell(w) - \rho^\ell(s, w/N)\big|^2\Big\}\Big] \leq 3\]
when $2\delta < 1/[8\rho_+(1-\rho_-)]$.

Hence, for such a $\delta$,
\begin{align*}
\frac{1}{\delta N}\log E_{\nu^N_s}\Big[\exp\Big\{\delta \sum_x q_x\Big\}\Big] &\leq \epsilon(N, \ell, \delta):=\frac{C}{\delta \ell} + \frac{C\ell}{N},
\end{align*}
which vanishes as $N\uparrow\infty$ and then $\ell\uparrow \infty$.

\medskip
{\it Step F.}
Putting things together, we observe that the entropy is bounded as
\begin{align*}
\frac{1}{N}\big(H(t) - H(0) \big)\leq \frac{C}{N\delta }\int_0^t H(s)ds + \epsilon(N, \ell, \delta)t.
\end{align*}
We have taken the initial distribution to be $\nu^N_0$, in which case $H(0)=0$ (an estimate $H(0)=o(N)$ with respect to the initial condition would also suffice).
By Gronwall's bound, we conclude as desired that
\[\frac{1}{N}H(t)\leq  (\delta/C)e^{C\delta^{-1} T} \epsilon(N, \ell, \delta)T = o(1)\]
as $N\uparrow\infty$ and $l\uparrow\infty$.
\qed

\section{Notes}
We have followed the scheme in \cite{Varadhan_notes}, although there are differences; see also Chapter VI in \cite{KL}.  
The hydrodynamic limit presented in Theorem \ref{thm:sec61} may be extended to all times $t\geq 0$, by different methods, where $\rho(t, u)$ is the `physical' solution mentioned earlier.  See for instance \cite{Rost}, when starting from a `step' initial profile $u_0 = 1(u\leq 0)$, and \cite{Ravi}, \cite{rezakhan-hyd} for general initial data.  

However, in general, hydrodynamics for all times has not been shown for processes with drift in Euler scale $\upsilon(N)=N$ that do not satisfy the `basic coupling' or are not `attractive' in the sense given in Subsection \ref{coupling}. It is an open problem to show such hydrodynamics. Here, TASEP can be verified to be `attractive', used in \cite{Ravi}, \cite{rezakhan-hyd}.  See also \cite{Sethuraman-Shahar} for hydrodynamics of long-range asymmetric exclusion and other processes.  

The `relative entropy' method is due to H.T. Yau \cite{yau}.  Other treatments of the method in different contexts may be found in \cite{KL}, \cite{Funaki}.  
Instead of using the concentration inequality for the last argument, one may invoke large deviations bounds as in \cite{KL} for instance.

While the `entropy' method of Guo-Papanicolaou-Varadhan \cite{GPV} presented in Section \ref{lec5} yields existence of a weak solution to the hydrodynamic partial differential equation, the `relative entropy' method shows uniqueness of classical solutions of the hydrodynamic equation: 
Although a requirement for the method is that a priori some smoothness of the hydrodynamic PDE solution is known.

 Interestingly, explicit knowledge of the invariant measures of the process is not needed for the relative entropy method, as Lemma \ref{lem:sec61-1} holds generally with respect to a reference measure $\nu_*$ (we used $\nu_{1/2}$ here for convenience) and approximations $\nu^N_t$ of $\mu^N_t$; see \cite{Jara-Menezes}[Lemma A.1] for the original derivation.

The relative entropy method has been helpful in several other problems, from deriving mean-curvature flows \cite{Funaki-Tsunoda}, \cite{EFHPS}, \cite{Funaki-1}, \cite{Funaki0}, \cite{Funaki1} to non-equilibrium fluctuations \cite{Dagalier-Landim}, \cite{Jara-Menezes}.  Here, the negative Dirichlet term in Lemma \ref{lem:sec61-1} is useful in the estimations.

\newpage
\chapter[Section $7$]{Construction of particle systems in infinite volume}
\label{lec7}

We construct, using the method of Liggett and Spitzer \cite{Liggett_Spitzer} and Andjel \cite{Andjel}, zero-range particle systems on $\Z^d$.  The approach is to take a limit of the processes restricted to large but finite sets.  Such a construction would also hold for exclusion and related models.  Other construction techniques are mentioned in the Notes subsection.

\section{What does it mean to construct a process?}

  In Section \ref{lec1}, on finite or countable state spaces $\Omega$, `construction of the process' meant determining the transition probabilities $P_t(x,y)= P(\eta(t)=y|\eta(0)=x)$ of a Markov chain, in terms of infinitesimal rates, from which probabilities $P(\eta(t)\in A|\eta(0)=x)$ could be found for $A\subset \Omega$.  One could then write down then the semigroup operator $P_t$, acting on bounded functions, and understand its connection to the generator $L$ in terms of backward and forward equations.
    
On more exotic spaces $\Omega$, to construct a Markov process usually means building a semigroup $P_t$, that is an operator, with the `Chapman-Kolmogorov' property $P_{t+s}=P_tP_s$ for $s,t\geq 0$, acting on a collection of functions on $\Omega$, and relating $P_t$ to a generator $L$ via backward and forward evolution equations.  In this context, one can usually associate, by Kolmogorov's extension theorem, a probability $P^\eta[\eta(t)\in d\zeta]$, with initial condition $\eta$, on the Borel sets in $\Omega$.  Sometimes, we will want the semigroup $P_t$ to have certain properties, such as `strong continuity' or the `Feller property', but this is not guaranteed (cf. \cite{Liggett_zr}[p. 247]).

For the particle systems we have studied, exclusion and zero-range models, we would like to extend the background space $\T^d_N$ to $\Z^d$.  Then, the configuration space would be $\Omega = \{0,1\}^{\S}$ for exclusion systems, and $\Omega = \N_0^{\S}$ for zero-range models.  With respect to exclusion systems, $\Omega$ is compact, which is helpful and allows different ways to construct the process.  However, for zero-range models, since $\Omega$ is not compact, some care must be taken and assumptions on the parameters, namely the rate function $g$ and transition probability $p$, should be made.  

Interestingly, if $g$ is in general unbounded, but Lipschitz, not all configurations $\eta\in \Omega$ may be `allowed`.  In other words, we will construct the zero-range process on a strict subset $\Omega'\subset \Omega$, so that once begun in $\Omega'$, the process will stay in $\Omega'$.  Part of the reason for this restriction is that if the rate function is large, particles jump faster, and the process may be influenced from particles at `infinity'; see the end of Section 2 in \cite{Andjel} for an example.

 On the other hand, if $g$ is bounded, then the process can be constructed on the full $\Omega$ by a different method \cite{Holley}.
 
 \section{Zero-range model and statement of results}
 
 We will assume in the following that
 $g:\N_0\rightarrow \R_+$ satisfies $g(0)=0$ and $g(k)>0$ for $k\geq 1$, and $g$ is Lipschitz:  There is a constant $a_0$ such that
 \begin{itemize}
 \item[(LIP)] $\sup_{k\geq 0} |g(k+1)-g(k)|\ \leq \ a_0.$
\end{itemize}
In particular, we do not assume in this Section \ref{lec7} that $g$ is increasing.

Also, we will take the jump probability $p$ on $\S$ ($p(x,y)\geq 0$ and $\sum_{y\in \S}p(x,y)=1$) to be irreducible and such that $\lim_{x\rightarrow\infty} p(x,y) = 0$ for all $y\in \S$.  Note that we do not assume that $p$ is translation-invariant, e.g. $p(x,y) = p(0,y-x) = p(y-x)$, or that $p$ is finite-range, that is when $p(x)=0$ for $|x|>R$ for some $R<\infty$.

\medskip
We now specify the allowed configuration space $\Omega'$.  For $x\in \S$,
let $$\beta(x) \ =\  \sum_{n\geq 0} \frac{p^{(n)}(x,0)}{2^n}$$
where $p^{(n)}(x,y)$ is the probability that a single particle reaches $y\in \S$ in $n$ steps from $x$.
Observe that 
\begin{equation}
\label{beta_ineq}
\sum_{y\in \S}p(x,y)\beta(y) \ \leq \ 2\beta(x).\end{equation}

\begin{exercise}\rm
Show \eqref{beta_ineq}.
\end{exercise}

Define for $\eta,\zeta\in \Omega = \N_0^\S$ the norm
$$\|\eta-\zeta\| \ = \ \sum_{x\in \S} |\eta(x)-\zeta(x)|\beta(x).$$
The collection of allowed configurations will be
\begin{equation}
\label{def:omega'}
\Omega' \ = \ \Big\{\eta\in \Omega| \|\eta\|= \sum_{x\in \S}\eta(x)\beta(x) <\infty\Big\}.
\end{equation}
Examples of configurations in $\Omega'$ include those with only a finite number of particles, for instance:  We say a configuration $\eta$ is finite if the number of particles $\sum_{x\in \S}\eta(x)<\infty$.

We now define a class of `Lipschitz' functions on which we construct the process.  We say that $f:\Omega'\rightarrow \R$ is Lipschitz if there is a constant $c$ such that
$$|f(\eta) - f(\zeta)| \ \leq \ c\|\eta-\zeta\|$$
for all $\eta,\zeta\in \Omega'$.  Let $c(f)$ be the smallest such constant $c$.

Denote by $\L$ the collection of all Lipschitz functions on $\Omega'$.  Note that $\L$ includes `simple' functions, those say depending on a finite number of variables $\eta(x)$ and taking on a finite number of values.  Moreover, trivially, $f(\eta)= \|\eta\|$ belongs to $\L$ with $c(f)=1$.

Define the operator $L$, which we will identify later as our generator, on functions in $\L$ and $\eta\in \Omega'$ by
$$(Lf)(\eta) \ = \ \sum_{x,y\in \S} p(y)g(\eta(x))\big[f(\eta^{x,x+y}) - f(\eta(x))\big].$$

\begin{lemma}
\label{lem:Lbound}
The operator $L$ is well defined for $f\in \L$ and $\eta\in \Omega'$, and
$$|Lf(\eta)| \ \leq \ 3a_0c(f)\|\eta\|.$$
\end{lemma}

\begin{proof}
Note that $\|\eta^{x, x+y} - \eta\| = \beta(x+y) + \beta(x)$.  The desired estimate follows from
\begin{eqnarray*}
|Lf(\eta)| & \leq & \sum_{x,y}p(x,x+y)g(\eta(x))\big|f(\eta^{x,x+y})-f(\eta)\big|\\
&\leq& c(f)\sum_{x,y}p(x,x+y)g(\eta(x))\big(\beta(x+y) - \beta(x)\big)\\
&\leq& a_0c(f)\sum_{x,y}p(x,x+y)\eta(x)\big(\beta(x+y) +\beta(x)\big)\\
&\leq& 3a_0c(f)\|\eta\|
\end{eqnarray*}
using (LIP), and \eqref{beta_ineq} in the last step.
\end{proof}

We now come to the main theorems.  

\begin{theorem}
\label{construction1}
There exists a semigroup $P_t$ on $\L$, which has specification $P_tf(\eta) = E^\eta[f(\eta_t)]$ for $f\in \L$ and finite configurations $\eta$ in terms of the countable-state process.

Moreover, for $\eta, \zeta\in \Omega'$ and $f\in \L$, the semigroup satisfies
$$\big|P_tf(\eta) - P_tf(\zeta)\big| \ \leq \ c(f)e^{4a_0t} \|\eta-\zeta\|$$
and $c(P_t f) \leq c(f)e^{4a_0t}$.
Also,
$$P_tf(\eta) \ = \ f(\eta) + \int_0^t LP_sf(\eta)ds.$$
\end{theorem}

We observe, when there are only a finite number of particles in the system, $P_t$ and $L$ are the semigroup and generator of a countable state Markov chain.

More properties are given in the following result.
\begin{theorem}
\label{construction2}
We have for $f\in \L$ and $\eta\in \Omega'$ that
\begin{itemize}
\item [(i)] $|L P_s f(\eta)| \leq 3c(f)e^{4a_0 s}\|\eta\|$ and $|P_t f(\eta) - f(\eta)| \leq (4a_0)^{-1}c(f)\|\eta\|(e^{4a_0t}-1)$
\item [(ii)] $\lim_{t\downarrow 0} t^{-1}\big[P_tf(\eta) -f(\eta)\big] = Lf(\eta)$
\item [(iii)] $LP_tf(\eta) = P_tLf(\eta)$.
\end{itemize}
\end{theorem}
We prove Theorems \ref{construction1} and \ref{construction2} in Subsections \ref{sec:construction1}, and \ref{sec:construction2}.

\subsection{The meaning of Theorem 7.2.3}
\label{sec:meaning_construction}

Recall that `simple' or cylinder functions, $f(\eta) = \prod_{j=1}^k 1_{A_j}(\eta(x_j))$ for $A_j\subset \N_0$ and $\{x_j\}_{j=1}^k \subset \S$, $k<\infty$, are Lipschitz functions.  Since, for a given $\eta\in \Omega'$, one may approximate $\eta$ by finite configurations $\zeta^n$ such that $\|\eta - \zeta^n\|\downarrow 0$ as $n\uparrow\infty$, by Theorem \ref{construction1}, the semigroup $P_tf(\eta)$ may be computed from that of countable state Markov chains:
$P_tf(\eta) = \lim P_tf(\zeta^n)=\lim E^{\zeta^n}[f(\eta_t]$.  Therefore, the finite dimensional distributions of $(\eta_t(x), t\geq 0, x\in \S)$ for a given initial configuration $\eta\in \Omega'$ may be identified.

Then, by Kolmogorov's extension theorem, there exists a probability measure $P^\eta[\eta(t)\in d\zeta]$ on Borel sets in $\Omega$ (not necessarily for the moment on $\Omega'$!) for each $t\geq 0$ and $\eta\in \Omega'$ such that
\begin{equation}
\label{finite}
P_tf(\eta) \ = \ \int P^\eta[\eta_t\in d\zeta]f(\zeta) \ = \ E^\eta[f(\eta_t)]\end{equation}
for cylinder functions $f$.

\begin{lemma}\label{stayinset} For $\eta\in \Omega'$, we have $E^\eta\big[\|\eta_t\|\big]<\infty$, and therefore the measure $P^\eta[\eta_t\in d\zeta]$ concentrates on $\Omega'$.
\end{lemma}

\begin{proof}
We apply Theorem \ref{construction1} to the function $f(\eta) = \|\eta\|$ belonging to $\L$ with $c(f)=1$.  Hence,
$$e^{4a_0t}\|\eta\| \ \geq \ |P_tf(\eta)| \ = \ |E^\eta[f(\eta_t)]| \ \ = \ E^\eta\big[\|\eta_t\|\big]$$
and so $\|\eta_t\|<\infty$ a.s. starting from $\eta$. \end{proof}

\begin{lemma} 
\label{lem:7.2.6}
For $\eta\in \Omega'$, we may identify for $f\geq 0$ or $f\in \L$ that
$$P_tf(\eta) \ = \ \int P^\eta[\eta_t\in d\zeta]f(\zeta) \ = \ E^\eta[f(\eta_t)].$$
\end{lemma}

\begin{proof} When $f\geq 0$, one can approximate $f$ by nonnegative simple functions (belonging to $\L$). Then, by monotone convergence, one may take a limit in \eqref{finite}, and define $P_tf(\eta)$ in this way.

When $f\in \L$, however, not necessarily positive, we may still approximate $f(\eta)$ by simple functions in $\L$, dominated by $|f(\eta)|$ pointwise for $\eta\in \Omega'$.  Since $|f(\eta)| \leq c(f)\|\eta\| + |f(0)|$ when $\eta\in \Omega'$, the 
 point is that, by Lemma \ref{stayinset} $\eta\mapsto \|\eta\|$ is integrable, and so we can pass to the limit in \eqref{finite} to define $P_tf(\eta)$.
 \end{proof}

\section{Construction estimates on a finite cube}

The strategy to construct an infinite volume semigroup $P_t$ is first to obtain estimates on the semigroup and generator which are well defined when the space is finite, and then use these estimates to define $P_t$ as the volume grows.  Throughout this subsection, the underlying space will be a cube of width $2L+1$: 
$$A_L\ = \ \big\{x: |x_i|\leq L, 1\leq i\leq d\big\}.$$ 
Let $P_t = P_t^{(A)}$ and $\eta_t=\eta_t^{(A)}$ be the (countable-state) zero-range process corresponding to a transition probability $p(x,y)=p_A(x,y)$ on $A=A_L$, with no transitions from $A$ to $A^c$.

We now bound the means of such a process at times $t\geq 0$.

\begin{lemma}
\label{branching}   For $t\geq 0$, and $y\in A$, we have
$$E^\eta[\eta_t(y)] \ \leq \ \sum_{x\in A}\eta(x) \sum_{\ell =0}^\infty \frac{(a_0t)^\ell}{\ell!} p^{(\ell)}(x,y).$$
Here, $p^{(\ell)}(x,y)$ is the $\ell$-step transition probability from $x$ to $y$.
\end{lemma}

\begin{proof}  The argument is by coupling the zero-range process to a continuous-time multitype branching process $\eta^+_t$ on $\N_0^A$ with generator
$$L^+f(\eta^+) \ = \ \sum_{x,y} a_0\eta(x)p(x,x+y)\big[f(\eta^+ + \delta_{x+y})-f(\eta^+)\big].
$$
Here, $\eta^+ + \delta_z$ is the configuration which adds a particle at site $z$ to $\eta^+$.  Note also the sum is over $x, y$ where $x, x+y\in A$.  An inspection of the formula reveals that $\eta_t^+$ is such that each particle at $x$ gives birth to a new particle at rate $a_0$.  This new particle is then displaced by $y$ with probability $p(x,x+y)$.  Alternatively, each particle gives birth to two new particles before it dies; one is kept at the birth location $x$, and the other is displaced by $y$.  

The coupling of $\eta_t$ and $\eta^+_t$ as follows:  Since by (LIP), $g(\eta(x))\leq a_0\eta(x)$, whenever a zero-range particle displaces from $x$ to $x+y$, one may arrange also that the branching process gives birth at $x$ and creates a new particle at $x+y$.  In this way, if the two processes are started from the same initial configuration, then $\eta_t(x)\leq \eta_t^+(x)$ for all $x\in A$ and $t\geq 0$.  

In particular, when $\eta\leq \eta^+$ coordinatewise, we have $E^\eta[\eta_t(x)] \leq E^{\eta^+}[\eta^+_t(x)]$ for all $x\in A$.

To finish, we need only calculate $E^{\eta^+}[\eta^+_t(x)]$.  Observe that
\begin{align*}
\frac{d}{dt}E^{\eta^+}[\eta^+_t(x)] &= E^{\eta^+}\big[L^+ \eta^+_t(x)\big]\\
&= \sum_z a_0 E^{\eta^+}[\eta_t^+(z)] p(z, x).
\end{align*}
Then, $\vec w_t :=\big(E^{\eta^+}[\eta_t^+(y)]: y\in A\big)$ satisfies
$\dot {\vec w}_t = a_0 \vec w_t P$ where $P= \big(p(z, x)\big)$ is the transition matrix.  Hence,
$\vec w_t = \vec w_0 e^{a_0 Pt}$.  In particular, 
\begin{align*}
E^{\eta^+}[\eta^+_t(y)]=& \sum_{x\in A} \eta(x) \sum_{\ell=0}^\infty \frac{(a_0t)^\ell}{\ell!} p^{(\ell)}(x,y). \qedhere
\end{align*}
\end{proof}

\begin{lemma}
\label{semigroup_Lip}
For $t\geq 0$ and $f\in \L$, we have $P_tf\in \L$ and $c(P_tf) \leq c(f)e^{3a_0t}$.
\end{lemma}

\begin{proof}
Consider the basic coupling in Subsection \ref{coupling}, the joint Markov process on $(\N_0^A)^2$ generated by
\begin{align*}
\bar L\phi(\eta,\zeta) =& \sum_{x,x+y\in A} \min\{g(\eta(x)),g(\zeta(x))\} p(x,x+y) \big[\phi(\eta^{x,x+y},\zeta^{x,x+y})-\phi(\eta,\zeta)\big]\\ 
&\ + \ \sum_{x,x+y\in A} \big(g(\eta(x))-g(\zeta(x))\big)_+ p(x,x+y) \big[\phi(\eta^{x,x+y},\zeta)-\phi(\eta,\zeta)\big]\\
&\ + \ \sum_{x,x+y\in A} \big(g(\zeta(x))-g(\eta(x))\big)_+ p(x,x+y) \big[\phi(\eta,\zeta^{x,x+y})-\phi(\eta,\zeta)\big].
\end{align*}
Let $\bar P_t$ be the coupled process semigroup.  Both marginals are zero-range process, and if initially $\eta \leq \zeta$, then the coordinatewise ordering is preserved, $\eta(t)\leq \zeta(t)$.

Now write
\begin{eqnarray*}
|P_tf(\eta) - P_t(\zeta)| \ = \ |\bar P_t(f(\eta) - f(\zeta)|\ 
\leq \ c(f) \bar P_t\|\eta - \zeta\|.
\end{eqnarray*}
We give now an estimate of the derivative of the right-hand side.  Note
\begin{align*}
\|\eta^{x,x+y} - \zeta\| - \|\eta - \zeta\| & =  \|\eta - \zeta^{x,x+y}\| - \|\eta - \zeta\| \\
& = \ \beta(x+y) + \beta(x)\end{align*}
and $\|\eta^{x, x+y} - \zeta^{x,x+y}\| - \|\eta-\zeta\| = 0$.
Then, adding together the positive and negative parts of $g(\eta(x))-g(\zeta(x))$, with an application of \eqref{beta_ineq}, we obtain
\begin{eqnarray*}
\bar L\|\eta - \zeta\| & \leq  & a_0 \sum_{x,x+y\in A} |\eta(x)-\zeta(x)|p(x,x+y) \big(\beta(x+y) + \beta(x)\big)\\
&\leq & 3a_0\|\eta - \zeta\|.
\end{eqnarray*}

Hence, since $\frac{d}{dt}\bar P_t \|\eta - \zeta\| = \bar P_t \bar L \|\eta - \zeta\|$,
by comparison, we obtain
\[|P_t f(\eta) - P_tf(\zeta)| \ \leq \ c(f)e^{3a_0t}\|\eta - \zeta\|, \]
which finishes the proof.  
\end{proof}

Let now $(P^1_t, L_1)$ and $(P^2_t, L_2)$ be zero-range processes on $A$ according to jump probabilities $p_1$ and $p_2$ that do not allow transitions from $A$ to $A^c$.

\begin{lemma}
\label{diff_gen_est}
For $f\in \L$, we have
\begin{align*}
&|(L_1-L_2)f(\eta)| \\
&\quad  \ \leq \ a_0c(f)\sum_{x,x+y\in A} \eta(x)|p_1(x,x+y)-p_2(x,x+y)|\big(\beta(x) + \beta(x+y)\big).
\end{align*}
\end{lemma}

\begin{lemma} 
\label{identity}
We have
$$P^1_tf(\eta) - P^2_tf(\eta) \ = \ \int_0^t P^1_s\big[L_1-L_2\big] P^2_{t-s}f(\eta) ds.$$
\end{lemma}

\begin{exercise}\rm
Prove, using previous estimates, Lemma \ref{diff_gen_est}.  Prove also Lemma \ref{identity}, noting that it is a standard inequality which holds generally for a pair of Markov processes.
\end{exercise}

\begin{lemma}
For $f\in \L$, we have
\begin{align}
\label{2.5}
&|P^1_tf(\eta) - P^2_tf(\eta)|\\
& \leq \
a_0c(f)\int_0^t e^{3a_0(t-s)}\sum_{x,y}P^1_s \eta(x) \nonumber\\
&\ \ \ \ \ \ \ \ \ \ \ \ \ \ \ \ \ \ \ \ \ \ \ \ \ \ \ \ \ \ \ \ \cdot |p_1(x,x+y)-p_2(x,x+y)|\big(\beta(x) + \beta(x+y)\big) ds \nonumber\\
& \leq \ a_0c(f)\int_0^t e^{3a_0(t-s)}\sum_{x,y}\Big[ \sum_{z\in A}\eta(z)\sum_{\ell\geq 0} \frac{(a_0s)^\ell}{\ell !} p_1^{(\ell)}(z,x)\Big]\nonumber \\
&\ \ \ \ \ \ \ \ \ \ \ \ \ \ \ \ \ \ \ \ \ \ \ \ \ \ \ \ \ \ \ \ \ \ \ \ \ \ |p_1(x,x+y)-p_2(x,x+y)|\big(\beta(x) + \beta(x+y)\big) ds.\nonumber
\end{align}
\end{lemma}

\begin{proof}  The inequalities follow by applying Lemma \ref{identity}, Lemma \ref{diff_gen_est}, and then Lemmas \ref{semigroup_Lip} and \ref{branching}. \end{proof}

\section{Extension to $\Z^d$ and proofs of the construction theorems}

Let $A=A_L$ increase to $\cup_{L\geq 1} A_L = \Z^d$.  For a transition probability $p$ on $\Z^d$, define the modified transition probabilities:
$$q_L(x,z) \ = \ \left\{\begin{array}{rl}
p(x,z) & {\rm if \ } x,z\in A_L, x\neq z\\
1& {\rm if \ } x=z\not\in A_L\\
p(x,x) + \sum_{w\not\in A_L}p(x,w) & {\rm if \ } x=z\in A_L.
\end{array}\right.$$
Note that $q_L$ does not allow transitions from $A_L$ to $A_L^c$.

It will be useful to observe that $q_L(x,z) \leq p(x,z) + 1(x=z)$.  Hence, $q_L$ satisfies \eqref{beta_ineq} with constant $3$ instead of $2$.

Let $P^{(L)}_t$ and $L^{(A)}$ be the zero-range semigroup and generator corresponding to $q_L$ on $A=A_L$.
 \begin{proposition}
 \label{cauchy}
 For $f\in \L$ and $\eta\in \Omega'$, we have
 $\lim_{L\uparrow\infty}P^{(L)}_tf(\eta)$ converges uniformly on sets of bounded $t$, sets of $\eta$ with bounded $\|\eta\|$, and sets of functions with bounded $c(f)$.
 \end{proposition}
 
 \begin{proof}
 We show that $\{P^{(L)}_tf(\eta)\}_{L\geq 1}$ forms a Cauchy sequence with the uniformity properties.
Consider $p_1 = q_L$ and $p_2= q_{L'}$, where $L\leq L'$.  Both $p_1, p_2$ are transition probabilities on $A_{L'}$.   We have that the integrand in \eqref{2.5} is bounded by
\begin{eqnarray*}
&&e^{3a_0(t-s)}\sum_{x\in A_{L'}}\Big[ \sum_{z\in A_{L'}}\eta(z)\sum_{\ell\geq 0} \frac{(a_0s)^\ell}{\ell !} q_{L}^{(\ell)}(z,x)\Big](2+ 2\cdot 3)\beta(x)\\
&&\ \ \leq  \ 8e^{3a_0(t-s)}\sum_{z\in A_{L'}} \eta(z)\sum_{\ell\geq 0}\frac{(3a_0s)^\ell}{\ell !}\beta(z)\
 \leq \ 8e^{3a_0t}\|\eta\|.
\end{eqnarray*}
Here, we used \eqref{beta_ineq} several times.  

Hence, the integrand is dominated.  Given 
$|p_1(x, x+y) - p_2(x,x+y)| \leq \sum_{w\not\in A_L} p(x, w) + \sum_{w\not\in A_{L'}}p(x,w)$ vanishes, the integrand vanishes pointwise for each $0\leq s\leq t$, as $L, L'\uparrow\infty$.

Plugging into \eqref{2.5}, we see that the convergence is uniform as desired.        
\end{proof}

\subsection{Proof of Theorem \ref{construction1}.}
\label{sec:construction1}
By Proposition \ref{cauchy}, we may now define, for $f\in \L$ and $\eta\in \Omega'$ that
$$P_tf(\eta) \ = \ \lim_{L\uparrow\infty} P^{(L)}_tf(\eta).$$ 
Moreover, by Lemma \ref{semigroup_Lip}, since $q_L$ satisfies \eqref{beta_ineq} with constant $3$, we have
$$c(P^{(L)}_tf) \ \leq \ c(f)e^{4a_0t}.$$

Consider now the two properties stated in Theorem \ref{construction1}.
We first establish the semigroup property for $P_t$:  Namely, $P_{t}P_s = P_{t+s}$ for $s,t\geq 0$.  Already this property holds for $P^{(L)}_t$.  It is sufficient to show for $f\in \L$ and $\eta\in \Omega'$ that
\begin{equation*}
 \lim_{L\uparrow\infty} [P_t-P^{(L)}_t]P^{(L)}_sf(\eta)  =  0 \ \ {\rm and} \ \ 
  \lim_{L\uparrow\infty} P_t[P^{(L)}_s-P_s]f(\eta)  =  0.
\end{equation*}

The first limit follows from the uniform convergence in Proposition \ref{cauchy}, given $c(P^{(L)}_sf) \leq c(f)e^{4a_0s}$ uniformly in $L$.

For the second limit, for fixed $t$ and $\eta\in \Omega'$, by construction, we may find by Kolmogorov's extension theorem (as discussed in Subsection \ref{sec:meaning_construction}) a probability 
$\mu:= P^\eta(\eta_t\in d\zeta)$ on $\Omega$ such that $\int \|\zeta\|\mu(d\zeta)<\infty$ and, for $h\in \L$,
$$P_t h(\eta) \ =\  \int h(\zeta) \mu(d\zeta).$$
Now observe 
$$|P^{(L)}_sf(\eta) - P_sf(\eta)| \ = \ |P^{(L)}_s(f(\eta) - f(0)) - P_s(f(\eta)-f(0))| \  \leq \ 2c(f)e^{4a_0s}\|\eta\|.$$
Hence, by dominated convergence, 
$$P_t[P^{(L)}_s-P_s]f(\eta)  = \int [P^{(L)}_sf(\zeta) - P_sf(\zeta)]\mu(d\zeta) \ \rightarrow \ 0.$$

To establish the integral property, note that it holds for the finite-volume semigroup $P^{(L)}_t$ and generator $L^{(A_L)}$:
$$P^{(L)}_tf(\eta) - f(\eta) \ = \ \int_0^t L^{(A_L)}P^{(L)}_sf(\eta)ds.$$
We can already pass to the limit on the left-hand side.  To pass on the right-hand side we need to show that
\begin{equation}
\label{lim_ex}
L^{(A_L)}P^{(L)}_s f(\eta) - LP_sf(\eta) \ \rightarrow \ 0\end{equation}
and that the the integrand is dominated.  Domination holds as
\begin{equation}
\label{gen_bound}
|L^{(A_L)}P^{(L)}_sf(\eta)| \ \leq \ 3a_0c(f)e^{4a_0s}\|\eta\|.\end{equation}
We leave to the reader to show \eqref{lim_ex}.
\qed
\medskip
\begin{exercise}\rm  Show \eqref{lim_ex} using estimates developed, and the explicit form of $L$.  Hint:  See the argument of Lemma 2.12 in \cite{Liggett_Spitzer}.
\end{exercise}

\subsection{Proof of Theorem 7.2.4}
\label{sec:construction2}
We will use the previous estimates to show the items in Theorem \ref{construction2}.  

\medskip
\noindent {\it Proof of (i).}  From Lemma \ref{lem:Lbound} and Theorem \ref{construction1} (see also \eqref{lim_ex} and \eqref{gen_bound}), we have that
$$|LP_sf(\eta)| \ \leq \  3c(f)e^{4a_0s}\|\eta\|.$$
Now, plugging into the integral expression in Theorem \ref{construction1}, and integrating, we complete (i).
\medskip

\noindent {\it Proof of (ii).}  If we can show that $LP_sf(\eta)$ is continuous at $t=0$, then (ii) follows from the integral formula in Theorem \ref{construction1}.  Now, from part (i), $P_t$ is continuous at $t=0$.  Write
$$LP_sf(\eta) \ = \ \sum_{x,x+y} g(\eta(x))p(x,x+y) \big[P_sf(\eta^{x,x+y})-P_sf(\eta)\big].$$
To show continuity of $LP_sf(\eta)$, we dominate
$$|P_sf(\eta^{x,x+y})-P_sf(\eta)| \ \leq \ c(P_sf)\|\eta^{x,x+y}-\eta\| \ \leq \ c(f)e^{3a_0t}[\beta(x+y) + \beta(x)].$$
The right-hand side bound is summable:
$$\sum_{x,x+y} g(\eta(x))p(x,x+y) [\beta(x+y) + \beta(x)] \ \leq \ 3a_0\|\eta\|.$$
Hence, the desired continuity follows from dominated convergence.

\medskip

\noindent {\it Proof of (iii).}  Write
\begin{eqnarray*}
LP_tf(\eta) & = & \lim_{s\downarrow 0} s^{-1}\big[ P_sP_tf(\eta) - P_tf(\eta)\big]\\
&=& \lim_{s\downarrow 0} P_t \frac{P_sf - f}{s}(\eta)\\
&=& P_t Lf(\eta).
\end{eqnarray*}
The first equality is from part (ii).  The second equation is from the semigroup property proved in Theorem \ref{construction1}.  The third equality is from dominated convergence and Lemma \ref{lem:7.2.6}, writing $P_th(\eta) = E^\eta[h(\eta_t)]$ for $h = [P_s f -f]/s$.  To justify these steps, we note the pointwise convergence is established already in part (ii).  The domination and $h$ belonging to $\mathcal L$ follows from part (i):  For $0<s\leq 1$, using $|e^z - 1|\leq |z|e^{|z|}$,
$$\big|s^{-1}\big[P_sf - f\big](\eta)\big| \ \leq \ (4a_0)^{-1}c(f)\|\eta\|s^{-1}|e^{4a_0s} - 1| \leq 2c(f)e^{4a_0}\|\eta\|.$$

This completes the argument. \hfill \qed

\section{Notes}

The material follows \cite{Andjel} which is almost always cited when working with zero-range processes.  On the other hand,
other construction methods exist.  For instance, if $g$ is bounded, one can construct the semigroup by the Hille-Yosida theorem as in \cite{Holley}.  \cite{Liggett_zr} gives another way, along the lines developed above, with different estimates, to construct the semigroup.  See also \cite{Marton} for yet another construction.

The exclusion process can also be constructed on $\Z^d$ in the way above, but can be done also using the Hille-Yosida theorem; see \cite{Liggett1}.  In this setting, local functions, those which depend on a finite number of variables $\{\eta(x): x\in \Z^d\}$, provide a natural core for the generator.

\newpage
\chapter[Section $8$]{Invariant measures: ergodicity and extremality}
\label{lec8}

We consider the invariant measures in zero-range processes on $\Z^d$.  Given the mass-conservative dynamics, there are several invariant measures corresponding to different particle densities.  
 To determine them, since the set of invariant measures is a convex set, it is useful to identify extreme points of this set.  In particular, if an invariant measure is an extreme point, the process run under it has interesting ergodic properties; see Subsection \ref{sec:extremal}. 
 
  Our focus here will be to show that a class of product invariant measures $\nu_\rho$ are extreme points.  With respect to zero-range models, many things are known, but there are still several open questions. 

\section{Invariant measures for zero-range processes}

To simplify the discussion, we will assume in the remainder of Section \ref{lec8} that $p$ is translation-invariant, but not necessarily finite-range.  However, a good theory exists when $p$ is not translation-invariant, remarked upon in the Notes subsection.  We will also keep the assumptions on $g$ made in Section \ref{lec7}:  Namely, $g(0)=0$, $g(k)>0$ for $k\geq 1$, and the (LIP) condition, $|g(k+1)-g(k)|\leq a_0$ for $k\geq 0$.

Recall the discussion of invariant measures for the zero-range process on a torus $\T^d_N$ in Subsection \ref{sec:zero-rangemodels}.  Our initial goal will be to extend these notions to $\Z^d$.  Recall the marginal $\kappa_\Psi$ where
$$\kappa_\Psi(k) \ = \ \left\{\begin{array}{rl}
\frac{1}{Z(\Psi)} \frac{\Psi^k}{g(k)!} & \ {\rm for \ }k\geq 1\\
\frac{1}{Z(\Psi)} & \ {\rm for \ } k=0
\end{array}
\right.
$$
for $0\leq \Psi < \liminf g(k)$, and form the measure
$$\bar\nu_\Psi \ = \ \prod_{x\in\Z^d}\kappa_\Psi.$$
As before, $\rho = \rho(\Psi) = E_{\bar\nu_\Psi}[\eta(0)]$ is a strictly increasing function of $\Psi$, and hence the inverse exists.  One then defines
$$\nu_\rho \ = \ \bar\nu_{\Psi(\rho)}$$
for $\rho < \rho^*:=\lim_{\Psi \uparrow \liminf g(k)}\rho(\Psi)$.

Recall $\Omega = \N_0^{\Z^d}$, and from Section \ref{lec7}, the norm $\|\eta\| = \sum_{x\in \Z^d}\eta(x)\beta(x)$ where $\beta(x) = \sum_{n\geq 0}p^{(n)}(x,0)/2^n$, and subset $\Omega' = \{\eta\in \Omega: \|\eta\|<\infty\}$.

\begin{lemma}
\label{lem:norm-squared}
We have $E_{\nu_\rho}\big[\|\eta\|\big]<\infty$ and hence the measure $\nu_\rho$ fully charges $\Omega'$, that is 
\begin{equation}
\label{fully_charge}
\nu_\rho(\Omega') \ =\  1.\end{equation}
Moreover, we have $E_{\nu_\rho}\big[\|\eta\|^2\big]<\infty$.
\end{lemma}

\begin{proof}
  Write
$$E_{\nu_\rho}\big [\|\eta\| \big ] \ = \ \sum_{x\in \Z^d} E_{\nu_\rho}[\eta(x)]\beta(x) \ = \ \rho\sum_{x\in \Z^d}\beta(x),$$
where
$
\sum_{x\in \Z^d}\beta(x) =  \sum_{x\in \Z^d} \sum_{n\geq 0}\frac{p^{(n)}(x,0)}{2^n}  =  2
$, since $p^{(n)}(x,0) = p^{(n)}(0,-x)$ by translation-invariance of $p$.

Moreover, by Schwarz inequality,
\begin{align*}
E_{\nu_\rho}\big[\|\eta\|^2\big] & = E_{\nu_\rho}\Big[\Big(\sum_x \eta(x)\beta(x)\Big)^2\Big]\\
&\leq \sum_x \beta(x) \cdot \sum_x E_{\nu_\rho}\big[(\eta(x))^2\big]\beta(x)\\
&\leq 4E_{\nu_\rho}\big[(\eta(0))^2\big].\qedhere
\end{align*}\end{proof}

We now update the finite-volume definition of an invariant measure to $\Z^d$.  Recall the Lipschitz functions $\L$ and the semigroup $P_t$ of the process on $\L$ from Section \ref{lec7}.
We will say that a probability measure $\mu$ on $\Omega'$ is an {\it invariant measure} if 
$$\int P_t f d\mu \ = \ \int f d\mu$$
for all bounded Lipschitz functions $f\in \L$.  Let $\I$ denote the convex set of invariant measures for the process on $\Z^d$.

Our first main theorem is the following infinite volume invariance.
\begin{theorem}[Invariance]
\label{inv_thm}
For $\rho< \rho^*$, we have $\nu_\rho\in \I$.  
\end{theorem}

\section{Proof of the invariance of $\nu_\rho$ in infinite volume}
It will be convenient to give a generator characterization of invariance of a measure.

\begin{proposition}
\label{generator_characterization}
Let $\mu$ be a probability measure on $\Omega$ such that $E^\mu\big[\|\eta\|\big]<\infty$.  Then, the following are equivalent:
\begin{itemize}
\item [(1)] $\int Lf d\mu = 0$ for all bounded $f\in \L$
\item [(2)] $\mu\in \I$.
\end{itemize}
\end{proposition}

In what follows, we will only use that `(1) implies (2)', and so we will prove this part of the result.

\begin{proof}[Proof of Proposition \ref{generator_characterization}: (1) $\Rightarrow$ (2)]
For bounded $f\in \L$, we have by Theorem \ref{construction1} that 
$$\int \Big[P_tf - f\Big] d\mu \ = \ \int \Big[\int_0^t LP_sf ds \Big]d\mu.$$
Since $E^\mu\big[\|\eta\|\big]<\infty$, we have $\mu$ is supported on $\Omega'$.
For $f\in \L$ and $\eta\in \Omega'$, by Lemma \ref{lem:Lbound} and Theorem \ref{construction1} we have $c(P_s f) \leq c(f)e^{4a_0 s}$ and
\begin{eqnarray*}
|LP_sf(\eta)| 
&\leq& 3a_0c(f)e^{4a_0s}\|\eta\|.
\end{eqnarray*}
Moreover, as $f$ is bounded, $P_s f$ is a bounded, Lipschitz function.

By the domination, we may interchange the integrals by Fubini's theorem.  Then, given (1), we have
$$\int \Big[\int_0^t LP_sf ds \Big]d\mu \ = \ \int_0^t \Big[\int LP_sf d\mu \Big]ds \ = 0$$
and so conclude (2).
\end{proof}

\begin{remark}
\rm
To prove the converse, `(2) implies (1)', we refer the reader to \cite{Andjel}[Lemma 2.9].
\end{remark}

 We now assert that
\begin{equation}
\label{finite_inv}
\nu^n_\rho:=\prod_{x\in A_n}\kappa_{\Psi(\rho)}
\end{equation}
 is an invariant measure of the process with translation-invariant transition probability restricted to the finite set $A_n$.
  Indeed, this follows immediately from (the proof of) Lemma \ref{lem:finite_inv} in Section \ref{lec4}.

\begin{proof}[Proof of Theorem \ref{inv_thm}]  For a transition probability $p$ on $\Z^d$, define the `truncation' (different than in the Section \ref{lec7}),
\begin{align*}
&p_n(x,y)\\
&   =  \left\{\begin{array}{ll} 1 &\ {\rm if \ } x=y\not\in A_n\\
p(x,y) + Q_n^{-1}\big[\sum_{z\not\in A_n}p(x,z)\big]\big[\sum_{z\not\in A_n}p(z,y)\big] & \ {\rm if \ }x,y\in A_n\\
0 & \ {\rm otherwise}
\end{array}\right.
\end{align*}
where
\begin{align*}
Q_n &= \sum_{z\in A_n, y\not\in A_n}p(z,y) = \sum_{z\in A_n}\Big[1 - \sum_{y\in A_n} p(z,y)\Big] \\
&= |A_n| - \sum_{z,y\in A_n}p(z,y)=\sum_{y\in A_n, z\not\in A_n}p(z,y).
\end{align*}

We observe the following properties of $p_n$:

(a).  First, $p_n$ is a transition probability: $\sum_y p_n(x,y) = 1$ for all $x\in \Z^d$.  If $x\not\in A_n$, it is trivial.  If $x\in A_n$, noting that $Q_n = \sum_{z\not\in A_n, y\in A_n}p(z,y)$, we have $\sum_y p_n(x,y) = \sum_{y\in A_n}p(x,y) +\sum_{y\not\in A_n}p(x,y)= 1$.  Moreover, $p_n$ has no transitions from $A_n$ to its complement.

(b).  In addition, $\sum_{x} p_n(x,y) = 1$ for all $y\in \Z^d$:  If $y\not\in A_n$, the claim is trivial.  If $y\in A_n$, write
\begin{align*}
\sum_xp_n(x,y) &=  \sum_{x\in A_n} p(x,y) + Q_n^{-1}\big[ \sum_{x\in A_n,z\not\in A_n}p(x,z)\big]\big[\sum_{z\not\in A_n}p(z,y)\big]  =  \sum_x p(x,y).
\end{align*}
The last quantity equals $1$ if $p$ is doubly-stochastic, the case if $p$ is translation-invariant.

(c).  Also, $p_n \rightarrow p$ for all $x,y\in Z^d$.  Eventually, $x,y\in A_n$.  The claim follows as $Q_n^{-1}\sum_{z\not\in A_n}p(z,y) \leq 1$.
\medskip 

Now, by the comments near \eqref{finite_inv}, as nothing moves away from $A_n$ when starting in $A_n$, we have $\nu_\rho$ restricted to $A_n$ is invariant.  Then, we have $E_{\nu_\rho}[L_n f]=0$ for bounded $f\in \L$ where $L_n$ is the generator for the zero-range process on $A_n$ according to transition probability $p_n$.  Therefore, to show $E_{\nu_\rho}[Lf]=0$, it is enough to show
$$E_{\nu_\rho}\big[|L_nf - Lf|\big] \ \rightarrow\ 0.$$

One may write, noting (LIP), $f\in \L$, $\|\eta^{x,y}-\eta\| = \beta(x)+\beta(y)$ when $\eta(x)\geq 1$ and $x\neq y$, and $f(\eta^{x,y})=f(\eta)$ for $x=y$, that
\begin{eqnarray*}
E_{\nu_\rho} \big[|L_nf - Lf|\big] & \leq & a_0c(f)\sum_{x\in \Z^d} E_{\nu_\rho}[\eta(x)]\sum_{y\neq x} |p(x,y) - p_n(x,y)|\big(\beta(x) + \beta(y)\big)\\
&=&a_0c(f) \rho \sum_{x\in \Z^d} \sum_{y\neq x} |p(x,y) - p_n(x,y)|\big(\beta(x) + \beta(y)\big).
\end{eqnarray*}

For fixed $x\neq y$, 
\[ |p(x,y) - p_n(x,y)|  \leq \frac{ 1(x,y\in A_n)}{Q_n}\sum_{z\not\in A_n}p(x,z) \sum_{z\not\in A_n}p(z,y).\]
Hence, since $\sum_z \beta(z)<\infty$ and $p$ is doubly stochastic,
 the penultimate display vanishes as $n\uparrow\infty$ by dominated convergence. 
 \end{proof}

\section{Extremality, harmonicity and ergodicity}
\label{sec:extremal}

We digress for the moment to an abstract $L^2$ setting.  Let $\Theta$ be a space with Borel sets $\B$.  Let $\eta_t$ be a Markov process on $\Theta$ with process semigroup $T_tf(\eta) = E^\eta[f(\eta_t)]$, acting on bounded functions $f$ (and well-defined extensions). We say here that $Q$ is an invariant measure if for all bounded functions $f$ on $\Theta$ (and therefore $L^2(Q)$ functions) and $t\geq 0$, we have $\int T_t f dQ = \int f dQ$.  
Let $\P_Q$ be the process measure on the path space with initial distribution $Q$.

\begin{definition}
We say $Q$ is an extremal invariant measure if the the following property holds.  When for $0<\epsilon<1$ and invariant probability measures $Q_1$ and $Q_2$ we have $Q = \epsilon Q_1 + (1-\epsilon)Q_2$, then $Q = Q_1=Q_2$.
\end{definition}

We note if the process has only one invariant measure $Q$, then of course it is extremal.  The import of the definition comes when the process is reducible in some way.  For instance, in finite state Markov chains, with exactly two irreducible components $C_1$ and $C_2$, the extreme invariant measures are exactly the unique invariant measures supported on $C_1$ and $C_2$ respectively.

\medskip
One might ask why is it useful to know when an invariant measure is extremal.  It turns out there is an interesting connection with harmonic functions and shift-ergodicity.  

\begin{exercise}
\label{ex:adjoint-harm}
With respect to invariant measure $Q$, noting the representation $T_t f(\eta) = E^\eta\big [f(\eta_t)\big]$, show that $T_t$ is an $L^2(Q)$ contraction.  Let $T^*_t$ be the $L^2(Q)$ adjoint of $T_t$.  Show also that $T^*_t$ is an $L^2(Q)$ contraction for each $t\geq 0$.
\end{exercise}

We will say that $f$ is harmonic, if $T_t f = f$ for all $t\geq 0$.  
\begin{lemma} 
\label{decomp}
The space $L^2(Q)$ is the direct sum of 
$I_0 \ = \ \big\{g\in L^2(Q): T_tg = g, t\geq 0\big\}$
 and the closure of
 $ I_1 \ = \ \big\{T_t h - h: h\in L^2(Q), t\geq 0\big\}$.
\end{lemma}

\begin{proof}
 Let $f$ be perpendicular to all functions in $\bar I_1$, the closure of $I_1$.  Then, $\langle f, T_th-h\rangle_Q = 0$ for all $h$ and $t\geq 0$.  This means $T^*_tf=f$ for all $t\geq 0$.  Therefore, as $T_t$ is a $L^2(Q)$ contraction, $E_Q\big[(T_t f - f)^2\big] = E_Q\big[(T_t f)^2\big] -2E_Q\big[f (T_t f)\big] + E_Q\big[f^2\big] \leq 2E_Q\big[f^2\big] - 2E_Q\big[(T^*_t f)f\big] = 0$. Hence, $T_tf =  f$ and $f\in I_0$ or $(\bar I_1)^\perp \subset I_0$.
 
 On the other hand, let $f\in I_0$.  Then, $\langle f, T_t h -h\rangle_Q = \langle T^*_t f - f, h\rangle_Q$.  Since $f\in I_0$, by a similar argument as above, we conclude $E_Q\big[(T^*_t f - f)^2\big] \leq 0$ and so $T^*_tf =f$ and $\langle f, T_t h - h\rangle_Q=0$.  Hence, $I_0\subset (\bar I_1)^\perp$.
\end{proof}

We have the following ergodic theorem due to Von Neumann.

\begin{proposition}
\label{prop:Von}
Let $f\in L^2(Q)$ and define $\hat f \in L^2(Q)$ to be the projection onto the subspace $H_0$.  In other words, $\hat f$ is the conditional expectation $\hat f = E[f| \mathcal T]$ with respect to the time-shift invariant sets $\mathcal T$.

Then, we have the $L^2(Q)$ convergence as $t\rightarrow\infty$,
$$\frac{1}{t}\int_0^t T_sf \; ds \ \rightarrow \ \hat{f}.$$
\end{proposition}

\begin{proof}
Decompose $f = \hat{f} + g$ where $g\in \bar I_1$.  Clearly, $T_t \hat{f} = \hat{f}$.  Since $\bar g\in \bar I_1$ can be approximated by $g\in I_1$ in $L^2(Q)$, we need to understand the contribution of the term $g$, in form $g = T_uh - h$ for some $h\in L^2(Q)$ and $u\geq 0$.
We see
$$\frac{1}{t}\int_0^t T_sg \; ds \ = \ \frac{1}{t}\Big[\int^{t+u}_t T_s h \; ds - \int^u_0 T_sh \; ds\Big] \ = \ O(t^{-1})$$
in $L^2(Q)$ from H\"older's inequality as $T_t$ is an $L^2(Q)$ contraction.
\end{proof}

We come now to the main result of the subsection.  Recall $\P_Q$ stands for the process measure when starting in $Q$.
\begin{proposition}
\label{equivalence}
Let $Q$ be an invariant measure.  All are equivalent:
\begin{itemize}
\item[(a)] For sets $A\in \B$, $T_tI(A) = I(A)$ $Q$ a.s. $\Rightarrow Q(A) = 0 {\rm \ or \ }1$.
\item[(b)] $\P_Q$ is ergodic:  For each $f\in L^2(Q)$, $\hat{f} = E_Q[f]$, $Q$ a.s.
\item[(c)] $Q$ is extremal.
\end{itemize}
\end{proposition}

Note that part (b) is the same as `time-shift ergodicity' or in other words that the shift invariant $\sigma$-field $\mathcal T$ is trivial:  Sets $\Lambda\in\mathcal T$ satisfy $\P_Q(\Lambda) = 0$ or $1$.

\begin{proof}

`b$\Rightarrow$c' Let $Q$ be an invariant measure whose path measure
is
ergodic.  Write $Q= \epsilon Q_1 + (1-\epsilon)Q_2$ for $0<\epsilon<1$
and invariant measures $Q_1$ and $Q_2$.  
Let now $f$ be a bounded function.  Then, by (b), as $t\rightarrow \infty$,
$\frac{1}{t}\int_0^t (T_sf)ds$ converges to $E_Q[f]$ in $L^2(Q)$ and therefore also in $L^2(Q_1)$.
Moreover, by Proposition \ref{prop:Von}, $\frac{1}{t}\int_0^t (T_sf)ds$ converges to $\hat{f}$ in $L^2(Q_1)$.  Hence,
$\hat{f} = E_Q[f]$ $Q_1$ a.s. and taking expectation, $E_{Q_1}[f]=E_Q[\hat f]=E_Q[f]$.
This gives
$Q_1(B)=Q(B)$ for $B\in \B$ and therefore $Q_1=Q$.

`a$\Rightarrow$b' Let $Q$ be an invariant measure and
suppose that $\P_Q$ is not ergodic.
Then there exists an $f\in
L^2(Q)$
such that $\hat{f}$ is not constant $Q$-a.s.  Let $c$ be such that
$Q(A)=\epsilon$, $0<\epsilon <1$ where $A= \{\hat{f}>c\}$. 

 Now, as
$T_t \hat{f}=\hat{f}$ $Q$-a.s. and
$T_t$ is a positive contraction taking $1$ to $1$, we have
that $T_tI(A)=I(A)$ $Q$-a.s.:  
Indeed, first, as $T_t$ is a positive
operator, $|\hat{f}|=|T_t\hat{f}| \leq
T_t|\hat{f}|$, so that, as $T_t$ is an $L^2$ contraction, 
we have $\big(T_t |\hat f| - |\hat f|\big)^2 \leq T_t |\hat f|^2 + |\hat f|^2 -2|\hat f|T_t|\hat f| \leq T_t |\hat f|^2 + |\hat f|^2 - 2|\hat f|^2$.  Note also $E_Q\big[T_t |\hat f|^2\big] = E_Q\big[|\hat f|^2\big]$.  Then, we have $E_Q\big[ \big(T_t |\hat f| - |\hat f|\big)^2\big] \leq 0$, and so
$Q$-a.s. $T_t|\hat{f}|=|\hat{f}|$.  
Therefore, $\max\{0, \hat{f}\} =
(\hat{f}+|\hat{f}|)/2$ is harmonic.  Further, if $f,g \in L^2$ are
harmonic, then $\max\{f,g\}=\max\{0, f-g\} +g$ is harmonic.
Correspondingly,
$\min\{f,g\} = -\max\{-f,-g\}$ is harmonic.  Of course, $1$ is
harmonic.
All of this gives that $\big(\min\{n\max(0,\hat{f}-c),1\}\big)_{n\geq 1}$
is a sequence of uniformly bounded 
harmonic functions.  The limit, as $n\rightarrow
\infty$,
is $I(A)$ which is therefore harmonic by dominated convergence.  Hence, by (a), $I(A)$ is constant, a contradiction.

`c$\Rightarrow$a' 
Let $A$ be such that $T_tI(A)=I(A)$ $Q$-a.s. and
$Q(A) = \epsilon$ for
$0<\epsilon<1$.  By this relation, the
process begun on $A$ stays in $A$ and, if begun in $A^c$ stays in $A^c$, with $Q$-probability $1$.  Then, we
have that
$Q_1(B) = \epsilon^{-1}Q(B\cap A)$ and
$Q_2(B) = (1-\epsilon)^{-1}Q(B\cap A^c)$ are distinct  invariant probability
measures such that $Q=\epsilon Q_1 +(1-\epsilon)Q_2$.  Therefore
$Q$ is not extremal.
\end{proof}

\section{Extremality of $\nu_\rho$}
We now address the `extremality' of $\nu_\rho$ in the convex set of invariant measures $\I$.  To do this rigorously, we will need to extend the process so that the extended semigroup $T_t^\rho$ and generator $L^\rho$ act on $L^2(\nu_\rho)$ functions, not just Lipschitz functions $\L$.  In this way, we can fit into Subsection \ref{sec:extremal}, and in particular $\nu_\rho$ will be an invariant measure in this setting.

However, to present the main ideas we assume this extension has been done, and that adjoints can be taken.  In the next subsection, we detail the extension to a $L^2(\nu_\rho)$ process.
Recall $s(x,y) = s(y-x) =  \frac{1}{2}\big(p(y-x) + p(x-y)\big)$
is the symmetrized transition probability.

\begin{theorem}
\label{thm:extremal}
When $s(\cdot)$ is irreducible on $\Z^d$, the invariant measure $\nu_\rho$ is extremal.
\end{theorem}

\begin{proof}
The idea is simple.   We will need to understand the Dirichlet form of the process.  In the next subsection, we show for $f\in {\rm Dom}(\rho)$, the domain of the extended generator $L^\rho$, that the associated Dirichlet form satisfies
$$D_\rho(f):= \langle f, -L^\rho f\rangle_{\nu_\rho} \ = \ \frac{1}{2}\sum_{x,y\in \Z^d} s(x,y)E_{\nu_\rho}\big[g(\eta(x))\big(f(\eta^{x,y})-f(\eta)\big)^2 \big].$$

Now, let $f$ be a bounded harmonic function, that is when $P^\rho_t f = f$, $\nu_\rho$ a.s.  Hence, as the limit,
$$\lim_{t\downarrow 0} \frac{1}{t}\big[P^\rho_t f - f\big] \ = \ 0,$$
(trivially) exists $\nu_\rho$ a.s., we conclude $f\in {\rm Dom}(\rho)$ and  $L^\rho f = 0$, $\nu_\rho$ a.s.

Therefore, $D_\rho(f)=0$, and so all summands vanish when $s(x,y)>0$.  In particular, as $g(k)>0$ exactly when $k\geq 1$,
$$f(\eta^{x,y}) = f(\eta) {\rm \ a.s.-}\nu_\rho \ \ {\rm for \ all \ }x,y \ {\rm such \ that \ } s(x,y)>0 {\rm \ when \ }\eta(x)\geq 1.$$
Consequently, $f$ is invariant to the motion of particles!  So, by irreducibility of $s(\cdot)$ and countability of $\Z^d$, one has
$$f(\eta_{x,y}) = f(\eta) \ \ {\rm for \ all \ }x,y \ {\rm a.s.-}\nu_\rho$$
where $\eta_{x,y}$ is the configuration which exchanges values $\eta(x)$ and $\eta(y)$.

Therefore, $f$ is finite permutation invariant.  Since $\nu_\rho$ is a product measure with i.i.d. marginals, by Hewitt-Savage $0-1$ law, $f$ is constant a.s.-$\nu_\rho$.  Inputting into Proposition \ref{equivalence} shows extremality.
\end{proof}

\section{Extension to an $L^2(\nu_\rho)$ process}

We first extend the semigroup defined on $\L$ to $L^2(\nu_\rho)$.  
We define the concept of a Markov semigroup on $\mc X=L^2(\nu_\rho)$.

\begin{definition}
A family of linear operators $P_t$ on $\mc X$ is a Markov semigroup if
(a) $P_0 = I$; (b) for $f\in \mc X$, $P_tf$ is right-continuous in $t$ on $\mc X$; (c) $P_t$ satisfies the semigroup property $P_{t+s} = P_tP_s$ for $s,t\geq 0$; (d) $P_t1 = 1$ for $t\geq 0$; (e) $P_tf\geq 0$ for $t\geq 0$ and nonnegative $f\in \mc X$.
\end{definition}

\begin{lemma}
$P_t$ on $\mathcal{L}$ extends by 
continuity to a Markov semigroup $P^\rho_t$ on 
$L^2(\nu_\rho)$.
\end{lemma}

\begin{proof}  For $f\in \L$, we have $|f(\eta)|\leq c(f)\|\eta\|$, and hence $\mc L \subset L^2(\nu_\rho)$ by Lemma \ref{lem:norm-squared}.  Also, for $f\in \mc L$, we may write $P_t f(\eta) = E^\eta[f(\eta_t)]$.
Therefore, by Jensen inequality,
$[P_tf(\eta)]^2 \leq [P_t |f(\eta)|]^2 \leq P_t f^2(\eta)$.
Then,
\begin{eqnarray*}
\|P_t f\|^2_{L^2(\nu_\rho)} &=& \int [P_t f]^2 d\nu_\rho\\
&\leq& \int P_t f^2 d\nu_\rho \\
&=& \int f^2 d\nu_\rho = \|f\|^2_{L^2(\nu_\rho)}.
\end{eqnarray*}
Given that simple functions, those supported on finitely many variables $\{\eta(x): x\in \Z^d\}$ and taking finitely many values, are contained in $\L$, we observe $\L$ 
is dense in $L^2(\nu_\rho)$.  Therefore, for $f\in L^2(\nu_\rho)$, we may define $P^\rho_t f$ as the Cauchy limit of $\{P_t f_n\}$ for a sequence of simple functions $f_n$ which converge to $f$ in $L^2(\nu_\rho)$.  Such an extension, given Theorems \ref{construction1} and \ref{construction2},
 fulfills the definition of a Markov semigroup on $L^2(\nu_\rho)$, details left to the reader.
\end{proof}

\begin{exercise}
\rm
Verify that $\nu_\rho$ satisfies the definition of an invariant measure given at the beginning of Subsection \ref{sec:extremal} with respect to $P_t^\rho$. 
\end{exercise}

There is a one-to-one correspondence between Markov semigroups and associated generators by the Hille-Yosida theorem; see \cite{EK}.
  
\begin{proposition}[Hille-Yosida Theorem]
For a Markov semigroup $T_t$ on $L^2(Q)$, define
\begin{eqnarray*}
{\rm Dom} &=& \big\{f\in L^2(Q):  \lim_{t\downarrow 0} \frac{T_tf - f}{t} \ {\rm exists}\big\}\\
Lf&=& \lim_{t\downarrow 0} \frac{T_tf - f}{t} \ {\rm for \ }f\in {\rm Dom}
\end{eqnarray*}

Then, if $f\in {\rm Dom}$, we have $T_tf\in {\rm Dom}$ and $\frac{d}{dt} T_tf = LT_tf = T_tLf$.
\end{proposition}

Let $L^\rho$ be the generator associated with semigroup $P_t^\rho$ with domain ${\rm Dom}(\rho)$.  
\begin{lemma}
\label{lem:closure} We have $\L\subset {\rm Dom}(\rho)$, $L^\rho$ on ${\rm Dom}(\rho)$ extends $L$ on $\L$, and $L^\rho$ is the closure of $L$ on $\L$, that is the graph of $L^\rho$ is the closure in $L^2(\nu_\rho)\times L^2(\nu_\rho)$ of the graph of $L$.
\end{lemma}

\begin{proof}
For $f\in \L$ and $\eta\in \Omega'$, by Lemma \ref{lem:Lbound} and Theorem \ref{construction1}, we have 
\begin{eqnarray*}
\frac{1}{t}\big[P_t f - f\big] &=& \frac{1}{t}\int_0^t LP_sf ds\\
&\leq& C(a_0)c(f)\|\eta\|t^{-1}\int_0^te^{4a_0s}ds.\end{eqnarray*}

Since $E_{\nu_\rho}\big[\|\eta\|^2\big]<\infty$ by Lemma \ref{lem:norm-squared}, and
$P^\rho_t$ on $L^2(\rho)$ extends $P_t$ on $\L$, we have by dominated convergence for $f\in \L$ that
$$\frac{1}{t}\big[P_t^\rho f - f\big] \ \rightarrow \ Lf$$
in $L^2(\nu_\rho)$ as $t\downarrow 0$.  Hence, $\L\subset {\rm Dom}(\rho)$, $Lf = L^\rho f$, and $L^\rho$ is an extension of $L$ on $\L$.

Finally, as $P_t: \L \rightarrow \L$ (cf. Lemma \ref{lem:Lbound}), and $\L$ is dense in $L^2(\nu_\rho)$ and therefore in ${\rm Dom}(\rho)$, one may conclude that $\L$ is a `core' for $L^\rho$, that is the closure of $L$ on $\L$ is equal to that of $L^\rho$ on ${\rm Dom}(\rho)$; see \cite[Lemma I.3.3]{EK}.  We comment for later use that bounded functions in $\L$ by the same token also form a `core' for $L^\rho$.
\end{proof}

We now give the formula for the Dirichlet form.  Recall the symmetrized transition probability $s(\cdot)$.
\begin{lemma}
For $f\in {\rm Dom}(\rho)$, we have that
\begin{equation}
\label{eq:Dirichlet_equality}
D_\rho(f) \ = \ \frac{1}{2}\sum_{x,y} s(x,y)E_{\nu_\rho}\big[g(\eta(x))\big(f(\eta^{x,y}) - f(\eta)\big)^2\big].
\end{equation}
\end{lemma}

\begin{proof}
We note, by the last line in the proof of Lemma \ref{lem:closure}, that bounded functions in $\L$ form a `core' for $L^\rho$.  First, the formula \eqref{eq:Dirichlet_equality} may be obtained for bounded $f\in \mc L$, given that $Lf$ is explicitly defined
and $D_\rho(f) = E_{\nu_\rho}\big[f(-Lf)\big]$, and is left as an exercise.

We now extend the representation to ${\rm Dom}(\rho)$.  Let $R(f)$ be the right-hand side of the display in the lemma. 
For $f\in {\rm Dom}(\rho)$, take bounded $f_n \in \L$ so that $f_n\rightarrow
f$
and $L^\rho f_n \rightarrow L^\rho f$ in 
$L^2(\nu_\rho)$.
Then
$$\lim_{n\rightarrow \infty} D_{\rho}(f_n) = D_{\rho}(f), \ \ {\rm
  and} \ \ 
\liminf_{n\rightarrow \infty} R(f_n)\geq R(f)$$
by Fatou's lemma.  Therefore, $R(f)\leq D_{\rho}(f)$ and in
particular,
$R(f)<\infty$ for $f\in {\rm Dom}(\rho)$.  However also,
$$0\leq D_{\rho}(f-f_n) \leq \|f-f_n\|_{L^2}\cdot
\|L^\rho f - L^\rho f_n\|_{L^2}$$
which vanishes as $n\rightarrow \infty$.  Hence, 
$\lim_{n\rightarrow \infty}R(f-f_n) =0$.  With the inequality $ab =\inf_{\epsilon>0}\big\{ \epsilon a^2/2 + \epsilon^{-1}b^2/2\big\}$, one may bound
\begin{align*}
\frac{1}{2}\sum_{x,y} s(x,y)E_{\nu_\rho}\big[g(\eta(x))\big(h_1(\eta^{x,y}) - h_1(\eta)\big)\big(h_2(\eta^{x,y}) - h_2(\eta)\big)\big] \leq \sqrt{R(h_1)R(h_2)}.
\end{align*}
Hence, writing $f_n = (f_n -f) + f$, we may conclude
$\lim_{n\rightarrow \infty}R(f_n) = R(f)$ to finish the proof. 
\end{proof}

\begin{exercise}
\rm
Show that $D_\rho(f)=-\langle f, Lf\rangle_{\nu_\rho}$ satisfies \eqref{eq:Dirichlet_equality} for bounded $f\in \mc L$.  Hint:  Show $f(\eta)Lf(\eta)$ is absolutely summable and integrable with respect to $\nu_\rho$.  Then, by Fubini's theorem, one can interchange the expectation and sum.  Noting $E_{\nu_\rho}\big[g(\eta(x))f(\eta^{x,y})f(\eta)\big]=E_{\nu_\rho}\big[g(\eta(y))f(\eta)f(\eta^{y,x})\big]$, one may reorganize the sum.  See also Lemma 2.4 in \cite{Sethuraman}.
\end{exercise}

\section{Notes}

In these notes, we have followed \cite{Andjel} for the invariance of $\nu_\rho$ and \cite{Sethuraman} for the extension to $L^2$ (see also Section IV.4 in \cite{Liggett1}) and extremality of $\nu_\rho$; see these papers for extensions to non translation-invariant transition probabilities.  See also \cite{Rosenblatt}[Section 2] for more on the connections between extremality, harmonicity, and ergodicity.

A natural question is what are all the extremals of the process.    In \cite{Andjel}, it is shown in $d=1,2$ when $g$ is an increasing function that $\{\bar\nu_\Psi: \Psi<\liminf g(k)\}$ are all the extremals!  When $p$ is not translation-invariant, also in $d=1,2$ when $g$ is increasing, all the extremals are found--in this case, the invariant measures are not necessarily translation-invariant \cite{Andjel}.  When $p$ is positive-recurrent, the extremals concentrate on configurations with a finite number of particles \cite{Waymire}, \cite{Andjel}.

For simple exclusion on $\Z^d$, since local functions, those depending on a finite number of variables $\{\eta(x):x\in \Z^d\}$, form a core, one may verify that the Bernoulli product measures $\nu_\rho = \prod_{x\in \Z^d}{\rm Bern}(\rho)$ for $\rho\in [0,1]$ are invariant for the infinite volume process (as compared to the finite-volume one in Proposition \ref{prop:exclusion_invariant}); see Chapter VIII in \cite{Liggett1}.  One may also use the arguments given here for the zero-range process.  Extremality of $\nu_\rho$ can also be seen from the arguments for Theorem \ref{thm:extremal}.  

However, for zero-range and exclusion models, there remain many open problems!  In particular, determining all the extremal invariant measures in all model settings is still not fully resolved.
See \cite{Liggett1}, \cite{Sethuraman}, as well as \cite{Amir}, \cite{BSS}, \cite{CRS}, \cite{Bramson1}, \cite{Bramson2}, \cite{Gantert0}, \cite{Jung}, \cite{Lin} for more references, and discussion.

\newpage
\chapter[Section $9$]{Additive functionals and central limit theorems}
\label{lec9}

The occupation time of a set, for instance, is an additive functional of a process.  One would like to know the first order and second order behaviors, namely the LLN and CLT for these objects.  This is an old subject, and in stochastic particle systems where the configuration space is difficult, with interesting scalings and limits.

We focus here mostly on symmetric exclusion processes on $\Z^d$, when starting from an invariant measure.  After a general introduction, we discuss ergodic and fluctuation behaviors, when starting under an invariant measure.  The Kipnis-Varadhan CLT, which has wide application, not only to particle systems, is stated and proved.   Comments on the behaviors in asymmetric simple exclusion are also made.

\section{Basic problem and ergodic theory}
\label{sec:9.1}

Consider a Markov process $\eta(t)$ on a state space $\Omega$.    Let $\mu$ be an invariant measure for the process.
Let $f:\Omega \rightarrow \R$ be an $L^2(\mu)$ function.  Recall that $E_\mu$ stands for the expectation under $\mu$, and $\P_\mu$ and $\E_\mu$ for the process measure and expectation when starting in $\mu$.

The basic problem is to determine the behavior as $t\rightarrow\infty$ of 
$$A_f(t) \ = \ \int_0^t f(\eta_s)ds.$$
In the following, we call $A_f(t)$ an `additive functional' since $A_f(t+s) - A_f(s)$ under $\mu$ has the same distribution as $A_f(t)$ under $\mu$.

If $\mu$ is extremal, then, as we have seen in Proposition \ref{equivalence}, the process started from initial distribution $\mu$ is ergodic with respect to time shifts, and hence (by Birkhoff's theorem)
$$\lim_{t\uparrow\infty} \frac{1}{t}\int_0^tf(\eta_s)ds \ = \ E_\mu[f] \ \ \mu-{\rm a.s.}$$
This covers the case when $\mu$ is unique, for instance, in positive-recurrent Markov chains.  When $\mu$ is null-recurrent, different behaviors may emerge (see \cite{Papa_Var_Duke}).

The next order question is to ask about the fluctuations:  Find the limits of
$$\frac{1}{\sqrt{t}} \int_0^t \big(f(\eta_s) - E_\mu[f]\big) ds.$$
One expects a Gaussian limit, with sufficient mixing.  In the following, we will assume $\mu$ is extremal and that $f$ is already centered, that is $E_\mu[f]=0$ to simplify expressions.

\medskip
Consider the example of finite-state irreducible Markov chains.  We may view $\mu = \langle \mu(x): x\in \Omega\rangle$ and $f = \langle f(x): x\in \Omega\rangle$ as elements of $\R^{|\Omega|}$.  Note that $(\mu L)(x)=0$ for all $x\in \Omega$ where $L$ is the generator of the process.  Since $\mu$ is the unique invariant measure, ${\rm Null}(L^T) = \{c\mu : c\in \R\}$.  The orthogonal complement of this null space is ${\rm Range}(L)$:  For $\phi\in {\rm Range}(L)$, we have $\phi = Lu$.  Then, $\sum_x \mu(x) (Lu)(x) = 0$ as $\mu$ is invariant.

As a consequence, given $E_\mu[f]=\sum_x f(x)\mu(x) = 0$, we have $f\perp \mu$ and hence $f$ belongs to the range ${\rm Range}(L)$, and $f = -Lu$ for some function $u$.  That is, $f$ solves a `Poisson' equation with respect to the generator.

  We may now write
\begin{equation}
\label{mart_decomp}
\frac{t}{\sqrt{t}}\int_0^t f(\eta_s)ds \ = \ \frac{1}{\sqrt{t}}M(t) + \frac{1}{\sqrt{t}}\big( u(\eta(0)) - u(\eta(t))\big)\end{equation}
where 
$$M(t) \ = \ u(\eta(t)) - u(\eta(0)) - \int_0^t Lu(\eta_s) ds$$
is a martingale with respect to natural sigma-fields.  The quadratic variation, as we have observed in Subsection \ref{sec:2.3}, equals
$$\langle M\rangle(t) \ = \ \int_0^t (L u^2 - 2u L u)(\eta_s)ds$$
whose mean is
$$\E_\mu\big[\langle M\rangle(t)\big] \ = \ t\E_\mu\big[M^2(1)\big]\ = \ 2tE_\mu[u(-Lu)] \ =: \ 2t D(u)$$
where $D(u)$ is the Dirichlet form.

The second term in \eqref{mart_decomp} is of order $t^{-1/2}$ since $u$ is a function on finite state space, and hence bounded.
The first term, however, can be treated with a martingale central limit theorem.

\begin{proposition}
\label{mclt}
  Let $(\mathcal{M}_t, \mathcal{F}_t)$ be a mean-zero martingale with stationary, ergodic increments such that
$\E\big [ \mathcal{M}^2_1\big] = \sigma^2$.
Then, 
$$\frac{1}{\sqrt{t}}\mathcal{M}_t \ \Rightarrow \ {\rm N}(0,\sigma^2).$$ 
\end{proposition}

\begin{exercise}\rm
This theorem can be found in discrete time in many places (cf. \cite{Hellend}[Theorem 2.5 (b)], \cite{MartingaleCLT}[Theorem 3]).  From the discrete time version, prove the proposition above, by considering the discrete time martingale $M_{\lfloor t\rfloor}$.
\end{exercise}

To complete the argument, note 
the martingale 
$M(t)$ has stationary and ergodic increments when the process is started from $\mu$.
Hence, by Proposition \ref{mclt}, we conclude
$$\frac{1}{\sqrt{t}}\int_0^t f(\eta_s)ds \ \Rightarrow \ {\rm N}(0, 2D(u)).$$

\section{Occupation functionals in exclusion systems}

Consider the $d$-dimensional exclusion process $\eta_t$ with semigroup $P_t$ and generator $L$,
\begin{equation*}
Lf(\eta) = \sum_{x,y\in \Z^d} \eta(x)\big(1-\eta(y)\big)p(x,y)\big[f(\eta^{x,y}) - f(\eta)\big],
\end{equation*}
lifting the finite-volume definition in Section \ref{lec2}.   Such an infinite-volume system on $\Omega =\{0,1\}^{\Z^d}$ may be well constructed; see the Notes in Section \ref{lec7}.  In the following, we will assume the jump probability $p$ is finite-range and translation-invariant.  

Let $\nu_\rho=\prod_{x\in \Z^d}{\rm Bern}(\rho)$ be the Bernoulli product measure with marginal success probability $\rho\in [0,1]$ on $\Omega$.  These measures can be shown invariant and extremal for the process, and when $p$ is symmetric they are also reversible; see the Notes in Section \ref{lec8}.  Compare also with the finite-volume discussion in Section \ref{lec2}.  In the following, we will work with the $L^2(\nu_\rho)$-exclusion processes.

Consider now the `additive functional' question for occupation functions $f(\eta)=\eta(0)-\rho$.
Given $\nu_\rho$ is extremal and therefore ergodic with respect to time shifts (cf. Proposition \ref{equivalence}), we see that the LLN behavior starting under $\nu_\rho$ is clear:  For $f\in L^2(\nu_\rho)$,
\[\lim_{t\uparrow\infty} \frac{1}{t}\int_0^t f(\eta_s) ds \ = \ E_{\nu_\rho}[f]\]
both in $L^2(\nu_\rho)$ and $\nu_\rho$-a.s. by Proposition \ref{prop:Von} and Birkhoff's ergodic theorem.

However, the fluctuation behavior of $A(t)$, with respect to $f(\eta)=\eta(0)-\rho$ is different depending on the dimension $d$, the symmetry/asymmetry of the jump probability $p$, and the density $\rho$.  A more or less complete theory is known, except for some interesting asymmetric cases in $d\leq 2$ which connect with `KPZ' class phenomenon, and `mean-zero asymmetric' cases.

We first consider the types of variances that might be obtained.
What is the variance of $A_f(t)$, what are its limits, and how does the limit depend on $f$?  For a general local function $f$, that is one that depends only on a finite number of variables $\{\eta(x):x\in \Z^d\}$, let us try to get a formula for ${\rm Var}(A_f(t))$.  Write, using stationarity, that
\begin{eqnarray*}
{\rm Var}(A_f(t)) & =& \E_{\nu_\rho} \Big[ \Big(\int_0^t f(\eta_s) ds \Big)^2\Big]\\
&= &2 \int_0^t (t-s) \E_{\nu_\rho}[f(\eta_s)f(\eta_0)]ds.
\end{eqnarray*}
By reversibility of $\nu_\rho$, the semigroup $P_s$ is self-adjoint, and we may express
\begin{eqnarray*} \E_{\nu_\rho}[f(\eta_s)f(\eta_0)] & =&  \langle P_sf, f\rangle_{\nu_\rho} \\
& = & \langle P_{s/2}f, P_{s/2}f\rangle_{\nu_\rho} \ = \ \|P_{s/2}f\|_{L^2(\nu_\rho}^2 \ \geq \ 0.\end{eqnarray*}
Recall, $\langle f, g\rangle_\mu :=E_{\mu}[fg]$.

Then, when the variance is scaled by $t$, its limit exists, possibly infinite, and is in form
\begin{eqnarray*}\sigma^2(f) & = & \lim_{t\uparrow\infty} \frac{1}{t}{\rm Var}(A_f(t))\\
&=& \lim_{t\uparrow\infty}
2\int_0^t \Big[1-\frac{s}{t}\Big]  \E_{\nu_\rho}[f(\eta_s)f(\eta_0)]ds\\
&=& 2\int_0^\infty \E_{\nu_\rho}[f(\eta_s)f(\eta_0)] ds.
\end{eqnarray*}
The last equality follows from monotone convergence.  This formula is the familiar `sum of correlations' expression with respect to sums of correlated random variables.

\subsection{Symmetric $p$ and duality relations}
\label{sec:duality}

 In symmetric exclusion, any mean-zero, local function can be written as
\begin{eqnarray*}
f(\eta) & = & \sum_{B} \f(B) \prod_{x\in B}\frac{\eta(x) - \rho}{\sqrt{\rho(1-\rho)}}\\
&=& \sum_{n\geq 1} \sum_{|B|=n} \f(B)\prod_{x\in B}\frac{\eta(x)-\rho}{\sqrt{\rho(1-\rho)}}.
\end{eqnarray*}
Here, the decomposition is over `degrees', that is, functions which are the product of $n$ centered, normalized occupation variables.

This basis is quite central to symmetric exclusion because of the `duality' relation.  Namely, for a function
$$f_B(\eta) \ = \ \prod_{x\in B}\frac{\eta(x) - \rho}{\sqrt{\rho(1-\rho)}}$$
where $|B|=n$, we have
$$(P_tf_B)(\eta)\ = \ \sum_{|U|=n} p_t(B,U) \prod_{x\in U}\frac{\eta(x)-\rho}{\sqrt{\rho(1-\rho)}}$$
where $p_t(B,U)$ is the transition probability of $n$-particle simple exclusion from a configuration $B$ to configuration $U$ in time $t$.

One way to prove this relation is to observe that 
$$Lf_B(\eta) \ = \ \sum_{|U|=n} q(B,U)f_U(\eta) $$
where $q(B,U)$ is the transition rate of $n$-particle symmetric exclusion; see Exercise \ref{ex:duality}.

At this point, we can use the duality relation, and independence of coordinates under $\nu_\rho$, to evaluate the term
\begin{eqnarray*}
E_{\nu_\rho}[(P_sf(\eta_0)f(\eta_0)] &=& E_{\nu_\rho} \Big[ \sum_B \f(B) \sum_U p_t(B,U) f_U(\eta_0) \cdot f(\eta_0)\Big]\\
&=& \sum_B \sum_U \f(B)\f(U) p_t(B,U)\\
&=& \sum_{n\geq 1} \sum_{|B|, |U|=n} \f(B)\f(U)p_t(B,U).
\end{eqnarray*}

Now, when $f(\eta) = \eta(0)-\rho$ is the centered occupation function, $\f(B)=0$ for all $B\neq \{0\}$, and $\f(\{0\}) = \sqrt{\rho(1-\rho)}$.  Then, the variance is calculated as
$${\rm Var}(A_f(t)) \ = \ 2\rho(1-\rho)\int_0^t (t-s)p_s(0,0)ds$$
where $p_t(0,0)$ is the return probability of a random walk according to jump probabilities $p$!

From local limit theorems, when the transition probability is simple, $p(e_i)=p(-e_i)=1/(2d)$ for the standard basis $\{e_i\}_{i=1}^d$, we have the following asymptotics.
\begin{proposition}
\label{variance_sym}
For $f(\eta)=\eta(0)-\rho$, under symmetric simple exclusion processes,
$${\rm Var}(A_f(t)) \ = \ \left\{\begin{array}{rl}
\frac{8\rho(1-\rho)}{3\sqrt{2\pi}}t^{3/2} + o(t^{3/2}) & \ {\rm in \ } d=1\\
\frac{2\rho(1-\rho)}{\pi}t\log t + o(t\log t) & \ {\rm in \ }d=2\\
2\rho(1-\rho)\int_0^\infty p_t(0,0)dt & \ {\rm in \ }d\geq 3.
\end{array}\right.
$$
\end{proposition}

\subsection{Asymmetric $p$}
\label{sec:asym-second}

By asymmetric, we mean that $\sum_x xp(x) \neq 0$.  The case that $p$ is not symmetric but $\sum_x xp(x)=0$, the `mean-zero' asymmetric situation is not considered here; see the Notes for citations.

As in symmetric exclusion, one can understand the general asymptotics by calculating the term $\E_{\nu_\rho}[f(\eta_s)f(\eta_0)]$.  However, the `duality relation' is no longer valid in the form given, since the generator does not preserve the `degree' of a function, that is $L$ applied to a function of $n$ coordinates is no longer a linear combination of $n$ coordinate functions.

However, for $f(\eta)= \eta(0) - \rho$, write
\begin{eqnarray*}
\E_{\nu_\rho}[f(\eta_s)f(\eta_0)]  & = & \E_{\nu_\rho}[\eta_s(0)\eta_0(0)] - \rho^2\\
&=& \rho\big\{\E_{\nu_\rho}[\eta_s(0)|\eta_0(0)=1] - \E_{\nu_\rho}[\eta_s(0)]\big\}\\
&=& \rho(1-\rho)\big\{\E_{\nu_\rho}[\eta_s(0)|\eta_0(0)=1]  - \E_{\nu_\rho}[\eta_s(0)|\eta_0(0)=0]\big\}.
\end{eqnarray*}
The basic coupling, with respect to exclusion, couples two copies of the exclusion process starting from configurations $\eta' \geq \eta'$ where $\eta'(x) = \eta'(x)$ for $x\neq 0$ and $\eta'(0)=1$, $\eta'(0)=0$.  Then, $(\eta'_t, \eta'_t)$ has generator
\begin{eqnarray*}
\bar L (\eta', \eta') & = & \sum_{x,y} p(y-x) 1(\eta'(x)=\eta'(x)=1)\big[f((\eta')^{x,y}, (\eta')^{x,y}) - f\big]\\
&& + \sum_{x,y}p(y-x)1(\eta'(x)=1, \eta'(x)=0)\big[f((\eta')^{x,y}, \eta') - f\big].
\end{eqnarray*}
The $\eta'_t$ process always majorizes $\eta'_t$, with exactly one discrepancy, whose position we label $R_t$.  

The dynamics of $R_t$ is as follows:
Infinitesimally, it displaces by $z$ with rate $p(z)(1-\eta(R_t +z)) + p(-z)\eta(R_t +z)$.  The term $p(z)(1-\eta(R_t +z))$ corresponds to the discrepancy, or `second-class' particle as it is sometimes known, moving by its own intention, and the term $p(-z)\eta(R_t+ z)$ refers to when a particle at $R_t +z$ would move to location $R_t$ in which case by the basic coupling, $R_t$ accedes and takes the place $R_t +z$.

Then, we have that
$$ \E_{\nu_\rho}[f(\eta_s)f(\eta_0)] \ = \ \rho(1-\rho) \bar P(R_t = 0) \ \geq \ 0.$$
Hence, by monotone convergence the limit exists,
$$\sigma^2_f \ = \ 2\rho(1-\rho)\int_0^\infty \bar P(R_t = 0) dt.$$

In mean value, substituting $\rho$ for $\eta(R_t + z)$, we see that a `mean' infinitesimal drift is
$$\sum z[p(z)(1-\rho) + p(-z)\rho]\ = \ (1-2\rho)\sum zp(z).$$
This leads to the conjecture that 
$$\sigma^2_f < \infty \ \Leftrightarrow \ \rho\neq 1/2$$
which has been proved; see the Notes.

\begin{proposition}
\label{variance_asym}
For $f(\eta)=\eta(0)-\rho$, with respect to asymmetric exclusion with drift $\sum zp(z)\neq 0$, we have
$${\rm Var}(A_f(t)) \ = \ 
\sigma^2_f t + o(t)$$
when $\rho\neq 1/2$ or when $d\geq 3$.

However, when $\rho=1/2$, we have ${\rm Var}(A_f(t)) \geq C_1t^{5/4}$.  In $d=2$, we have ${\rm Var}(A_f(t)) \geq C_2t\log\log t$.
\end{proposition}

\begin{remark}\label{rem:add-conj}
\rm
We remark the orders expected when $\rho\neq 1/2$ in $d=1,2$ are $t^{4/3}$ and $t(\log t)^{2/3}$ respectively.  These orders connect with certain KPZ class phenomena, and are open to verify.  See the Notes for more discussion.
\end{remark}

\section{Central limit theorems and $H_{-1}$ norms}
Our goal now is to prove asymptotic normality of $t^{-1/2}A_f(t)$ when $\sigma^2_f<\infty$.  Except for the two open cases in Proposition \ref{variance_asym}, a central limit theorem also holds when $A_f(t)$ is normalized by the square root of its variance, although we do not discuss these results here; see the following Notes.

We focus on the symmetric case where the Kipnis-Varadhan CLT applies.  The Kipnis-Varadhan theorem is a CLT for additive functionals of a Markov process, reversible and ergodic, on general state space $\Omega$.

\begin{theorem}
\label{Kipnis-Varadhan}
Consider an $L^2(\mu)$ Markov process $\eta_t$ begun with ergodic, reversible invariant measure $\mu$.
Let $f:\Omega\rightarrow\R$ be an $L^2(\mu)$ function such that $\sigma^2_f<\infty$. 
Then,
$$\frac{1}{\sqrt{t}}\int_0^t f(\eta_s)ds \ \Rightarrow \ {\rm N}(0,\sigma^2_f).$$
\end{theorem}
A functional CLT/invariance principle can also be proved under the assumption of the theorem, but we do not discuss this extension here.

The main tool to prove Theorem \ref{Kipnis-Varadhan} is to approximate $t^{-1/2}A_f(t)$ by a martingale, and then to use Proposition \ref{mclt}.  
The idea is to try to write $f = -Lu$.  However, this is not possible in general.  If it were possible, one could use the results in \cite{Bhattacharya}.  A resolvent type equation always holds however, that can be worked with.
The following, which can be proved by this approach, is sufficient for the martingale approximation.

\begin{proposition}
\label{KV_prop}
Under the assumptions of Theorem \ref{Kipnis-Varadhan}, there is a martingale $M(t)$ with stationary and ergodic increments such that
$$\frac{1}{\sqrt{t}}A_f(t) \ = \ \frac{1}{\sqrt{t}}M(t) + \frac{1}{\sqrt{t}}\zeta(t)$$
where 
$$\lim_{t\uparrow\infty}\frac{1}{t}\E_\mu[\zeta^2(t)] \ = \ 0 \ \ \ {\rm and \ \ \ }
\E_\mu\big[ M^2(t)\big] =t \sigma^2_f.$$
\end{proposition}

\vskip .1cm
\subsection{$\mc H_{1}$ and $\mc H_{-1}$ norms}
Before proving Proposition \ref{KV_prop}, some definitions will be useful.  The generator $L$ of exclusion is well defined on the core of local functions.  Define, for local functions $\phi$ and $\lambda\geq 0$, the semi-norm $\|\phi\|_{1,\lambda}$ by
$$\|\phi\|_{1,\lambda}^2 \ = \ \langle \phi, (\lambda-L)\phi\rangle_{\mu} \ = \ D(\phi) + \lambda \|\phi\|^2_{L^2(\mu)}.$$
Note that, since $-L$ is a nonnegative self-adjoint operator, 
$$\langle \phi, (\lambda -L) \psi\rangle_{\mu} \ = \langle (\lambda -L)^{1/2}\phi, (\lambda -L)^{1/2}\psi\rangle_{\mu} \ \leq \ \|\phi\|_{1,\lambda} \|\psi\|_{1,\lambda}.$$
After modding out by functions with $\|\phi\|_{1,\lambda}=0$, define the space $\H_{1,\lambda}$ as the completion with respect to $\|\cdot \|_{1,\lambda}$. 

Now, we define for a local function $\phi$ that
$$\|\phi\|_{-1,\lambda} \ = \ \sup \left\{\frac{\langle \phi, \psi\rangle_{\mu}}{\|\psi\|_{1,\lambda}}: \psi \ {\rm local}\right\}.$$
Again, after modding out by functions $\|\phi\|_{-1,\lambda}=0$, define $\H_{-1,\lambda}$ as the completion with respect to $\|\cdot \|_{-1,\lambda}$.
When $\lambda>0$, $(\lambda -L)^{-1}$ is a bounded operator (bounded by $\lambda^{-1}$ from resolvent formulas--see the next subsection), and 
$$\|\phi\|_{-1,\lambda}^2 \ = \ \langle \phi, (\lambda -L)^{-1}\phi \rangle_{\mu}.$$

Both spaces $\H_{1,\lambda}$, $\H_{-1,\lambda}$ are Hilbert spaces where innerproducts are given by polarization:
$$\langle\langle \phi, \psi\rangle\rangle \ = \ \frac{1}{4}\left\{\|\phi + \psi\| - \|\phi-\psi\|\right\}.$$ 

These two norms are dual to each other:  For local functions, 
$$\langle \phi, \psi\rangle_{\mu} \ \leq \ \|\phi\|_{-1,\lambda}\|\psi\|_{1,\lambda}.$$
When $\lambda =0$, $\|\cdot\|_1:=\|\cdot \|_{1,0}$ and $\|\cdot\|_{-1}:=\|\cdot\|_{-1,0}$ are called the $\H_{1}$ and $\H_{-1}$ norms respectively.  Note that $\|\phi\|_1^2=D(\phi)$.

We note there is another formula for $\|\phi\|_{-1,\lambda}$ of note:
\begin{align}
\label{eq:H-1formula}
\|\phi\|_{-1,\lambda}^2 &=  \sup \left\{ 2\langle \phi, \psi\rangle_{\mu} -\|\psi\|^2_{1,\lambda}: \psi \ {\rm local}\right\}.
\end{align}

We remark, when $L$ is not self-adjoint, there are useful notions of $\mc H_{1,\lambda}$ and $\mc H_{-1, \lambda}$ spaces; see \cite{B}[Lemma 2.1], \cite{Ko_La_Ol} for instance.

\begin{exercise}
\rm
Verify formula \eqref{eq:H-1formula} from the definitions.  

Also, show $\|\phi\|_{-1, \lambda}\uparrow \|\phi\|_{-1}$ and $\|\phi\|_{1,\lambda}\downarrow \|\phi\|_1$ as $\lambda\downarrow 0$.
\end{exercise}

Sometimes the $\H_{-1}$ norm is referred to as the `variance' norm.  Since $\|\phi\|_{-1,\lambda}$ is increasing as $\lambda \downarrow 0$, and $\lim_{\lambda\downarrow 0}\|\phi\|_{-1,\lambda} = \|\phi\|_{-1}$, we have
\begin{eqnarray*}
2\|\phi\|_{-1}^2 & = & 2\lim_{\lambda\downarrow 0 }\|\phi\|_{-1,\lambda}^2 \\
& = & 2\lim_{\lambda\downarrow 0} \langle \phi, (\lambda -L)^{-1}\phi\rangle_{\mu}\\
&=& 2\lim_{\lambda\downarrow 0}\int_0^\infty e^{-\lambda t} E_{\mu}[\phi(P_t\phi)] dt\\
&=& 2\int_0^\infty \E_{\mu}[\phi(\eta_t)\phi(\eta_0)]dt \ = \ \sigma^2_\phi.
\end{eqnarray*}
The last evaluation of the limit follows as $E_{\mu}[\phi(P_t\phi)] = \E_{\mu}[\phi(\eta_t)\phi(\eta_0)]$ is nonnegative since $P_t$ is self-adjoint.

In particular, we observe the following.
\begin{lemma}
\label{lem:asymptoticH-1}
We have $\sigma^2_f = 2\|f\|^2_{-1}$ and so
\[{\rm the \ condition \ }\sigma^2_f<\infty \Leftrightarrow \|f\|^2_{-1}<\infty. \]
\end{lemma}

\subsection{Step 1: Resolvent calculations}
To facilitate the proof of Proposition \ref{KV_prop},
consider the resolvent equation
\begin{equation}
\label{resolvent_eqn}
 \lambda u_\lambda - Lu_\lambda \ = \ f
 \end{equation}
where $$u_\lambda(\eta) \ = \ (\lambda -L)^{-1}f\ = \ \int_0^\infty e^{-\lambda t} P_tf(\eta) dt.$$
\begin{exercise}
\rm
Show that $E_{\nu'_\rho}[u^2_\lambda] \leq \lambda^{-1}E_{\nu'_\rho}[f^2]$.  This bound is part of the standard theory of `resolvents'.
\end{exercise}

Let us multiply the resolvent equation by $u_\lambda$ and take expectation:
$$ \lambda \|u_\lambda\|_{L^2(\mu)} + D(u_\lambda) \ = \ \langle f, u_\lambda\rangle_{\mu}.$$
Then,
$$ \lambda \|u_\lambda\|_{L^2(\mu)} + \|u_\lambda\|_1^2 \ \leq \ \|f\|_{-1}\|u_\lambda\|_1$$
which gives immediately that
$$\|u_\lambda\|_1\ \leq \ \|f\|_{-1} \ \ {\rm and  \ \ } \lambda \|u_\lambda\|^2_{L^2(\mu)}\ \leq \ \|f\|_{-1}^2$$
uniformly in $\lambda>0$.

Also, since $|\langle L\phi, \psi\rangle_{\mu}| \leq \|\phi\|_1\|\psi\|_1$, we observe
$L: \H_1 \rightarrow \H_{-1}$ is a bounded operator, with bound $1$.  Hence, for all $\lambda>0$, 
$$\|Lu_\lambda\|_{-1} \ \leq \ \|u_\lambda\|_1 \ \leq \ \|f\|_{-1}.$$

Now, by the uniform boundedness principle, one can find a subsequence which converges weakly to an element $w\in\H_1$:
$$u_{\lambda_n} \ \rightarrow \ w.$$
For $\phi$ local, and as local functions are dense in $H_1$, we have
$$\lambda \langle u_\lambda, \phi\rangle_{\mu} \ \leq \ \sqrt{\lambda} \|f\|_{-1} \|\phi\|_{L^2},$$
and since $\|\lambda u_\lambda\|_{-1} \leq 2\|f\|_{-1}$, we conclude $\lambda u_\lambda \rightarrow 0$ weakly in $\H_{-1}$.
Therefore,
\begin{equation}
\label{weak_convergence_Lu}
-Lu_{\lambda_n} \ \rightarrow \ f \ \ {\rm weakly \ in \ }\H_{-1}.
\end{equation}

\subsection{Step 2: Strong approximations}
At this point, we would like to claim that
$$\langle f, u_\lambda\rangle \rightarrow \langle f, w\rangle, \ \langle u_\lambda, -Lu_\lambda\rangle = \|u_\lambda\|^2_1 \rightarrow \|w\|^2_1, \ {\rm and \ } \lambda \|u_\lambda\|^2_{L^2} \rightarrow 0$$
so that $\langle f, w\rangle = \|w\|^1_1$ and $f$ satisfies a `weak' Poisson equation.  However, the weak convergence shown, as there are limits in $\lambda$ in both components of $\langle u_\lambda, -Lu\lambda\rangle$, is not strong enough to conclude this relation immediately.

Nevertheless, we will show that the convergences can be strengthened as follows.
\begin{proposition}
\label{1_2}
We have
\begin{itemize}
\item[(i)] $\lim_{\lambda\downarrow 0} \lambda \|u_\lambda\|^2_{L^2(\mu)} \ = \ 0$.

\item[(ii)] There is a $w\in \H_1$ such that $u_\lambda \rightarrow w$ strongly in $\H_1$.  

Also $\|w\|_1^2 = \langle w,f\rangle_{\mu}$.
\end{itemize}
\end{proposition}

\medskip

\begin{proof}[Proof of Proposition \ref{KV_prop}]
Given (i) and (ii) above, to be shown later, we now finish the proof of the Proposition \ref{KV_prop}.  Write
\begin{eqnarray*}
A_f(t) &=& M_\lambda(t) + \lambda \int_0^t u_\lambda(\eta_s)ds + u_\lambda(\eta_0) - u_\lambda(\eta_s)\\
&=:& M_\lambda(t) + \zeta_\lambda(t)
\end{eqnarray*}
where
$$M_\lambda(t) \ = \ u_\lambda(\eta_t) - u_\lambda(\eta_0) - \int_0^t Lu_\lambda(\eta_s)ds$$
is a martingale and $\zeta_\lambda(t)$ is the remainder.  

We claim now $\E_\mu\big[M^2_\lambda(t)\big] = 2t\|u_\lambda\|_1^2$.  Indeed, we need only evaluate 
\begin{align}
\label{eq:M-lim}
\lim_{t\downarrow 0}t^{-1}\E_\mu\big[M^2_\lambda(t)\big] = 2\|u_\lambda\|_1^2.
\end{align}
  First, 
as $Lu_\lambda = \lambda u_\lambda -f$ and $u_\lambda, f\in L^2(\mu)$, we have $E_\mu\big[(Lu_\lambda)^2\big]<\infty$ and so 
$$\E_\mu\Big[\Big(t^{-1}\int_0^t Lu_\lambda (\eta_s)ds\Big)^2\Big]\leq Ct\big(\|u_\lambda\|_{L^2} + \|f\|_{L^2}\big) \rightarrow 0.$$  Second, to finish,
\begin{align*}
t^{-1}\E_\mu\Big[ \Big(u_\lambda (\eta_t) - u_\lambda(\eta_0)\Big)^2\Big]
& = 2t^{-1}E_\mu\big[u^2_\lambda - u_\lambda P_t u_\lambda\big]\\
& = -2E_\mu\Big[ u_\lambda\Big(\frac{P_t u_\lambda - u_\lambda}{t}\Big)\Big]\\
& \rightarrow -2E_\mu\big[u_\lambda Lu_\lambda\big] = 2\|u_\lambda\|_1^2.
\end{align*}

As a consequence, we observe
$$\E_{\mu}\big[|M_\lambda(t) - M_\theta(t)|^2\big] \ =\  2t\|u_\lambda - u_\theta\|_1^2 \ \rightarrow \ 0$$
as $\lambda,\theta\downarrow 0$ from (ii).  Hence, being a Cauchy sequence, $M_\lambda(t) \rightarrow M(t)$ in $L^2(\mu)$.  Therefore, $M(t)$ (with respect to the natural filtration of $\eta_t$) is a martingale with ergodic and stationary increments such that
$$\E_\mu\big[M^2(t)\big] = t \E_\mu\big[M^2(1)\big]$$
and
$$\E_{\mu}\big[M^2(1)\big] \ = \ \lim_{\lambda\downarrow 0} \E_{\mu}\big[M^2_\lambda(1)\big] \ = \ 2\lim_{\lambda\downarrow 0} \|u_\lambda\|^2_1 \ = \ 2 \|w\|^2_1.$$

Now, the error $\zeta_\lambda(t)$ is handled as follows:  By the equation $A_f(t) = M_\lambda(t) + \zeta_\lambda(t)$, we see that
a limit $\zeta_\lambda(t)\rightarrow \zeta(t)$ holds as $\lambda\downarrow 0$ in $L^2(\mu)$.  
Hence, $A_f(t) = M(t) + \zeta(t)$ and
$$\zeta(t) \ = \ M_\lambda(t) - M(t) + \zeta_\lambda(t).$$

By Schwarz inequality, using that $\mu$ is invariant, we have by (i), choosing $\lambda = t^{-1}$, that
\begin{eqnarray*}
t^{-1}\|\zeta_\lambda(t)\|_{L^2(\mu)}^2 & \leq & 3\big[ t^{-1}t^2 \lambda^2\|u_\lambda\|^2_{L^2} + 2t^{-1}\|u_\lambda\|^2_{L^2}\big]\\
&=& \frac{C}{t}\|u_{t^{-1}}\|_{L^2}^2 \ \rightarrow \ 0
\end{eqnarray*}
 as $t\uparrow\infty$.

On the other hand, 
\begin{eqnarray*}t^{-1}\|M_{t^{-1}}(t) - M(t) \|^2_{L^2} & = & t^{-1}\lim_{\theta\downarrow 0} \|M_{t^{-1}}(t) - M_\theta(t)\|^2_{L^2}\\
& = & t^{-1}\cdot 2t\|u_{t^{-1}} - w\|_1^2 \ \rightarrow \ 0
\end{eqnarray*}
as $t\uparrow\infty$.

Finally, we need to identify the variance:  By (ii) and Lemma \ref{lem:asymptoticH-1},
$$2\|w\|_1^2 \ = \ 2\langle w, f\rangle_{\mu} \ = \ \lim_{\lambda\downarrow 0} 2\langle u_\lambda, f\rangle_{\mu} \ = \ 2\|f\|_{-1}^2 \ = \ \sigma^2_f.$$ 

This finishes the proof of Proposition \ref{KV_prop}, subject to Proposition \ref{1_2}.
\end{proof}

\subsection{Step 3: Proof of Theorem \ref{Kipnis-Varadhan}}
We now have a representation 
$$\frac{1}{\sqrt{t}}\int_0^t f(\eta_s)ds \ = \ \frac{1}{\sqrt{t}}M(t) + \frac{1}{\sqrt{t}}\zeta(t).$$
The error vanishes in $L^2(\mu)$ and $\E_\mu\big[M^2(1)\big]=\sigma^2_f$.

 The martingale central limit theorem, Proposition \ref{mclt}, finishes the proof. \qed

 \subsection{Step 4: Proof of Proposition \ref{1_2}}

 We now recall Mazur's theorem (see \cite{Liggett2}[Lemma 4.38]).  A proof is given at the end.
 \begin{lemma}
 In a Hilbert space, if $x_n\rightarrow x$ weakly, there is a convex combination of the $\{x_1,x_2,\ldots, x_n\}$ which converges strongly to $x$, that is $\|x_n -x\|\rightarrow 0$.
 \end{lemma}

Now, let $v_n$ be a a convex combination of $\{u_{\lambda_k}\}$ such that $v_n \rightarrow w$ strongly in $\H_1$.  Then, as $-Lu_{\lambda_n}\rightarrow f$ weakly in $\H_1$ and $L$ is linear, $-Lv_n \rightarrow f$ weakly in $\H_1$.  In fact, since $\|L(v_n - v_m)\|_{-1} \leq \|v_n-v_m\|_1$, as $L$ is self-adjoint, we see that $\{-Lv_n\}$ is a Cauchy in $H_{-1}$, and hence converges strongly to $f$ in $H_{-1}$.

Then,
$$\|w\|^2_1\ = \ \lim \|v_n\|_1^2 \ = \ \lim\langle v_n, -Lv_n\rangle \ = \ \langle w, f\rangle.$$

Now, returning to the subsequence $u_{\lambda_n}$, by lower semicontinuity and weak convergence of $u_{\lambda_n}\rightarrow w$ in $H_1$, and \eqref{resolvent_eqn}, we have 
\begin{eqnarray*}\|w\|_1^2 & \leq & \liminf \|u_{\lambda_n}\|_1^2\\
&\leq& \liminf \lambda_n \|u_{\lambda_n}\|^2_{L^2} + \|u_{\lambda_n}\|^2_1 \\
&=&\liminf \langle f, u_{\lambda_n}\rangle\\
&=& \langle f, w\rangle \ = \ \|w\|_1^2.
\end{eqnarray*}
The same calculation can be repeated with `$\liminf$' replaced by `$\limsup$'.  

Therefore, we conclude $\|u_{\lambda_n}\|_1^2 \rightarrow \|w\|^2_1$, which means $u_{\lambda_n}\rightarrow w$ strongly in $\H_1$, and also that
$\lambda_n \|u_{\lambda_n}\|^2_{L^2} \rightarrow 0$.

By the same arguments, on any subsequence of $\{u_\lambda\}$, a further subsequence $u_{\lambda'_n}$ can be found so that $\lambda'_n\|u_{\lambda'_n}\|^2_{L^2}\rightarrow 0$.  Hence, $\lim_{\lambda\downarrow 0}\lambda \|u_\lambda\|^2_1 = 0$, showing part (i).

To show part (ii), we need to show that the limit $w$ is obtained on any subsequential limit of $u_\lambda$ strongly in $\H_1$.  To this end, as before, suppose $\{u_{\lambda'_n}\}$ is a subsequence converging strongly to $w'$ in $\H_1$.  We conclude as before that $-Lu_{\lambda'_n}\rightarrow f$ strongly in $\H_{-1}$.
We now show that $\|w-w'\|_1^2 = 0$ which will finish the claim.

Write, 
\begin{eqnarray*}
\|w-w'\|_1^2 & = & \lim_{n} \|u_{\lambda_n} - u_{\lambda'_n}\|^2_1\\
&=&\lim_n  \langle u_{\lambda_n} - u_{\lambda'_n}, -Lu_{\lambda_n} + Lu_{\lambda'_n}\rangle \\
&=& \lim_n \langle u_{\lambda_n} - u_{\lambda'_n}, -\lambda_nu_{\lambda_n} + \lambda'_nu_{\lambda'_n}\rangle
\end{eqnarray*}
since $-Lu_{\lambda_n} + Lu_{\lambda'_n} = -\lambda_nu_{\lambda_n} + \lambda'_nu_{\lambda'_n}$ noting \eqref{resolvent_eqn}.
Now, $\lambda'_nu_{\lambda'_n}, \lambda_n u_{\lambda_n} \rightarrow 0$ weakly in $H_{-1}$ (cf. before \eqref{weak_convergence_Lu}).  Hence, via the penultimate equation, and the relation $\|Lu\|_{-1}\leq \|u\|_1$, we may approximate $u_{\lambda_n}$ by $w$ and $u_{\lambda'_n}$ by $w'$.  Then, by the last equation, we see that the right-hand side vanishes.
 Therefore, $w=w'$ in $H_1$.
This finishes the proof of Proposition \ref{1_2}. \qed
\medskip

\noindent {\bf Proof of Mazur's Theorem.}  Let us assume the limit is $w=0$ without loss of generality. From the weak convergence, one can find iteratively a subsequence $\{n_k\}$ with $n_1=1$ such that 
$$|\langle\langle u_{n_k}, u_{n_j}\rangle\rangle | \ \leq \ k^{-1}, \ \ \ {\rm for \ } 1\leq j< k.$$
Then, 
$$\big\|\frac{1}{k}\sum_{j=1}^k u_{n_j}\big\|^2 \ \leq \ \frac{1}{k^2}\sum_{j=1}^k \|u_{n_j}\|^2 + \frac{2}{k^2}\sum_{1\leq i<j\leq k} \frac{1}{j}$$
which vanishes as $k\uparrow\infty$.  Hence, we can take $v_n = k^{-1}\sum_{j=1}^k u_{n_j}$ for $n_k\leq n<n_{k+1}$.
\qed

\section{Notes}
Aside from the original paper \cite{KV}, treatments of the Kipnis-Varadhan theorem, and applications to additive functionals and tagged particles, can be found in \cite{Ko_La_Ol}, \cite{Liggett2}.
There have been generalizations of the Kipnis-Varadhan theorem to nonreversible situations \cite{SVY}, \cite{Varadhan}, \cite{Woodroofe}, some of which has been used here, albeit in the reversible setting.  In the nonreversible case, as we have seen, it may be that $\sigma^2_f<\infty$ but $\|f\|_{-1, S} =\infty$ where the $\H_{-1}$ norm is with respect to the symmetrized operator $-S = -(L + L^*)/2$, e.g. $f(\eta)=\eta(0)-\rho$ for $\rho\neq 1/2$ in $d=1$.  
However, a non-asymptotic bound ${\rm Var}\big(A_f(t)\big)\leq \|f\|^2_{-1, S}$ holds; see Lemma \ref{H-1bound}, Remark \ref{rem:non-asymptotic}.
An open problem of interest is to show in the general nonreversible case, with a condition such as $\|f\|_{-1,S}<\infty$, that a CLT holds.

The asymptotic variances in Proposition \ref{variance_sym} and associated CLT's were first proved in \cite{Kip_flu}.  For more general additive functionals, when $f$ is not necessarily the occupation function of a site, CLT's and diffusive variance criteria are proved in \cite{SX} for reversible mass conservative systems.  

For the asymmetric process with a non-zero drift,
Proposition \ref{variance_asym} is proved in a combination of papers \cite{B}, \cite{SS}, \cite{S_comp}, \cite{S_superdiff}, and associated CLT's are shown in \cite{Sclt}.  In particular, the paper \cite{S_comp} shows that $\H_{-1}$ norms are comparable across $d\geq 1$ finite-range exclusion processes with the same drift.  For instance such comparisons would allow to lift variance calculations from $d=1$ asymmetric simple exclusion processes, e.g. `ASEP' or `TASEP', as in \cite{SS}, to more general finite-range processes.

The `second-class' particle $R_t$, introduced in Subsection \ref{sec:asym-second} in the context of ASEP is of interest by itself.
 When the process starts under $\nu_\rho$, in $d=1$, ${\rm Var}\big(R_t\big)$ scales as $t^{4/3}$ \cite{Balasz}, \cite{Quastel-Valko}, and in $d=2$ scales as $t(\log t)^{2/3}$ \cite{yau-second}, both `super-diffusive'.  In $d\geq 3$, the variance is diffusive of order $O(t)$; see \cite{LY}.  However, when the initial condition is a `step' condition (to the left and right with different densities) in $d=1$, the variance may have different orders \cite{FF}, \cite{Ferrari-Patrik}; see also \cite{Adams-J-M} and references therein.  Moreover, among other properties, the second-class particle in $d=1$ tracks the location of microscopic `shocks' in the system \cite{FF}, \cite{FKS}.  See also \cite{Corwin-asep} and references therein for a `speed' process coupling between second-class particles with different priorities.

 The superdiffusive conjectures in Remark \ref{rem:add-conj} regarding ${\rm Var}\big(A_f(t)\big)$, for $\eta(0)-1/2$ and $\rho=1/2$, rely on a `Gaussian ansatz', that the local limit terms $\bar P(R_t = 0)\sim t^{-2/3}$ in $d=1$ and $t^{-1/2}(\log t)^{-1/3}$ in $d=2$ scale like the power $\big[{\rm Var}(R_t)\big]^{-1/2}$.  Verification is an open problem; see however \cite{BFP} for a weak form of the ansatz in $d=1$.

In the mean-zero asymmetric setting, variance orders are the same as in the symmetrized process, and CLTs hold when the $\H_{-1}$ norm is finite.  This would include occupation functionals in $d\geq 3$. See \cite{Varadhan} for a Kipnis-Varadhan-type CLT using a `sector inequality' for the generators of mean-zero processes.
In $d=1$ mean-zero asymmetric processes, fractional Brownian motion limits for properly scaled, centered occupation variables have been shown \cite{GJ-addfclt}.  For $d=2$ mean-zero asymmetric systems, it is open to show a fluctuation limit for the properly scaled, centered occupation functionals.
See \cite{BGS2}, \cite{BGS} for more discussion of these problems in simple exclusion and zero-range processes, when starting from an invariant measure.

Starting from a non-invariant initial condition, LLNs can be inferred for additive functionals in attractive systems such as exclusion processes via `local equilibrium' results; see \cite{KL}[Chapter 9].  However, less is known about associated fluctuations; see however \cite{Franco-Dirk}, \cite{Xu-Zhao}.

Finally, we comment that the study of `duality' (cf. Subsection \ref{sec:duality}), and its applications in particle systems and other models has advanced in recent years (cf. \cite{Fran}, \cite{Giardina}, \cite{Redig}, and references therein).

\newpage

 \chapter[Section $10$]{A tagged particle in symmetric exclusion on $Z^d$}
\label{lec10}

Understanding the motion of a tracer particle, as it interacts with others, is a basic applied concern. 
We consider here the problem in the finite range symmetric exclusion processes when started under an invariant measure $\nu_\rho$.  Except for the case in dimension $d=1$ when the jump law $p$ is nearest-neighbor, the tagged particle motion is diffusive and converges to a Brownian motion.  In the exceptional case, it can be shown that the motion is subdiffusive and converges to a fractional Brownian motion of Hurst parameter $1/4$.  

First, we discuss LLNs.  Then, the diffusive behavior is considered using the Kipnis-Varadhan theorem.  Finally, we discuss an exceptional subdiffusive case.

\section{Tagged problem setting}

Consider the $d$-dimensional exclusion process $\eta_t$ with finite-range translation-invariant jump probability $p(\cdot)$, as discussed in Section \ref{lec9}.  We will assume that the symmetrized jump probability $s(\cdot) = \big(p(\cdot) + p(-\cdot)\big)/2$ is irreducible.  The system consists typically of an infinite number particles.  Let us identify one of them initially and let $x_t$ be its position at time $t\geq 0$.  To fix things, suppose it starts at the origin initially, $x_0=0$.

How to capture the evolution of $x_t$, say LLN and CLT statements?  With respect to its own history, it is not Markovian because of influence of the other particles.  However, if we consider $x_t$ and $\eta_t$ together, then the joint process $(x_t, \eta_t)$ is Markovian with generator
\begin{eqnarray*}
\tilde Lf(x,\eta) &=& \sum_{u,v\neq x} p(v-u)\eta(u)(1-\eta(v))\big[f(x,\eta^{u,v}) - f(x,\eta)\big]\\
&&\ \ + \sum_{v} p(v)(1-\eta(x+v))\big[f(x+v,\eta^{x,x+v}) - f(x,\eta)\big].
\end{eqnarray*}

It will be convenient to consider `Lagrangian' coordinates, or those in the reference frame of the tagged motion.  Define $\zeta_t = \tau_{x_t}\eta_t$, that is $\zeta_t(y) = \eta_t(y+x_t)$ for all $y\in \Z^d$ where $\tau_x$ is shift by $x$. Then, the joint process $(x_t,\zeta_t)$ is also Markovian and has generator
\begin{eqnarray*}
\hat Lf(x,\zeta) &=& \sum_{u,v\neq 0} p(v-u)\zeta(u)(1-\zeta(v))\big[f(x,\zeta^{u,v}) - f(x,\zeta)\big]\\
&&\ \ + \sum_{v} p(v)(1-\zeta(v))\big[f(x+v,\theta_v\zeta) - f(x,\zeta)\big]
\end{eqnarray*}
where $\theta_v\zeta$ is the configuration which exchanges the values $\zeta(0)$ and $\zeta(v)$, displacing the particle at the origin, namely the tagged particle by $v$, and then shifts the reference frame to this new origin.

In particular, an explicit description of $\theta_v\zeta$ is given by
$$(\theta_v\zeta)(y) \ = \ \left\{\begin{array}{rl} \zeta(y+v) & \ {\rm for \ }y\neq -v,0\\
\zeta(v) & \ {\rm for \ }y = -v\\
1& \ {\rm for \ }y = 0.
\end{array}\right.
$$

Both joint processes can be constructed with a core of local functions.
Interestingly, the process $\zeta_t$ by itself is a Markov process.  This can be seen as the generator $L$ acting on functions of $\zeta$ alone,
 \begin{eqnarray*}
Lf(\zeta) &=& \sum_{u,v\neq 0} p(v-u)\zeta(u)(1-\zeta(v))\big[f(\zeta^{u,v}) - f(\zeta)\big]\\
&&\ \ + \sum_{v} p(v)(1-\zeta(v))\big[f(\theta_v\zeta) - f(\zeta)\big],
\end{eqnarray*}
does not depend on $x$, and hence the associated semigroup also does not depend on $x$.
The process $\zeta_t$ `drives' the process $x_t$ in that $x_t$ can be recovered in terms of reference frame shifts. 

 Define $N_v(t)$ as the count of shifts of $\zeta_\cdot$ of displacement $v$ up to time $t$.  Then, 
\begin{equation}
\label{x_equation}
x_t \ = \ \sum_v vN_v(t)\end{equation}
where the sum may be restricted to $v$ in the support of $p(\cdot)$.  For instance, in one dimension, when $p$ is nearest-neighbor, $x_t$ is the number of right shifts minus the number of left shifts.

Moreover, it is not difficult to find invariant measures for $\zeta_t$ generated by $L$.
\begin{lemma}
\label{lemma1_x}
The Bernoulli product measure $\nu'_\rho=\nu_\rho(\cdot|\eta(0)=1)$, conditioned to have a particle at the origin, is invariant for $\zeta_t$.  When $p$ is symmetric, $\nu'_\rho$ is reversible.
\end{lemma}

The Dirichlet form can be computed also on local functions:
\begin{eqnarray}
\label{eq:dirichlet-De}
D(f) &=& E_{\nu'_\rho}[f(-Lf)]\\
&=& \frac{1}{2}\sum_{u,v\neq 0} s(v-u)E_{\nu'_\rho}[(f(\zeta^{u,v}) - f(\zeta))^2] \nonumber\\
&&\ \ \ + \frac{1}{2}\sum_v s(v) E_{\nu'_\rho}[(1-\zeta(v))f(\theta_v\zeta) - f(\zeta))^2]\nonumber\\
&=& D_e(f) + D_t(f)\nonumber
\end{eqnarray}
where $s(\cdot)$ is the symmetrized probability, $s(v) = (p(v)+p(-v))/2$, assumed to be irreducible on $\Z^d$.

\begin{lemma}
\label{lemma2_x}
$\nu'_\rho$ is extremal for the process $\zeta_t$.
\end{lemma}

\begin{exercise}\rm
Since $\nu'_\rho$ is invariant to exchanges of coordinates and also the operator $\theta_v$, the proof of Lemma \ref{lemma1_x} is similar to the proof of Proposition \ref{prop:exclusion_invariant} which shows $\nu^N_\rho$ is invariant in the finite-volume setting. 
Verify these calculations.
\end{exercise}

\begin{exercise}\rm
Perform the computations to derive the Dirichlet form, and deduce Lemma \ref{lemma2_x} following the proof of Theorem \ref{thm:extremal}.
\end{exercise}

Now, as before, with zero-range processes, we can extend the process $\zeta_t$ to an $L^2(\nu_\rho(\cdot|\eta(0)=1))$ process.  The problem now is to understand the LLN and fluctuations for $x_t$ in terms of the formula \eqref{x_equation}.

\subsection{Martingales for $x_t$}

Each count $N_v(t)$, as with a Poisson process, can be compensated by its intensity, $\int_0^t p(v)(1-\zeta_s(v))ds$, to form a martingle
$$M_v(t) \ = \ N_v(t) - \int_0^t p(v)(1-\zeta_s(v))ds$$
with quadratic variation
$$\langle M_v\rangle(t) \ = \ \int_0^t p(v)(1-\zeta_s(v))ds.$$
Then,
$$\M_v:= M^2_v(t) - \int_0^t p(v)(1-\zeta_s(v))ds$$
is a martingale.

How to verify the above statements?  Technically, we should have formed the generator $\vec{L}$ for the process $(\{N_v(t)\}, \zeta_t)$.  Then, with $f(\{N_v\},\zeta) = N_v$, we may compute $\vec{L}f$ and $\vec{L}f^2 - 2f\vec{L}f$ to see that
$$M_v(t) \ = \ f(\{N_v(t)\}, \zeta_t) - f(\{N_v(0)\},\zeta_0) - \int_0^t \vec{L}f ds$$
and $$ \langle M_v\rangle(t) \ = \ \int_0^t \big\{\vec{L}f^2 - 2f\vec{L}f \big\}ds.$$
Although $f$ is not a bounded function, one may apply truncations to bring it in to the domain of the generator $\vec{L}$, while noting $N_v(t) \leq N(t):=\sum_w N_w(t)$, a Poisson$(1)$ distributed r.v., to help remove the truncations.

Now, we may write
\begin{equation}
\label{mart_decomp_x}
x_t \ = \ \sum_v vM_v(t) + \int_0^t \sum_v  vp(v)(1-\zeta_s(v))ds.
\end{equation}
The first term, $\hat{M}(t) = \sum vM_v(t)$ is another (vector valued) martingale with quadratic variation, for $\ell\in \Z^d$, given by
\begin{equation} \label{quad_var_x}\langle \hat{M}\cdot \ell \rangle(t) \ = \ \int_0^t \sum_v (v\cdot \ell)^2 p(v)(1-\zeta_s(v))ds.
\end{equation}

\begin{remark}
\label{rem:construction}
\rm
At this point, we comment, we will need only formulas \eqref{mart_decomp_x} and \eqref{quad_var_x} in the following.  These could have been derived from the generator action of $\hat{L}$.  The purpose above was to explain more physically their interpretation in terms of counts $\{N_v\}$.
\end{remark}

We also note, since $\nu'_\rho$ is extremal,
by Proposition \ref{equivalence}, $\hat{M}(t)$ has stationary and ergodic increments and, by \eqref{quad_var_x},
\begin{align}
\label{eq:sec_mom_hatM}
\E_{\nu'_\rho}\big[\big(\hat{M}(t)\cdot \ell\big)^2\big] = t (1-\rho)\sum_v (v\cdot\ell)^2p(v).
\end{align}

The fourth moment of $\hat M(t)\cdot\ell$ may also be bounded via a generalization of the Burkholder-Gundy-Davis inequalities:
\begin{align*}
\E_{\nu'_\rho}\big[ \big(\hat M(t)\cdot\ell\big)^4\big]& \leq C \E_{\nu'_\rho}\big[ \max\big\{\langle \hat M(t)\cdot\ell\rangle^2 (t), A_t^{(2)}\big\}\big] \\
&\leq C\E_{\nu'_\rho}\big[\langle \hat M(t)\cdot\ell\rangle^2 (t)\big] + C\E_{\nu'_\rho}\big[A_t^{(2)}\big]
\end{align*}
where $A_t^{(2)}$ is the compensator of the jumps process $D_t:=\sum_{0\leq s\leq t} \big(\Delta (\hat M(s)\cdot \ell)\big)^2= \sum_{0\leq s\leq t}\big(\Delta( x_s\cdot \ell)\big)^2$ and $\Delta y(t) = |y(t)-y(t-)|$ measures the size of a jump; see \cite{JS}[Theorem I.3.17] and \cite{HJ}[Corollary 2.4].  It is also known that $\E_{\nu'_\rho}\big[A_t^{(2)}\big] \leq 4\E_{\nu'_\rho}\big[D^2_t\big]$; see \cite{HJ}[page 4] and \cite{Len}[Lemma 4.1].

 As the range of jumps $R<\infty$, we have $|\Delta (x_s\cdot \ell)|\leq R\|\ell\|$.  Also, the integrand of $\langle \hat M\cdot \ell\rangle(t)$ is bounded.  Then, we may bound
  \[\E_{\nu'_\rho}\big[ \big(\hat M(t)\cdot\ell\big)^4\big] \leq Ct^2.\]

\section{LLN for $x_t$}

The law of large numbers for $x_t$ is already interesting.  In some sense $x_t$ should be a random walk, but it is impeded by the other particles.
How much is the question.

\begin{theorem}
\label{lln_x}
Starting under $\nu'_\rho$, we have a.s. and in $L^1$ that
$$\frac{1}{t}{x_t} \ \rightarrow \ (1-\rho) \sum_v vp(v).$$
\end{theorem}

\begin{proof}
Write
$$\frac{1}{t}{x_t} \ = \ \frac{1}{t}\hat{M}(t) + \frac{1}{t}\int_0^t \sum_v vp(v)(1-\zeta_s(v))ds.$$
Since the fourth moment of $|\hat M(t)|$ is bounded by $t^2$ and $\nu'_\rho$ is extremal and therefore ergodic to time shifts (Proposition \ref{equivalence}), we have by the ergodic theorem (Birkhoff) that the right side converges to its mean, equal to $(1-\rho)\sum_v vp(v)$, both a.s. and in $L^1$.
\end{proof}

We comment the more simple estimate $\E_{\nu'_\rho}\big[|\hat{M}(t)|^2\big] \leq Ct$ (cf. \eqref{eq:sec_mom_hatM}) 
 does not lead immediately to a.s. convergence, although $L^1$ convergence would hold.
  
\section{General CLT for $x_t$}

We now discuss a central limit theorem for $x_t$ when the jump probability $p$ is symmetric, first proved in \cite{KV}.  It can be extended to a functional CLT with respect to a Brownian motion limit, but that is not done here.  We mostly follow the scheme in \cite{KV}.

\begin{theorem}
\label{CLT_x}
Let $p$ be symmetric.  Starting under $\nu'_\rho$, we have 
$$\frac{1}{\sqrt{t}}x_t \ \Rightarrow \ {\rm N}(0,\mathcal C)$$
where $\mathcal C= \mathcal C(\rho,p)$ is a covariance matrix. 
\end{theorem}

The limiting covariance $\mathcal C$ is not explicit, although there are physics conjectures, and some rigorous results for its behavior as a function of $\rho$ \cite{LOV}.  However, it is nondegenerate, except for a particular case.

\begin{proposition}
\label{nondeg_C}
Except in the case, $d=1$ and $p$ is nearest-neighbor ($p(1)=p(-1)=1/2$), the covariance $\mathcal C$ is nondegenerate.
\end{proposition}

We first prove Theorem \ref{CLT_x} and then later Proposition \ref{nondeg_C}.  Let $h(\zeta) = \sum (v\cdot \ell)p(v)(1-\zeta(v))$ be the local, bounded drift function.
For $\ell\in \Z^d$, it will be useful to bound the $H_{-1}$ norm, with respect to $L$, of $h(\zeta)\cdot \ell$.
\begin{lemma}
\label{hcdotell}
We have $h(\zeta)\cdot \ell = \sum (v\cdot \ell)p(v)(1-\zeta(v))$ belongs to 
$H_{-1}$.
\end{lemma}
\begin{proof}  Since $p$ is symmetric, $h$ is a linear combination of $(1-\zeta(w)) -(1- \zeta(-w))$ for $w$ in the support of 
$p$.  Consider, as $\nu'_\rho$ is invariant with respect to the operator $\theta_w$, for local functions $\phi$ that
\begin{align*}
&\langle (1-\zeta(w))-(1-\zeta(-w)), \phi\rangle_{\nu'_\rho} \\
&\quad\quad =  E_{\nu'_\rho}[(1-\zeta(w))\phi(\zeta)] - E_{\nu'_\rho}[(1-\theta_w(\zeta)(-w))\phi(\theta_w(\zeta))]\\
&\quad\quad = E_{\nu'_\rho}[(1-\zeta(w))(\phi(\zeta)-\phi(\theta_w(\zeta))]\\
&\quad \quad \leq E_{\nu'_\rho}[(1-\zeta(w))^2 (\phi(\zeta)-\phi(\theta_w(\zeta))^2]^{1/2}\\
&\quad \quad = E_{\nu'_\rho}[(1-\zeta(w)) (\phi(\zeta)-\phi(\theta_w(\zeta))^2]^{1/2}.
\end{align*}
Now, the last quantity is less than $\sqrt{2/p_{min}}D(\phi)^{1/2}$ where $p_{min}=\min\{p(v): v\in {\rm Supp}(p)\}$.  Hence, $h\cdot \ell$ belongs to $H_{-1}$.
\end{proof}

\begin{proof}[Proof of Theorem \ref{CLT_x}]  From the martingale decomposition \eqref{mart_decomp_x}, we have
$$\frac{1}{\sqrt{t}}{x_t} \ = \ \frac{1}{\sqrt{t}}\hat{M}(t) + \frac{1}{\sqrt{t}}\int_0^t \sum vp(v)(1-\zeta_s(v))ds.$$
Since $p$ is symmetric, the process $\zeta_t$ is reversible with respect to $\nu'_\rho$.  We want to apply the Kipnis-Varadhan apparatus to the second term.

By Lemma \ref{hcdotell}, we have $h\cdot\ell\in H_{-1}$.  Then, by Proposition \ref{KV_prop}, there is a square integrable martingale $M(t)\cdot\ell$ and error $\chi(t)$ such that $|\chi(t)|/\sqrt{t}$ vanishes in $L^2$,
$$\int_0^t h(\zeta_s)\cdot \ell\ ds \ = \ M(t)\cdot\ell + \chi(t),$$
and 
\begin{align}
\label{eq:M-secMom}
\E_{\nu'_\rho}\big[\big(M(t)\cdot \ell\big)^2\big] = t \E_{\nu'_\rho}\big[\big(M(1)\cdot\ell\big)^2\big]<\infty.
\end{align}

Hence, 
$$\frac{1}{\sqrt{t}}{x_t} \ = \ \frac{\hat{M}(t)\cdot\ell + M(t)\cdot\ell}{\sqrt{t}} + \frac{1}{\sqrt{t}}\chi(t).$$
At this point, one
completes the argument by the martingale central limit Theorem \ref{mclt}, since $\hat{M}(t) + M(t)$ has stationary and ergodic increments and, noting \eqref{eq:M-secMom}, \eqref{eq:sec_mom_hatM},
$$\E_{\nu'_\rho}\big[\big(\hat{M}(t)\cdot \ell + M(t)\cdot \ell\big)^2\big] = t  \E_{\nu'_\rho}\big[\big((\hat{M} + M)(1)\cdot \ell\big)^2\big]<\infty.$$

Finally, we comment that the covariance $\mathcal C$ satisfies 
\begin{align*}
2\mathcal C_{i,j}  &= \E_{\nu'_\rho}\big[ \big((\hat M + M)(1)\cdot (e_i + e_j)\big)^2\big] \\
&\quad\quad - \E_{\nu'_\rho}\big[\big((\hat M + M)(1)\cdot e_i \big)^2\big] -\E_{\nu'_\rho}\big[\big((\hat M + M)(1)\cdot e_j\big)^2\big].
\qedhere
\end{align*}
\end{proof}

\subsection{Proof of Proposition \ref{nondeg_C}}

In the exceptional case, the two martingales $\hat{M}$ and $M$ cancel each other, leading to subdiffusive fluctuations for $x_t$ discussed in the next subsection.  However, in all other situations, full cancellation does not occur.
 
The argument relies on the following bound.  Recall the Dirichlet form $D_e$ from \eqref{eq:dirichlet-De}.

\begin{lemma}
\label{nondeg-lem}
We have for $\ell\in \Z^d$ and $\phi\in L^2(\nu'_\rho)\cap H_1$ that
$$|\langle h\cdot\ell, \phi\rangle_{\nu'_\rho}| \ \leq \ C D_e(\phi)^{1/2}.$$ 
\end{lemma}

\begin{proof}  From the form of $h$, we need only show the result for $\zeta(w)-\zeta(-w)$ where $w$ is in the support of $p$.  The key point is that, aside from the exceptional case, one can either go around, in terms of exchanges of coordinates, the origin in $d\geq 2$ or hop over it in $d=1$, without disturbing the tagged particle sitting at the origin.  

One can rewrite, by irreducibility of $p$, that 
$$\zeta(w)-\zeta(-w)  \ = \ \sum_{k=0}^m \zeta(q_k) - \zeta(q_{k+1})$$
in terms of a sequence from $w=q_0$ to $-w=q_m$ so that $p(q_{k+1}-q_k)>0$ for $0\leq k\leq m$.  Then,
\begin{align*}
|\langle \zeta(q_k) - \zeta(q_{k+1}), \phi\rangle_{\nu'_\rho}| & =  E_{\nu'_\rho}[\zeta(q_k)[\phi(\zeta) - \phi(\zeta^{q_k,q_{k+1}})]\\
&\leq  \rho^{1/2}E_{\nu'_\rho}[(\phi(\zeta) - \phi(\zeta^{q_k,q_{k+1}}))^2]^{1/2}\\
&\leq  \sqrt{\frac{2\rho}{p_{min}}} D_e(\phi)^{1/2}. \qedhere
\end{align*}
\end{proof}
\medskip

\begin{proof}[Proof of Proposition \ref{nondeg_C}]  From the formula for the limiting variance, we need to show there is an $\ell\in \Z^d$ such that
$$\lim_{\lambda\downarrow 0} \E_{\nu'_\rho}\big[(\hat{M}(1)\cdot\ell + M_\lambda(1)\cdot\ell)^2\big]  \ > \ 0.$$
In the argument below, $\ell$ in the support of $p$ will suffice.

\medskip
{\it Step 1.}  Note that $h$ is bounded, and consider the resolvent equation
$$\lambda u_\lambda -Lu_\lambda = h\cdot\ell.$$
Multiply by $u_\lambda$ and take expectation to obtain $\lambda E_{\nu'_\rho}[u^2_\lambda] + \|u_\lambda\|^2_1 = \langle h\cdot \ell, u_\lambda\rangle_{\nu'_\rho}$.   Since $h\cdot\ell\in H_{-1}$ by Lemma \ref{hcdotell}, we obtain that $u_\lambda \in L^2(\nu'_\rho)\cap H_{1}$.
Hence, 
we obtain the inequality
\begin{equation}
\label{Dboundedbelow}
\|u_\lambda\|_1^2 \ \leq \ \langle h\cdot \ell, u_\lambda\rangle_{\nu'_\rho} \ \leq \ C D_e(u_\lambda)\end{equation}
by Lemma \ref{nondeg-lem}.

 Let $\g(x,\zeta) = x\cdot\ell + u_\lambda(\zeta)$.  Then, 
$$\hat{M}(t)\cdot\ell + M_\lambda(t)\cdot\ell \ = \ \g(x_t,\zeta_t) - \g(x_0,\zeta_0) - \int_0^t \hat{L}\g(x_s,\zeta_s)ds.$$
By a similar argument as for \eqref{eq:M-lim} (in the proof of the Kipnis-Varadhan Proposition \ref{KV_prop}),  we have
$$ \E_{\nu'_\rho}\big[(\hat{M}(1)\cdot\ell + M_\lambda(1)\cdot\ell)^2\big] \ = \ \frac{1}{t}\E_{\nu'_\rho}\big[(\g(x_t, \zeta_t)-\g(x_0, \zeta_0))^2\big]\ = \ -2E_{\nu'_\rho}\big[\g\hat{L}\g\big].$$
An explicit computation gives that
the right hand side equals
\begin{eqnarray*}
&&E_{\nu'_\rho}\Big[ \sum_{u,v\neq 0} p(v-u)\big[\g(x,\zeta^{u,v}) - \g(x,\zeta)\big]^2\Big]\\
&& \ \ + E_{\nu'_\rho}\Big[ \sum_v p(v) (1-\zeta(v))\big[\g(x+v,\theta_v\zeta)-\g(x,\zeta)\big]^2\Big].
\end{eqnarray*}
The first term is $2D_e(u_\lambda)$.  Then, by dropping the second term, we have the lower bound
$$\E_{\nu'_\rho}\big[(\hat{M}(1)\cdot\ell + M_\lambda(1)\cdot \ell)^2\big] \ \geq \ 2D_e(u_\lambda).$$

\medskip
{\it Step 2.}  
Now, if $\|u_\lambda\|_1$ vanishes as $\lambda\downarrow 0$, we have
$$\E_{\nu'_\rho}\big[ \big(M_\lambda(1)\cdot \ell\big)^2\big] \ = \ 2\|u_\lambda\|_1^2 \ \rightarrow \ 0$$
and so the limiting variance would be $\E_{\nu'_\rho}\big[ (\hat{M}(1)\cdot\ell)^2\big]= 2\sum_v (v\cdot \ell)^2p(v)(1-\rho)$ which is positive for $\ell$ in the support of $p$.

On the other hand, if $\|u_\lambda\|_1\not\rightarrow 0$, then $D_e(u_\lambda)$ also does not vanish by \eqref{Dboundedbelow}.
Hence, still the limiting variance is positive.
\end{proof}

\section{Subdiffusive CLT in the exceptional case}
\label{sec:subdiffusive}

Since particles are ordered in $d=1$ when the transition probability $p$ is nearest-neighbor, one may suspect that the expected displacement of a particle n time $t$ may be less than diffusive when $p$ is additionally symmetric.  This is indeed the case, when starting under $\nu'_\rho$.

\begin{theorem}
\label{tgp_thm}
Starting under $\nu'_\rho$, we have
$$\frac{1}{t^{1/4}}x_t \ \Rightarrow \ {\rm N}(0, \sigma^2)$$
where
$\sigma^2 = \sqrt{2/\pi}[(1-\rho)/\rho]$.
\end{theorem}
A functional CLT is also known, but here instead of Brownian motion the limit will be a fractional BM with Hurst parameter $1/4$ \cite{PS}.

There are a couple of ways to prove this subdiffusive CLT, first proved in \cite{Arratia}. 
Another proof was given in \cite{RV}.  
We present a method which makes a connection with the current through the bond $(-1,0)$ based on \cite{DemF}.

Define
$$J_{-1,0}(t) \ = \ N_+(t) - N_-(t)$$
where $N_+(t)$ is the number of particles which start on the left of $x=1/2$ and move across the bond $(-1,0)$ up to time $t$, and $N_-(t)$ is the number of particles to the right of $x=1/2$ crossing to the left.

Then, we have the following relations between $x_t$ and $J_{-1,0}(t)$:
\begin{eqnarray}
\{x_t \geq A\} &=& \{ J_{-1,0}(t) \geq \sum_{x=1}^{A-1}\eta_t(x)\} \ \ {\rm when \ } A\geq 1\nonumber\\
\{x_t = 0\} & = & \{J_{-1,0}(t) = 0\}\nonumber\\
\{x_t\leq A\} & = & \{J_{-1,0}(t) \leq \sum_{x = A+1}^{0} \eta_t(x)\} \ \ {\rm when \ }A\leq -1.
\label{tgp_current}\end{eqnarray}

We will show the following result which will imply Theorem \ref{tgp_thm}.  The proof of Theorem \ref{current_thm} is deferred to Subsection \ref{sec:currentthm}.
\begin{theorem}
\label{current_thm}Starting under $\nu'_\rho$, we have
$$\frac{1}{t^{1/4}}J_{-1,0}(t) \ \Rightarrow \ {\rm N}(0,\rho^2\sigma^2).$$
\end{theorem}

We now give the proof of Theorem \ref{tgp_thm}.  
\medskip

\begin{proof}[Proof of Theorem \ref{tgp_thm}] Write, for $A>0$, that
\begin{eqnarray*}
\P_{\nu'_\rho}\big(x_t\geq At^{1/4}\big) & = & \P_{\nu'_\rho}\Big(J_{-1,0}(t) \geq \sum_{x=1}^{\lceil At^{1/4}\rceil -1}\eta_t(x)\Big)\\
&=&\P_{\nu'_\rho}\Big(t^{-1/4}J_{-1,0}(t) \geq t^{-1/4}\sum_{x=1}^{\lceil At^{1/4}\rceil -1}\eta_t(x)\Big).
\end{eqnarray*}
Since, under $\nu'_\rho$, $\{\eta_t(x)\}_{x\geq 1}$ are i.i.d. Bernoulli random variables,
$$t^{-1/4}\sum_{x=1}^{\lceil At^{1/4}\rceil-1}\eta_t(x) \ \rightarrow \ A\rho \ \ \ {\rm a.s.}$$
Then, from Theorem \ref{current_thm}, we have
\begin{align*}
\lim_{t\rightarrow\infty} \P_{\nu'_\rho}(x_t\geq At^{1/4}) &=\lim_{t\rightarrow\infty} \P_{\nu'_\rho}\big(t^{-1/4}J_{-1,0}(t) \geq A\rho\big) \\
&= P\big({\rm N}(0,\sigma^2\rho^2) \geq A\rho\big).
\end{align*}
The right-hand side is the same as $P({\rm N}(0,\sigma^2)\geq A)$.

A similar argument works for $A<0$. \end{proof}

\subsection{Stirring representation}
One way to prove Theorem \ref{current_thm} is to use a `graphical' representation of symmetric exclusion process (cf. \cite{Liggett1}[Chapter 8).
We present the representation in the nearest-neighbor setting in $d=1$.  

Recall that the generator for symmetric simple exclusion has form (cf. \eqref{simplification})
\[
L\phi(\eta) \ = \ \sum_{x\in \Z}\big[\phi(\eta^{x,x+1})-\phi(\eta)\big].
\]
In other words, the process evolves by `exchanging' values with neighboring coordinate variables.

Now, consider a collection of Poisson processes with intensity $1/2$ indexed by the bonds $\{(x,x+1): x\in \Z\}$.  Place a particle at each vertex in $\Z$.  At the event times of the Poisson process with index $(x,x+1)$, the particles at $x$ and $x+1$ exchange positions.  Let $\xi^x_t$ be the location of the particle intially at $x$ at time $t\geq 0$.  Marginally, each $\xi^x_t$ has the statistics of a nearest-neighbor symmetric random walk, although jointly they are dependent.

We claim that
$$\eta_t(x) \ = \ 1\big\{ x\in \{\xi^i_t: \eta_0(i)=1\}\big\},$$
that is $\eta_t(x)=1$ exactly when there is an $i\in \Z$ such that $\xi^i_t = x$ and $\eta_0(i)=1$.  Infinitesimally, the possible transitions are exchanges of nearest-neighbor values which is the same as given in the generator $L$.
Then, 
$$J_{-1,0}(t) \ = \ \sum_{i<0}1(\xi^i_t\geq 0)\eta_0(i) - \sum_{i\geq 0}1(\xi^i_t<0)\eta_0(i).$$

Define now
$$K_+(t) = \sum_{i< 0}1\{\xi^i_t\geq 0\}, \ \ {\rm and \ \ } K_-(t) = \sum_{i\geq 0} 1\{\xi^i_t< 0\}$$
which represent the number of stirring particles starting from the left of the origin which end up on the right at time $t$, and vice versa.

Since, in the stirring process, a crossing of the bond $(-1,0)$ in one direction corresponds to a crossing in the other direction, as all sites are occupied by stirring particles, we have $K_+(t)-K_-(t)$ is constant in time.  But, $K_+(0)=K_-(0)=0$, and hence $K(t):=K_+(t)=K_-(t)$ for all $t\geq 0$.
When $K(t)\geq 1$, let $i_1<i_2<\cdots <i_{K(t)}<0$ be the random locations where $\xi_t^{i_k}\geq 0$, and $0\leq j_1<\cdots < j_{K(t)}$ be the random locations where $\xi_t^{j_k}<0$.  Then, we can write the current as
\begin{align}
\label{eq:JK}
J_{-1,0}(t) \ = \ \sum_{k=1}^{K(t)} a_k
\end{align}
where $a_k(t) = \eta_0(i_k) - \eta_0(j_k)$; Note the time dependence on $t$ comes from the random locations $i_k, j_k$ which depend on $t$.  Conditionally on $K(t)$, and the random locations, the variables $\{a_k\}_{k=1}^{K(t)}$ are independent with respect to $\nu'_\rho$.  Moreover, conditionally, only the variable $\alpha_1$, if $j_1=0$, is different from $\{\alpha_k\}_{k=2}^{K(t)}$ which are i.i.d., taking values $-1, 0, 1$ with mean $0$ and variance $2\rho(1-\rho)$.

\subsection{Estimates on $K$}
In the following, $P$ and $E$ refer to the stirring process probability and expectation when $\eta_0$ is distributed according to $\nu'_\rho$.
We have the following properties of $K$.

\begin{lemma}
\label{lem:Kest}
We have
$$t^{-1/2}E[K(t)] \ \rightarrow \ (2\pi)^{-1/2}.$$
\end{lemma}

\begin{proof}
Write, by symmetry,
\begin{eqnarray*}
E[K(t)] & = & \sum_{i<0} P(\xi^i_t\geq 0)\
= \ \sum_{i<0} P(\xi^0_t\geq -i)\\
&=& \sum_{i>0} P(\xi^0_t\geq i)\ =\ E[(z(t))_+]
\end{eqnarray*}
where $z(t)$ is a simple random walk starting at the origin and $(z(t))_+$ is the positive part.

Now, $\{t^{-1/2}z(t)\}$ is a uniformly integrable sequence of random variables which converges in distribution to ${\rm N}(0,1)$.  Hence, 
\[t^{-1/2}E[K(t)] \ \rightarrow \ \int_0^\infty (2\pi)^{-1/2}xe^{-x^2/2}dx \ = \ 1.\qedhere\]
\end{proof}

We also have that the variables $1(\xi^x_t\geq 0)$ and $1(\xi^y_t\geq 0)$ are negatively correlated, which makes intuitive sense since the exchange mechanism is repulsive to an extent.

\begin{lemma}
\label{lem:correlation}
For $x,y<0$, we have that
$$P(\xi^x_t,\xi^y_t\geq 0) \ \leq \ P(\xi^x_t\geq 0)P(\xi^y_t\geq 0).$$
\end{lemma}

We refer the reader to Lemma 4.12 \cite{Liggett1} or page 365 \cite{Arratia} for proofs of the lemma (in a more general context). See also the discussion in \cite{Liggett3}, which shows a very general form of negative association, a `strong Rayleigh' condition for symmetric exclusion processes started from product measures.
\medskip

A consequence of the Lemma \ref{lem:correlation} is that ${\rm Var}(K(t)) \leq E[K(t)]$.

\subsection{Proof of Theorem \ref{current_thm}}
\label{sec:currentthm}
Consider the following steps.

{\it Step 1.}
To prove the theorem, noting \eqref{eq:JK}, we show that the following difference converges in probability, as $t\rightarrow\infty$,
$$W(t) \ = \ t^{-1/4}\Big[\sum_{k=1}^{K(t)} a_k - \sum_{k=1}^{\lfloor (2\pi)^{-1/2}t^{1/2}\rfloor }a_k \Big] \ \rightarrow \ 0.$$
We make explicit the contribution of the term $\alpha_1$ which might reflect the tagged particle initially at the origin:
Write, noting empty sums vanish,
\begin{align*}
W(t) &= W(t)1(K(t)=0) + W(t)1(K(t)\geq 1)\\
&=  t^{-1/4}\Big[\sum_{k=2}^{K(t)} a_k - \sum_{k=2}^{\lfloor (2\pi)^{-1/2}t^{1/2}\rfloor }a_k\Big]1(K(t)\geq 1),
\end{align*}
with respect to variables $\{\alpha_k\}_{k=2}^{K(t)}$ which are i.i.d. mean $0$ and variance $2\rho(1-\rho)$, conditionally on $K(t), \{i_k, j_k\}$.

  Then, the variance, noting the conditional mean vanishes,
\begin{eqnarray*}
{\rm Var}(W(t)) & = & E\Big [E[W(t)^2|K(t), \{i_k,j_k\}]\Big] \\
& \leq & 2\rho(1-\rho)t^{-1/2}E\Big[\big|K(t) - \lfloor (2\pi)^{-1/2}t^{1/2}\rfloor \big|\Big]\\
&\leq& 2\rho(1-\rho)t^{-1/2}E\Big[\big|K(t)-E[K(t)]\big|\Big] \\
&&\ \ \ \ + 2\rho(1-\rho)t^{-1/2}\big|E[K(t)] - \lfloor (2\pi)^{-1/2}t^{1/2}\rfloor\big |.
\end{eqnarray*}

Since $K(t)$ is the sum of negatively correlated random variables, we have ${\rm Var}(K(t))\leq E[K(t)] = O(t^{1/2})$ and also $t^{-1/2}\big|E[K(t)] - (2\pi)^{-1/2}\big| = o(1)$ by Lemma \ref{lem:Kest}.
Hence, for $b>0$, we can apply Chebychev's inequality to get
$$P(W(t)> b) \ \leq \  b^{-2}{\rm Var}(W(t)) \ \leq \ b^{-2}\big[O(t^{-1/4})  + o(1)\big]\ \rightarrow \ 0.$$

\medskip
{\it Step 2.} Now, since $\alpha_11(K(t)\geq1)$ is bounded, $t^{-1/4}\alpha_11(K(t)\geq 1)\rightarrow 0$ a.s.  Hence, to finish the argument, we will take the limit of
\begin{eqnarray*}
&&E\Big[\exp\Big\{i\theta \frac{1}{t^{1/4}}\sum_{k=2}^{\lfloor (2\pi)^{-1/2}t^{1/2}\rfloor}a_k\Big\}\Big] \\
&&\ \ \ = \ E\Big[ E\Big[\exp\Big\{i\theta \frac{1}{t^{1/4}}\sum_{k=2}^{\lfloor(2\pi)^{-1/2}t^{1/2}\rfloor}a_k \Big\}\Big| K(t), \{i_k,j_k\}\Big]\Big].
\end{eqnarray*}
Note, by Lemma \ref{lem:Kest}, that $((2\pi)^{-1/2}-b)t^{1/2}\leq K(t)\leq ((2\pi)^{-1/2}+b)t^{1/2}$ with high probability for each $b>0$.  
Then, there are roughly $\lfloor ((2\pi)^{-1/2}\pm b)t^{1/2}\rfloor$ summands in the sum above, with high probability.

Hence, the last display may be approximated by
\begin{align*}
E\Big[\psi\big(\theta t^{-1/4}\big)^{\lfloor (2\pi)^{-1/2}t^{1/2}\rfloor +o(t^{1/2})\rfloor }\Big], 
\end{align*}
where $\psi$ is the characteristic function of $a_k$ conditioned on $K(t)$ and the random locations, which has an explicit distribution on values $0, -1,1$ with mean $0$ and variance $2\rho(1-\rho)$.  One now follows the strategy for the usual CLT to obtain the last term converges to the characteristic function of a Normal r.v. with variance $2\rho(1-\rho)\cdot (2\pi)^{-1/2}=\sqrt{2/\pi}\rho(1-\rho)$ as desired. \qed

\section{Notes}

The tagged particle problem, starting from $\nu_\rho$, is somewhat complete at the level of LLN and CLT for symmetric exclusion processes with finite range on $\Z^d$.  See \cite{Liggett1}, \cite{Liggett2} and references therein for treatments.  See also \cite{Jara} for longer range symmetric exclusion processes.  Large and moderate deviations have also been studied \cite{Sasamoto}, \cite{SV}, \cite{SV1}, \cite{Xue-Zhao}.  See also \cite{Toth} and references therein for related work on random walk in random environment.

For asymmetric simple exclusion, a CLT has been shown in diffusive scale when $d\geq 3$ \cite{SVY}, or when $d=1$ and the jumps are nearest-neighbor \cite{Kipnis}.  Also, when the jumps are mean-zero, finite-range, but asymmetric in $d\geq 1$, the CLT also holds \cite{Varadhan}.  

When the jumps are asymmetric with a nonzero drift, the variance of the tagged particle at time $t$ is shown to be diffusive, in the sense of Laplace transforms for all $d=1, 2$ \cite{Svarbound}.
However, the associated CLT, when $d=2$ and when $d=1$ and the jumps are not nearest-neighbor, is open.

When starting out of the invariant measure, the LLN behavior has been studied in \cite{Rezakhan-tagged}, but less is known for the fluctuations or other limits; see however \cite{JL}, 
\cite{CGS}, \cite{CS}, \cite{CS2},
\cite{Franco-Dirk}, \cite{Tracy-Widom}.

See also \cite{HLS}, \cite{jls1}, \cite{jls2}, \cite{Szrtg} for discussions and some results for a tagged particle in zero-range models.  
On other spaces, such as trees, see \cite{Chen}, \cite{Gantert-tagged} for recent developments. 

We comment that the tagged particle problem is an old one going back at least to Einstein's 1905 opus.  Recently, derivation of Brownian motion rigorously from deterministic hard sphere dynamics with elastic collisions, where initially positions of particles are in thermal equilibrium, has been addressed in \cite{Bodineau-inv}.  See \cite{Gal}, \cite{Ard}, \cite{Spohn}, \cite{Toth1} for reviews and discussions.

\newpage
\chapter[Section $11$]{Equilibrium fluctuations of symmetric exclusion}
\label{lec11}

We motivate the study of `equilibrium' fluctuations of $d$-dimensional symmetric exclusion via another proof of the 
$t^{1/4}$-scaling limit of the current, Theorem \ref{current_thm}, in the one dimensional nearest-neighbor process.  `Equilibrium' fluctuations capture the CLT behavior of the empirical mass density around the constant solution of the hydrodynamic equation, when started from an invariant measure $\nu_\rho$.  The limiting equation is a type of linear stochastic heat equation, an SPDE whose solution is an infinite dimensional `Ornstein-Uhlenbeck' process.  

We will focus our discussion in the symmetric simple exclusion context when $p(\pm e_i)=(2d)^{-1}$ with respect to the standard basis $\{e_i\}_{i=1}^d$.

\section{Another derivation of the current fluctuations limit}
Consider the $d=1$ nearest-neighbor, symmetric exclusion setting discussed in Subsection \ref{sec:subdiffusive}.  
Recall the current $J_{x,x+1}(t)$ through the bond $(x,x+1)$ in $\Z$ is the number of particles crossing the bond from left to right minus the number crossing from right to left up to time $t$.

A moment's thought gives, reflecting the boundary values at $x$, that
$$J_{x-1,x}(t) - J_{x,x+1}(t) \ = \ \eta_t(x) - \eta_0(x),$$
the difference being equal to $0$, $1$ or $-1$.
Then, formally,
$$J_{-1,0}(t) \ = \ \sum_{x\geq 0} J_{x-1,x}(t) - J_{x,x+1}(t) \ = \ \sum_{x\geq 0}\eta_t(x) - \eta_0(x).$$
However, the display does not make sense as typically there are an infinite number of particles in the system.

We may truncate however in the following way:  Let 
$$G_n(x) \ = \ \Big(1-\frac{x}{n}\Big)1(0\leq x\leq n).$$
Write, in terms of a scaling parameter $N$, that
\begin{eqnarray*}
&&\sum_{x\geq 0}G_n(x/N)\big[J_{x-1,x}(t) - J_{x,x+1}(t)\big] \\
&&\ \ \ \  =  J_{-1,0}(t) + \sum_{x\geq 1} \big[G_n(x/N) - G_n(x-1/N)\big]J_{x-1,x}(t)\\
 &&\ \ \ \  = J_{-1,0}(t) + \frac{1}{nN}\sum_{x=1}^{N+1}J_{x-1,x}(t).
 \end{eqnarray*}
 At the same time, we have
 \begin{eqnarray*}
 \sum_{x\geq 0}G_n(x/N)\big[J_{x-1,x}(t) - J_{x,x+1}(t)\big] &=& \sum_{x\geq 0} G_n(x/N)\big[\eta_t(x) - \eta_0(x)\big].
 \end{eqnarray*}

Following our custom, since space is scaled by $N$, we will speed up time by $N^2$, and define the mass `fluctuation field' $W^N_t$ with respect to $\Z^d$ by its action on $G$:
 \begin{equation}
 \label{eq:fluc_field}
 W^N_t(G) \ = \ \frac{1}{N^{d/2}}\sum_{x\in \Z^d} G(x/N)\big(\eta_{N^2t}(x) - \rho\big),
 \end{equation}
 here written for $d\geq 1$.
 
In $d=1$, scaling the current by the fourth root of the time scaling, we have for all $n\geq 1$ that
\begin{align}
\label{eq:current_rep}
\frac{1}{\sqrt{N}}J_{-1,0}(N^2t)  =  W^N_t(G_n) - W^N_0(G_n) -\frac{1}{nN^{3/2}}\sum_{x=1}^{N+1}J_{x-1,x}(N^2t).
\end{align}

The idea now is that the first two terms on the right should be given in terms of a limit fluctuation field $W_t$, which we must define, while the last term on the right hand side should vanish as $N\uparrow\infty$ and $n\uparrow\infty$.  

One can adjust the construction of the process (cf. Remark \ref{rem:construction}), to include counts $N^+_{x,x+1}(t)$ and $N^-_{x,x+1}(t)$, keeping track of the numbers of particles crossing $(x,x+1)$ from left to right and vice versa.  Then, $J_{x,x+1}(t) = N^+_{x,x+1}(t)-N^-_{x,x+1}(t)$ and 
\begin{align*}
M_{x,x+1}(t) & =  J_{x,x+1}(t) - \frac{1}{2}\int_0^t \big[\eta_s(x)(1-\eta_s(x+1)) - \eta_s(x+1)(1-\eta_s(x)) \big]ds   \\
&= J_{x,x+1}(t) - \frac{1}{2}\int_0^t \big[\eta_s(x)- \eta_s(x+1)\big] ds \ \ {\rm and }
\end{align*}
\[
M_{x,x+1}(t)^2 - \frac{1}{2}\int_0^t \big[ \eta_s(x)(1-\eta_s(x+1)) + \eta_s(x+1)(1-\eta_s(x)) \big]ds 
\]
are martingales.  
Since jumps are not simultaneous in the process, $\{M_{x,x+1}(t)\}_{x\in \Z}$ are orthogonal martingales.

\begin{exercise}
\rm Show $E[M_{x,x+1}(t)M_{y,y+1}(t)]=0$ by decomposing on the possible crossing times of bonds $(x,x+1)$ and $(y,y+1)$, which are stopping times occurring a.s. at distinct times, and the martingale property.  That is, write 
$M_{x,x+1}(t) = \sum (M_{x,x+1}(\tau_{k+1})-M_{x,x+1}(\tau_k))$ and a similar formula for $M_{y,y+1}(t)$ where $0=\tau_0<\cdots <\tau_{N(t)}<\tau_{N(t)+1}=t$ are the $N(t)$ jumps on these bonds up to time $t$.  Then, we have $$E\big[(M_{x,x+1}(\tau_{k+1})-M_{x,x+1}(\tau_k))(M_{y,y+1}(\tau_{j+1})-M_{y,y+1}(\tau_j))\big]$$
equals zero if $k =j$ since jumps are not simultaneous.  But, if $k<j$, then since $M_{y,y+1}(\tau_{j+1}\wedge t)$ is a martingale, the display also vanishes.
\end{exercise}

In the following, recall as before that $E_\mu$ stands for the probability and expectation under $\mu$, and $\P_\mu$ and $\E_\mu$ for the process measure and expectation when starting in $\mu$.

\begin{lemma}
\label{error}
Starting in an invariant measure, $\nu_\rho$, we have
$$\lim_{n\uparrow\infty}\sup_{N\geq 1} \E_{\nu_\rho}\Big[ \Big(\frac{1}{nN^{3/2}}\sum_{x=1}^{N+1}J_{x-1,x}(N^2t)\Big)^2\Big] \ = \ 0.$$
\end{lemma}

\begin{proof}
Write 
\begin{align}\label{current1}
&\frac{1}{nN^{3/2}}\sum_{x=1}^{N+1}J_{x-1,x}(N^2t) \nonumber\\
&\quad\quad= \frac{1}{nN^{3/2}}\sum_{x=1}^{N+1} \frac{1}{2}\int_0^{N^2t} \big[ \eta(x-1)(1-\eta(x)) - \eta(x)(1-\eta(x-1)) \big]ds \nonumber \\
&\quad\quad \ \ \ \ \ + \ \frac{1}{nN^{3/2}}\sum_{x=1}^{N+1} M_{x-1,x}(N^2t) \\
&\quad\quad = \frac{1}{nN^{3/2}}\frac{1}{2}\int_0^{N^2t} \big[\eta_{s}(0) - \eta_{s}(N)\big] ds + \frac{1}{nN^{3/2}}\sum_{x=1}^{N+1} M_{x-1,x}(N^2t).\nonumber
\end{align}
Here, we used that $\eta(x-1)(1-\eta(x)) - \eta(x)(1-\eta(x-1)) = \eta(x-1)-\eta(x)$.

Now, from an $H_{-1}$-norm/variance estimate Proposition \ref{variance_sym}, we have for all $y$ that
$$\E_{\nu_\rho}\Big[ \Big(\int_0^{N^2t}\big( \eta_s(y) - \rho\big) ds\Big)^2\Big] \ \leq \ C(\rho)N^{3/2}t^{3/2}.$$
Hence, the integral term in \eqref{current1} is bounded by
$Cn^{-2}N^{-3}N^3 t^{3/2} = Cn^{-2}$ which vanishes as $n\uparrow\infty$.

For the martingale term in \eqref{current1}, using orthogonality of martingales and $\eta(x)\leq 1$, we obtain
\begin{eqnarray*}
&&\E_{\nu_\rho}\Big[ \Big(\frac{1}{nN^{3/2}}\sum_{x=1}^{N+1} M_{x-1,x}(N^2t)\Big)^2\Big]\\
&&\ \ \ \leq \ \frac{1}{n^2N^3}\sum_{x=1}^{N+1}\E_{\nu_\rho} \big[M^2_{x-1,x}(N^2t)\big]\\
&&\ \ \ \leq \ \frac{1}{2n^2N^3}\sum_{x=1}^{N+1}\E_{\nu_\rho} \int_0^{N^2t}\big[\eta(x)(1-\eta(x+1)) + \eta(x+1)(1-\eta(x))\big] ds\\
&&\ \ \ =\  O(n^{-2})
\end{eqnarray*}
To conclude the proof. \end{proof}

To complete the argument, we need to understand the limit fluctuation fields.  This is the subject of the next subsection.  In Subsection \ref{sec:returntocurrent}, we will complete the proof of Theorem \ref{current_thm}.

\section{Infinite dimensional Ornstein-Uhlenbeck process limit}
\label{sec:infinite-OUlimit}
We consider now the general $d\geq 1$ setting.  Consider the space $\mf D= C^\infty_c(\R^d)$, consisting of compactly supported, $C^\infty$ real functions on $\R^d$.  Its dual $\mf D'$ is the space of distributions, that is the continuous linear functionals acting on $\mf D$.

Recall the definition of $W^N_t(G)$ in \eqref{eq:fluc_field}. 
To capture the evolution of $W^N_t(G)$ for a fixed $G\in \mf D$, write, with respect to the finite-range symmetric exclusion generator $L$ (cf. \eqref{simplification}),
\begin{eqnarray*}
W^N_t(G) &=& W^N_0(G) + N^2\int_0^t LW^N_s(G) ds + \M^N_t(G)
\end{eqnarray*}
where, after the usual twice summation-by-parts under symmetry,
$$N^2LW^N_t(G) \ = \ \frac{1}{2N^{d/2}}\sum_{x,y\in \Z^d} p(y)\triangle^N_{x,y}G(x/N) \eta_{N^2s}(x)$$
and
$\M^N_t(G)$ is a martingale such that
$${\mc N}^{N}_t(G)\ = \ (\M^N_t(G))^2 - N^2\int_0^t\big[ L(W^N_s(G))^2 - 2Y^N_s(G)LW^N_s(G)\big]ds$$
is a martingale.
The last integral in the display, the quadratic variation of $\M^N_t(G)$, can be evaluated as
\begin{eqnarray*}
&&\big\langle \M^N(G)\big\rangle_t \ = \ N^2\int_0^t \big[L(W^N_s(G))^2 - 2Y^N_s(G)LW^N_s(G)\big] ds \\
&&\ \ \  = \ \frac{1}{N^{d}}\int_0^t \sum_{x\in \Z^d} p(y)[\nabla^N_{x,y}G(x/N)]^2\eta_{N^2s}(x)(1-\eta_{N^2s}(x+y)) ds.\end{eqnarray*}
Here, similar to the hydrodynamics calculations (cf. Subsection \ref{sec:sketch-exclusionhydro}), $\nabla^N_{x,y}= N\big[G((x+y)/N) - G(x/N)\big]$ and $\triangle^N_{x,y}G= N^2\big[G((x+y)/N) -2G(x/N) + G((x-y)/N)\big]$ are the scaled discrete gradient and Laplacian.

Since $\sum_x p(y) \triangle^N_{x,y} G(x/N) =0$, we may subtract $\rho$ and write
\begin{align}
\label{eq:fluc-mart-rep}
W^N_t(G) = W^N_0(G) + \frac{1}{2} \int_0^t W^N_s\Big(\sum_y p(y)\triangle^N_{\cdot,y}G(\cdot/N)\Big)ds  + \M^N_t(G) \ \ \ \ \  {\rm and}
\end{align}
\begin{align*}
&{\mc N}^N_t(G) \\
&\ = \big(\M^N_t(G)\big)^2 - \frac{1}{N^d}\int_0^t \sum_{x,y\in \Z^d}p(y)\big[\nabla^N_{x,y} G(x/N)\big]^2\eta_{N^2s}(x)\big(1-\eta_{N^2s}(x+y)\big)ds.
\end{align*}

\medskip
Now, as before with respect to the hydrodynamics, we have two steps:
\medskip

Step 1:  Show tightness of $\{W^N_t: t\in [0,T]\}_{N\geq 1}$ in an appropriate space, and continuity of limit trajectories under limit points.  

Step 2:  Identify the limit points in terms of a unique `infinite dimensional Ornstein-Uhlenbeck' process.
\medskip

Putting it together, we will arrive at the following informal description: $W^N_t$ converges to $W_t$ which solves, in the nearest-neighbor symmetric setting,
\begin{equation}
\label{infinity_eqn}
dY_t \ = \ \frac{1}{2d}\triangle W_t dt + \sqrt{d^{-1}\rho(1-\rho)}\nabla d\B_t,
\end{equation}
more explained in the following subsections.

\medskip
\subsubsection{Spaces}
A natural space will be $D([0,T], {\mf D}')$, the distribution valued right-continuous with left limits
 trajectories on ${\mf D}$, with the strong dual topology.

 We remark there are other natural spaces one could consider.  For instance, $W^N_t$ could act on the Schwartz space of rapidly decreasing functions $\mc S$, in which case $W^N_t$ would be a member of $\mc S'$, the space of `tempered distributions.'  One could work also with the Hermite basis based $\mc H_{k}$ spaces containing $\mc S$. Then, $W^N_t$ would be a member of $\mc H_{-k}\subset \mc S'$. 
 In another direction, we could have also specified the lattice as $\T_N^d$, instead of $\Z^d$.
 See Chapter 11 \cite{KL} for instance where spaces $\mc H_{-k}$ with respect to $\T_N^d$ are used; on $\R^d$, the Hermite spaces would use definitions in \cite{Reed-Simon}[p. 142].

Both $\mc S'$ and $\mc H_{-k}$ are nuclear, Frech\'et spaces, whereas the space of distributions $\mf D'$ is a `strict inductive limit' of nuclear Frech\'et spaces.  Importantly, in each of these spaces, tightness of $\{W^N_t: t\in [0,T]\}$ is implied by tightness of $\{W^N_t(G): t\in [0,T]\}$ for each $G$ in $\mc S'$, $\mc H_{-k}$, or $\mc D'$; see \cite{Fonseca}, \cite{Fouque}, \cite{Mitoma} for more discussion.

\subsection{A precise statement of \eqref{infinity_eqn}}  
Let $Q_N$ be the probability measure on $D([0,T], \mf D')$ governing $\{W^N_t: t\in [0,T]\}$, when the underlying nearest-neighbor, symmetric exclusion process starts from invariant measure $\nu_\rho$.

\begin{theorem}
\label{fluc_thm}
We have $Q_N$ converges to $Q$, concentrated on $C([0,T], \mf D')$, governing a Gaussian Markov random field with mean $0$ and covariance,
\begin{eqnarray*}
&&E_Q[W_s(G)W_t(H)] \\
&&\ = \ \frac{\rho(1-\rho)}{(2\pi d^{-1}(t-s))^{d/2}}\int_{\R^d}\int_{\R^d} G(v)H(u)\exp\Big\{-\frac{d|u-v|^2}{2(t-s)}\Big\}dudv
\end{eqnarray*}
for all $0\leq s\leq t$, $G, H\in \mf D$.
\end{theorem}

The process $W_t$ governed by $Q$ is sometimes called a  `generalized Ornstein-Uhlenbeck process', as discussed in Holley-Stroock \cite{HS}.  The scalings and the limit in Theorem \ref{fluc_thm}, where space is scaled by $1/N$, time by $v(N)=N^2$ and the empirical measure by $1/\sqrt{N}$, are sometimes referred to as a `Edwards-Wilkinson' fluctuation limit.
\medskip

\subsubsection{Holley-Stroock martingale problem}
The existence/uniqueness of $Q$ in Theorem \ref{fluc_thm} follows from a `martingale problem' characterization. 

 Let $U= (2d)^{-1}\triangle$ be the nonnegative self-adjoint operator defined on domain $L^2(\R^d)$ and let $T_t$ be the associated heat semigroup.  Define $B= \sqrt{d^{-1}\rho(1-\rho)}\nabla$ to be the linear gradient operator.  Let also $\F_t$ be the sigma-field in $D([0,T], \mf D)$ generated by a process $W_s(H)$ for $s\leq t$ and $H\in \mf D$.

\begin{theorem}
\label{OU}
Suppose $Q$ is a probability measure governing $W_t$, concentrating on $C([0,T], \mf D')$, and for each $H\in \mf D$,
$$M_t^{U,H} \ = \ W_t(H) - W_0(H) - \int_0^t W_s(UH)ds$$
and
$$N^{U,H}_t \ = \ (M_t^{U,H})^2 - \|BH\|^2_{L^2} t$$
are $L^1(Q)$, $\F_t$-martingales.  Then, for all $0\leq s<t$, and subsets $A\subset \R^d$, $Q$ a.s.,
\begin{eqnarray}
\label{eq:HS-fd}
&&Q\Big[W_t(H)\in A| \F_s\big] \\ 
&& \ \ = \ \int_A \frac{1}{\sqrt{2\pi \int_0^{t-s}\|BT_rH\|^2_{L^2} dr}}\exp\Big\{\frac{-|y- W_s(T_{t-s}H)|^2}{2\int_0^{t-s}\|BT_rH\|^2_{L^2} dr}\Big\}dy.\nonumber
\end{eqnarray}
Therefore, the finite-dimensional distributions of $Q$, and hence the measure $Q$ on $C([0,T], \mf D)$, are determined by its restriction to $\F_0$.
\end{theorem}

We refer to the work of Holley-Stroock \cite{HS} for more details; see also Chapter 11 \cite{KL}.
\medskip

\subsubsection{Sketch of proof of Theorem \ref{fluc_thm}}
Assuming Steps 1,2, we outline the proof of Theorem \ref{fluc_thm}.  In our context, the restriction to $\F_0$ is already known:  Under the invariant measure $\nu_\rho$, we have
$W^N_0$ converges to a Gaussian field with mean zero and covariance
\begin{equation}
\label{0_cov}
E_Q[W_0(G)W_0(H)] \ = \ \rho(1-\rho)\langle G, H\rangle_{L^2}.\end{equation}

\begin{exercise}\rm
Show that the joint distribution of $W^N_0(H_1), \ldots, W^N_0(H_\ell)$ is Gaussian with covariance given via \eqref{0_cov}.  One can use the Cram\'er-Wold device.
\end{exercise}

Also, from the martingale property in Theorem \ref{OU}, the formula \eqref{eq:HS-fd}, and that $\frac{d}{dt} T_t = T_t U$, we can argue
\begin{eqnarray}
\label{eq:W-covariance}
E_Q[W_s(G)W_t(H)] & = & E_Q\big[W_s(G)(M_t^{U,H} - M_s^{U,H})\big] \\
&&\ \ \ \ + \int_s^t E_Q\big[W_s(G)W_u(UH)\big]du + E_Q\big[W_s(G)W_s(H)\big]\nonumber\\
&=& E_Q\big[W_s(G)W_s(T_{t-s}H)\big] \ = \ E_Q\big[W_0(G)W_0(T_{t-s}H)\big].\nonumber
\end{eqnarray}
This covariance is exactly what is given in Theorem \ref{fluc_thm}.

Now, since $\B_t(H) = \|BH\|^{-1}_{L^2}M_t^{U,H}$ is a continuous martingale with quadratic variation $t$, by Levy's characterization, we have that $\B_t(H)$ is distributed as a Brownian motion.  Hence, we have
\begin{equation}
\label{eq:W_t-decomp}
W_t(H) \ = \ W_0(H) + \int_0^t W_s(UH)ds + \|BH\|_{L^2} \B_t(H)
\end{equation}
where $\B_t$ is the infinite dimensional Brownian motion (Gaussian Markov random field) with covariance
$$E_Q[\B_s(G)\B_t(H)] \ = \ (\min\{s,t\})\int_{\R^d} \frac{\nabla G(u)}{\|\nabla G\|_{L^2}}\frac{\nabla H(u)}{\|\nabla H\|_{L^2}} du.$$
In this way, \eqref{eq:W_t-decomp} gives a meaning to the integral form of \eqref{infinity_eqn}.

\begin{exercise}\rm
Use polarization with the martingales $M^{U,G}_s$ and $M^{U,H}_t$ to show the formula in the last display.
\end{exercise}

What remains in the proof of Theorem \ref{fluc_thm} is to show Steps 1,2 from which one may conclude that $Q_N$ converges to a $Q$, supported on continuous trajectories, satisfying the `martingale problem' conditions in Theorem \ref{OU}.  This is done in Subsection \ref{sec:proofs-eq-fluc-12}.

\subsection{Application: Proof of Theorem \ref{current_thm}}
\label{sec:returntocurrent}

We return to our motivating example with respect to current fluctuations.  By \eqref{eq:current_rep} and Lemma \ref{error}, since $N^{-1/2}J_{-1,0}(N^2t)$ does not depend on $n$, we have
uniformly in $N\geq 1$ that 
$$\{W^N_t(G_n) + W^N_0(G_n): n\geq 1\}$$
is a Cauchy sequence in $L^2(\nu_\rho)$.  
Note, by stationarity, that
$$\sup_{N\geq 1}\E_{\nu_\rho}\big[\big(W^N_t(H)\big)^2\big] \leq 2\|H\|^2_{L^2(\R^d)}.$$
Then, by Theorem \ref{fluc_thm}, approximating $G_n$ by smooth compactly supported functions, we have for fixed $n$ as $N\uparrow\infty$ that
$$W^N_t(G_n) - W_0^N(G_n) \ \Rightarrow \ W_t(G_n)- W_0(G_n).$$
Since $\{W_t(G_n)-W_0(G_n): n\geq 1\}$ is Cauchy in $L^2(\nu_\rho)$, we denote its limit by $W_t(H_0) - W_0(H_0)$, whose distribution is a mean-zero Gaussian.  

In particular, 
$$\frac{1}{\sqrt{N}}J_{-1,0}(N^2t) \ \Rightarrow \ W_t(H_0)-W_0(H_0).$$
By Lemma \ref{error}, one may identify the limiting variance by computing
\begin{align}
\label{eq:duality-variance}
\lim_{n\uparrow\infty}\lim_{N\uparrow\infty}{\rm Var}\big(N^{-1/2}J_{-1,0}(N^2t)\big)  =  \lim_{n\uparrow\infty} \lim_{N\uparrow\infty}\E_{\nu_\rho}\Big[\big(W^N_t(G_n) - W^N_0(G_n)\big)^2\Big].
\end{align}

\begin{exercise}\rm
Use duality, that is $\E_{\nu_\rho}[(\eta_{N^2t}(x)-\rho)(\eta_0(y)-\rho)] = \rho(1-\rho)P_{N^2t}(0,y-x)= \rho(1-\rho)P\big(Z_{N^2t}=y-x\big)$, where $Z_\cdot$ is a continuous-time simple symmetric random walk starting at $0$ (cf. Subsection \ref{sec:duality}), to verify that the variance \eqref{eq:duality-variance} with $t=1$ equals $\sqrt{2/\pi}\rho(1-\rho)$.  Hint:  To analyze the cross term, weak convergence $Z_{N^2}/N\Rightarrow {\rm N}(0, 1)$ may be useful.
\end{exercise}

\section{Proof of Steps 1,2: Tightness and identification}
\label{sec:proofs-eq-fluc-12}
We first restate Steps 1, 2 given informally before \eqref{infinity_eqn} in the setting of nearest-neighbor symmetric exclusion.
\medskip

Step 1:  Show tightness of $\{Q_N\}_{N\geq 1}$ governing $\{W^N_t: t\in [0,T]\}_{N\geq 1}$, members in $D([0,T], \mf D')$, in the uniform topology of $C([0,T], \mf D')$.
Hence, trajectories under limit points $Q$ will be continuous.

Step 2:  Identify the limit points $Q$ in terms of the unique `infinite dimensional Ornstein-Uhlenbeck' process given in Theorem \ref{fluc_thm}.
\medskip

In the following, recall that $p$ is nearest-neighbor, symmetric and translation-invariant.

\medskip
\subsection{Moments of $\M^N_t(H)$}
Recall the martingales $\M^N_t(H)$ and $\mc N^N_t(H)$ specified in \eqref{eq:fluc-mart-rep}.   Before going to the proofs of Steps 1,2, we record some moment estimates of $\M^N_t(H)$ for $H\in \mf D$.

By taking expectation of the quadratic variation, we obtain
\begin{align}
\label{eq:secondmomentofM}
\E_{\nu_\rho}\big[ (\M^N_t(H))^2\big] 
 = \frac{\rho(1-\rho)T}{N^d}\sum_{x,y\in \Z^d}p(y)\big[\nabla^N_{x,y} H(x/N)\big]^2 \leq C(\rho, T)\|\nabla H\|^2_{L^2(\R^d)}.
\end{align}

To bound the fourth moment, we will take the following approach.  Other arguments, via Burkholder-Gundy-Davis inequalities may also be used.
\begin{lemma}
For all local functions $F$, $s,t\in [0,T]$, and $\lambda \in \R$, we have
$$Z^\lambda_{s,t} \ = \ \exp\Big\{\lambda F(\eta_t) - \lambda F(\eta_s) - \int_s^t e^{-\lambda F(\eta_u)}Le^{\lambda F(\eta_u)}du\Big\}$$
is a martingale.
\end{lemma}

\begin{proof} The lemma is a type of `Girsanov' formula.  See \cite{EK}[around p. 175] which shows $Z_{s,t}$ is a local martingale.   Since we are dealing with the exclusion process, where occupation numbers are bounded, it is integrable and so a martingale.
\end{proof}

\begin{lemma}
\label{lem:fourthmoment}
For $H\in \mf D$, and $t\in [0, T]$, we have 
$$\E_{\nu_\rho}\big[\big(\M^N_t(H)\big)^4\big] \ \leq \ C(H)\big( t^2 + N^{-d-2}t\big).$$
\end{lemma}

\begin{proof}
Let $Z^\lambda_{N^2s,N^2t}$ be the martingale with 
\[F(\eta) = N^{-d/2}\sum_x H(x/N)\big(\eta(x)-\rho\big).\]
By explicit calculation,
\begin{align*}
&e^{-\lambda F(\eta_{N^2u})}(N^2L)e^{\lambda F(\eta_{N^2u})} \\
&\quad\quad = \ N^2\sum_{x,y} p(y-x)\big[e^{\lambda N^{-d/2}[H(y/N)-H(x/N)]} - 1\big]\eta_{N^2u}(x)\big(1-\eta_{N^2u}\big).
\end{align*}
Morevoer,
\begin{equation}\label{mart=1}
\E_{\nu_\rho} \big[Z^\lambda_{N^2s,N^2t}\big] \ = \ 1.\end{equation}

Now, we may expand the left-hand side and equate in powers of $\lambda$.  By symmetry of $p$, we evaluate \eqref{mart=1} as
\begin{eqnarray*}
&&\E_{\nu_\rho}\Big[\exp\Big\{\lambda M^N_t(H)\big)\\
&&\ \ \  - \frac{\lambda^2N^2}{2N^d}\int_0^t \sum_{x,y} p(y-x)\big(H(y/N)-H(x/N)\big)^2\eta_{N^2u}(x)\big(1-\eta_{N^2u}(y)\big) du \\
&&\ \ \ \ \ \ +\lambda^3 \int_0^t \mc R^1 ds+ \lambda^4 \int_0^t \mc R^2ds + \lambda^5\int_0^t \mc R^4ds\Big\}\Big] \ = \ 1.\end{eqnarray*}
It is not difficult now, although a long calculation, to estimate $\{\mc R^i\}_{i=1}^3$ and obtain the desired statement. Note that by matching $\lambda^2$ terms, one recovers \eqref{eq:secondmomentofM} for instance.
\end{proof}

\begin{exercise}\rm
Make the computation in the proof of the above lemma to match fourth powers of $\lambda$ to obtain the lemma statement.
\end{exercise}

\subsection{Proof of Step 2, given Step 1}
We will apply results in the limit theory of martingales in the context of symmetric simple exclusion.  References in this vein include \cite{EK}, \cite{JS}, \cite{Whitt}, among others.

Suppose $Q$ is a limit point found with respect to the uniform topology of $\{Q_N\}_{N\geq 1}$, necessarily supported on continuous $\mf D'$-valued trajectories.  We will identify the subsequence by $\{N: N\geq 1\}$ itself to streamline notation. 

 We now show that $M_t^{U,H}$ and $N^{U,H}_t$ are $L^1(Q)$ martingales, and thereby identify $Q$ via Theorem \ref{OU}.

Consider the formula for $\M^N_t(G)$ in \eqref{eq:fluc-mart-rep}.  Given Step 1 and that the function $w_\cdot(H)\mapsto w_t(H) -w_0(H) - (1/2)\int_0^tw_s\big(\sum_y p(y) \triangle^N_{\cdot, y}H(\cdot/N)\big)ds$ is continuous in the uniform topology, the term $\M^N_t(G)$ would converge in distribution to a limit
$$M^{U,H}_t:=W_t(H)-W_0(H) - \frac{1}{2d}\int_0^t W_s(\triangle H)ds.$$

 Observe, 
 by \eqref{eq:secondmomentofM}, that $\sup_{N\geq 1}\E_{\nu_\rho}\big[\big(\M^N_t(H)\big)^2\big]<\infty$.  Therefore, we conclude
  $\{\M^N_t(H)\}_{N\geq 1}$ is uniformly integrable for each $t\in [0,T]$.  
 Since the martingales $\M^N_t(H)\Rightarrow M^{U, H}_t$, we conclude the limit $M^{U, H}_t$ is an $L^2(Q)$ martingale by Theorem IX.1.12 in \cite{JS}.

Also, as
\begin{align}
\label{eq:quad_var_prob}
&\E_{\nu_\rho}\Big[\Big\{\frac{1}{N^d}\int_0^t \sum_{x,y\in \Z^d} p(y) \big(\nabla^N_{x,y} H(x/N)\big)^2\\
&\quad \quad \cdot \Big( \eta_{N^2s}(x)\big(1-\eta_{N^2s}(x+y)\big) - \rho(1-\rho)\Big)ds\Big\}^2\Big]
\ \leq \ \frac{C(\rho)t}{N^d}\|\nabla H\|^2_{L^2(\R^d)},\nonumber
\end{align}
we have $\langle \M^N(H)\rangle_t$ converges to $\rho(1-\rho)t\|\nabla H\|^2_{L^2(\R^d)}$ in probability.

By the continuous mapping theorem, $\big(\M^N_t(H)\big)^2\Rightarrow \big(M^{U, H}_t\big)^2$.  Moreover, by the convergence in probability implied by \eqref{eq:quad_var_prob}, we observe $\big(\M^N_t(H)\big)^2 - \langle \M^N(H)\rangle_t \Rightarrow \big(M^{U, H}_t\big)^2 - \rho(1-\rho)t\|\nabla H\|^2_{L^2(\R^d)}=: N^{U,H}_t$.  To see that the limit $N^{U, H}_t$ is an $L^1(Q)$ martingale, we observe $\big\{\big(\M^N_t(H)\big)^2 - \langle \M^N(H)\rangle_t: N\geq 1\big\}$ is uniformly integrable by the fourth moment estimate Lemma \ref{lem:fourthmoment}.  
Hence, as desired $N^{U, H}_t$ is a martingale, again by Theorem IX.1.12 in \cite{JS}.

\medskip

\subsection{Proof of Step 1}
To show tightness of $\{Q_N\}$ 
with respect to the uniform topology, 
as $\mf D$ is the `inductive limit' of nuclear Frech\'et spaces, it is sufficient to verify tightness of $\{W^N_t(H): N\geq 1\}$ for each $H\in \mf D$; see \cite{Fouque}.

We remark, as in the study of hydrodynamics, tightness with respect to uniform topology implies that in the Skorohod topology.
We have therefore the following tightness criterion. 
Denote the uniform modulus of continuity by
$$w_{W}(\delta) \ = \ \sup_{\stackrel{|t-s|\leq \delta}{s,t\in [0,T]}} |W_t - W_s|.$$
\begin{lemma}
\label{tightness}
A family of probability measures $\{Q_N\}$ on $D([0,T],\mf D')$ is tight if for each $H\in \mf D$,
\begin{itemize}
\item[(a)] $\lim_{A\uparrow\infty}\lim_{N\uparrow\infty} Q_N\big[ |W^N_0(H)| >A\big] \ = \ 0 $
and, 
\item[(b)] for all $\epsilon>0$,
$\lim_{\delta\downarrow 0} \lim_{N\uparrow\infty} Q_N\big[w_{W^N(H)}(\delta)\geq \epsilon\big] \ = \ 0$.
\end{itemize}
\end{lemma}

\begin{proof}
Prelimit, we have
\begin{align}\label{W_decom}
W^N_t(H) &=  \M^N_t(H) + W^N_0(H) \\
&\quad + \frac{1}{2N^{d/2}}\int_0^t \sum_{x,y}p(y) (\triangle^N_{x,y} H)(x/N) (\eta_{N^2s}(x)-\rho)ds.\nonumber
\end{align}
We need to show each of these terms satisfies the criteria.  
\vskip .1cm

{\it The term $W^N_0(H)$.}  The mean square, starting from the invariant measure $\nu_\rho$, is bounded in the limit by $2\rho(1-\rho)\|H\|^2_{L^2(\R^d)}$, and hence item (a) holds.  Part (b) holds trivially.
\vskip .1cm

{\it The integral term.}  Part (a) holds trivially.  For item (b), 
 by Chebychev and Schwarz inequalities, we may bound
\begin{align*}
&Q_N\Big[\sup_{\stackrel{|t-s|\leq \delta}{s,t\in [0,T]}} \Big|\int_s^t \frac{1}{2N^{d/2}}\sum_{x,y} p(y)\big(\triangle^N_{x,y}G\big)(x/N)(\eta_{N^2u}(x)-\rho)du\Big|>\epsilon\Big]\\
&\quad \leq \frac{\delta}{\epsilon^2} \int_0^T E_{Q_N}\Big[\Big(\frac{1}{2N^{d/2}}\sum_{x,y} p(y)\big(\triangle^N_{x,y}G\big)(x/N)(\eta_{N^2u}(x)-\rho)\Big)^2\Big] du.
\end{align*}
Starting from $\nu_\rho$, the last quantity is bounded by $2d^{-1}\rho(1-\rho)T\delta \epsilon^{-2}\|\triangle G\|_{L^2(\R^d)}^2$, which vanishes as $\delta\downarrow 0$.

\vskip .1cm
{\it The term $\M^N_t(H)$.} Part (a) holds trivially.  For item (b), 
We employ the standard `three $\epsilon$ argument, dividing the interval $[0,T]$ into $O(T/\delta)$ subintervals (a similar scheme was used in Subsection \ref{sec:3.3.3}).  Noting when $s,t$ belong to a single subinterval or are in adjacent ones, we may bound 
$$\sup_{|t-s|\leq \delta}|\M^N_t(h_z) - \M^N_s(h_z)| \leq \max_i\sup_{t\in I_i} \ 3|\M^N_t(h_z) - \M^N_{t_i}(h_z)|$$
where the max is over the subintervals $\{I_i\}$ and $t_i$ is the left endpoint of the $i$th one.
Then, by stationarity, Chebychev and Doob's inequality and Lemma \ref{lem:fourthmoment},
\begin{align*}
&Q_N\Big[\sup_{\stackrel{|t-s|\leq \delta}{t,s\in [0,T]}}| \M^N_t(H)-\M^N_s(H) |> \epsilon\Big] \\
&\quad \leq \sum_i Q_N\Big[\sup_{t\in I_i}| \M^N_t(H)-\M^N_{t_i}(H) |> \epsilon/3\Big]\\
&\quad \leq \sum_i \frac{C}{(\epsilon/3)^4}E_{Q_N}\Big[\big(\M^N_\delta(H)\big)^4\Big] \ \leq\  \frac{C(H)|T}{\delta} \cdot  \big[\delta^2 + {\delta}/{N^{d+1}}\big]. 
\end{align*}
The right-hand side of the display vanishes as $N\uparrow\infty$ and $\delta\downarrow 0$.
\end{proof}

\section{Notes}
The proof given here for the current fluctuations Theorem \ref{current_thm} follows \cite{JL}, and
the idea of truncation goes back to \cite{RV}.

We comment, with respect to equilibrium fluctuations, in the context of symmetric exclusion models, as for the hydrodynamic limit, there is no `replacement' estimate needed; see also the treatment in \cite{KL}.     The theory of generalized Ornstein Uhlenbeck processes originates in \cite{HS}.

Recently, a proof of fluctuations starting from `flat' states via `discrete regularity structures' has been given in symmetric exclusion \cite{HMW}.  

`Nonequilibrium' fluctuations have also been shown in symmetric exclusion \cite{Ravishankar}. 
In asymmetric models, fluctuations starting from an invariant measure are also known in $d\geq 3$ \cite{Chang}.

  \subsection{KPZ etc.}
In $d=1$, for the nearest-neighbor asymmetric process, including TASEP, there are many works describing the behaviors of currents, `height' functions, and fluctuation fields starting from an invariant measure and other initial conditions.
 Instead of a generalized OU process, under time scaling $v(N) = N^{3/2}$, space scaling $1/N$, and say scaling $1/\sqrt{N}$ of the `height' functions, the limits are different, those with respect to a `KPZ' (Kardar-Parisi-Zhang) fixed point (cf. \cite{MQR}).
 
 On the other hand, if the asymmetry is weak in the sense the drift of the transition probability is $O(1/\sqrt{N})$, under diffusive time/space scaling, the limit of certain translated fluctuation fields solves a `stochastic Burgers' or sometimes called `KPZ-Burgers' equation.  
 
 These limits connect with integrable probability, last passage percolation and the `Directed Landscape', polymers, singular KPZ SPDE, and random matrices.  See \cite{ACQ}, \cite{Baik}, \cite{BG}, \cite{Corwin}, \cite{Das}, \cite{DOV}, \cite{GJ}, \cite{GJS}, \cite{GK}, \cite{GPS}, \cite{GP}, \cite{GP1}, \cite{Hairer}, \cite{IGP}, \cite{Johansson}, \cite{Liu}, \cite{Quastel}, \cite{Spohnreview}, \cite{Yang}, \cite{Yang1} for reviews and discussion, among other references.
 
 In $d\geq 2$, the work \cite{CGT} on scaled continuum stochastic Burgers equation shows Gaussian fluctuations. In $d=2$, Gaussian behaviors have been shown with respect to scaled `subcritical' continuum KPZ equations \cite{CD}, \cite{CSZ}, \cite{G}.     In $d\geq 3$, Gaussian behaviors have been seen starting from asymmetric exclusion \cite{LOY1}.  As mentioned in \cite{CGT}, in $d=2$, it is open to see such behavior starting from asymmetric exclusion, or other particle systems.  

In $d=1$, the fluctuation fields of many-component systems of particles, and coupled SPDE's including coupled stochastic Burgers equations have been considered, although many open problems remain; see \cite{ABC}, \cite{BFS}, \cite{FFSV}, \cite{Schutz}, \cite{RDKKS}, \cite{RDKS}, \cite{SKP}, \cite{Spohn-coupled} and references therein.

\newpage
\chapter[Section $12$]{Boltzmann-Gibbs principle for symmetric zero-range processes}
\label{lec12}

We discuss the equilibrium fluctuations of $d$-dimensional symmetric finite-range zero-range processes, that is the scaling limit of the empirical mass fluctuation field, starting from an invariant measure $\nu_\rho$.  Unlike for symmetric exclusion models, a `replacement' at the fluctuation level must be made in order to identify the limit field as a generalized OU process.  This replacement, sometimes called the `Boltzmann-Gibbs' principle, is of its own interest.  The proof we give makes use of a `spectral gap' estimate.

\section{Statement of equilibrium fluctuations}
Recall, from Sections \ref{lec4} and \ref{lec7}, that the rate function $g:\N_0\rightarrow \R_+$ and the jump probability $p(\cdot)$ define the zero-range process with generator
$$Lf(\eta) = \sum_{x,y\in \Z^d}p(y)g(\eta(x))\big[f(\eta^{x,x+y})-f(\eta)\big].$$
 To reduce notation, we will assume that $p$ is symmetric, translation-invariant, and nearest-neighbor: $p(\pm e_i)= 1/(2d)$ for the standard basis $\{e_i\}_{i=1}^d$.  Also, we remind that $g$ satisfies $g(0)=0$, $g(k)>0$ for $k\geq 1$.  We will also impose the following.
\begin{itemize}
\item[(Lip)] 
$|g(k+1)-g(k)|\leq a_0$ for all $k\geq 0$.
\item[(M)] There exists $k_0$ and $\epsilon_0>0$ such that $g(k+k+0) - g(k)\geq \epsilon_0$ for all $k\geq 0$.
\end{itemize}
The first assumption is something we have already seen in the construction of the process on $\Z^d$ in Section \ref{lec7}, which includes the independent particle process when $g(k)\equiv k$.  The second assures a uniform `spectral gap' that will be discussed later in Subsection \ref{sec:specgap}.

We will also fix an invariant measure $\nu_\rho=\prod_{x\in \Z^d}\kappa_{\Psi(\rho)}$, a product of `Poisson' like marginals, and the process will be assumed to begin under this extremal invariant measure.  We will assume the process is an $L^2(\nu_\rho)$ process on $\Z^d$. Recall Section \ref{lec8} for specifications and details.
There is no difficulty in assuming if preferred that the process is on the torus $\T^d_N$ where there are no construction issues.
As before, $E_\mu$ denotes the probability and expectation under $\mu$, and $\P_\mu$ and $\E_\mu$ the process measure and expectation when starting in $\mu$.

Recall, as in Section \ref{lec11}, that $W^N_t(G)$ for fixed $G$ smooth with compact support stands for the fluctuation field,
$$W^N_t(G) \ = \ \frac{1}{N^{d/2}}\sum_x G(x/N)\big(\eta_{N^2t}(x) - \rho\big).$$
To derive the limit field, write as before
\begin{eqnarray*}
W^N_t(G) &=& W^N_0(G) + N^2\int_0^t LW^N_s(G) ds + \M^N_t(G)
\end{eqnarray*}
where, after the usual summation-by-parts,
$$LW^N_t(G) \ = \ \frac{1}{2N^{d/2}}\sum_{x,y\in \Z^d}p(y) \triangle^N_{x,y}G(x/N) \big(g\big(\eta_{N^2s}(x)\big)-\Psi(\rho)\big).$$
Here, $\Psi(a) = E_{\nu_a}[g(\eta(0))]$,
and
$\M^N_t(G)$ is a martingale such that
$${\mc N}^{N}_t(G)\ = \ (\M^N_t(G))^2 - N^2\int_0^t L(Y^N_s(G))^2 - 2Y^N_s(G)LY^N_s(G)ds$$
is a martingale.
The last integral in the display, equal to $\langle \M^N(G)\rangle_t$, can be evaluated as
\begin{eqnarray*}
&&N^2\int_0^t \Big[L(Y^N_s(G))^2 - 2Y^N_s(G)LY^N_s(G)\Big]ds \\
&&\ \ \  = \ \frac{1}{N^{d}}\int_0^t \sum_{x,y\in \Z^d}p(y) [\nabla^N_{x,y}G(x/N)]^2g\big(\eta_{N^2s}(x)\big) ds.\end{eqnarray*}
These calculations are analogous to those for hydrodynamics in Sections \ref{lec4}, \ref{lec5}.

\medskip
Following the method described for exclusion processes in Section \ref{lec11}, we have two steps:
\medskip

Step 1:  Show tightness of $\{W^N_t: t\in [0,T]\}$ in an appropriate space, and continuity of limit trajectories under limit points.

Step 2:  Identify the limit points in terms of a unique `infinite dimensional Ornstein-Uhlenbeck' process 
\medskip

As before, the space of trajectories is $D([0,T], \mc D')$, where $\mc D'$ is the dual space of distributions with respect to $\mc D= C^\infty_c(\R^d)$.  
Step 1, tightness, is accomplished as before, and it is left to the reader to verify the proof in the zero-range setting.

The more interesting part, for zero-range processes, is Step 2 where instead of the occupation variable $\eta(x)$ we have a function of it, namely $g(\eta(x))$ in both martingales above.  If the normalization were $N^{-d}$ instead of $N^{-d/2}$, replacing the nonlinear function with a function of the mass empirical density is the `standard' hydrodynamic replacement.  Here, we have to work a little harder.  However, starting in equilibrium $\nu_\rho$ helps.

The following `Boltzmann-Gibbs' estimate allows the replacement in $\M^N_t(G)$.

\begin{theorem}
\label{BG}
For smooth, compactly supported $G$, we have
\begin{align*}
&\lim_{N\uparrow\infty}  \E_{\nu_\rho}\Big[ \Big| \int_0^t\frac{1}{N^{d/2}}\sum_{x,y\in \Z^d} p(y)\triangle^N_{x,y}G(x/N) \\
&\quad\quad\quad \cdot \Big(g\big(\eta_{N^2s}(x)\big) - \Psi(\rho) - \Psi'(\rho)\big(\eta_{N^2s}(x)-\rho\big)\Big)]ds\Big|^2\Big] \ = \ 0.
\end{align*}
\end{theorem}

The replacement, however, in the square martingale ${\mc N}^N_t(G)$ will be a consequence of the following limit.
\begin{equation}
\label{ergodic}
\lim_{N\uparrow\infty} \E_{\nu_\rho}\Big[ \Big|\int_0^t\frac{1}{N^d}\sum_{x,y}p(y) [\nabla^N_{x,y}G(x/N)]^2\Big[g\big(\eta_{N^2s}(x)\big) - \Psi(\rho)\Big] ds\Big|^2\Big] \ = \ 0.
\end{equation}

Hence, we arrive at the following result, specializing to the nearest-neighbor setting, when $p(\pm e_i)=1/(2d)$ with respect to the standard basis $\{e_i\}_{i=1}^d$.

\begin{theorem}
We have that $W^N_t$ converges to $W_t$ where
\begin{equation}
\label{infinity_eqn-zr}
dW_t \ = \ \frac{\Psi'(\rho)}{2d}\triangle 
W_t dt + \sqrt{d^{-1}\Psi(\rho)}\nabla d\B_t.\end{equation}
\end{theorem}

The characterization of $W_t$ in terms of the generalized OU martingale problem of Holley and Stroock is as before with symmetric exclusion in Section \ref{lec11}.  Though, the operators $U = (\Psi'(\rho)/(2d))\triangle$ and $B = I\Psi(\rho)/d)\nabla$ differ in prefactor constants.

Analogous to before with symmetric exclusion,
$$E_{\nu_\rho}[W_0(G)W_0(H)] \ = \ \sigma^2\langle G, H\rangle$$
where $\sigma^2(\rho) = E_{\nu_\rho}[(\eta(0)-\rho)^2]$.  Also, the covariance of $W_t(G)$ and $W_s(H)$ satisfies the formula \eqref{eq:W-covariance} where $T_t$ is the semigroup associated to $(\Psi'(\rho)/(2d))\triangle$.

One way to look at \eqref{infinity_eqn-zr} is to relate it to the hydrodynamic equation:
$$\partial_t \rho \ = \ \frac{1}{2d}\triangle \Psi(\rho)$$
with initial condition $\rho(0,x) = \rho_0(x)$.  The quantity $W^N_t$ can be seen as an `error' via a `linearization' of the hydrodynamic equation about the equilibrium density $\rho$ (the constant solution when starting from $\nu_\rho$), where $\Psi(\rho(t,x)) \sim \Psi(\rho) + \Psi'(\rho)\rho(t,x)$.  See \cite{Spohn} for more physical intuition behind this intepretation.

We remark in passing, in the case $g(k)\equiv k$, the setting of independent particles, the replacement Theorem \ref{BG} is not needed as, analogous to symmetric exclusion, the martingale $\mc M^N_t(G)$ is already `closed' with respect to $W^N_t$, that is fully expressed in terms of $W^N_t(G)$ and $W^N_t(\triangle^N_{x,y}G)$.

\section{Kipnis-Varadhan estimate}

We begin with the following `Kipnis-Varadhan' \cite{KV} non-asymptotic bound of independent interest, helpful to establish the Boltzmann-Gibbs principle.  Recall the notions of $H_1$ and $H_{-1}$ norms from Section 
\ref{lec9}.  These extend to the zero-range context, where the local functions used to define spaces $H_{1}$ and $H_{-1}$ are local $L^2(\nu_\rho)$ functions.

\begin{lemma}
\label{H-1bound}
For all local $L^2(\nu_\rho)$ functions, we have for the symmetric zero-range process
$$\E_{\nu_\rho}\Big[\Big(\int_0^t f(\eta_s) ds \Big)^2\Big] \ \leq \ 12t \|f\|_{-1}^2.$$
\end{lemma}
Note that it may be that $\|f\|_{-1}$ diverges for a given $f\in L^1(\nu_\rho)$.

\begin{proof}
Write the resolvent equation, for $\lambda>0$,
$$\lambda u_\lambda - Lu_\lambda \ = \ f.$$
Multiplying by $u_\lambda$ and integrating, we have
$$\lambda \|u_\lambda\|^2_0 + \|u_\lambda\|^2_1 \ = \ \langle f, u_\lambda\rangle_{\nu_\rho},$$
where $\langle f,h\rangle_{\nu_\rho} = E_{\nu_\rho}[fg]$.  Note that $\langle f, u_\lambda\rangle_{\nu_\rho}\leq \|f\|_{-1}\|u_\lambda\|_{1}$, and so $\|u_\lambda\|_1\leq \|f\|_{-1}$.
Now, consider the martingale
$$M_\lambda(t) \ = \ u_\lambda(\eta_t) - u_\lambda(\eta_0) - \int_0^t Lu_\lambda(\eta_s)ds,$$
with quadratic variation
$\langle M_\lambda \rangle = \int_0^t\big[Lu_\lambda^2 - 2u_\lambda Lu_\lambda\big]ds$.

Write
$$\int_0^t f(\eta_s)ds \ = \ M_\lambda(t) - \lambda\int_0^t u_\lambda(\eta_s)ds +u_\lambda(\eta_0) - u_\lambda(\eta_t).$$
Since $\E_{\nu_\rho}\big[M_\lambda^2(t)\big] = 2t\|u_\lambda\|_1^2$, from squaring the left hand side of the above display, using $(a+b+c+d)^2 \leq 4(a^2+b^2+c^2+d^2)$, by stationarity, we have
$$\E_{\nu_\rho}\Big[\Big(\int_0^t f(\eta_s) ds \Big)^2\Big] \ \leq \ 4\Big\{ 2t\|u_\lambda\|_1^2 + \lambda^2 t^2 \|u_\lambda\|^2_0 + 2\|u_\lambda\|^2_0\Big\}.$$
Choosing $\lambda = t^{-1}$, we see that the left hand side is bounded by 
\[12t\langle f, u_\lambda\rangle_{\nu_\rho} 
 \ \leq \ 12t\|f\|^2_{-1}. \qedhere\]  \end{proof}

\begin{remark}
\label{rem:non-asymptotic}
\rm
We comment Lemma \ref{H-1bound} also holds for asymmetric zero-range processes, among others, where the $H_1$ and $H_{-1}$ norms are with respect to the symmetrized generator $S = (L+L^*)/2$.  One can also put a `$\sup_{0\leq t\leq T}$' inside the $\E_{\nu_\rho}$-expectation, via a martingale argument.  See \cite{Ko_La_Ol}, \cite{S_comp} for these generalizations.
\end{remark}

\section{Derivation of Boltzmann-Gibbs estimate}

We need to show the variance of 
$$\int_0^t\frac{1}{N^{d/2}}\sum_{x,y\in \Z^d} p(y)\triangle^N_{x,y}G(x/N) \Big(g\big(\eta_{N^2s}(x)\big) - \Psi(\rho) - \Psi'(\rho)\big(\eta_{N^2s}(x)-\rho\big)\Big)ds$$
vanishes in the $N\uparrow\infty$ limit. 
 It will be enough to bound the variance of 
\begin{align}
\label{eq:BG-H}
\int_0^t\frac{1}{N^{d/2}}\sum_{x\in \Z^d} H(x/N) \Big(g\big(\eta_{N^2s}(x)\big) - \Psi(\rho) - \Psi'(\rho)\big(\eta_{N^2s}(x)-\rho\big)\Big)ds
\end{align}
for $H\in \mc D$, from which Theorem \ref{BG} can be deduced.

Write \eqref{eq:BG-H} as
\begin{align}
\label{eq:BG-2}
&\int_0^t \sum_x N^{-d/2}H(x/N) \Big\{g(\eta_{N^2s}(x)) 
-\E_{\nu_\rho}\big[g(\eta_{N^2s}(x))\big | \sum_{y\in B_{\ell,x}}\eta_{N^2s}(y)\big]\Big\} ds\nonumber\\
&\ \ \ \ + \ \int_0^t\sum_x N^{-d/2}H(x/N) \Big\{\E_{\nu_\rho}\big[g(\eta_{N^2s}(x))\big | \sum_{y\in B_{\ell,x}}\eta_{N^2s}(y)\big] \\
&\ \ \ \ \ \ \ \ \ \ \ \ \ \ \ \ \  - \Psi(\rho)- \Psi'(\rho)\big(\eta_{N^2s}(x)-\rho\big)\Big\} ds \ = \ A_1 + A_2.\nonumber
\end{align}
Here, $B_{\ell,x}$ is a block of width $\ell\geq 2$ centered at $x$, and  $\ell\ll N$ is another scaling parameter.  Related to the study of the hydrodynamic limit, the idea is that $A_1$ considers the approximation of $g(\eta_{N^2s}(x)$ by its conditional expectation average with respect to a local density.  The term $A_1$ gives the error with respect to the leading order terms in this conditional expectation.  

The strategy will be to bound the $H_{-1}$ norm of $A_1$, and to use Schwarz inequality and Taylor expansions with $A_2$, a sort of `equivalence of ensembles' estimate.  These are done in the next two subsections.  In Subsection \ref{sec:proofofBG}, we assemble these bounds to prove Theorem \ref{BG}.

\subsection{Bound on $A_1$}
To bound the variance of $A_1$, we will bound the $H_{-1}$ norm of its integrand, denoted $\mf A_1$.  Since time has been sped up by $v(N)=N^2$, the generator of $\eta_{N^2s}$ is $N^2L$.  For local $L^2(\nu_\rho)$ functions $\phi$, we would like to demonstrate
$$\big\langle \mf A_1, \phi\big\rangle_{\nu_\rho} \leq {\mf Y_1(N)} \big(N^2D(\phi)\big)^{1/2},$$
where the Dirichlet form
$N^2D(\phi) = \big\langle \phi, (-N^2L)\phi\big\rangle_{\nu_\rho}$ and $\mf Y_1(N)$ is a bound of the maximum ratio $\big\langle \mf A_1,\phi\rangle_{\nu_\rho}/\big(N^2D(\phi)\big)^{1/2}$.  Then, the $H_{-1}$ norm of $\mf A_1$ is bounded by $\mf Y_1(N)$, which we show vanishes as $N\uparrow\infty$.

To this end,  write
$$
\big\langle N^{-d/2}H(x/N) \tau_x V(\eta),\phi\big\rangle_{\nu_\rho} \ = \ \big\langle N^{-d/2}H(x/N) \tau_x V(\eta),\phi_\ell\big\rangle_{\nu^\ell_\rho},$$
given $V(\eta) = g(\eta(0)) - E_{\nu_\rho}[g(\eta(0))|\sum_{y\in B_{\ell, 0}}\eta(y)]$ depends only on variables indexed by $B_{\ell,0}$, where $\tau_x$ is the shift operator,  $\nu^\ell_\rho =\prod_{y\in B_{\ell,x}}\kappa_{\Psi(\rho)}$ is the restriction, and $\phi_\ell = E_{\nu_\rho}\big[\phi|\{\eta(y): y\in B_{\ell,x}\}\big]$.  We comment, as $\nu_\rho$ is a product measure, that $V(\eta) = g(\eta(0)) - E_{\nu^\ell_\rho}[g(\eta(0))|\sum_{y\in B_{\ell, 0}}\eta(y)]$.

We may further evaluate the right-hand side as
$$
\sum_{k\geq 0} P(k,\ell) \big\langle N^{-d/2}H(x/N) \tau_xV(\eta), \phi_\ell \big\rangle_{\nu_{k, \ell}}
$$
where $\nu_{k, \ell}$ is the `canonical' measure, that is $\nu^\ell_\rho$ conditioned
on there being $k$ particles in $B_{\ell, x}$, and $P(k,\ell) = P_{\nu_\rho}\big(\sum_{y\in B_{\ell,x}}\eta(y) = k\big)$.

Recall $\eta^\ell(x) = (2\ell+1)^{-d}\sum_{y\in B_{\ell, x}}\eta(y)$ (cf. \eqref{eq:eta-l} from Section \ref{lec4}).  Underlying the above localization, that is for adding and subtracting the conditional expectation $E_{\nu_\rho}[g(\eta(x))|\eta^\ell(x)]$, is that we can now solve a certain `Poisson' equation, no matter the particle number $k$, since $E_{\nu_{k, \ell}}[\tau_x V]=0$ for all $k$.

Indeed, note that $\nu_{k,\ell}$ is the unique invariant measure for the irreducible zero-range dynamics localized to $B_{\ell,x}$ with generator
$$L_{x, \ell}f(\eta) = \sum_{z, z+y\in B_{\ell, x}}p(y)g(\eta(z))\big[f(\eta^{z, z+y})-f(\eta)\big],$$
 and $\tau_xV$ is mean-zero with respect to $\nu_{k, \ell}$, which assigns probabilities to configurations in $\{0, 1, \ldots, k\}^{B_{\ell, x}}$.  Moreover, one may verify that $\nu_{k, \ell}$ is reversible with respect to $L_{x, \ell}$.  Also, the Dirichlet form with respect to $-L_{x, \ell}$ may be computed,
 \begin{align*}
 D_{x, \ell}(f;\nu_{k,\ell}) &= \big\langle f, (-L_{x, \ell})f\big\rangle_{\nu_{k, \ell}} \\
 &= \sum_{z, z+y\in B_{\ell, x}}p(y)E_{\nu_{k, \ell}}\Big[g(\eta(z))\big(f(\eta^{z, z+y})-f(\eta)\big)^2\Big].
 \end{align*}
 So, in particular, $-L_{x, \ell}$ is a nonnegative, reversible operator. 
 
\begin{exercise}\rm
Verify the reversibility of $\nu_{k,\ell}$ and the form of the Dirichlet form $D_{x,\ell}$. 
\end{exercise}
 
Now, thinking of $\nu_{k, \ell}$ as a vector in $\R^{|B_{\ell, x}|(k+1)}$, since it is in null space of the transpose $N^2 L^{T}_{x, \ell}$ as it is a stationary measure, and $E_{\nu_{k, \ell}}[\tau_x V]=0$ so that $\tau_x V$ is orthogonal to $\nu_{k, \ell}$, we conclude that $\tau_x V$ belongs to the range of $N^2 L_{x,\ell}$.  See Subsection \ref{sec:9.1} for related decompositions.

Hence, we may solve
$\tau_x V \ = \ -N^2L_{x,\ell} u$
for some function $u$.  
In particular, as $-L_{x,\ell}$ is nonnegative, and reversible with respect to $\nu_{k, \ell}$,
\begin{align*}
&\big\langle N^{-d/2}H(x/N) \tau_x V(\eta), \phi_\ell \big\rangle_{\nu_{k, \ell}}  =  \big\langle N^{-d/2}H(x/N) \big(-N^2L_{x,\ell}\big)u, \phi_\ell\big \rangle_{\nu_{k,\ell}} \\
&\quad \quad =\big\langle N^{-d/2}H(x/N)\big(-N^2L_{x,\ell}\big)^{1/2}u, \big(-N^2L_{x, \ell}\big)^{1/2}\phi_\ell\big\rangle_{\nu_{k, \ell}}\\
&\quad\quad =\big\langle N^{-d/2}H(x/N) \big(-N^2L_{\ell,k}\big)^{-1/2}\tau_x V, \big(-N^2L_{\ell,k}\big)^{1/2}\phi_\ell \big\rangle_{\nu_{k,\ell}}.
\end{align*}

At this point, we state a spectral gap inequality to be discussed later.  Namely,
for mean-zero functions, we have that
$$\big\|\big(-L_{x,\ell}\big)^{-1/2} \tau_x V\big\|^2_{L^2(\nu_{k,\ell})} \ \leq \ {\mf M}(\ell, k)\big\|\tau_xV\big\|^2_{L^2(\nu_{\ell,k})}$$
where a bound of the inverse of the `spectral gap' ${\mf M}(\ell, k) \leq C_0\ell^2$ is independent of $k$ when (Lip) and (M) hold for the rate $g$ \cite{LSV}.
    
Hence, by the relation $2ab =\inf_{\epsilon>0}\big\{ \epsilon a^2 + \epsilon^{-1}b^2\big\}$ for $a,b>0$, we have
\begin{align*}
&\big\langle N^{-d/2}H(x/N) \tau_xV(\eta), \phi_\ell \big\rangle_{\nu_{k, \ell}} \\
&\quad\quad  \leq \ \frac{\epsilon}{2} N^{-d}H^2(x/N) N^{-2} C_0^2\ell^2\big\|\tau_xV\big\|^2_{L^2(\nu_{k,\ell}} + \frac{\epsilon^{-1}N^2}{2}D_{x,\ell}(\phi_\ell; \nu_{k,\ell}).
\end{align*}

Note, by unraveling the measures, that 
\begin{align*}
\sum_{k\geq 0} P(k,\ell) D_{x,\ell}(\phi_\ell; \nu_{k,\ell}) & =  \sum_{z,z+y\in B_{\ell,x}} p(y)E_{\nu_\rho}\Big[g(\eta(z))\big(\phi(\eta^{z,z+y}) - \phi(\eta)\big)^2\Big] \\
& :=  D_{x,\ell}(\phi;\nu_\rho).
\end{align*}
Also, by estimating the overcount of terms over bonds $(z, z+y)$, we have
$$\sum_x D_{x,\ell}(\phi; \nu_\rho) \ \leq \ C_1\ell^d D(\phi),$$
where the full Dirichlet form is evaluated
$$D(\phi) = \sum_{z, y\in \Z^d}p(y) E_{\nu_\rho}\Big[g(\eta(z))\big(\phi(\eta^{z, z+y}) - \phi(\eta)\big)^2\Big].
$$
In addition,
$\sum_{k\geq 0} P(k,\ell) \big\|\tau_x V\big\|^2_{L^2(\nu_{k,\ell})} \ = \ \big\|V\big\|^2_{L^2(\nu_\rho)}$.

Then, coming back to \eqref{eq:BG-2}, summing over $x$, 
we have
\begin{align*}
\big\langle \mf A_1, \phi\big\rangle_{\nu_\rho} &= \sum_{x\in \Z^d}\sum_{k\geq 0}
P(k,\ell)\big\langle N^{-d/2}H(x/N)\tau_x V, \phi_\ell\big\rangle_{\nu_{k,\ell}}\\
&\leq \frac{\epsilon}{2N^d}\frac{C^2\ell^2}{N^2} \sum_x H^2(x/N)\sum_{k\geq 0}P(k,\ell)\big\|\tau_x V\big\|^2_{\nu_{k,\ell}}\\
&\quad\quad + \frac{\epsilon^{-1}N^2}{2}\sum_x\sum_{k\geq 0}P(k,\ell)D_{x,\ell}(\phi_\ell; \nu_{k,\ell})\\
&\leq \frac{\epsilon}{2N^d}\frac{C^2\ell^2}{N^2} \sum_x H^2(x/N)\big\|\tau_x V\big\|^2_{\nu_\rho} + \frac{\epsilon^{-1}N^2}{2}\cdot C\ell^d D(\phi)\\
&\leq  C\|V\|_{L^2(\nu_\rho)} \|H\|_{L^2(\R^d)}\frac{\ell^{d/2 +1}}{N}\big(N^2D(\phi)\big)^{1/2},
\end{align*}
after optimizing on $\epsilon>0$.
Here, we used translation-invariance to yield $\|\tau_x V\|_{\nu_\rho} = \|V\|_{\nu_\rho}$ and the constant $C$ may have changed line to line.

Hence, the $H_{-1}$ norm of $\mf A_1$ is bounded 
$$\mf Y_1(N)= C\|V\|_{L^2(\nu_\rho)} \|H\|_{L^2(\R^d)}\frac{\ell^{d/2 +1}}{N}=O(\ell^{d/2+1}/N),$$ and therefore the variance of $A_1$ by Lemma \ref{H-1bound} is bounded
 \begin{align}
 \label{eq:A1bound}
 \E_{\nu_\rho}\big[(A_1)^2\big] \leq Ct\|V\|^2_{L^2(\nu_\rho)}\|H\|^2_{L^2(\nu_\rho)}\frac{\ell^{d+2}}{N^2}.\rightarrow 0,
 \end{align}
   as $N\uparrow\infty$ for each $\ell$ fixed.

\subsection{Bound on $A_2$}
Since $H$ is smooth, we may replace $\eta(x)$ with 
$\eta^\ell(x)$
in $A_2$.  By invariance of $\nu_\rho$, adding and subtracting $\rho$, and Schwarz inequality,
\begin{eqnarray*}
&&\lim_{N\uparrow\infty}
\E_{\nu_\rho}\Big[ \Big|\int_0^t \frac{1}{N^{d/2}}\sum_x H(x/N)\Big[\eta_{N^2s}(x)- \eta^\ell_{N^2s}(x)\Big] ds \Big|^2\Big]\\
&& = \ \lim_{N\uparrow\infty} \E_{\nu_\rho}\Big[ \Big|\int_0^t \frac{1}{N^{d/2}}\sum_x\\
&&\quad\quad \quad  \frac{1}{(2\ell+1)^d}\sum_{y\in B_{\ell,x}}\Big[H(x/N)-H((y-x)/N)\Big](\eta_{N^2s}(x)- \rho) ds \Big|^2\Big]\\
&& \leq \ C(\rho, H,t)\ell^2 N^{-2} \ \rightarrow \ 0.
\end{eqnarray*}

Then, by Schwarz inequality, using invariance and translation-invariance of $\nu_\rho$, 
we have
\begin{align}
\label{A_2_1}
&\E_{\nu_\rho}\Big[\Big|\int_0^t\sum_x N^{-d/2}H(x/N) \\
&\quad\quad\quad \times \Big\{\E_{\nu_\rho}\big[g(\eta_{N^2s}(x))\big | \sum_{y\in B_{\ell,x}}\eta_{N^2s}(y)\big] - \Psi\rho)- \Psi'(\rho)\big(\eta_{N^2s}^\ell(x)-\rho\big)\Big\} ds\Big|^2\Big]\nonumber\\
&\ \leq  \  Ct^2\ell^d\Big(\frac{1}{N^d}\sum_x H^2(x/N)\Big)\nonumber \\
&\quad\quad \quad \times E_{\nu_\rho}\Big[ \Big(E_{\nu_\rho}[g(\eta(0))| \sum_{y\in B_{\ell,0}}\eta(y)] - \Psi(\rho)- \Psi'(\rho)\big(\eta^\ell(0)-\rho\big)\Big)^2\Big]\nonumber\\
&\ \leq \ C(t, H) \ell^d E_{\nu_\rho}\Big[ \Big(E_{\nu_\rho}[g(\eta(0))| \sum_{y\in B_{\ell,0}}\eta(y)] - \Psi(\rho)- \Psi'(\rho)\big(\eta^\ell(0)-\rho\big)\Big)^2\Big]. \nonumber
\end{align}
 The factor $\ell^d$ arises since the summands over $x$ are independent only when they are separated by distance $2\ell +1$.

We now estimate the mean-square of
$$ E_{\nu_\rho}\Big[g(\eta(0))- \Psi(\rho)- \Psi'(\rho)\big(\eta^\ell(0)-\rho\big)\Big| \sum_{y\in B_{\ell,0}}\eta(y)\Big].$$

\subsubsection{Outside truncation} We first truncate the number of particles in $B_{\ell,0}$.  Let $K=\rho/2$.  Recall $g(\eta(0))\leq a_0\eta(0)$.  Then, by Schwarz inequality,
\begin{align}
\label{eq:outside}
&E_{\nu_\rho}\Big[\Big(E_{\nu_\rho}\Big[g(\eta(0))- \Psi(\rho)- \Psi'(\rho)\big(\eta^\ell(0)-\rho\big)\Big| \sum_{y\in B_{\ell,0}}\eta(y)\Big]\Big)^21\big(|\eta^\ell(0)-\rho |>K)\Big]\nonumber\\
&\quad \leq C(a_0, \rho)P_{\nu_\rho}\big(\eta^\ell(0)>K\big)^{1/2}\  \leq \ \frac{C(a_0, \rho)}{\ell^{2d} K^4}.
\end{align}
Here, we used Chebychev to bound $P_{\nu_\rho}\big(|\eta^\ell(0)-\rho|>K\big)\leq E_{\nu_\rho}[(\eta^\ell(0)-\rho)^8]/K^8$.

Inserting into \eqref{A_2_1}, we see that the multiplying $\ell^d$ factor is compensated, and the cost of truncation is small, uniformly over $N$, for large $\ell$.

\medskip

\subsubsection{Inside truncation} 
We now use a Taylor expansion and a local central limit theorem.
For $M = \sum_{y\in B_{\ell,0}}\eta(y)$, let $M/(2\ell +1)^d:=\bar M_\ell$.  Given the truncation, we have that $\rho/2\leq \bar M_\ell\leq 3\rho/2$ is bounded away from $0$ and $\infty$ uniformly in $\ell$.

Since $\nu_\rho$ is a product measure with marginals in the form given in Subsection \ref{sec:zero-rangemodels}, the conditional expectation
$E_{\nu_\rho}[h|\sum_{y\in B_{\ell, 0}}\eta(y)=M]$ does not depend on the density $\rho$ and can be changed to another, say $\bar M_\ell$; see near \eqref{1-block-final} for a similar discussion with respect to the $1$-block hydrodynamics lemma.
Recall the calculation in Exercise \ref{ex:zr-particle},
$$E_{\nu_a}\big[g(\eta(0) h(\eta)] = \Psi(a)E_{\nu_a}[h(\eta + \delta_0)]$$
where $\delta_0$ is the configuration with a single particle at $0$.
Then,
\begin{align}
\label{eq:Lan}
E_{\nu_\rho}\big[g(\eta(0)\big | \sum_{y\in B_{\ell,0}}\eta(y)=M\big]&= 
E_{\nu_{\bar M_\ell}}\big[g(\eta(0)\big | \sum_{y\in B_{\ell,0}}\eta(y)=M\big] \\
& =  \Psi(\bar M_\ell)\frac{P_{\nu_{\bar M_\ell}}[\sum_{y\in B_{\ell,0}}\eta(y) = M-1]}{P_{\nu_{\bar M_\ell}}[\sum_{y\in B_{\ell,0}}\eta(y) = M]}.\nonumber
\end{align}

Recall $\sigma^2(\bar M_\ell) = E_{\nu_{\bar M_\ell}}\big[(\eta(0) - \bar M_\ell)^2\big]$. 
By the local central limit Theorem VII.13 in \cite{Petrov} (applied with $k=3$ there), when $|S-M|<\infty$ and $\bar M_\ell$ is bounded uniformly away from $0$ and $\infty$, we have for $z = (S-M)/(\kappa(2\ell+1)^{d/2})$ and large $\ell$,
\begin{align*}
\Big|\sigma(\bar M_\ell)(2\ell + 1)^{d/2}P_{\nu_{\bar M_\ell}}[\sum_{y\in B_{\ell,0}}\eta(y) = S]  - \frac{1}{\sqrt{2\pi}} e^{-\frac{z^2}{2}} 
\Big| \ = \ o(\ell^{-d/2}).
\end{align*}
For $S=M-1$, we see $z= -\kappa^{-1} (2\ell+1)^{-d/2}$, whereas for $S=M$, we have $z=0$, in which case $\sigma(\bar M_\ell)(2\ell+1)^{d/2}P_{\nu_{\bar M_\ell}}\big(\sum_{y\in B_{\ell, 0}}\eta(y) = M\big) \sim (2\pi)^{-1/2}$ for large $\ell$.

Hence, multiplying and dividing \eqref{eq:Lan} by $\sigma(\bar M_\ell)(2\ell+1)^{d/2}$, on the set $\big\{\rho/2\leq \eta^\ell(0)=\bar M_\ell\leq 3\rho/2\big\}$, we have
\begin{eqnarray*}
E_{\nu_{\bar M_\ell}}\big[g(\eta(0))\big | \sum_{y\in B_{\ell,0}}\eta(y) = M\big] & = & \Psi(\bar M_\ell)\big(1+o(\ell^{-d/2}) \big).
\end{eqnarray*}
Then, 
\begin{align*}
&\Big|E_{\nu_\rho}\Big[g(\eta(0))- \Psi(\rho)- \Psi'(\rho)\big(\eta^\ell(0)-\rho\big)\Big| \sum_{y\in B_{\ell,0}}\eta(y)\Big]\Big|^2 1\big(|\eta^\ell(0)-\rho|\leq \rho/2\big)\\
&\quad  = \Big|\Psi(\bar M_\ell) -\Psi(\rho)- \Psi'(\rho)\big[\eta^\ell(0) - \rho\big) + o(\ell^{-d/2})\Big|^2 1\big(|\eta^\ell(0)-\rho|\leq \rho/2\big)\\
&\quad \leq \sup_{|a-\rho|\leq \rho/2}|\Psi'(a)|^2 \big[\eta^\ell(0) -\rho\big]^4 + o(\ell^{-d}),
\end{align*}
from Taylor approximation.

Since $E_{\nu_\rho}\big[(\eta^\ell(0)-\rho)^4\big] = O(\ell^{-2d})$, we observe
\begin{align}
\label{eq:inside}
&E_{\nu_\rho}\Big[\Big|E_{\nu_\rho}\Big[g(\eta(0))- \Psi(\rho)- \Psi'(\rho)\big(\eta^\ell(0)-\rho\big)\Big| \sum_{y\in B_{\ell,0}}\eta(y)\Big]\Big|^2 1\big(|\eta^\ell(0)-\rho|\leq \rho/2\big)\Big]\nonumber\\
& \quad = o(\ell^{-d}).
\end{align}
 Hence, with respect to \eqref{A_2_1}, the quantity \eqref{eq:inside} multiplied by $\ell^d$ vanishes, uniformly in $N$. 
 \medskip
 
 Finally, with estimates \eqref{eq:outside} and \eqref{eq:inside}, inserting into \eqref{A_2_1}, we have
\begin{align}
\label{eq:A2bound}
\lim_{\ell\uparrow\infty}\lim_{N\uparrow\infty}\E_{\nu_\rho}\big[(A_2)^2\big] = 0.
\end{align}

\subsection{Proof of Theorem \ref{BG}}  
\label{sec:proofofBG}
We need only show that $\E_{\nu_\rho}\big[(A_i)^2\big]$, for $i=1,2$, vanish as $N\uparrow\infty$ and $\ell\uparrow\infty$, where $A_1, A_2$ are given in \eqref{eq:BG-2}.  These limits are furnished in \eqref{eq:A1bound} and \eqref{eq:A2bound}, concluding the argument.\qed

\section{Spectral gap}
\label{sec:specgap}
One can define the `spectral gap' for a reversible process generator $L_{k,\ell}$ in terms of a Poincare inequality:
$${\rm Var}_{\nu_{k,\ell}}[f] \ \leq \ {\mf M}(k,\ell) E_{\nu_{k,\ell}}\big[f(-L_{k,\ell})f\big].$$
Then, the gap would be the reciprocal of the smallest ${\mf M}(k,\ell)$ for which the inequality is true for all $f$.

Equivalently, the gap is the difference between the two largest eigenvalues of $-L_{k,\ell}$.  Since $0$ is the largest eigenvalue, the gap would be the negative of the second largest eigenvalue, $-\lambda_2$.  Since $L_{k,\ell}$ is a matrix, one can in principle compute it exactly, but this is usually hard to do when the state space gets large.  

The spectral gap has connections with the mixing time of the process.  Recall that 
$${\rm Var}_{\nu_{k,\ell}}(P_t f) \ \leq \ Ce^{\lambda_2 t}.$$  The larger $-\lambda_2$, the faster the approach of $P_t f$ to its mean $E_{\nu_{k,\ell}}[f]$.

In zero-range processes, since the jump times are controlled by the rate $g$, one expects the orders of the mixing time to depend on $g$.  This is in fact the case.  When $g(k)=k$, the independent case, the mixing is rapid, and corresponds to the mixing time of a single random walk on $B_{\ell,0}$ which is $O(\ell^{-2})$.  For instance, consider $d=1$, then a nearest-neighbor walk, in moving from one end of the interval to the other, has to take $\ell$ steps which usually takes $\ell^2$ units of time.  

If $g$ is not too different from the independent case, one expects similar behavior, not dependent on the number of particles, and this is the result quoted and used above, valid under our assumptions (Lip) and (M) \cite{LSV}.

If $g$ is sublinear, that is $g(k) = k^\gamma$ for $0<\gamma<1$, then it has been shown the spectral gap depends on the number of particles $k$, namely the gap is $O((1+\rho)^{-\gamma} \ell^{-2})$ where $\rho = k/(2\ell +1)^d$ \cite{Nagahata}.  For instance, if $\rho\uparrow \infty$, movement becomes difficult, and the gap vanishes!

If $g(k) = 1(k\geq 1)$, then it can be seen that the gap is $O((1+\rho)^{-2}\ell^{-2})$ \cite{Morris}.  When $d=1$, this is the case where there is a connection between the zero-range process and simple exclusion:  The number of spaces between particles in simple exclusion correspond to the number of particles in the zero-range process. 

It seems to be an open problem to characterize the gap when $g$ is bounded in general. One presumably expects the behavior as in the last case.

\section{Notes}

The Boltzmann-Gibbs principle, under a stationary measure, was named and proved first by Brox and Rost \cite{BR} which did not use a spectral gap assumption. See also \cite{KL} for a presentation along this line.  The notion of using the spectral gap estimate to prove the Boltzmann-Gibbs principle seems to be newer; see also \cite{GJ},  \cite{GJS}.

While equilibrium fluctuations of particle systems with finite-range symmetric interactions is relatively well understood (cf. \cite{KL}, \cite{Ko_La_Ol}, \cite{Spohn}), a main open problem is to understand `nonequilibrium' fluctuations of the density field in a general class of particle systems.  A main difficulty is to perform the needed `replacement' estimates in the fluctuation scales.  Results in dimension $d=1$ are known for finite-range symmetric models; see \cite{CY}, \cite{Ferrari}.  Recently, these results have been generalized to $d\leq 3$, for a class of particle systems \cite{Jara-Menezes}; see also \cite{Dagalier-Landim},  \cite{FLS}, \cite{JL-fluc}, and discussions therein.

There are many works on mixing time, spectral gap and related log-Sobolev estimates.  See for instance \cite{BDagT}, \cite{Caputo-Japan}, \cite{CLR}, \cite{Caputo}, \cite{Cancrini}, \cite{DaiPrai}, \cite{Diaconis}, \cite{Gantert}, \cite{Hermon}, \cite{Labbe}, \cite{Rameh}, \cite{Salez} and references therein.

Beyond fluctuations, other scalings and limit fields have been considered.  We mention in passing, cogent to the present narrative, these include the `Navier-Stokes' corrections \cite{LOY}, and large deviations of the empirical density and ,currents  \cite{Bertini-etal}, \cite{Bertini1}, \cite{Bodineau}, \cite{KOV}, \cite{Mariani}, \cite{QT}, \cite{Vjapan}.

\end{document}